\documentclass[10pt]{amsart}
\usepackage{a4wide,color}
\usepackage{amsmath,amssymb}
\usepackage[active]{srcltx}
\usepackage{ulem}
\usepackage{enumitem}
\usepackage{bm}
\usepackage{graphicx}
\usepackage[dvipsnames]{xcolor}
\usepackage{subfigure}

\allowdisplaybreaks

\let \eps \varepsilon

\newcommand{\R}{{\mathbb R}}

\newtheorem{theorem}{Theorem}
\newtheorem{lemma}[theorem]{Lemma}
\newtheorem{proposition}[theorem]{Proposition}

\newtheorem{corollary}[theorem]{Corollary}

\let \de=\delta
\let \eps=\varepsilon

\let \la=\lambda

\let \ka=\kappa

\let \om=\omega

\definecolor{mypink1}{RGB}{114,110,161}
\definecolor{mypink2}{RGB}{124,115,196}
\definecolor{mypink3}{RGB}{96,156,176}

\begin{document}
\title[Boltzmann equation]{3D hard sphere Boltzmann equation: explicit
structure and the transition process from polynomial tail to Gaussian tail}
\author{Yu-Chu Lin}
\address{Yu-Chu Lin, Department of Mathematics, National Cheng Kung
University, Tainan, Taiwan}
\email{yuchu@mail.ncku.edu.tw }
\author{Haitao Wang}
\address{Haitao Wang, School of Mathematical Sciences, Institute of Natural
Sciences, MOE-LSC, IMA-Shanghai, Shanghai Jiao Tong University, Shanghai,
China}
\email{haitallica@sjtu.edu.cn}
\author{Kung-Chien Wu}
\address{Kung-Chien Wu, Department of Mathematics, National Cheng Kung
University, Tainan, Taiwan and National Center for Theoretical Sciences,
National Taiwan University, Taipei, Taiwan}
\email{kungchienwu@gmail.com}
\date{\today }

\begin{abstract}
We study the Boltzmann equation with hard sphere in a near-equilibrium
setting. The initial data is compactly supported in the space variable and
has a polynomial tail in the microscopic velocity. We show that the solution
can be decomposed into a particle-like part (polynomial tail) and a
fluid-like part (Gaussian tail). The particle-like part decays exponentially
in both space and time, while the fluid-like part corresponds to the behavior of the compressible Navier-Stokes equation, which dominates the long time
behavior and exhibits rich wave motion. The nonlinear wave interactions in
the fluid-like part are precisely characterized and therefore we are able to distinguish the linear and nonlinear wave of the solution. It is notable that although the solution has polynomial tail in the velocity initially, the transition process from the polynomial to the Gaussian tail can be
quantitatively revealed due to the collision with the background global Maxwellian.

%
%
%
\end{abstract}

\maketitle
\tableofcontents

\addtocontents{toc}{\setcounter{tocdepth}{1}}


\section{Introduction}

\subsection{The model}

The Boltzmann equation is a fundamental model in the collisional kinetic
theory, which describes the evolution of a phase space distribution function
of moderately dilute gas. Precisely, the Boltzmann equation reads
\begin{equation}
\left \{
\begin{array}{l}
\partial _{t}F+\xi \cdot \nabla _{x}F=Q(F,F)\text{,} \\[4mm]
F(0,x,\xi )=F_{0}(x,\xi )\text{,}%
\end{array}%
\right. \quad (t,x,\xi )\in {\mathbb{R}}^{+}\times {\mathbb{R}}^{3}\times {%
\mathbb{R}}^{3}\text{,}  \label{bot.1.a}
\end{equation}%
where $F\left( t,x,\xi \right) $ is the distribution function for particles
at time $t>0$, position $x=\left( x_{1}\text{, }x_{2}\text{, }x_{3}\right)
\in \mathbb{R}^{3}$ and microscopic velocity $\xi =\left( \xi _{1}\text{, }%
\xi _{2}\text{, }\xi _{3}\right) \in \mathbb{R}^{3}$, and initial data $%
F_{0}\left( x,\xi \right) \geq 0$ is given. Here the Boltzmann collision
operator $Q\left( \cdot ,\cdot \right) $ is given by
\begin{equation*}
Q\left( G,F\right) =\int_{\mathbb{R}^{3}}\int_{\mathbb{S}^{2}}\left \vert
q\cdot n \right \vert \left[ G\left( \xi _{\ast }^{\prime }\right) F\left(
\xi ^{\prime }\right) -G\left( \xi _{\ast }\right) F\left( \xi \right) %
\right] dn d\xi _{\ast }  \label{collision}
\end{equation*}%
where $q=\xi -\xi _{\ast }$ is the relative velocity and the
post-collisional velocities satisfy
\begin{equation*}
\xi ^{\prime }=\xi -\left[ \left( \xi -\xi _{\ast }\right) \cdot n \right] n
\text{, \  \  \ }\xi _{\ast }^{\prime }=\xi _{\ast }+\left[ \left( \xi -\xi
_{\ast }\right) \cdot n \right] n \text{.}
\end{equation*}%
It is well known that the global Maxwellians are the steady solutions to the
Boltzmann equation. In the perturbation regime near the Maxwellian, we look
for the solution in the form of
\begin{equation}
F=\mathcal{M}+f\text{, \  \  \ }F_{0}(x,\xi )=\mathcal{M}+\eps f_{0}\text{,}
\label{F}
\end{equation}%
where $\eps>0$ sufficiently small, with a perturbation function $f$ to $%
\mathcal{M} $. Here the global Maxwellian $\mathcal{M}$ is normalized as
\begin{equation*}
\mathcal{M=}\frac{1}{\left( 2\pi \right) ^{3/2}}\exp \left( -\frac{\left
\vert \xi \right \vert ^{2}}{2}\right) \text{.}
\end{equation*}%
Substituting \eqref{F} into \eqref{bot.1.a}, the perturbation function $f$
satisfies the equation
\begin{equation}
\left \{
\begin{array}{l}
\partial _{t}f+\xi \cdot \nabla _{x}f=\mathcal{L}f+Q\left( f,f\right),
\vspace {3mm}
\\
f|_{t=0}=F_{0}(x,\xi )-\mathcal{M=}\eps f_{0}\text{, \  \ }\left( x,\xi
\right) \in \mathbb{R}^{3}\times \mathbb{R}^{3}\text{,}%
\end{array}%
\right.  \label{Linearized}
\end{equation}%
where
\begin{equation*}
\mathcal{L}f=Q\left( \mathcal{M},f\right) +Q\left( f,\mathcal{M}\right)
\text{.}
\end{equation*}%
In fact, $\mathcal{L}$ can be split into%
\begin{equation*}
\mathcal{L}f=-\nu +\mathcal{K}\text{,}
\end{equation*}%
with \
\begin{equation*}
\nu \left( \xi \right) =\int_{\mathbb{R}^{3}}\int_{\mathbb{S}^{2}}\left
\vert q\cdot \omega \right \vert \mathcal{M}\left( \xi _{\ast }\right)
d\omega d\xi _{\ast }\sim \left( 1+\left \vert \xi \right \vert \right)
\text{,}
\end{equation*}%
\begin{equation*}
\mathcal{K}f=\int_{\mathbb{R}^{3}}\int_{\mathbb{S}^{2}}\left \vert q\cdot n
\right \vert \left[ \mathcal{M}\left( \xi _{\ast }^{\prime }\right) f\left(
\xi ^{\prime }\right) +f\left( \xi _{\ast }^{\prime }\right) \mathcal{M}%
\left( \xi ^{\prime }\right) -f\left( \xi _{\ast }\right) \mathcal{M}\left(
\xi \right) \right] dn d\xi _{\ast }\text{.}
\end{equation*}

Note that this kind of perturbation \eqref{F} allows the initial data to
have a polynomial tail in the microscopic velocity, which is different from
the standard perturbation, $F=\mathcal{M}+\sqrt{\mathcal{M}}f$, where
initial data is assumed to have a Gaussian tail.

It is known that there are extensive studies on the standard perturbation,
including the global existence, time-asymptotic behavior, and even the
pointwise structure, see \cite{[CeIlPu],[Glassey],[Liu-Yu3],[Ukai-Yang]} and
the references therein. It would be very interesting to see that can we
still obtain the precise space-time structure of the solution for initial
data with a polynomial tail?

Moreover, since the perturbation setting describes the collisions between a
small amount of released particles and the ambient particles that have
reached thermal equilibrium, the physical intuition suggests that the
distribution of the released particles will also approach thermal
equilibrium over a long period, namely, it will become close to a Gaussian
in terms of the microscopic velocity. It is a challenge problem that is it
possible to give a quantitative description of the transition process?

The main goal of this paper is to answer the above two questions.
Specifically, we will construct a pointwise estimate of solution to %
\eqref{Linearized} with respect to all variables, space, time, and velocity.
The estimate not only exhibits the wave motion in space-time, but also
reveals how the solution transitions from a polynomial tail to a Gaussian
tail.

\subsection{Review of previous works}

Concerning the polynomial tail perturbation for collisional kinetic
equation, there has been substantial progress recently. For the torus case,
it was initiated by Gualdani-Mischler-Mouhot \cite{[Gualdani]} for the
Boltzmann equation with hard sphere. It was then generalized by
\cite{K-T-Wu} for Landau equation, by 
\cite{[Yang],[Herau]} 
for the Boltzmann equation without angular cutoff, and by 
\cite{[Cao]} for the Boltzmann equation with soft potentials. For the whole
space case, we refer to \cite{K-P,[Cao-Duan-Li]} for the non-cutoff
Boltzmann equation and 
\cite{[DLL]} for the cutoff Boltzmann equation. These works mainly focused
on global existence and large time decay of the solution, whereas our study
aim at providing a more quantitative description of the structure of the
solution.

Next, we review some space-time pointwise results closely related to the
current study. It was initiated by Liu \cite{[Liu]} for 1D viscous
conservation laws, then developed to multi-dimensional compressible
Navier-Stokes equation \cite{[Deng-Yu], [Du-Wu], [Liu-Noh],[Liu-Wang]}.
There are two key ingredients. The first is the construction of Green's
function for the linearized system, where rich wave phenomena, such as
dissipative Huygens wave, diffusion wave, Riesz wave are identified. The
other one is the careful estimate of nonlinear wave couplings between the
above basic wave patterns. As is known, the long time behavior of the
Boltzmann equation is governed by macroscopic fluid. There are some parallel
results for Boltzmann equation with standard Gaussian tail. The result for 1D hard sphere Boltzmann
equation was constructed by Liu-Yu \cite{[LiuYu1]}. As
the nonlinear interaction is strong in 1D, the authors need to extract the
so called \textquotedblleft kinetic Burger equations\textquotedblright \ to
close the nonlinear problem. Later, \cite{[Liu-Yu2],[Liu-Yu3]} obtained the
explicit structure of Green's function for the linearized Boltzmann equation
with hard sphere in 3D. Recently, \cite{[Li-Zhong]} constructed the explicit
structure of the relativistic Boltzmann equation for \textquotedblleft hard
ball\textquotedblright , an exponentially sharp ansatz similar to structure
in \cite{[Deng-Yu]} was justified. In these works, the nonlinear
interactions have been estimated to the extent necessary to close the
nonlinear ansatz.

The transition from polynomial tail to Gaussian tail is related to the decay
estimates for large velocities in the Boltzmann equation. For space
homogeneous case, there are extensive studies on $L^1_v$ moments or
pointwise decay, both for polynomial and exponential weight, see \cite%
{[ACGM], [Bobylev], [De], [GPV]} and references therein. For space
inhomogeneous case, the results are relatively fewer. Generation of
polynomial moments in $L^1_v$ or pointwise sense was established under
suitable moment bound conditions for hard potential with or without cutoff,
see \cite{[CaSne],[Gualdani],[ImMoSi]}. Different from the conditional
results, our result provides a dynamic process for Gaussian tail generation
in the perturbation regime.


\subsection{Notations}

Before stating our main results, we introduce some notations used in this
paper. We denote $w_{\beta }\left( \xi \right) =\left \langle \xi
\right
\rangle ^{\beta }=(1+|\xi |^{2})^{\beta /2}$ , $\beta \in {\mathbb{R}%
} $. For the microscopic variable $\xi $, we denote
\begin{equation*}
|g|_{L_{\xi }^{p}}=\Big(\int_{{\mathbb{R}}^{3}}|g|^{p}d\xi \Big)^{1/p}%
\hbox{
if }1\leq p<\infty \hbox{,}\quad \quad |g|_{L_{\xi }^{\infty }}=\sup_{\xi
\in {\mathbb{R}}^{3}}|g(\xi )|\hbox{,}
\end{equation*}%
and the weighted norms can be defined by
\begin{equation*}
|g|_{L_{\xi ,\beta }^{p}}=\Big(\int_{{\mathbb{R}}^{3}}\left \vert \left
\langle \xi \right \rangle ^{\beta }g\right \vert ^{p}d\xi \Big)^{1/p}%
\hbox{
if }1\leq p<\infty \hbox{,}\quad \quad |g|_{L_{\xi ,\beta }^{\infty
}}=\sup_{\xi \in {\mathbb{R}}^{3}}\left \vert \left \langle \xi \right
\rangle ^{\beta }g(\xi )\right \vert \hbox{.}
\end{equation*}
The $L_{\xi }^{2}$ inner product in ${\mathbb{R}}^{3}$ will be denoted by $%
\left \langle \cdot ,\cdot \right \rangle _{\xi }$, i.e.,
\begin{equation*}
\left \langle f,g\right \rangle _{\xi }=\int f(\xi )\overline{g(\xi )}d\xi %
\hbox{.}
\end{equation*}%
For the Boltzmann equation with hard sphere, the natural norm in $\xi $ is $%
|\cdot |_{L_{\sigma }^{2}}$, which is defined as
\begin{equation*}
|g|_{L_{\sigma }^{2}}^{2}=\left \vert \left \langle \xi \right \rangle ^{%
\frac{1}{2}}g\right \vert _{L_{\xi }^{2}}^{2}\hbox{.}
\end{equation*}%
For the space variable $x$, we have similar notations, namely,
\begin{equation*}
|g|_{L_{x}^{p}}=\left( \int_{{\mathbb{R}^{3}}}|g|^{p}dx\right) ^{1/p}%
\hbox{
if }1\leq p<\infty \hbox{,}\quad \quad |g|_{L_{x}^{\infty }}=\sup_{x\in {%
\mathbb{R}^{3}}}|g(x)|\hbox{.}
\end{equation*}

Finally, with $\mathcal{X}$ and $\mathcal{Y}$ being two normed spaces, we
define%
\begin{equation*}
\left \Vert g\right \Vert _{\mathcal{XY}}=\left \vert \left \vert g\right
\vert _{\mathcal{Y}}\right \vert _{\mathcal{X}}\hbox{,}
\end{equation*}%
and for simplicity, we denote
\begin{equation*}
\left \Vert g\right \Vert _{L^{2}}=\left \Vert g\right \Vert _{L_{\xi
}^{2}L_{x}^{2}}=\left( \int_{\mathbb{R}^{3}}\left \vert g\right \vert
_{L_{x}^{2}}^{2}d\xi \right) ^{1/2}\hbox{.}
\end{equation*}%
For any two functions $f(x,t)$ and $g(x,t)$, we define the space-time
convolution as
\begin{equation*}
f(x,t)\ast _{x,t}g(x,t)=\int_{0}^{t}\int_{{\mathbb{R}}^{3}}f(x-y,t-\tau
)g(y,\tau )dyd\tau .
\end{equation*}%
For any two real numbers $a$ and $b$, we define $a\wedge b=\min \{a,b\}$.

For simplicity of notations, hereafter, we abbreviate $\leq C$ to $\lesssim $%
, where $C$ is a constant depending only on fixed numbers.

\subsection{Main theorem and significant points of our results}

In order to achieve our goal, we introduce the decomposition $f=f_{1}+\sqrt{%
\mathcal{M}}f_{2}$. Here $f_1$ corresponds to the part with only polynomial
tail, while $\sqrt{\mathcal{M}}f_2$ has Gaussian tail. Heuristically, the
closer a distribution function to a Gaussian, the closer the behavior
resembles macroscopic hydrodynamics. Therefore, one may expect $f_1$ behaves
like particle, and $f_2$ behaves like fluid.

The following coupled system designed for $f_{1}$ and $f_{2}$ is to realize
one's intuition:
\begin{equation}
\left \{
\begin{array}{l}
\partial _{t}f_{1}+\xi \cdot \nabla _{x}f_{1}=-\nu \left( \xi \right) f_{1}+%
\mathcal{K}_{s}f_{1}+Q\left( f_{1},f_{1}\right) +Q\left( f_{1},\sqrt{%
\mathcal{M}}f_{2}\right) +Q\left( \sqrt{\mathcal{M}}f_{2},f_{1}\right) ,%
\vspace {3mm}
\\
\partial _{t}f_{2}+\xi \cdot \nabla _{x}f_{2}=Lf_{2}+\mathcal{K}%
_{b}f_{1}+\Gamma \left( f_{2},f_{2}\right) ,%
\end{array}%
\right.  \label{decom-System}
\end{equation}%
with initial data
\begin{equation*}
f_{1}\left( 0,x,\xi \right) =f_{1,0}\left( x,\xi \right) =\varepsilon
f_{0}\left( x,\xi \right) \hbox{, \  \  \ }f_{2}\left( 0,x,\xi \right)
=f_{2,0}\left( x,\xi \right) =0\hbox{,}
\end{equation*}%
where
\begin{equation*}
Lg=\frac{1}{\sqrt{\mathcal{M}}}\left[ Q\left( \mathcal{M},\sqrt{\mathcal{M}}%
g\right) +Q\left( \sqrt{\mathcal{M}}g,\mathcal{M}\right) \right]
\end{equation*}%
is the so called linearized Boltzmann collision operator,
\begin{equation*}
\Gamma \left( f_{2},f_{2}\right) =\frac{1}{\sqrt{\mathcal{M}}}Q\left( \sqrt{%
\mathcal{M}}f_{2},\sqrt{\mathcal{M}}f_{2}\right)
\end{equation*}%
is the nonlinear Boltzmann collision operator, and $\mathcal{K=K}_{s}+%
\mathcal{M}^{1/2}\mathcal{K}_{b}$ \ with
\begin{equation*}
\mathcal{K}_{s}:=\chi _{\{ \left \vert \xi \right \vert \geq R\}}\mathcal{K}%
\hbox{, \  \  \  \ }\mathcal{K}_{b}:=\mathcal{M}^{-1/2}\chi _{\{ \left \vert
\xi \right \vert <R\}}\mathcal{K}
\end{equation*}%
for a constant $R>0$ large enough, $\chi _{\{ \cdot \}}$ being the indicator
function.

It is noted that similar decomposition was also employed in \cite%
{[Cao-Duan-Li]} for Boltzmann equation without angular cutoff. Based on this
decomposition, we have the main theorem as follows:

\begin{theorem}
\label{thm-main1} Let $\beta >4$ be sufficiently large. Assume that $%
f_{0}\in L_{\xi ,\beta }^{\infty }L_{x}^{\infty }$ is compactly supported in
the variable $x$ for all $\xi $
\begin{equation*}
f_{0}\left( x,\xi \right) \equiv 0\text{ for }\left \vert x\right \vert \geq
1\text{, }\xi \in \mathbb{R}^{3}\text{.}
\end{equation*}%
Then for any fixed $\de>0$, there exists $\eps>0$ small enough such that the
solution $f$ of $(\ref{Linearized})$ exists for $t>0$ and it can be
decomposed as $f=f_{1}+\sqrt{\mathcal{M}}f_{2}$, where $f_{1}$ and $f_{2}$
satisfy $\left( \ref{decom-System}\right)$ with the following estimates:

For $f_{1}$, 
we have
\begin{equation*}
\left|\mathbf{1}_{\{ \left<\xi \right> \leq t\}}f_{1}\right|\lesssim
\varepsilon \left \langle \xi \right \rangle ^{-\beta } e^{-\overline{c}_{0}%
\left[\left<\xi \right>^{2}+(t+|x|)\right] }
\end{equation*}
and
\begin{equation*}
\left|\mathbf{1}_{\{ \left<\xi \right> > t\}}f_{1}\right|\lesssim
\varepsilon \left \langle \xi \right \rangle ^{-\beta } e^{-\overline{c}_{0} %
\left[\left<\xi \right>t+(t+|x|)\right] }
\end{equation*}
for some positive constant $\overline{c}_{0}$.

For $f_{2}$, we have
\begin{align*}
\left \vert f_{2}\right \vert _{L_{\xi ,\beta }^{\infty }} &\lesssim
\varepsilon \left[
\begin{array}{l}
\mathbf{1}_{\{ \left \vert x\right \vert \leq \mathbf{c}t\}}\left(
1+t\right) ^{-3/2}\left( 1+\frac{\left \vert x\right \vert ^{2}}{1+t}\right)
^{-3/2} +\left( 1+t\right) ^{-3/2}e^{-\frac{|x|^{2}}{\widehat{D}\left(
1+t\right) }} \\
\\
+\left( 1+t\right) ^{-2}e^{-\frac{\left( \left \vert x\right \vert -\mathbf{c%
}t\right) ^{2}}{\widehat{D}_{0}\left( 1+t\right) }}+e^{-\frac{t+|x|}{%
\widehat{c}_{\de}}}%
\end{array}%
\right] \\
\\
&\quad+ \varepsilon^{2}\mathbf{1}_{\{ \left \vert x\right \vert \leq (%
\mathbf{c}+\de)t\}}\left[ \left( 1+t\right) ^{-2}\left( 1+\frac{\left \vert
x\right \vert ^{2}}{1+t}\right) ^{-3/2}+\left( 1+t\right) ^{-\frac{5}{2}%
}\left( 1+\frac{\left( \left \vert x\right \vert -\mathbf{c}t\right) ^{2}}{%
1+t}\right) ^{-1}\right]
\end{align*}%
for some positive constants $\hat{D}, \hat{D}_{0}, \hat{c}_{\de}$. Here $%
\mathbf{c}=\sqrt{5/3}$ is the sound speed associated with global Maxwellian $%
\mathcal{M}$.
\end{theorem}

Several remarks on the main theorem are in order:

\begin{itemize}

\item The main result is the combinations of Theorems \ref{thm-main}, \ref%
{outside the cone}, and \ref{Improvement-f1}. In Theorem \ref{thm-main}, we
obtain the global wave structure, which is accurate in the time-like region
but only shows polynomial decay for $f_2$ in the space-like region. In
Theorem \ref{outside the cone}, the estimate in space-like region is further
improved to be exponentially sharp. Theorem \ref{Improvement-f1} describes
the dynamic process of transition from polynomial tail to Gaussian tail for $%
f_1$.

\item The result shows that the polynomial tail part $f_1$ decays
exponentially in both space and time, while Gaussian tail part $\sqrt{%
\mathcal{M}}f_2$ exhibits rich wave phenomena and dominates the solution at
large time. This is consistent with our intuition: the polynomial and
Gaussian tail parts are associated with particle-like and fluid-like
behaviors, respectively.

\item In the estimate of $f_2$, we see that the terms consist of $%
\varepsilon $ and $\varepsilon^2$ orders. The $\varepsilon$ order terms
represent linear waves, such as Huygens, diffusion, and Riesz waves, as
given by the Green's function. The $\varepsilon^2$ order terms arise from
nonlinear interactions between these basic waves. They consist of polynomial
versions of Huygens and diffusion waves, primarily concentrate inside the
acoustic wave cone. Compared to the linear waves, the nonlinear waves not
only have a $\varepsilon^2$ order magnitude, but also decay faster by $%
(1+t)^{-1/2}$ than their linear counterparts. However, previous works \cite%
{[Deng-Yu],[Du-Wu],[Li-Zhong],[Liu-Noh],[Liu-Wang]} only showed that they
have the same decay rates. Thus, our results provide a more accurate
description for the nonlinear effect, based on sharper estimates of
nonlinear wave couplings.

\item The polynomial tail of the solution is fully captured by $f_{1}$ part.
The pointwise estimate of $f_{1}$ in velocity variable thus shows how the
polynomial tail transitions to a Gaussian tail. Specifically, $f_{1}$ only
exhibits a polynomial tail initially, as time evolves it immediately
acquires an exponential tail. As time continues to evolve, the particles
with velocity $\left \langle \xi \right \rangle <t$ will become Gaussian
tail, while the non-Gaussian part is of the order $e^{-t^{2}}$ and any
moments generated by the non-Gaussian part are bounded by $e^{-t^{2}}$.
Therefore, as time tends to infinity, the distribution function will
eventually transit to a Gaussian tail; but however, the transition cannot be
completed in any finite time.

\item The assumption that the initial data has compact support in space is
unessential. It is not hard to generalize to the case where the initial data
decays polynomially in space, but in this case, the space-like behavior
should be modified accordingly. We do not pursue this, as our focus is on
the quantitative description of wave motion and the transition to a Gaussian
tail.

\end{itemize}

\subsection{Ideas and strategies}

We now outline the ideas and strategies for the proof of the main theorem.

\begin{subsubsection}{\underline{Global wave structure}}
Let us begin with Theorem \ref{thm-main}, the global wave structure of the solution. Consider the coupled system \eqref{decom-System}. Let $\mathbb{S}^t$ and $\mathbb{G}^t$ be the solution operators to the damped transport equation and (standard) linearized Boltzmann equation respectively, namely, $g(t)\equiv \mathbb{S}
^{t}g_{0}  $ solves
\begin{equation*}
\partial _{t}g+\xi \cdot \nabla _{x}g+\nu \left( \xi \right) g=0\hbox{,}%
\quad g(0,x,\xi )=g_{0}\,,
\end{equation*}
and $g(t)\equiv \mathbb{G}^t g_0$ solves
\begin{equation*}
\partial _{t}g+\xi \cdot \nabla _{x}g-L g=0\hbox{,}%
\quad g(0,x,\xi )=g_{0}\,.
\end{equation*}
Then by Duhamel principle,  solutions $f_1$ and $f_2$ to \eqref{decom-System} satisfy the following coupled integral system
\begin{equation}\label{eq:integral}
\left \{
\begin{aligned}
&f_1(t)=\varepsilon \mathbb{S}^t  f_0+	\int_{0}^{t}\mathbb{S}^{t-\tau} \mathcal{K}_s f_1(\tau)d\tau +\int_{0}^{t}\mathbb{S}^{t-\tau}\Bigl[ Q(f_1,f_1)+Q(f_1,\sqrt{\mathcal{M}}f_2)+Q(\sqrt{\mathcal{M}}f_2,f_1)	\Bigr](\tau)d\tau,\\
&f_2(t)= \int_{0}^{t} \mathbb{G}^{t-\tau}\mathcal{K}_b  f_1(\tau)d\tau+ \int_{0}^{t} \mathbb{G}^{t-\tau} \Gamma(f_2,f_2)(\tau) d\tau.
\end{aligned}
\right.
\end{equation} The essential step for constructing global wave structure is to find out the accurate ansatz for the solution.

We neglect the nonlinear effects for the time being. In the equation for $f_1$, if $R$ is chosen sufficiently large, the integral operator $\mathcal{K}_s$ can be regarded as a perturbation of the damped transport operator. This results in an  exponential decay of $f_1$ in both space and time. We then substitute the estimate of $f_1$ into the integral equation for $f_2$, requiring consideration of  $\int_{0}^{t}\mathbb{G}^{t-\tau} \mathcal{K}_b f_1(\tau)d \tau$. The explicit structure of Green's function $\mathbb{G}$ is constructed in \cite{[Liu-Yu2], [Liu-Yu3]} (we stated it in Lemma \ref{Green}), showing rich wave structures, including a dissipative Huygens wave (acoustic wave described by a moving heat kernel), a diffusion wave (thermal wave described by a stationary heat kernel), a Riesz wave (related to vorticity of macroscopic fluid, described by a polynomial analogue of diffusion wave confined in  the wave cone), and a space-time exponential decay term.  The estimate of $f_2$ is given by the convolution of the Green's function and the source term $\mathcal{K}_b f_1$, 
inheriting a structure similar to that of the Green's function (see Lemmas \ref{exponential-easy-1}-\ref{poly-exponential}). This indicates that $f_1$ can be viewed as  particle-like wave, while $f_2$ can be viewed as fluid-like wave.

Next, we incorporate nonlinear effects. In designing the coupled system, we intentionally placed all nonlinear terms involving $f_1$ in the first equation of \eqref{decom-System}, as it describes the particle-like behavior. The term $\Gamma(f_2,f_2)$ is included in the second equation, as it is associated with  the fluid-like behavior. The key  term  is $\int_0^t \mathbb{G}^{t-\tau}\Gamma(f_2,f_2)(\tau)d\tau$, which accounts for the nonlinear wave interactions. A fundamental property is that the nonlinear operator $\Gamma$ is purely microscopic, and when acted upon by the Green's operator, it  gains an extra $\frac{1}{2}$-order time decay. Substituting the linear estimates leads to the  convolution estimates:
{\small \begin{equation*}
 		 \frac{\bigl(\mbox{Diffusion+Huygens+Riesz+Exponential decay}\bigr)}{\sqrt{1+t}}\underset{(x,t)}{\ast} \bigl(\mbox{Diffusion+Huygens+Riesz+Exponential decay}\bigr)^2.
\end{equation*}
}

The main effort  is  to provide sharp estimates of these convolutions. Without delving into the details, we provide a heuristic explanation of the interaction process here. We illustrate this with the convolution of two Huygens waves.
\[
\int_{0}^{t} \int_{\R^3} (1+t-s)^{-5/2} e^{-\frac{(\lvert x-y\rvert-(t-s))^2}{D_0(t-s)}} (1+s)^{-4} e^{-\frac{ (	\lvert y\rvert -s )^2}{C_1(s+1)}} dy ds,
\]
where  $\mathbf{c}=1$ for simplicity.
The following figure is  to explain the interaction process:

\begin{figure}[h]
	\centering
	\includegraphics[width=0.6\textwidth]{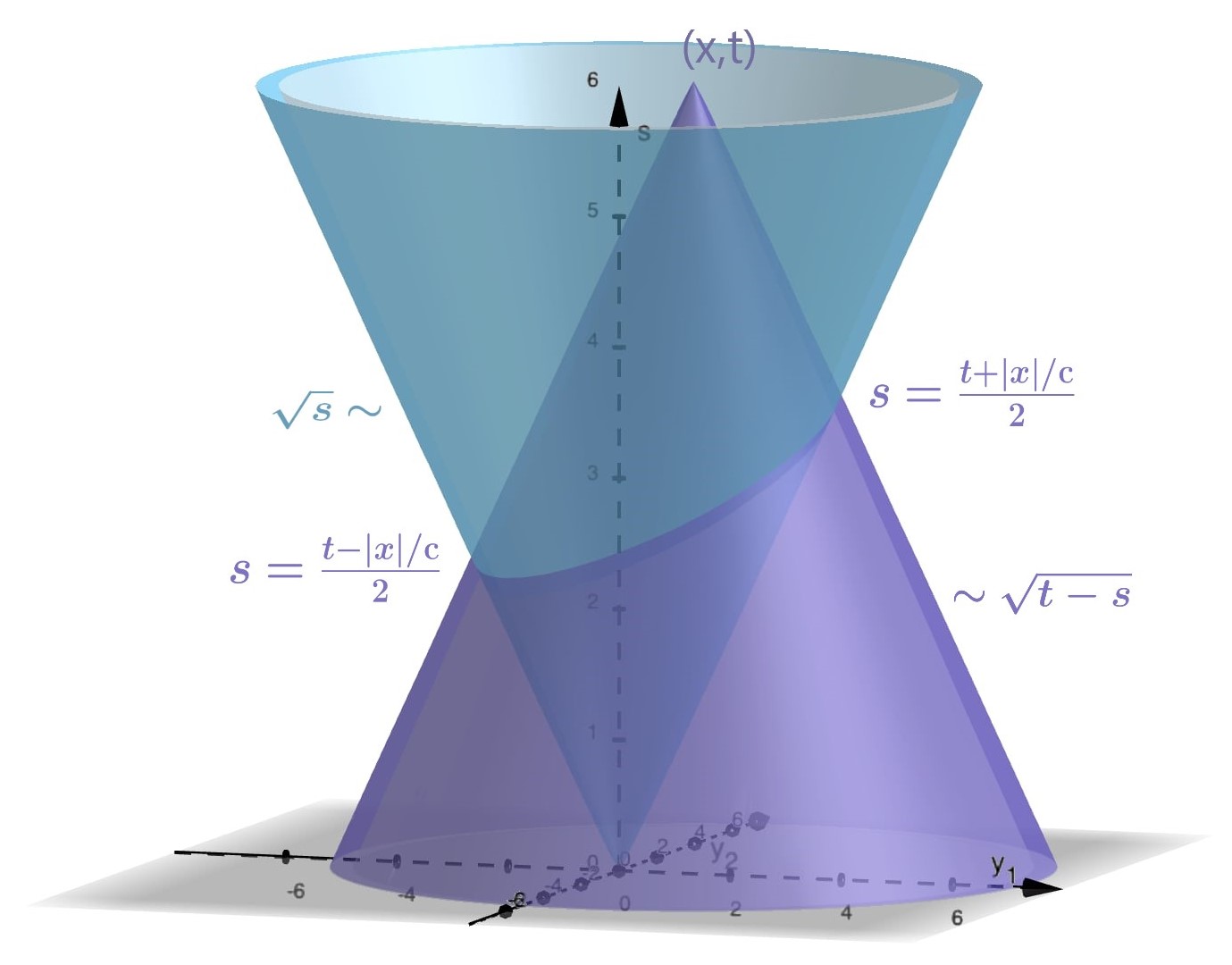}
	\caption{Interaction between two Huygens waves}
\end{figure}

Here, the source $(1+s)^{-4} e^{-\frac{ (	\lvert y\rvert -s )^2}{C_1(s+1)}}$ is plotted as a forward cone with thickness $\sqrt{s}$, and the propagator  $ (1+t-s)^{-5/2} e^{-\frac{(\lvert x-y\rvert-(t-s))^2}{D_0(t-s)}}  $ is plotted as a backward cone with thickness $\sqrt{t-s}$. The space is represented in 2D in the figure.
%
%
The  interaction essentially occurs in the following space-time region:
\[
\left \{ (y,s) \big|\, \lvert \lvert x-y\rvert -(t-s)\rvert \leq O(1)\sqrt{t-s} \mbox{ and } \lvert \lvert y\rvert -s\rvert \leq O(1)\sqrt{s}
\right \}.
\]
Inside this region, the exponential term in Huygens wave is not effective, and the decay is mainly due to time decay factor. The key point is the sharp estimate of the  volume for this space-time region.  In Section \ref{nonlinear wave interaction}, we first provide some heuristic calculations. In fact, our rigorous estimates are greatly  motivated by the heuristics. We   identify the strong interaction region (which appears in the regions $D_{4}$ and $D_{5}$ of Section \ref{wave-interaction}) and perform very careful estimates there. The results match those obtained by the heuristic argument.

Through these sharp convolution estimates, we  propose an appropriate ansatz, which is exponential decay for $f_1$ and  polynomially sharp for $f_2$ as the main focus here is  the region inside wave cone, where only polynomial decay can be expected. Justifying the ansatz involves even more complicated convolution estimates, but the underlying idea remains similar. Additionally,  the damped transport operator $\mathbb{S}^{t}$ is used to compensate the loss of velocity decay from the nonlinear operator $\Gamma$ in the justification for  $f_1$ (see Lemmas \ref{Lemma-gain decay} and \ref{Lemma-gain decay-nonmoving}).
Our result distinguishes between the linear and nonlinear parts of the solution, significantly improving upon previous results in \cite{[Deng-Yu],[Du-Wu],[Li-Zhong], [Liu-Noh],[Liu-Wang]}, where the nonlinear couplings and the linear part have the same decay rate.

\end{subsubsection}

\begin{subsubsection}{\underline{Exponential decay outside the acoustic wave cone}}

In Theorem \ref{thm-main}, we obtain space-time exponential decay for $f_1$, but for $f_2$, we only achieve a polynomial decay estimate. This is because we designed a polynomial-type ansatz to facilitate the control of the nonlinear part of $f_2$. This ansatz is accurate for the structure inside the acoustic wave cone. However, since the initial data has compact support in space and $f_2$ corresponds to a fluid structure that propagates at a finite speed, we expect the solution to  decay exponentially outside the wave cone.

To observe the behavior of $f_2$ outside the cone, we multiply a suitable weight function on $f_2$ and prove an $L^\infty$ bound of the weighted solution through regularization and energy estimates. This approach for obtaining the space asymptotic behavior of the Boltzmann equation was developed in our previous work \cite{[LinWangLyuWu]}, and here it is adapted to handle $f_2$.

It is worth mentioning that in our previous work, we proved exponential decay outside a wave cone $\lvert x\rvert < Mt$   for a sufficiently large $M$. Here, through careful calculation of the micro projection, we show that $f_2$ indeed decays space-time exponentially outside the wave cone $\lvert x\rvert < (\mathbf{c}+\delta) t$ for any positive $\delta$ (see Theorem \ref{outside the cone}), where $\mathbf{c}$ is the sound speed. This result is more consistent with physical reality.

\end{subsubsection}

\begin{subsubsection}{\underline{The transition from polynomial tail to Gaussian tail}}

In the decomposition $f=f_1+\sqrt{\mathcal{M}}f_2$, the latter  already exhibits Gaussian tail. Therefore, studying the transition process is equivalent to examining the generation of the Gaussian tail for $f_1$. The mechanism for generating  velocity decay comes from $\displaystyle e^{-\nu(\xi)t}$  in the damped transport operator $\mathbb{S}^t$:
\[
\mathbb{S}^t g_0=e^{-\nu(\xi)t}g_0(x-\xi t,\xi).
\]
At first glance, it seems that $f_1$ can gain arbitrary velocity decay as time evolves. However, the upper bound for velocity decay is limited by the coupling between $f_1$ and $\sqrt{\mathcal{M}}f_2$, specifically by the term
\[
\int_{0}^{t}\mathbb{S}^{t-\tau}\Bigl[Q(f_1,\sqrt{\mathcal{M}}f_2)+Q(\sqrt{\mathcal{M}}f_2,f_1)	\Bigr](\tau)d\tau
\]
in the second equation of \eqref{eq:integral}. The velocity weight will ultimately be slowed down by $\sqrt{\mathcal{M}}f_2$ as it decay at most as a Gaussian. We design a suitable weighted function
\[
e^{\kappa \langle \xi \rangle \min \{ \langle \xi \rangle \wedge t\}}
\] to capture this feature. By carefully analyzing the commutator between this weight function and the  $\mathcal{K}_s,\,Q$ operators, we complete the description of dynamic transition process (see Theorem \ref{Improvement-f1}).

It is interesting to note that  the  mechanism for gaining velocity weight and the limitation of the maximal generation  both stem from collisions with the global Maxwellian. This is entirely consistent with our physical intuition.

%
%
%
%
%
%

\end{subsubsection}

\subsection{Organization of the paper}

The rest of this paper is organized as follows: In Section \ref{pre}, we
first prepare some basic properties of the integral operator $\mathcal{K}$,
and present some estimates for the damped transport equation and the
linearized Boltzmann equation. In Section \ref{global-structure}, we
construct the global wave structures of the solution, fully utilizing sharp
nonlinear wave interactions. In Section \ref{weighted}, we apply the
weighted energy estimate to prove that the solution indeed decays
exponentially in space-time outside the wave cone.
In Section \ref{f1-velocity}, we provide a quantitative description of how
solution approaches a Gaussian tail in terms of the microscopic velocity.
Finally, we present the proof of all kinds of wave interactions in Section %
\ref{wave-interaction}.

\section{Preliminaries}

\label{pre}

To begin with, we study some essential properties of the collision frequency
$\nu \left( \xi \right) $, the operator $\mathcal{K}$ and the collision
operator $Q$. It is well known that there exist two positive constants $\nu
_{0}$ and $\nu _{1}$ such that%
\begin{equation*}
\nu _{0}\left \langle \xi \right \rangle \leq \nu \left( \xi \right) =\int_{%
\mathbb{R}^{3}}\int_{S^{2}}\left \vert q\cdot n\right \vert \mathcal{M}%
\left( \xi _{\ast }\right) dnd\xi _{\ast }\leq \nu _{1}\left \langle \xi
\right \rangle
\end{equation*}%
for all $\xi \in \mathbb{R}^{3}$. In \cite[Lemma 2.1 and Lemma 6.2]{[Cao]},
the following estimates of the collision kernel have been proved.

\begin{lemma}[{\protect \cite{[Cao]}}]
For any $\beta >4$, there exists a constant $C_{\beta }>0$ depending only on
$\beta $ such that%
\begin{equation*}
\int_{\mathbb{R}^{3}}\int_{S^{2}}\left \vert \xi -\xi _{\ast }\right \vert
\frac{w_{\beta }\left( \xi \right) }{ w_{\beta }\left( \xi _{\ast }^{\prime
}\right) }e^{-\frac{1}{4}|\xi^{\prime 2}}dnd\xi _{\ast }\leq \frac{C}{\beta }%
\nu \left( \xi \right) +C_{\beta }\frac{\nu \left( \xi \right) }{\left(
1+\left \vert \xi \right \vert \right) ^{2}}
\end{equation*}%
and
\begin{equation*}
\int_{\mathbb{R}^{3}}\int_{S^{2}}\left \vert \xi -\xi _{\ast }\right \vert
\frac{w_{\beta }\left( \xi \right) }{w_{\beta }\left( \xi ^{\prime }\right)
w_{\beta }\left( \xi _{\ast }^{\prime }\right) }dnd\xi _{\ast }\leq \frac{C}{%
\beta }\nu \left( \xi \right) +C_{\beta }\frac{\nu \left( \xi \right) }{%
\left( 1+\left \vert \xi \right \vert \right) ^{2}}\hbox{.}
\end{equation*}%
where $C>0$ is a universal constant independent of $\beta $. Here $w_{\beta
}\left( \xi \right) =\left \langle \xi \right \rangle ^{\beta }$.
\end{lemma}

Using this lemma, one can get the estimates for the operators $\mathcal{K}$
and $Q$.

\begin{lemma}[{\protect \cite{[Cao]}}]
\label{Estimate-KQ}For any $\beta >4$, there exists a constant $C_{\beta }>0$
depending only on $\beta$ such that%
\begin{equation*}
\left \langle \xi \right \rangle ^{\beta }\left \vert \mathcal{K}f\right
\vert \leq \left( \frac{C}{\beta }\nu \left( \xi \right) +C_{\beta }\frac{%
\nu \left( \xi \right) }{\left( 1+\left \vert \xi \right \vert \right) ^{2}}%
\right) \left \vert f\right \vert _{L_{\xi ,\beta }^{\infty }}
\end{equation*}%
and
\begin{equation*}
\left \vert \left \langle \xi \right \rangle ^{\beta }Q\left( f,g\right)
\right \vert \leq C_{\beta }\nu \left( \xi \right) \left \vert f\right \vert
_{L_{\xi ,\beta }^{\infty }}\left \vert g\right \vert _{L_{\xi ,\beta
}^{\infty }}\hbox{.}
\end{equation*}
Moreover, for any $R>0$,%
\begin{equation*}
\chi _{\{ \left \vert \xi \right \vert \geq R\}}\left \langle \xi \right
\rangle ^{\beta }\left \vert \mathcal{K}f\right \vert \leq \left( \frac{C}{%
\beta }+\frac{C_{\beta }}{R^{2}}\right) \nu \left( \xi \right) \left \vert
f\right \vert _{L_{\xi ,\beta }^{\infty }}\,.
\end{equation*}
\end{lemma}

In order to study the first equation of (\ref{decom-System}), we introduce
the damped transport operator $\mathbb{S}^{t}$, that is, $\mathbb{S}%
^{t}g_{0} $ is the solution of the equation:
\begin{equation*}
\partial _{t}g+\xi \cdot \nabla _{x}g+\nu \left( \xi \right) g=0\hbox{,}%
\quad g(0,x,\xi )=g_{0}\,.
\end{equation*}

\begin{lemma}
\label{Lemma-gain decay}Let $\beta >4$ and $M>0$. Assume that $U\left(
t,x,\xi \right) $ and $h(t,x,\xi )$ satisfy
\begin{equation*}
\left \vert \nu \left( \xi \right) ^{-1}U\left( t,x,\xi \right) \right \vert
_{L_{\xi ,\beta }^{\infty }}\leq Ae^{-\frac{t+\left \vert x\right \vert }{c}}
\end{equation*}%
and
\begin{equation*}
\left \vert h\left( t,x,\xi \right) \right \vert _{L_{\xi ,\beta }^{\infty
}}\leq Be^{-\frac{t+\left \vert x\right \vert }{c}}
\end{equation*}%
for some constants $A,c>0$ with $\frac{\nu _{0}}{2}>\frac{1}{c}$. If $g$
satisfies the integral equation
\begin{equation*}
g\left( t,x,\xi \right) =\int_{0}^{t}\mathbb{S}^{t-\tau }\left( \mathcal{K}%
_{s}h+U\right) \left( \tau ,x,\xi \right) d\tau \hbox{,}
\end{equation*}%
then
\begin{equation*}
\left \vert g(t,x,\xi )\right \vert _{L_{\xi ,\beta }^{\infty }}\leq
2(A+\eta B)e^{-\frac{t+\left \vert x\right \vert }{c}}\hbox{,}
\end{equation*}%
where
\begin{equation*}
\eta =\eta (\beta ,R)=\left( \frac{C}{\beta }+\frac{C_{\beta }}{R^{2}}%
\right) \,.
\end{equation*}
\end{lemma}

\proof%
Let $y=x-\left( t-\tau \right) \xi $, we have
\begin{eqnarray*}
\left \vert \left \langle \xi \right \rangle ^{\beta }g\left( t,x,\xi
\right) \right \vert &=&\left \vert \int_{0}^{t}e^{-\nu \left( \xi \right)
\left( t-\tau \right) }\left \langle \xi \right \rangle ^{\beta }\left[
\mathcal{K}_{s}h\left( \tau ,x-\xi \left( t-\tau \right) ,\xi \right)
+U\left( \tau ,x-\xi \left( t-\tau \right) ,\xi \right) \right] d\tau \right
\vert \\
&\leq &\int_{0}^{t}e^{-\frac{\nu \left( \xi \right) }{2}\left( t-\tau
\right) }\nu \left( \xi \right) \left( \nu \left( \xi \right) \right)
^{-1}e^{-\frac{\nu _{0}}{2}\left( t-\tau +\left \vert x-y\right \vert
\right) }\left \langle \xi \right \rangle ^{\beta }\left[ \left \vert
\mathcal{K}_{s}h\left( \tau ,y,\xi \right) \right \vert \right] d\tau \\
&&+\int_{0}^{t}e^{-\frac{\nu \left( \xi \right) }{2}\left( t-\tau \right)
}\nu \left( \xi \right) \left( \nu \left( \xi \right) \right) ^{-1}e^{-\frac{%
\nu _{0}}{2}\left( t-\tau +\left \vert x-y\right \vert \right) }\left
\langle \xi \right \rangle ^{\beta }\left \vert U\left( \tau ,y,\xi \right)
\right \vert d\tau \\
&\leq &\eta \int_{0}^{t}e^{-\frac{\nu \left( \xi \right) }{2}\left( t-\tau
\right) }\nu \left( \xi \right) \cdot \sup_{y\in \mathbb{R}^{3}}e^{-\frac{%
\nu _{0}}{2}\left( t-\tau +\left \vert x-y\right \vert \right) }\left \vert
h\left( \tau ,y,\xi \right) \right \vert _{L_{\xi ,\beta }^{\infty }}d\tau \\
&&+\int_{0}^{t}e^{-\frac{\nu \left( \xi \right) }{2}\left( t-\tau \right)
}\nu \left( \xi \right) \cdot \sup_{y\in \mathbb{R}^{3}}e^{-\frac{\nu _{0}}{2%
}\left( t-\tau +\left \vert x-y\right \vert \right) }\left \vert \nu
^{-1}(\xi )U\left( \tau ,y,\xi \right) \right \vert _{L_{\xi ,\beta
}^{\infty }}d\tau \\
&\leq &2(A+\eta B)e^{-\frac{t+\left \vert x\right \vert }{c}}\text{.}
\end{eqnarray*}%
The proof is completed.$%
\hfill%
\square $

\begin{lemma}
\label{Lemma-gain decay-nonmoving} Let $\alpha >0$ and $\rho >0$. Assume
that $V\left( t,x,\xi \right) $ satisfies
\begin{equation*}
\left \vert \nu \left( \xi \right) ^{-1}V\left( t,x,\xi \right) \right \vert
_{L_{\xi ,\beta }^{\infty }}\leq \left( 1+t\right) ^{-\alpha }\left( 1+\frac{%
\left \vert x\right \vert ^{2}}{1+t}\right) ^{-\rho }\,.
\end{equation*}%
If $g$ satisfies the integral equation
\begin{equation*}
g\left( t,x,\xi \right) =\int_{0}^{t}\mathbb{S}^{t-\tau }V\left( \tau ,x,\xi
\right) d\tau \hbox{,}
\end{equation*}%
then
\begin{equation*}
\left \vert g(t,x,\xi )\right \vert _{L_{\xi ,\beta }^{\infty }}\leq
C_{\alpha ,\rho }\left( 1+t\right) ^{-\alpha }\left( 1+\frac{\left \vert
x\right \vert ^{2}}{1+t}\right) ^{-\rho }
\end{equation*}%
for some constant $C_{\alpha ,\rho }>0$.
\end{lemma}

\begin{proof}
The proof is similar to Lemma \ref{Lemma-gain decay}. It suffices to verify
that
\begin{eqnarray*}
&&\int_{0}^{t}e^{-\frac{\nu \left( \xi \right) }{2}\left( t-\tau \right)
}\nu \left( \xi \right) \sup_{y\in \mathbb{R}^{3}}e^{-\frac{\nu _{0}}{2}%
\left( t-\tau +\left \vert x-y\right \vert \right) }\left( 1+\tau \right)
^{-\alpha }\left( 1+\frac{\left \vert y\right \vert ^{2}}{1+\tau }\right)
^{-\rho }d\tau \\
&\leq &C_{\alpha ,\rho }\left( 1+t\right) ^{-\alpha }\left( 1+\frac{\left
\vert x\right \vert ^{2}}{1+t}\right) ^{-\rho }\,.
\end{eqnarray*}

Let $0\leq \tau \leq t$. If $\left \vert y\right \vert >\left \vert
x\right
\vert /2$, then
\begin{equation*}
\left( 1+\frac{\left \vert y\right \vert ^{2}}{1+\tau }\right) ^{-\rho }\leq
\left( 1+\frac{\left \vert x\right \vert ^{2}}{4\left( 1+t\right) }\right)
^{-\rho }\leq 4^{\rho }\left( 1+\frac{\left \vert x\right \vert ^{2}}{1+t}%
\right) ^{-\rho }\hbox{,}
\end{equation*}%
and thus%
\begin{equation*}
e^{-\frac{\nu _{0}}{2}\left( t-\tau +\left \vert x-y\right \vert \right)
}\left( 1+\tau \right) ^{-\alpha }\left( 1+\frac{\left \vert y\right \vert
^{2}}{1+\tau }\right) ^{-\rho }\leq 4^{\rho }\left( 1+\tau \right) ^{-\alpha
}\left( 1+\frac{\left \vert x\right \vert ^{2}}{1+t}\right) ^{-\rho }\hbox{.}
\end{equation*}%
If $\left \vert y\right \vert \leq \left \vert x\right \vert /2$, then%
\begin{equation*}
e^{-\frac{\nu _{0}}{2}\left \vert x-y\right \vert }\leq e^{-\frac{\nu
_{0}\left \vert x\right \vert }{4}}\leq C_{\rho }\left( 1+\left \vert
x\right \vert ^{2}\right) ^{-\rho }\leq C_{\rho }\left( 1+\frac{\left \vert
x\right \vert ^{2}}{1+t}\right) ^{-\rho }\hbox{,}
\end{equation*}%
and so%
\begin{equation*}
e^{-\frac{\nu _{0}}{2}\left( t-\tau +\left \vert x-y\right \vert \right)
}\left( 1+\tau \right) ^{-\alpha }\left( 1+\frac{\left \vert y\right \vert
^{2}}{1+\tau }\right) ^{-\rho }\leq C_{\rho }\left( 1+\tau \right) ^{-\alpha
}\left( 1+\frac{\left \vert x\right \vert ^{2}}{1+t}\right) ^{-\rho }\hbox{.}
\end{equation*}%
Therefore,%
\begin{eqnarray*}
&&\int_{0}^{t}e^{-\frac{\nu \left( \xi \right) }{2}\left( t-\tau \right)
}\nu \left( \xi \right) \sup_{y\in \mathbb{R}^{3}}e^{-\frac{\nu _{0}}{2}%
\left( t-\tau +\left \vert x-y\right \vert \right) }\left( 1+\tau \right)
^{-\alpha }\left( 1+\frac{\left \vert y\right \vert ^{2}}{1+\tau }\right)
^{-\rho }d\tau \\
&\leq &\left( C_{\rho }+4^{\rho }\right) \left( 1+\frac{\left \vert x\right
\vert ^{2}}{1+t}\right) ^{-\rho }\int_{0}^{t}e^{-\frac{\nu \left( \xi
\right) }{2}\left( t-\tau \right) }\nu \left( \xi \right) \left( 1+\tau
\right) ^{-\alpha }d\tau \\
&\leq &C_{\alpha }\left( C_{\rho }+4^{\rho }\right) \left( 1+t\right)
^{-\alpha }\left( 1+\frac{\left \vert x\right \vert ^{2}}{1+t}\right)
^{-\rho }\hbox{.}
\end{eqnarray*}
\end{proof}

\begin{lemma}
\label{Lemma-gain decay-moving} Let $\alpha >0$ and $\rho >0$. Assume that $%
W\left( t,x,\xi \right) $ satisfies
\begin{equation*}
\left \vert \nu \left( \xi \right) ^{-1}W\left( t,x,\xi \right) \right \vert
_{L_{\xi ,\beta }^{\infty }}\leq \left( 1+t\right) ^{-\alpha }\left( 1+\frac{%
\left( \mathbf{c}t-\left \vert x\right \vert \right) ^{2}}{1+t}\right)
^{-\rho }\,.
\end{equation*}%
If $g$ satisfies the integral equation
\begin{equation*}
g\left( t,x,\xi \right) =\int_{0}^{t}\mathbb{S}^{t-\tau }W\left( \tau ,x,\xi
\right) d\tau \hbox{,}
\end{equation*}%
then
\begin{equation*}
\left \vert g(t,x,\xi )\right \vert _{L_{\xi ,\beta }^{\infty }}\leq
C_{\alpha ,\rho }\left( 1+t\right) ^{-\alpha }\left( 1+\frac{\left( \mathbf{c%
}t-\left \vert x\right \vert \right) ^{2}}{1+t}\right) ^{-\rho }
\end{equation*}%
for some constant $C_{\alpha ,\rho }>0$.
\end{lemma}

\begin{proof}
The proof is similar to Lemma \ref{Lemma-gain decay} and it suffices to
verify that
\begin{eqnarray*}
&&\int_{0}^{t}e^{-\frac{\nu \left( \xi \right) }{2}\left( t-\tau \right)
}\nu \left( \xi \right) \sup_{y\in \mathbb{R}^{3}}e^{-\frac{\nu _{0}}{2}%
\left( t-\tau +\left \vert x-y\right \vert \right) }\left( 1+\tau \right)
^{-\alpha }\left( 1+\frac{\left( \left \vert y\right \vert -\mathbf{c}\tau
\right) ^{2}}{1+\tau }\right) ^{-\rho }d\tau \\
&\leq &C_{\alpha ,\rho }\left( 1+t\right) ^{-\alpha }\left( 1+\frac{\left(
\mathbf{c}t-\left \vert x\right \vert \right) ^{2}}{1+t}\right) ^{-\rho }
\end{eqnarray*}

We consider two cases $\left \vert x\right \vert \geq \mathbf{c}t$ and $%
\left \vert x\right \vert \leq \mathbf{c}t$. \newline
\newline
\textbf{Case 1:} $\left \vert x\right \vert \geq \mathbf{c}t$. We split $%
\mathbb{R}^{3}$ into two parts
\begin{equation*}
\{y\in \mathbb{R}^{3}:\left \vert \left \vert y\right \vert -\mathbf{c}\tau
\right \vert >\frac{\left \vert x\right \vert -\mathbf{c}t}{2}\} \hbox{ and }%
\{y\in \mathbb{R}^{3}:\left \vert \left \vert y\right \vert -\mathbf{c}\tau
\right \vert \leq \frac{\left \vert x\right \vert -\mathbf{c}t}{2}\} \hbox{.}
\end{equation*}%
If $\left \vert \left \vert y\right \vert -\mathbf{c}\tau \right \vert \leq
\frac{\left \vert x\right \vert -\mathbf{c}t}{2}$, then
\begin{eqnarray*}
\left \vert x-y\right \vert &\geq &\left \vert \left \vert x\right \vert
-\left \vert y\right \vert \right \vert =\left \vert \left \vert x\right
\vert -\mathbf{c}\tau +\mathbf{c}\tau -\left \vert y\right \vert \right
\vert \geq \left \vert \left \vert x\right \vert -\mathbf{c}\tau \right
\vert -\left \vert \left \vert y\right \vert -\mathbf{c}\tau \right \vert \\
&\geq &\left \vert x\right \vert -\mathbf{c}t-\frac{\left \vert x\right
\vert -\mathbf{c}t}{2}=\frac{\left \vert x\right \vert -\mathbf{c}t}{2}%
\hbox{,}
\end{eqnarray*}%
and thus%
\begin{equation*}
e^{-\frac{\nu _{0}}{2}\left \vert x-y\right \vert }\leq e^{-\frac{\nu
_{0}\left \vert \left \vert x\right \vert -\mathbf{c}t\right \vert }{4}}\leq
C_{\rho }\left( 1+\left( \left \vert x\right \vert -\mathbf{c}t\right)
^{2}\right) ^{-\rho }\leq C_{\rho }\left( 1+\frac{\left( \left \vert x\right
\vert -\mathbf{c}t\right) ^{2}}{1+t}\right) ^{-\rho }\hbox{.}
\end{equation*}%
If $\left \vert \left \vert y\right \vert -\mathbf{c}\tau \right \vert >%
\frac{\left \vert x\right \vert -\mathbf{c}t}{2}$, then
\begin{equation*}
\left( 1+\frac{\left( \left \vert y\right \vert -\mathbf{c}\tau \right) ^{2}%
}{1+\tau }\right) ^{-\rho }\leq \left( 1+\frac{\left( \left \vert x\right
\vert -\mathbf{c}t\right) ^{2}}{4\left( 1+t\right) }\right) ^{-\rho }\leq
4^{\rho }\left( 1+\frac{\left( \left \vert x\right \vert -\mathbf{c}t\right)
^{2}}{1+t}\right) ^{-\rho }\hbox{.}
\end{equation*}%
Therefore,%
\begin{eqnarray*}
&&\int_{0}^{t}e^{-\frac{\nu \left( \xi \right) }{2}\left( t-\tau \right)
}\nu \left( \xi \right) \sup_{y\in \mathbb{R}^{3}}e^{-\frac{\nu _{0}}{2}%
\left( t-\tau +\left \vert x-y\right \vert \right) }\left( 1+\tau \right)
^{-\alpha }\left( 1+\frac{\left( \left \vert y\right \vert -\mathbf{c}\tau
\right) ^{2}}{\left( 1+\tau \right) }\right) ^{-\rho }d\tau \\
&\leq &\left( C_{\rho }+4^{\rho }\right) \left( 1+\frac{\left( \left \vert
x\right \vert -\mathbf{c}t\right) ^{2}}{1+t}\right) ^{-\rho }\int_{0}^{t}e^{-%
\frac{\nu \left( \xi \right) }{2}\left( t-\tau \right) }\nu \left( \xi
\right) \left( 1+\tau \right) ^{-\alpha }d\tau \\
&\leq &C_{\alpha }\left( C_{\rho }+4^{\rho }\right) \left( 1+t\right)
^{-\alpha }\left( 1+\frac{\left( \left \vert x\right \vert -\mathbf{c}%
t\right) ^{2}}{1+t}\right) ^{-\rho }\hbox{.}
\end{eqnarray*}
\newline
\textbf{Case 2:} $\left \vert x\right \vert \leq \mathbf{c}t$. If $0\leq
\tau \leq \frac{t}{2}+\frac{\left \vert x\right \vert }{2\mathbf{c}}$, then
\begin{equation*}
t-\tau \geq t-\left( \frac{t}{2}+\frac{\left \vert x\right \vert }{2\mathbf{c%
}}\right) =\frac{\mathbf{c}t-\left \vert x\right \vert }{2\mathbf{c}}\geq 0%
\hbox{,}
\end{equation*}%
and thus%
\begin{equation*}
e^{-\nu _{0}\left( t-\tau \right) }\leq e^{-\frac{\nu _{0}}{2\mathbf{c}}%
\left( \mathbf{c}t-\left \vert x\right \vert \right) }\leq C_{\rho }\left(
1+\left( \mathbf{c}t-\left \vert x\right \vert \right) ^{2}\right) ^{-\rho
}\leq C_{\rho }\left( 1+\frac{\left( \mathbf{c}t-\left \vert x\right \vert
\right) ^{2}}{1+t}\right) ^{-\rho }\hbox{.}
\end{equation*}%
It implies that
\begin{eqnarray*}
&&\int_{0}^{\frac{t}{2}+\frac{\left \vert x\right \vert }{2\mathbf{c}}}e^{-%
\frac{\nu \left( \xi \right) }{2}\left( t-\tau \right) }\nu \left( \xi
\right) \sup_{y\in \mathbb{R}^{3}}e^{-\frac{\nu _{0}}{2}\left( t-\tau +\left
\vert x-y\right \vert \right) }\left( 1+\tau \right) ^{-\alpha }\left( 1+%
\frac{\left( \left \vert y\right \vert -\mathbf{c}\tau \right) ^{2}}{1+\tau }%
\right) ^{-\rho }d\tau \\
&\leq &C_{\rho }\left( 1+\frac{\left( \mathbf{c}t-\left \vert x\right \vert
\right) ^{2}}{1+t}\right) ^{-\rho }\int_{0}^{t}e^{-\frac{\nu \left( \xi
\right) }{2}\left( t-\tau \right) }\nu \left( \xi \right) \left( 1+\tau
\right) ^{-\alpha }d\tau \\
&\leq &C_{\alpha }C_{\rho }\left( 1+\tau \right) ^{-\alpha }\left( 1+\frac{%
\left( \mathbf{c}t-\left \vert x\right \vert \right) ^{2}}{1+t}\right)
^{-\rho }\hbox{.}
\end{eqnarray*}%
If $\frac{t}{2}+\frac{\left \vert x\right \vert }{2\mathbf{c}}\leq \tau \leq
t$, $\left \vert y\right \vert \leq \frac{\mathbf{c}\tau +\left \vert
x\right \vert }{2}$, then
\begin{equation*}
\mathbf{c}\tau -\left \vert y\right \vert \geq \frac{\mathbf{c}\tau -\left
\vert x\right \vert }{2}\geq \frac{\mathbf{c}t-\left \vert x\right \vert }{4}%
\geq 0\hbox{,}
\end{equation*}%
and so
\begin{equation*}
\left( 1+\frac{\left( \left \vert y\right \vert -\mathbf{c}\tau \right) ^{2}%
}{1+\tau }\right) ^{-\rho }\leq \left( 1+\frac{\left( \mathbf{c}t-\left
\vert x\right \vert \right) ^{2}}{16\left( 1+t\right) }\right) ^{-\rho }\leq
16^{\rho }\left( 1+\frac{\left( \mathbf{c}t-\left \vert x\right \vert
\right) ^{2}}{1+t}\right) ^{-\rho }\hbox{.}
\end{equation*}%
If $\frac{t}{2}+\frac{\left \vert x\right \vert }{2\mathbf{c}}\leq \tau \leq
t$, $\left \vert y\right \vert >\frac{\mathbf{c}\tau +\left \vert
x\right
\vert }{2}$, then%
\begin{equation*}
\left \vert x-y\right \vert \geq \left \vert y\right \vert -\left \vert
x\right \vert \geq \frac{\mathbf{c}\tau -\left \vert x\right \vert }{2}\geq
\frac{\mathbf{c}t-\left \vert x\right \vert }{4}\hbox{,}
\end{equation*}%
and thus%
\begin{equation*}
e^{-\frac{\nu _{0}}{2}\left \vert x-y\right \vert }\leq e^{-\frac{\nu _{0}}{8%
}\left( \mathbf{c}t-\left \vert x\right \vert \right) }\leq C_{\rho }\left(
1+\left( \mathbf{c}t-\left \vert x\right \vert \right) ^{2}\right) ^{-\rho
}\leq C_{\rho }\left( 1+\frac{\left( \mathbf{c}t-\left \vert x\right \vert
\right) ^{2}}{1+t}\right) ^{-\rho }\text{.}
\end{equation*}%
Consequently,%
\begin{eqnarray*}
&&\int_{\frac{t}{2}+\frac{\left \vert x\right \vert }{2\mathbf{c}}}^{t}e^{-%
\frac{\nu \left( \xi \right) }{2}\left( t-\tau \right) }\nu \left( \xi
\right) \sup_{y\in \mathbb{R}^{3}}e^{-\frac{\nu _{0}}{2}\left( t-\tau +\left
\vert x-y\right \vert \right) }\left( 1+\tau \right) ^{-\alpha }\left( 1+%
\frac{\left( \left \vert y\right \vert -\mathbf{c}\tau \right) ^{2}}{\left(
1+\tau \right) }\right) ^{-\rho }d\tau \\
&\leq &\left( 16^{\rho }+C_{\rho }\right) \left( 1+\frac{\left( \mathbf{c}%
t-\left \vert x\right \vert \right) ^{2}}{1+t}\right) ^{-\rho
}\int_{0}^{t}e^{-\frac{\nu \left( \xi \right) }{2}\left( t-\tau \right) }\nu
\left( \xi \right) \left( 1+\tau \right) ^{-\alpha }d\tau \\
&\leq &C_{\alpha }\left( 16^{\rho }+C_{\rho }\right) \left( 1+t\right)
^{-\alpha }\left( 1+\frac{\left( \mathbf{c}t-\left \vert x\right \vert
\right) ^{2}}{1+t}\right) ^{-\rho }\hbox{.}
\end{eqnarray*}

Combining all the discussion, there exists a constant $C_{\alpha ,\rho }>0$
such that%
\begin{eqnarray*}
&&\int_{0}^{t}e^{-\frac{\nu \left( \xi \right) }{2}\left( t-\tau \right)
}\nu \left( \xi \right) \sup_{y\in \mathbb{R}^{3}}e^{-\frac{\nu _{0}}{2}%
\left( t-\tau +\left \vert x-y\right \vert \right) }\left( 1+\tau \right)
^{-\alpha }\left( 1+\frac{\left( \left \vert y\right \vert -\mathbf{c}\tau
\right) ^{2}}{1+\tau }\right) ^{-\rho }d\tau \\
&\leq &C_{\alpha ,\rho }\left( 1+t\right) ^{-\alpha }\left( 1+\frac{\left(
\mathbf{c}t-\left \vert x\right \vert \right) ^{2}}{1+t}\right) ^{-\rho }%
\hbox{,}
\end{eqnarray*}%
as desired.
\end{proof}

For the second equation of decomposition (\ref{decom-System}), we list some
basic properties of the linearized Boltzmann collision operator $L$ and
nonlinear operator $\Gamma $ as below.

It is well known that the null space of \ $L$
\begin{equation*}
\ker \left( L\right) =\text{span }\{ \chi_{0},\chi_{1}, \chi_{2}, \chi_{3},
\chi_{4}\} \text{,}
\end{equation*}%
is a five-dimensional vector space, where
\begin{equation*}
\chi_{0}=\mathcal{M}^{1/2},\text{ }\chi_{i}=\xi _{i}\mathcal{M}^{1/2},\text{
}\chi_{4}=\frac{1}{\sqrt{6}}\left( \left \vert \xi \right \vert
^{2}-3\right) \mathcal{M}^{1/2} \text{, \ }i=1\text{, }2\text{, }3\text{.}
\end{equation*}
Let $\mathrm{P}_{0}$ be the orthogonal projection with respect to the $%
L_{\xi }^{2}$ inner product onto $\mathrm{\ker }(L)$, and $\mathrm{P}%
_{1}\equiv \mathrm{Id}-\mathrm{P}_{0}$. That is, for any $g\in L^{2}_{\xi}$,
\begin{equation*}
\mathrm{P}_{0}g=\sum_{i=0}^{4}\left<\chi_{i}, g\right>_{\xi}\chi_{i}\,,
\quad \mathrm{P}_{1}g=g-\mathrm{P}_{0}g\,.
\end{equation*}
The solution of the wave propagation is connected to the operator $\mathrm{P}%
_{0}(\xi \cdot \om)\mathrm{P}_0$ for $\om \in \mathbb{S}^{2}$
\begin{equation}  \label{speed}
\left \{%
\begin{array}{l}
\mathrm{P}_{0}\xi \cdot \om E_{j}=\la_{j}E_{j}\,, \\[2mm]
\la_{0}=\mathbf{c}\,,\quad \la_{1}=-\mathbf{c}\,,\quad \la_{2}=\la_{3}=\la%
_{4}=0\,, \mathbf{c}=\sqrt{\frac{5}{3}} \,, \\[2mm]
E_{0}=\sqrt{\frac{3}{10}}\chi_{0}+\sqrt{\frac{1}{2}}\om \cdot \overline{\chi}%
+\sqrt{\frac{1}{5}}\chi_{4}\,, \\[2mm]
E_{1}=\sqrt{\frac{3}{10}}\chi_{0}-\sqrt{\frac{1}{2}}\om \cdot \overline{\chi}%
+\sqrt{\frac{1}{5}}\chi_{4}\,, \\[2mm]
E_{2}= -\sqrt{\frac{2}{5}}\chi_{0}+\sqrt{\frac{3}{5}}\chi_{4}\,, \\[2mm]
E_{3}= \om_{1}\cdot \overline{\chi}\,, \\[2mm]
E_{4}= \om_{2}\cdot \overline{\chi}\,.%
\end{array}
\right.
\end{equation}
where $\overline{\chi}=(\chi_{1},\chi_{2},\chi_{3})$, and $\{ \om_{1},\om%
_{2}, \om \}$ is an orthonormal basis of ${\mathbb{R}}^{3}$.


\begin{lemma}[{\protect \cite{[Liu-Yu2]}}]
\label{pro1} The collision operator $L$ consists of a multiplicative
operator $\nu (\xi )$ and an integral operator $K$, namely,
\begin{equation*}
Lf=-\nu (\xi )f+Kf\,,
\end{equation*}
where $\nu (\xi )\sim (1+|\xi |)$ and
\begin{equation*}
\left \vert Kg\right \vert _{L_{\xi ,\eta+1 }^{\infty }}\leq C\left \vert
g\right \vert _{L_{\xi ,\eta }^{\infty }}\text{, }\eta \in \mathbb{R}\text{,}
\end{equation*}%
for some universal constant $C>0$.

The nonlinear operator $\Gamma $ has the following estimate
\begin{equation*}
\left \vert \nu ^{-1}(\xi )\Gamma (g,h)\right \vert _{L_{\xi ,\eta }^{\infty
}}\leq C\left \vert g\right \vert _{L_{\xi ,\eta }^{\infty }}\left \vert
h\right \vert _{L_{\xi ,\eta }^{\infty }}\text{, }\eta \in \mathbb{R}\text{,}
\end{equation*}%
for some universal constant $C>0$.
\end{lemma}

It is vital to get the space-time structure of the solution of the
linearized Boltzmann equation
\begin{equation*}
\partial _{t}h+\xi \cdot \nabla _{x}h=Lh\text{,}
\end{equation*}%
and we denote $\mathbb{G}^{t}$ as its solution operator.

\begin{lemma}[{\protect \cite{[Liu-Yu2], [Liu-Yu3]}}]
\label{Green} The solution $h=\mathbb{G}^{t}h_{0}$ of the equation
\begin{equation*}
\partial _{t}h+\xi \cdot \nabla _{x}h=Lh\hbox{, }h\left( 0,x,\xi \right)
=h_{0}\left( x,\xi \right) \hbox{.}
\end{equation*}%
satisfies%
\begin{equation*}
\left \vert \mathbb{G}^{t}h_{0}\right \vert _{L_{\xi ,\beta }^{\infty }}\leq
\mathcal{C}_{1}\left[
\begin{array}{l}
\left( 1+t\right) ^{-2}e^{-\frac{\left( \left \vert x\right \vert -\mathbf{c}%
t\right) ^{2}}{D_{0}\left( 1+t\right) }}+\left( 1+t\right) ^{-3/2}e^{-\frac{%
\left \vert x\right \vert ^{2}}{D_{0}\left( 1+t\right) }} \\[2mm]
+\mathbf{1}_{\{ \left \vert x\right \vert \leq \mathbf{c}t\}}\left(
1+t\right) ^{-3/2}\left( 1+\frac{\left \vert x\right \vert ^{2}}{1+t}\right)
^{-3/2}+e^{-\frac{t+|x|}{c_{0}}}%
\end{array}%
\right] \ast _{x}\left \vert h_{0}\right \vert _{L_{\xi ,\beta }^{\infty }}
\end{equation*}%
for any $\beta >3/2$ and some constants $\mathcal{C}_{1}$, $c_{0}$, $D_{0}>0
$, where $h_{0}\in L_{\xi ,\beta }^{\infty }(L_{x}^{\infty }\cap L_{x}^{1})$%
. Moreover, if $\mathrm{P}_{1}h_{0}=0$, then
\begin{equation*}
\left \vert \mathbb{G}^{t}h_{0}\right \vert _{L_{\xi ,\beta }^{\infty }}\leq
\mathcal{C}_{1}\left[
\begin{array}{l}
\left( 1+t\right) ^{-5/2}e^{-\frac{\left( \left \vert x\right \vert -\mathbf{%
c}t\right) ^{2}}{D_{0}\left( 1+t\right) }}+\left( 1+t\right) ^{-2}e^{-\frac{%
\left \vert x\right \vert ^{2}}{D_{0}\left( 1+t\right) }} \\[2mm]
+\mathbf{1}_{\{ \left \vert x\right \vert \leq \mathbf{c}t\}}\left(
1+t\right) ^{-2}\left( 1+\frac{\left \vert x\right \vert ^{2}}{1+t}\right)
^{-3/2}+e^{-\frac{t+|x|}{c_{0}}}%
\end{array}%
\right] \ast _{x}\left \vert h_{0}\right \vert _{L_{\xi ,\beta }^{\infty }}.
\end{equation*}
\end{lemma}

\section{Global wave structure}

\label{global-structure}

In this section we will prove the global wave structure of the equation $(%
\ref{Linearized})$ which is described in the following theorem.

\begin{theorem}
\label{thm-main} Let $\beta >4$ be sufficiently large. Assume that $f_{0}\in
L_{\xi ,\beta }^{\infty }L_{x}^{\infty }$ is compactly supported in the
variable $x$ for all $\xi $
\begin{equation*}
f_{0}\left( x,\xi \right) \equiv 0\text{ for }\left \vert x\right \vert \geq
1\text{, }\xi \in \mathbb{R}^{3}\text{.}
\end{equation*}%
Then there exists $\eps>0$ small enough such that the solution $f$ of $(\ref%
{Linearized})$ exists for $t>0$ and it can be written as $f=f_{1}+\sqrt{%
\mathcal{M}}f_{2}$, where $f_{1}$ and $f_{2}$ satisfy $\left( \ref%
{decom-System}\right) $ with
\begin{equation*}
\left \vert f_{1}\right \vert _{L_{\xi ,\beta }^{\infty }}\leq 2\mathcal{C}%
_{1}\varepsilon \left \Vert f_{0}\right \Vert _{L_{\xi ,\beta }^{\infty
}L_{x}^{\infty }}e^{-\frac{\nu _{0}}{2}\left( t+\left \vert x\right \vert
\right) }\text{,}
\end{equation*}%
and
\begin{align*}
& \quad \left \vert f_{2}\right \vert _{L_{\xi ,\beta }^{\infty }} \\
& \leq \mathfrak{B}\varepsilon \left \Vert f_{0}\right \Vert _{L_{\xi ,\beta
}^{\infty }L_{x}^{\infty }}\left[
\begin{array}{l}
\left( 1+t\right) ^{-2}e^{-\frac{\left( \left \vert x\right \vert -\mathbf{c}%
t\right) ^{2}}{\widehat{D}_{0}\left( 1+t\right) }}+\left( 1+t\right)
^{-3/2}e^{-\frac{\left \vert x\right \vert ^{2}}{\widehat{D}_{0}\left(
1+t\right) }} \\[2mm]
+\mathbf{1}_{\{ \left \vert x\right \vert \leq \mathbf{c}t\}}\left(
1+t\right) ^{-3/2}\left( 1+\frac{\left \vert x\right \vert ^{2}}{1+t}\right)
^{-3/2}+e^{-\frac{t+|x|}{\widehat{c}_{0}}}%
\end{array}%
\right] \\
& \quad +2\mathfrak{C}\left( \mathfrak{B}\varepsilon \left \Vert f_{0}\right
\Vert _{L_{\xi ,\beta }^{\infty }L_{x}^{\infty }}\right) ^{2}\left[ \left(
1+t\right) ^{-2}\left( 1+\frac{\left \vert x\right \vert ^{2}}{1+t}\right)
^{-\frac{3}{2}}+\left( 1+t\right) ^{-\frac{5}{2}}\left( 1+\frac{\left( \left
\vert x\right \vert -\mathbf{c}t\right) ^{2}}{1+t}\right) ^{-1}\right]
\end{align*}%
for some positive constants $\mathfrak{B},\mathcal{C}_{1},\mathfrak{C},\hat{D%
}_{0},\hat{c}_{0}$.
\end{theorem}

\begin{proof}
To clarify the space-time structures of the solutions $f_{1}$ and $f_{2}$ to
$\left( \ref{decom-System}\right) $, we design an iteration $\{f^{n}\}$ with
$f^{n}=f_{1}^{n}+\sqrt{\mathcal{M}}f_{2}^{n}$ , as follows:

\begin{equation}
\left \{
\begin{array}{l}
\partial _{t}f_{1}^{n+1}+\xi \cdot \nabla _{x}f_{1}^{n+1}+\nu \left( \xi
\right) f_{1}^{n+1}=\mathcal{K}_{s}f_{1}^{n}+U(f_{1}^{n},f_{2}^{n})%
\vspace {3mm}
\\
\partial _{t}f_{2}^{n+1}+\xi \cdot \nabla _{x}f_{2}^{n+1}=Lf_{2}^{n+1}+%
\mathcal{K}_{b}f_{1}^{n+1}+\Gamma (f_{2}^{n},f_{2}^{n})%
\end{array}%
\right.  \label{iteration scheme}
\end{equation}%
with
\begin{equation*}
f_{1}^{n+1}\left( 0,x,\xi \right) =\varepsilon f_{0}\left( x,\xi \right)
\text{, }f_{2}^{n+1}\left( 0,x,\xi \right) =0\text{, }\ f_{1}^{0}\left(
t,x,\xi \right) =f_{2}^{0}\left( t,x,\xi \right) =0\text{,}
\end{equation*}%
where

\begin{equation*}
U(f_{1}^{n},f_{2}^{n})=Q\left( f_{1}^{n},f_{1}^{n}\right) +Q\left( f_{1}^{n},%
\sqrt{\mathcal{M}}f_{2}^{n}\right) +Q\left( \sqrt{\mathcal{M}}%
f_{2}^{n},f_{1}^{n}\right) \text{.}
\end{equation*}%
We shall study the space-time structures of $f_{1}^{n}$ and $f_{2}^{n}$. Now
we rewrite (\ref{iteration scheme}) as integral forms:%
\begin{equation*}
\left \{
\begin{array}{l}
f_{1}^{n+1}=\varepsilon \mathbb{S}^{t}f_{0}+\int_{0}^{t}\mathbb{S}^{t-\tau }%
\left[ \mathcal{K}_{s}f_{1}^{n}(\tau )+U(f_{1}^{n},f_{2}^{n})(\tau )\right]
d\tau =:f_{1,1}^{n+1}+f_{1,2}^{n+1}\text{,}%
\vspace {3mm}
\\
f_{2}^{n+1}=\int_{0}^{t}\mathbb{G}^{t-\tau }\mathcal{K}_{b}\left(
f_{1}^{n+1}\right) (\tau )d\tau +\int_{0}^{t}\mathbb{G}^{t-\tau }\Gamma
(f_{2}^{n},f_{2}^{n})(\tau )d\tau =:f_{2,1}^{n+1}+f_{2,2}^{n+1}\text{.}%
\end{array}%
\right.
\end{equation*}

First, one can see that $f_{1}^{1}=f_{1,1}^{1}$, $f_{2}^{1}=f_{2,1}^{1}$ and
$f_{1,2}^{1}=f_{2,2}^{1}=0$. Since $f_{0}$ is compactly supported in the
unit ball with respect to the variable $x$ uniformly for all $\xi $, we have
\begin{equation*}
\left \vert f_{1}^{1}\right \vert _{L_{\xi ,\beta }^{\infty }}\leq \mathcal{C%
}_{1}\varepsilon \Vert f_{0}\Vert _{L_{\xi ,\beta }^{\infty }L_{x}^{\infty
}}e^{-\frac{\nu _{0}}{2}(t+|x|)}
\end{equation*}%
for some constant $\mathcal{C}_{1}$. In view of Lemmas \ref{Estimate-KQ}, %
\ref{Green} and Lemmas \ref{exponential-easy-1} -- \ref{poly-exponential},
we get%
\begin{align*}
\left \vert f_{2}^{1}\right \vert _{L_{\xi ,\beta }^{\infty }}& =\left \vert
\int_{0}^{t}\mathbb{G}^{t-\tau }\mathcal{K}_{b}\left( f_{1}^{1}\right) (\tau
)d\tau \right \vert _{L_{\xi ,\beta }^{\infty }} \\
& \leq 2\mathcal{C}_{1}\varepsilon \left \Vert f_{0}\right \Vert _{L_{\xi
,\beta }^{\infty }L_{x}^{\infty }}\left[ \left( 2\pi \right) ^{\frac{3}{4}%
}e^{\frac{R^{2}}{4}}\left( \frac{C\nu _{1}}{\beta }\left( 1+R\right)
+C_{\beta }\right) \right] \\
& \cdot \mathcal{C}_{1}\left[
\begin{array}{l}
\left( 1+t\right) ^{-2}e^{-\frac{\left( \left \vert x\right \vert -\mathbf{c}%
t\right) ^{2}}{D_{0}\left( 1+t\right) }}+\left( 1+t\right) ^{-3/2}e^{-\frac{%
\left \vert x\right \vert ^{2}}{D_{0}\left( 1+t\right) }} \\[2mm]
+\mathbf{1}_{\{ \left \vert x\right \vert \leq \mathbf{c}t\}}\left(
1+t\right) ^{-3/2}\left( 1+\frac{\left \vert x\right \vert ^{2}}{1+t}\right)
^{-3/2}+e^{-\frac{t+|x|}{c_{0}}}%
\end{array}%
\right] \ast _{x,t}e^{-\frac{\nu _{0}}{2}(t+|x|)} \\
& \leq 2\mathcal{C}_{2}\mathcal{C}_{1}^{2}\left[ \left( 2\pi \right) ^{\frac{%
3}{4}}e^{\frac{R^{2}}{4}}\left( \frac{C\nu _{1}}{\beta }\left( 1+R\right)
+C_{\beta }\right) \right] \\
& \cdot \varepsilon \left \Vert f_{0}\right \Vert _{L_{\xi ,\beta }^{\infty
}L_{x}^{\infty }}\left[
\begin{array}{l}
\left( 1+t\right) ^{-2}e^{-\frac{\left( \left \vert x\right \vert -\mathbf{c}%
t\right) ^{2}}{\widehat{D}_{0}\left( 1+t\right) }}+\left( 1+t\right)
^{-3/2}e^{-\frac{\left \vert x\right \vert ^{2}}{\widehat{D}_{0}\left(
1+t\right) }} \\[2mm]
+\mathbf{1}_{\{ \left \vert x\right \vert \leq \mathbf{c}t\}}\left(
1+t\right) ^{-3/2}\left( 1+\frac{\left \vert x\right \vert ^{2}}{1+t}\right)
^{-3/2}+e^{-\frac{t+|x|}{\widehat{c}_{0}}}%
\end{array}%
\right] \\
& \leq \mathfrak{B}\varepsilon \left \Vert f_{0}\right \Vert _{L_{\xi ,\beta
}^{\infty }L_{x}^{\infty }}\left[
\begin{array}{l}
\left( 1+t\right) ^{-2}e^{-\frac{\left( \left \vert x\right \vert -\mathbf{c}%
t\right) ^{2}}{\widehat{D}_{0}\left( 1+t\right) }}+\left( 1+t\right)
^{-3/2}e^{-\frac{\left \vert x\right \vert ^{2}}{\widehat{D}_{0}\left(
1+t\right) }} \\[2mm]
+\mathbf{1}_{\{ \left \vert x\right \vert \leq \mathbf{c}t\}}\left(
1+t\right) ^{-3/2}\left( 1+\frac{\left \vert x\right \vert ^{2}}{1+t}\right)
^{-3/2}+e^{-\frac{t+|x|}{\widehat{c}_{0}}}%
\end{array}%
\right] \text{.}
\end{align*}

In fact, by definition of $f_{1,1}^{n}$, we have $f_{1,1}^{n}=f_{1,1}^{1}$
for all $n$ and so
\begin{equation*}
\left \vert f_{1,1}^{n}\right \vert _{L_{\xi ,\beta }^{\infty }}\leq
\mathcal{C}_{1}\varepsilon \Vert f_{0}\Vert _{L_{\xi ,\beta }^{\infty
}L_{x}^{\infty }}e^{-\frac{\nu _{0}}{2}(t+|x|)}\text{.}  \label{f11}
\end{equation*}%
Now we choose $\beta >4$ and $R>0$ sufficiently large such that the constant
$\eta \left( \beta ,R\right) $ defined in Lemma \ref{Lemma-gain decay}
satisfies
\begin{equation*}
\eta \left( \beta ,R\right) =\frac{C}{\beta }+\frac{C_{\beta }}{R^{2}}<\frac{%
1}{8}\text{.}
\end{equation*}%
Fixing such $\beta $ and $R$, we shall prove that if $\varepsilon >0$ is
sufficiently small, we have%
\begin{equation*}
\left \vert f_{1,2}^{n}\right \vert _{L_{\xi ,\beta }^{\infty }}\leq
\mathcal{C}_{1}\varepsilon \Vert f_{0}\Vert _{L_{\xi ,\beta }^{\infty
}L_{x}^{\infty }}e^{-\frac{\nu _{0}}{2}(t+|x|)}\text{,}  \label{f12}
\end{equation*}%
\begin{equation*}
\left \vert f_{2,1}^{n}\right \vert _{L_{\xi ,\beta }^{\infty }}\leq
\mathfrak{B}\varepsilon \left \Vert f_{0}\right \Vert _{L_{\xi ,\beta
}^{\infty }L_{x}^{\infty }}\left[
\begin{array}{l}
\left( 1+t\right) ^{-2}e^{-\frac{\left( \left \vert x\right \vert -\mathbf{c}%
t\right) ^{2}}{\widehat{D}_{0}\left( 1+t\right) }}+\left( 1+t\right)
^{-3/2}e^{-\frac{\left \vert x\right \vert ^{2}}{\widehat{D}_{0}\left(
1+t\right) }} \\[2mm]
+\mathbf{1}_{\{ \left \vert x\right \vert \leq \mathbf{c}t\}}\left(
1+t\right) ^{-3/2}\left( 1+\frac{\left \vert x\right \vert ^{2}}{1+t}\right)
^{-3/2}+e^{-\frac{t+|x|}{\widehat{c}_{0}}}%
\end{array}%
\right] \text{,}  \label{f21}
\end{equation*}%
\begin{equation*}
\left \vert f_{2,2}^{n}\right \vert _{L_{\xi ,\beta }^{\infty }}\leq 2%
\mathfrak{C}\left( \mathfrak{B}\varepsilon \left \Vert f_{0}\right \Vert
_{L_{\xi ,\beta }^{\infty }L_{x}^{\infty }}\right) ^{2}\left[ \left(
1+t\right) ^{-2}\left( 1+\frac{\left \vert x\right \vert ^{2}}{1+t}\right)
^{-\frac{3}{2}}+\left( 1+t\right) ^{-\frac{5}{2}}\left( 1+\frac{\left( \left
\vert x\right \vert -\mathbf{c}t\right) ^{2}}{1+t}\right) ^{-1}\right] \text{%
,}  \label{f22}
\end{equation*}%
for all $n\in \mathbb{N}$ by induction on $n$, here the large constant $%
\mathfrak{C}$ $>0$ will be determined later.

By induction hypothesis and Lemma \ref{Estimate-KQ}, we have%
\begin{eqnarray*}
\left \vert \nu ^{-1}U\left( f_{1}^{n},f_{2}^{n}\right) \right \vert
_{L_{\xi ,\beta }^{\infty }} &\leq &C\left( \left \vert f_{1}^{n}\right
\vert _{L_{\xi ,\beta }^{\infty }}^{2}+2\left \vert f_{1}^{n}\right \vert
_{L_{\xi ,\beta }^{\infty }}\left \vert f_{2}^{n}\right \vert _{L_{\xi
,\beta }^{\infty }}\right) \\
&\leq &C\varepsilon \Vert f_{0}\Vert _{L_{\xi ,\beta }^{\infty
}L_{x}^{\infty }}\left( 4\mathcal{C}_{1}+\mathfrak{B}\left( 1+2\mathfrak{CB}%
\varepsilon \left \Vert f_{0}\right \Vert _{L_{\xi ,\beta }^{\infty
}L_{x}^{\infty }}\right) \right) \\
&&\cdot \mathcal{C}_{1}\varepsilon \Vert f_{0}\Vert _{L_{\xi ,\beta
}^{\infty }L_{x}^{\infty }}e^{-\frac{\nu _{0}}{2}(t+|x|)}\text{.}
\end{eqnarray*}%
By Lemma \ref{Lemma-gain decay},
\begin{eqnarray*}
\left \vert f_{1,2}^{n+1}\right \vert _{L_{\xi ,\beta }^{\infty }} &=&\left
\vert \int_{0}^{t}\left \langle \xi \right \rangle ^{\beta }\mathbb{S}%
^{t-\tau }\left[ \mathcal{K}_{s}f_{1}^{n}(\tau )+U(f_{1}^{n},f_{2}^{n})(\tau
)\right] d\tau \right \vert \\
&\leq &\left[ 4\eta \left( \beta ,R\right) +2C\varepsilon \Vert f_{0}\Vert
_{L_{\xi ,\beta }^{\infty }L_{x}^{\infty }}\left( 4\mathcal{C}_{1}+\mathfrak{%
B}\left( 1+2\mathfrak{CB}\varepsilon \left \Vert f_{0}\right \Vert _{L_{\xi
,\beta }^{\infty }L_{x}^{\infty }}\right) \right) \right] \\
&&\cdot \mathcal{C}_{1}\varepsilon \Vert f_{0}\Vert _{L_{\xi ,\beta
}^{\infty }L_{x}^{\infty }}e^{-\frac{\nu _{0}}{2}(t+|x|)} \\
&\leq &\mathcal{C}_{1}\varepsilon \Vert f_{0}\Vert _{L_{\xi ,\beta }^{\infty
}L_{x}^{\infty }}e^{-\frac{\nu _{0}}{2}(t+|x|)}\text{,}
\end{eqnarray*}%
after we choose $\varepsilon >0$ sufficiently small such that%
\begin{equation*}
2C\varepsilon \Vert f_{0}\Vert _{L_{\xi ,\beta }^{\infty }L_{x}^{\infty
}}\left( 4\mathcal{C}_{1}+2\mathfrak{CB}\left( 1+\mathfrak{B}\varepsilon
\left \Vert f_{0}\right \Vert _{L_{\xi ,\beta }^{\infty }L_{x}^{\infty
}}\right) \right) <1/2\text{.}
\end{equation*}%
Therefore,
\begin{equation*}
\left \vert f_{1,2}^{n+1}\right \vert _{L_{\xi ,\beta }^{\infty }}\leq
\mathcal{C}_{1}\varepsilon \Vert f_{0}\Vert _{L_{\xi ,\beta }^{\infty
}L_{x}^{\infty }}e^{-\frac{\nu _{0}}{2}(t+|x|)}\text{,}
\end{equation*}%
and so
\begin{equation*}
\left \vert f_{1}^{n+1}\right \vert _{L_{\xi ,\beta }^{\infty }}\leq 2%
\mathcal{C}_{1}\varepsilon \Vert f_{0}\Vert _{L_{\xi ,\beta }^{\infty
}L_{x}^{\infty }}e^{-\frac{\nu _{0}}{2}(t+|x|)}\text{.}
\end{equation*}%
Likewise $f_{2,1}^{1}$, we have%
\begin{eqnarray*}
\left \vert f_{2,1}^{n+1}\right \vert _{L_{\xi ,\beta }^{\infty }} &=&\left
\vert \int_{0}^{t}\mathbb{G}^{t-\tau }\mathcal{K}_{b}\left(
f_{1}^{n+1}\right) (\tau )d\tau \right \vert _{L_{\xi ,\beta }^{\infty }} \\
&\leq &2\mathcal{C}_{1}\varepsilon \Vert f_{0}\Vert _{L_{\xi ,\beta
}^{\infty }L_{x}^{\infty }}\left[ \left( 2\pi \right) ^{\frac{3}{4}}e^{\frac{%
R^{2}}{4}}\left( \frac{C\nu _{1}}{\beta }\left( 1+R\right) +C_{\beta
}\right) \right] \\
&&\cdot \mathcal{C}_{1}\left[
\begin{array}{l}
\left( 1+t\right) ^{-2}e^{-\frac{\left( \left \vert x\right \vert -\mathbf{c}%
t\right) ^{2}}{D_{0}\left( 1+t\right) }}+\left( 1+t\right) ^{-3/2}e^{-\frac{%
\left \vert x\right \vert ^{2}}{D_{0}\left( 1+t\right) }} \\[2mm]
+\mathbf{1}_{\{ \left \vert x\right \vert \leq \mathbf{c}t\}}\left(
1+t\right) ^{-3/2}\left( 1+\frac{\left \vert x\right \vert ^{2}}{1+t}\right)
^{-3/2}+e^{-\frac{t+|x|}{c_{0}}}%
\end{array}%
\right] \ast _{x,t}e^{-\frac{\nu _{0}}{2}(t+|x|)} \\
&\leq &2\mathcal{C}_{2}\mathcal{C}_{1}^{2}\left[ \left( 2\pi \right) ^{\frac{%
3}{4}}e^{\frac{R^{2}}{4}}\left( \frac{C\nu _{1}}{\beta }\left( 1+R\right)
+C_{\beta }\right) \right] \\
&&\cdot \varepsilon \left \Vert f_{0}\right \Vert _{L_{\xi ,\beta }^{\infty
}L_{x}^{\infty }}\left[
\begin{array}{l}
\left( 1+t\right) ^{-2}e^{-\frac{\left( \left \vert x\right \vert -\mathbf{c}%
t\right) ^{2}}{\widehat{D}_{0}\left( 1+t\right) }}+\left( 1+t\right)
^{-3/2}e^{-\frac{\left \vert x\right \vert ^{2}}{\widehat{D}_{0}\left(
1+t\right) }} \\[2mm]
+\mathbf{1}_{\{ \left \vert x\right \vert \leq \mathbf{c}t\}}\left(
1+t\right) ^{-3/2}\left( 1+\frac{\left \vert x\right \vert ^{2}}{1+t}\right)
^{-3/2}+e^{-\frac{t+|x|}{\widehat{c}_{0}}}%
\end{array}%
\right] \\
&\leq &\mathfrak{B}\varepsilon \left \Vert f_{0}\right \Vert _{L_{\xi ,\beta
}^{\infty }L_{x}^{\infty }}\left[
\begin{array}{l}
\left( 1+t\right) ^{-2}e^{-\frac{\left( \left \vert x\right \vert -\mathbf{c}%
t\right) ^{2}}{\widehat{D}_{0}\left( 1+t\right) }}+\left( 1+t\right)
^{-3/2}e^{-\frac{\left \vert x\right \vert ^{2}}{\widehat{D}_{0}\left(
1+t\right) }} \\[2mm]
+\mathbf{1}_{\{ \left \vert x\right \vert \leq \mathbf{c}t\}}\left(
1+t\right) ^{-3/2}\left( 1+\frac{\left \vert x\right \vert ^{2}}{1+t}\right)
^{-3/2}+e^{-\frac{t+|x|}{\widehat{c}_{0}}}%
\end{array}%
\right] \text{.}
\end{eqnarray*}%
Also, there exits a constant $\mathcal{C}_{3}>1$ depending on $\widehat{D}%
_{0}$ and $\widehat{c}_{0}$ such that%
\begin{eqnarray}
&&\left \vert f_{2,1}^{n+1}\right \vert _{L_{\xi ,\beta }^{\infty }}  \notag
\\
&\leq &2\mathcal{C}_{2}\mathcal{C}_{1}^{2}\left[ \left( 2\pi \right) ^{\frac{%
3}{4}}e^{\frac{R^{2}}{4}}\left( \frac{C\nu _{1}}{\beta }\left( 1+R\right)
+C_{\beta }\right) \right]  \label{f21-1} \\
&&\cdot \varepsilon \left \Vert f_{0}\right \Vert _{L_{\xi ,\beta }^{\infty
}L_{x}^{\infty }}\left[
\begin{array}{l}
\left( 1+t\right) ^{-2}e^{-\frac{\left( \left \vert x\right \vert -\mathbf{c}%
t\right) ^{2}}{\widehat{D}_{0}\left( 1+t\right) }}+\left( 1+t\right)
^{-3/2}e^{-\frac{\left \vert x\right \vert ^{2}}{\widehat{D}_{0}\left(
1+t\right) }} \\[2mm]
+\mathbf{1}_{\{ \left \vert x\right \vert \leq \mathbf{c}t\}}\left(
1+t\right) ^{-\frac{3}{2}}\left( 1+\frac{\left \vert x\right \vert ^{2}}{1+t}%
\right) ^{-\frac{3}{2}}+e^{-\frac{t+|x|}{\widehat{c}_{0}}}%
\end{array}%
\right]  \notag \\
&\leq &2\mathcal{C}_{3}\mathcal{C}_{2}\mathcal{C}_{1}^{2}\left[ \left( 2\pi
\right) ^{\frac{3}{4}}e^{\frac{R^{2}}{4}}\left( \frac{C\nu _{1}}{\beta }%
\left( 1+R\right) +C_{\beta }\right) \right]  \notag \\
&&\cdot \varepsilon \left \Vert f_{0}\right \Vert _{L_{\xi ,\beta }^{\infty
}L_{x}^{\infty }}\left[ \left( 1+t\right) ^{-\frac{3}{2}}\left( 1+\frac{%
\left \vert x\right \vert ^{2}}{1+t}\right) ^{-\frac{3}{2}}+\left(
1+t\right) ^{-2}\left( 1+\frac{\left( \left \vert x\right \vert -\mathbf{c}%
t\right) ^{2}}{1+t}\right) ^{-1}\right]  \notag \\
&\leq &\mathfrak{B}\varepsilon \left \Vert f_{0}\right \Vert _{L_{\xi ,\beta
}^{\infty }L_{x}^{\infty }}\left[ \left( 1+t\right) ^{-\frac{3}{2}}\left( 1+%
\frac{\left \vert x\right \vert ^{2}}{1+t}\right) ^{-\frac{3}{2}}+\left(
1+t\right) ^{-2}\left( 1+\frac{\left( \left \vert x\right \vert -\mathbf{c}%
t\right) ^{2}}{1+t}\right) ^{-1}\right] \text{.}  \notag
\end{eqnarray}

Since $f_{2,2}^{n+1}$ satisfies the equation%
\begin{equation*}
\partial _{t}f_{2,2}^{n+1}+\xi \cdot \nabla
_{x}f_{2,2}^{n+1}=Lf_{2,2}^{n+1}+\Gamma \left( f_{2}^{n},f_{2}^{n}\right)
\text{ with }f_{2,2}^{n+1}\left( 0,x,\xi \right) =0\text{, }
\end{equation*}%
we further decompose $f_{2,2}^{n+1}$ into two parts as $%
f_{2,2}^{n+1}=h_{1}^{n+1}+h_{2}^{n+1}$ where $h_{1}^{n+1}$, $h_{2}^{n+1}$
satisfy the equations
\begin{equation*}
\left \{
\begin{array}{l}
\partial _{t}h_{1}^{n+1}+\xi \cdot \nabla _{x}h_{1}^{n+1}+\nu \left( \xi
\right) h_{1}^{n+1}=\Gamma \left( f_{2}^{n},f_{2}^{n}\right) \text{,}%
\vspace {3mm}
\\
\partial _{t}h_{2}^{n+1}+\xi \cdot \nabla _{x}h_{2}^{n+1}+\nu \left( \xi
\right) h_{2}^{n+1}=Kf_{2,2}^{n+1}\text{,}%
\end{array}%
\right.
\end{equation*}%
with $h_{1}^{n+1}\left( 0,x,\xi \right) =h_{2}^{n+1}\left( 0,x,\xi \right)
=0 $. By induction hypothesis and (\ref{f21-1}),
\begin{eqnarray*}
\left \vert f_{2}^{n}\right \vert _{L_{\xi ,\beta }^{\infty }} &=&\left
\vert f_{2,1}^{n}+f_{2,2}^{n}\right \vert _{_{L_{\xi ,\beta }^{\infty }}} \\
&\leq &\left[ \mathfrak{B}\varepsilon \left \Vert f_{0}\right \Vert _{L_{\xi
,\beta }^{\infty }L_{x}^{\infty }}+2\mathfrak{C}\left( \mathfrak{B}%
\varepsilon \left \Vert f_{0}\right \Vert _{L_{\xi ,\beta }^{\infty
}L_{x}^{\infty }}\right) ^{2}\right] \\
&&\cdot \left[ \left( 1+t\right) ^{-2}\left( 1+\frac{\left( \left \vert
x\right \vert -\mathbf{c}t\right) ^{2}}{1+t}\right) ^{-1}+\left( 1+t\right)
^{-3/2}\left( 1+\frac{\left \vert x\right \vert ^{2}}{1+t}\right) ^{-3/2}%
\right] \\
&\leq &2\mathfrak{B}\varepsilon \left \Vert f_{0}\right \Vert _{L_{\xi
,\beta }^{\infty }L_{x}^{\infty }}\left[ \left( 1+t\right) ^{-2}\left( 1+%
\frac{\left( \left \vert x\right \vert -\mathbf{c}t\right) ^{2}}{1+t}\right)
^{-1}+\left( 1+t\right) ^{-3/2}\left( 1+\frac{\left \vert x\right \vert ^{2}%
}{1+t}\right) ^{-3/2}\right] \text{,}
\end{eqnarray*}%
so that%
\begin{eqnarray}
\left \vert \nu ^{-1}\Gamma \left( f_{2}^{n},f_{2}^{n}\right) \right \vert
_{L_{\xi ,\beta }^{\infty }} &\leq &C\left \vert f_{2}^{n}\right \vert
_{L_{\xi ,\beta }^{\infty }}^{2}  \notag \\
&\leq &8C\left( \mathfrak{B}\varepsilon \left \Vert f_{0}\right \Vert
_{L_{\xi ,\beta }^{\infty }L_{x}^{\infty }}\right) ^{2}  \label{gamma-f2} \\
&&\cdot \left[ \left( 1+t\right) ^{-4}\left( 1+\frac{\left( \left \vert
x\right \vert -\mathbf{c}t\right) ^{2}}{1+t}\right) ^{-2}+\left( 1+t\right)
^{-3}\left( 1+\frac{\left \vert x\right \vert ^{2}}{1+t}\right) ^{-3}\right]
\text{.}  \notag
\end{eqnarray}
In view of Lemmas \ref{Lemma-gain decay-nonmoving} and \ref{Lemma-gain
decay-moving}, we have
\begin{eqnarray}
&&\left \vert \left \langle \xi \right \rangle ^{\beta }h_{1}^{n+1}\right
\vert  \notag \\
&=&\left \vert \left \langle \xi \right \rangle ^{\beta }\int_{0}^{t}\mathbb{%
S}^{t-\tau }\Gamma \left( f_{2}^{n},f_{2}^{n}\right) \left( \tau \right)
d\tau \right \vert  \label{h1} \\
&\leq &8C^{2}\left( \mathfrak{B}\varepsilon \left \Vert f_{0}\right \Vert
_{L_{\xi ,\beta }^{\infty }L_{x}^{\infty }}\right) ^{2}  \notag \\
&&\cdot \left[ \left( 1+t\right) ^{-4}\left( 1+\frac{\left( \left \vert
x\right \vert -\mathbf{c}t\right) ^{2}}{\left( 1+t\right) }\right)
^{-2}+\left( 1+t\right) ^{-3}\left( 1+\frac{\left \vert x\right \vert ^{2}}{%
\left( 1+t\right) }\right) ^{-3}\right]  \notag \\
&\leq &\mathfrak{C}\left( \mathfrak{B}\varepsilon \left \Vert f_{0}\right
\Vert _{L_{\xi ,\beta }^{\infty }L_{x}^{\infty }}\right) ^{2}\left[ \left(
1+t\right) ^{-4}\left( 1+\frac{\left( \left \vert x\right \vert -\mathbf{c}%
t\right) ^{2}}{\left( 1+t\right) }\right) ^{-2}+\left( 1+t\right)
^{-3}\left( 1+\frac{\left \vert x\right \vert ^{2}}{\left( 1+t\right) }%
\right) ^{-3}\right] \text{.}  \notag
\end{eqnarray}

As for $h_{2}^{n+1}$, we have%
\begin{equation*}
\left \vert \left \langle \xi \right \rangle ^{\beta }h_{2}^{n+1}\right
\vert =\left \vert \left \langle \xi \right \rangle ^{\beta }\int_{0}^{t}%
\mathbb{S}^{t-\tau }K\left( f_{2,2}^{n+1}\right) \left( \tau ,x,\xi \right)
d\tau \right \vert \text{.}
\end{equation*}%
It follows from Lemma \ref{pro1} that
\begin{equation*}
\left \vert \nu \left( \xi \right) ^{-1}\left \langle \xi \right \rangle
^{\beta }Kf_{2,2}^{n+1}\right \vert _{L_{\xi }^{\infty }}\leq \frac{C}{\nu
_{0}}\left \vert f_{2,2}^{n+1}\right \vert _{L_{\xi ,\beta -1}^{\infty }}%
\text{.}
\end{equation*}%
In view of Lemma \ref{Green} and (\ref{gamma-f2}), together with convolution
estimates in Section \ref{nonlinear wave interaction}, to be more specific,
Lemmas \ref{Easy1}--\ref{nonmoving-moving}, we obtain
\begin{eqnarray*}
\left \vert f_{2,2}^{n+1}\right \vert _{L_{\xi ,\beta -1}^{\infty }}
&=&\left \vert \int_{0}^{t}\left \langle \xi \right \rangle ^{\beta -1}%
\mathbb{G}^{t-\tau }\Gamma \left( f_{2}^{n},f_{2}^{n}\right) \left( \tau
\right) d\tau \right \vert \\
&\leq &\mathcal{C}_{1}\nu _{1}\left[
\begin{array}{l}
\left( 1+t\right) ^{-\frac{5}{2}}e^{-\frac{\left( \left \vert x\right \vert -%
\mathbf{c}t\right) ^{2}}{D_{0}\left( 1+t\right) }}+\left( 1+t\right)
^{-2}e^{-\frac{\left \vert x\right \vert ^{2}}{D_{0}\left( 1+t\right) }} \\%
[2mm]
+\mathbf{1}_{\{ \left \vert x\right \vert \leq \mathbf{c}t\}}\left(
1+t\right) ^{-2}\left( 1+\frac{\left \vert x\right \vert ^{2}}{1+t}\right)
^{-3/2}+e^{-\frac{t+|x|}{c_{0}}}%
\end{array}%
\right] \ast _{x,t}\left \vert \nu ^{-1}\Gamma \left(
f_{2}^{n},f_{2}^{n}\right) \right \vert _{L_{\xi ,\beta }^{\infty }} \\
&\leq &\mathcal{C}_{1}\nu _{1}8C\left( \mathfrak{B}\varepsilon \left \Vert
f_{0}\right \Vert _{L_{\xi ,\beta }^{\infty }L_{x}^{\infty }}\right) ^{2}%
\left[
\begin{array}{l}
\left( 1+t\right) ^{-\frac{5}{2}}e^{-\frac{\left( \left \vert x\right \vert -%
\mathbf{c}t\right) ^{2}}{D_{0}\left( 1+t\right) }}+\left( 1+t\right)
^{-2}e^{-\frac{\left \vert x\right \vert ^{2}}{D_{0}\left( 1+t\right) }} \\%
[2mm]
+\mathbf{1}_{\{ \left \vert x\right \vert \leq \mathbf{c}t\}}\left(
1+t\right) ^{-2}\left( 1+\frac{\left \vert x\right \vert ^{2}}{1+t}\right)
^{-3/2}+e^{-\frac{t+|x|}{c_{0}}}%
\end{array}%
\right] \\
&&\ast _{x,t}\left[ \left( 1+t\right) ^{-4}\left( 1+\frac{\left( \left \vert
x\right \vert -\mathbf{c}t\right) ^{2}}{1+t}\right) ^{-2}+\left( 1+t\right)
^{-3}\left( 1+\frac{\left \vert x\right \vert ^{2}}{1+t}\right) ^{-3}\right]
\\
&\leq &8C^{2}\mathcal{C}_{1}\nu _{1}\left( \mathfrak{B}\varepsilon \left
\Vert f_{0}\right \Vert _{L_{\xi ,\beta }^{\infty }L_{x}^{\infty }}\right)
^{2} \\
&&\cdot \left[ \left( 1+t\right) ^{-\frac{5}{2}}\left( 1+\frac{\left( \left
\vert x\right \vert -\mathbf{c}t\right) ^{2}}{1+t}\right) ^{-1}+\left(
1+t\right) ^{-2}\left( 1+\frac{\left \vert x\right \vert ^{2}}{1+t}\right)
^{-\frac{3}{2}}\right] \text{.}
\end{eqnarray*}%
Therefore, using Lemmas \ref{Lemma-gain decay-nonmoving} and \ref{Lemma-gain
decay-moving},
\begin{eqnarray}
&&\left \vert h_{2}^{n+1}\right \vert _{L_{\xi ,\beta }^{\infty }}  \notag \\
&\leq &C\cdot \frac{C}{\nu _{0}}8C^{2}\mathcal{C}_{1}\nu _{1}\left(
\mathfrak{B}\varepsilon \left \Vert f_{0}\right \Vert _{L_{\xi ,\beta
}^{\infty }L_{x}^{\infty }}\right) ^{2}  \label{h2} \\
&&\cdot \left[ \left( 1+t\right) ^{-2}\left( 1+\frac{\left \vert x\right
\vert ^{2}}{1+t}\right) ^{-\frac{3}{2}}+\left( 1+t\right) ^{-\frac{5}{2}%
}\left( 1+\frac{\left( \left \vert x\right \vert -\mathbf{c}t\right) ^{2}}{%
1+t}\right) ^{-1}\right]  \notag \\
&\leq &\mathfrak{C}\left( \mathfrak{B}\varepsilon \left \Vert f_{0}\right
\Vert _{L_{\xi ,\beta }^{\infty }L_{x}^{\infty }}\right) ^{2}\left[ \left(
1+t\right) ^{-2}\left( 1+\frac{\left \vert x\right \vert ^{2}}{1+t}\right)
^{-\frac{3}{2}}+\left( 1+t\right) ^{-\frac{5}{2}}\left( 1+\frac{\left( \left
\vert x\right \vert -\mathbf{c}t\right) ^{2}}{1+t}\right) ^{-1}\right]
\notag
\end{eqnarray}%
after we choose $\mathfrak{C}>0$ sufficiently large. Combining (\ref{h1})
and (\ref{h2}) gives%
\begin{equation*}
\left \vert f_{2,2}^{n+1}\right \vert _{L_{\xi ,\beta }^{\infty }}\leq 2%
\mathfrak{C}\left( \mathfrak{B}\varepsilon \left \Vert f_{0}\right \Vert
_{L_{\xi ,\beta }^{\infty }L_{x}^{\infty }}\right) ^{2}\left[ \left(
1+t\right) ^{-2}\left( 1+\frac{\left \vert x\right \vert ^{2}}{1+t}\right)
^{-\frac{3}{2}}+\left( 1+t\right) ^{-\frac{5}{2}}\left( 1+\frac{\left( \left
\vert x\right \vert -\mathbf{c}t\right) ^{2}}{1+t}\right) ^{-1}\right] \text{%
.}
\end{equation*}%
Consequently, we get the desired estimates for $f_{1,2}^{n}$, $f_{2,1}^{n}$,
$f_{2,2}^{n}$ by induction on $n$.

Finally, it is straightforward to prove that $\left \{ \left(
f_{1}^{n},f_{2}^{n}\right) \right \} $ is a Cauchy sequence in $L_{\xi
,\beta }^{\infty }\left( L_{x}^{2}\cap L_{x}^{\infty }\right) $. As a
consequence, $\left( f_{1}^{n},f_{2}^{n}\right) $ converges to $\left(
f_{1},f_{2}\right) $ in $L_{\xi ,\beta }^{\infty }\left( L_{x}^{2}\cap
L_{x}^{\infty }\right) $ and $\left( f_{1},f_{2}\right) $ solves the problem
(\ref{decom-System}) with%
\begin{equation*}
\left \vert f_{1}\right \vert _{L_{\xi ,\beta }^{\infty }}\leq 2\mathcal{C}%
_{1} \varepsilon \left \Vert f_{0}\right \Vert _{L_{\xi ,\beta }^{\infty
}L_{x}^{\infty }}e^{-\frac{\nu _{0}}{2}\left( t+\left \vert x\right \vert
\right) }\text{,}
\end{equation*}%
and
\begin{align*}
&\quad \left \vert f_{2}\right \vert _{L_{\xi ,\beta }^{\infty }} \\
&\leq \mathfrak{B}\varepsilon \left \Vert f_{0}\right \Vert _{L_{\xi ,\beta
}^{\infty }L_{x}^{\infty }}\left[
\begin{array}{l}
\left( 1+t\right) ^{-2}e^{-\frac{\left( \left \vert x\right \vert -\mathbf{c}%
t\right) ^{2}}{\widehat{D}_{0}\left( 1+t\right) }}+\left( 1+t\right)
^{-3/2}e^{-\frac{\left \vert x\right \vert ^{2}}{\widehat{D}_{0}\left(
1+t\right) }} \\[2mm]
+\mathbf{1}_{\{ \left \vert x\right \vert \leq \mathbf{c}t\}}\left(
1+t\right) ^{-3/2}\left( 1+\frac{\left \vert x\right \vert ^{2}}{1+t}\right)
^{-3/2}+e^{-\frac{t+|x|}{\widehat{c}_{0}}}%
\end{array}%
\right] \\
&\quad+2\mathfrak{C}\left( \mathfrak{B}\varepsilon \left \Vert f_{0}\right
\Vert _{L_{\xi ,\beta }^{\infty }L_{x}^{\infty }}\right) ^{2 }\left[ \left(
1+t\right) ^{-2}\left( 1+\frac{\left \vert x\right \vert ^{2}}{1+t}\right)
^{-\frac{3}{2}}+\left( 1+t\right) ^{-\frac{5}{2}}\left( 1+\frac{\left( \left
\vert x\right \vert -\mathbf{c}t\right) ^{2}}{1+t}\right) ^{-1}\right] \,.
\end{align*}%
The proof of Theorem \ref{thm-main} is completed.
\end{proof}

\section{Exponential decay outside the wave cone}

\label{weighted}

In Theorem \ref{thm-main}, the estimate for $f_{1}$ is exponentially sharp,
while the estimate of $f_{2}$ is only polynomially sharp. The reason is
that, to facilitate the closure of nonlinearity, we focused on the structure
inside sound wave cone, so we chose a polynomial ansatz. However, since our
initial data is compactly supported in space and $f_{2}$ represents the
fluid part with an essentially finite propagation speed. Therefore, we
expect a faster decay in the space-like region.

In this section, we will improve the behavior of $f_{2}$ outside the sound
wave cone. Indeed, we can prove that $f_{2}$ has exponential decay both in
space and time there. The result is stated as follows.

\begin{theorem}
\label{outside the cone}Under the same assumption of Theorem \ref{thm-main},
for any $0<\delta \ll 1$, if $\varepsilon >0$ is sufficiently small, there
exists a large positive constant $D$ depending on $\delta $, such that for $%
\left \vert x\right \vert >\left( \mathbf{c}+\delta \right) t$,
\begin{equation*}
\left \vert f_{2}\right \vert _{L_{\xi ,\beta }^{\infty }}\leq C_{\delta }%
\eps e^{-\frac{\left \langle x\right \rangle +t}{D}}\text{,}
\end{equation*}%
where the constant $C_{\delta }>0$ is independent of time and $C_{\delta
}\rightarrow \infty $ as $\delta \rightarrow 0$.
\end{theorem}

To attain this end, we consider the weighted nonlinear equation
corresponding to $\left( \ref{decom-System}\right) $. That is, let $%
u=f_{w}:=wf=u_{1}+\sqrt{\mathcal{M}}u_{2}$, $u_{i}=wf_{i}$ $\left( i=1\text{%
, }2\right) $, where the weight function is given by
\begin{equation*}
w\left( t,x\right) =\exp \left( \frac{\left \langle x\right \rangle -Mt}{%
\ell }\right) \text{,}
\end{equation*}%
$\mathbf{c}<M<\mathbf{c}+1$ and sufficiently large $\ell >0$. Note that by
Theorem \ref{thm-main}, $f_{1}$ decays in space and time exponentially, so
that $u_{1}$ can be controlled if $\ell $ is large enough, that is, there
exists $c_{0}>0$ such that
\begin{equation*}
\Vert u_{1}\Vert _{L_{\xi ,\beta }^{\infty }L_{x}^{\infty }}\text{, }\Vert
u_{1}\Vert _{L_{\xi ,\beta }^{\infty }L_{x}^{2}}\lesssim \eps e^{-c_{0}t}.
\end{equation*}%
In view of $\left( \ref{decom-System}\right) $, $u_{2}$ satisfies the
equation
\begin{equation*}
\left \{
\begin{array}{l}
\partial _{t}u_{2}+\xi \cdot \nabla _{x}u_{2}-w^{-1}\left( \partial
_{t}w+\xi \cdot \nabla _{x}w\right) u_{2}=Lu_{2}+\mathcal{K}_{b}u_{1}+\Gamma
(f_{2},u_{2})%
\vspace {3mm}
\\
u_{2}\left( 0,x,\xi \right) =0\text{.}%
\end{array}%
\right.  \label{Eq-u2}
\end{equation*}%
After choosing $\ell >0$ large such that \ $M\ell ^{-1}$ is small, we have
\begin{equation*}
\widetilde{\nu }\left( t,x,\xi \right) =\nu \left( \xi \right) +w^{-1}\left(
\partial _{t}w+\xi \cdot \nabla _{x}w\right) \geq \frac{3}{4}\nu \left( \xi
\right) \text{.}
\end{equation*}%
Under this situation, we are ready to estimate $u_{2}$. Let $T>0$ be a
finite number. Denote
\begin{equation*}
C_{u_{2},T}^{\infty }=\eps^{-1}\sup_{0\leq t\leq T}\left \Vert u_{2}\right
\Vert _{L_{\xi ,\beta }^{\infty }L_{x}^{\infty }}\text{, \  \  \  \ }%
C_{u_{2},T}^{2}=\eps^{-1}\sup_{0\leq t\leq T}\left \Vert u_{2}\right \Vert
_{L_{\xi ,\beta }^{\infty }L_{x}^{2}}\text{.}
\end{equation*}%
Also denote
\begin{equation*}
\begin{array}{ccc}
C_{f_{2}}^{\infty }=\eps^{-1}\sup \limits_{t\geq 0}\left( 1+t\right)
^{3/2}\left \Vert f_{2}\right \Vert _{L_{\xi ,\beta }^{\infty }L_{x}^{\infty
}}\text{,} &  & C_{f_{2}}^{2}=\eps^{-1}\sup \limits_{t\geq 0}\left(
1+t\right) ^{3/4}\left \Vert f_{2}\right \Vert _{L_{\xi ,\beta }^{\infty
}L_{x}^{2}}\text{,}%
\end{array}%
\end{equation*}%
which are finite due to Theorem \ref{thm-main}.

To estimate $u_{2}$, we design a Picard-type iteration: the zeroth order
approximation $u_{2}^{\left( 0\right) }$ is defined as%
\begin{equation*}
\partial _{t}u_{2}^{\left( 0\right) }+\xi \cdot \nabla _{x}u_{2}^{\left(
0\right) }+\widetilde{\nu }\left( \xi \right) u_{2}^{\left( 0\right) }=%
\mathcal{K}_{b}u_{1}+\Gamma(f_{2}, u_{2})\text{, \ }u_{2}^{\left( 0\right)
}\left( 0,x,\xi \right) =0\text{,\ }
\end{equation*}%
and define the $j$th order approximation $u_{2}^{\left( j\right) }$, $j\geq
1 $, inductively as
\begin{equation*}
\partial _{t}u_{2}^{\left( j\right) }+\xi \cdot \nabla _{x}u_{2}^{\left(
j\right) }+\widetilde{\nu }\left( \xi \right) u_{2}^{\left( j\right)
}=Ku_{2}^{\left( j-1\right) }\text{, \ }u_{2}^{\left( j-1\right) }\left(
0,x,\xi \right) =0\text{.}
\end{equation*}%
Thus, the wave part and the remainder part are defined respectively as
follows:%
\begin{equation*}
W_{w}^{(m)}=\sum_{j=0}^{m}u_{2}^{\left( j\right) }\text{, \ }\mathcal{R}%
_{w}^{\left( m\right) }=u_{2}-W_{w}^{\left( m\right) }\text{,}
\end{equation*}%
$\mathcal{R}_{w}^{\left( m\right) }$ solving the equation%
\begin{equation*}
\partial _{t}\mathcal{R}_{w}^{\left( m\right) }+\xi \cdot \nabla _{x}%
\mathcal{R}_{w}^{\left( m\right) } +\widetilde{\nu }\left( \xi \right)%
\mathcal{R}_{w}^{\left( m\right) }=K\mathcal{R}_{w}^{\left( m\right)
}+Ku_{2}^{\left( m\right) }\text{, \ }\mathcal{R}_{w}^{\left( m\right)
}\left( 0,x,\xi \right) =0\text{.}  \label{Remainder}
\end{equation*}

In fact, $m=6$ is enough to get the $H_{x}^{2}$ regularization estimate for
the remainder part. Following a similar argument as \cite{[LinWangLyuWu]},
we have

\begin{lemma}
\label{Estimate-u^(j)}Let $\beta >4$ be large enough. Then for $0\leq t\leq
T $,%
\begin{equation*}
\left \Vert u_{2}^{\left( j\right) }\right \Vert _{L_{\xi ,\beta }^{\infty
}L_{x}^{\infty }}\lesssim \eps e^{-c_{0}t}+\eps^{2}C_{u_{2},T}^{\infty
}\left( 1+t\right) ^{-\frac{3}{2}}\text{, \  \ }\left \Vert u_{2}^{\left(
j\right) }\right \Vert _{L_{\xi ,\beta }^{\infty }L_{x}^{2}}\lesssim \eps %
e^{-c_{0}t}+\eps^{2} C_{u_{2},T}^{2}\left( 1+t\right) ^{-\frac{3}{2}}\text{,}
\end{equation*}%
for all integers $j\geq 0$.
\end{lemma}

\begin{lemma}
\label{Regularization -u^(6)}For $0\leq t\leq T$,%
\begin{equation*}
\left \Vert u_{2}^{\left( 6\right) }\right \Vert _{L_{\xi
}^{2}H_{x}^{2}}\lesssim \eps e^{-c_{0}t}+\eps^{2} C_{u_{2},T}^{2}\left(
1+t\right) ^{-\frac{3}{2}} \text{.}
\end{equation*}
\end{lemma}

\begin{proposition}
\label{Regularization-R^(6)}(the $H_{x}^{2}$ regularization estimate for $%
\mathcal{R}_{w}^{\left( 6\right) }$) Let $0<\delta \ll 1$. Then for $M=%
\mathbf{c}+25\delta $ and large $\ell >0$, the corresponding weighted
remainder $\mathcal{R}_{w}^{\left( 6\right) }$ satisfies
\begin{equation*}
\left \Vert \mathcal{R}_{w}^{\left( 6\right) }\right \Vert _{L_{\xi
}^{2}H_{x}^{2}}\leq C_{\delta }\left(\eps+\eps^{2} C_{u_{2},T}^{2}\right)
\end{equation*}%
for some positive constant $C_{\delta }$, $C_{\delta }\rightarrow \infty $
as $\delta \rightarrow 0$.
\end{proposition}

\begin{proof}
Note that
\begin{equation*}
w^{-1}\left( \partial _{t}w+\xi \cdot \nabla _{x}w\right) =-\frac{M}{\ell }+%
\frac{\xi \cdot x}{\ell \left \langle x\right \rangle }\text{,}
\end{equation*}%
\begin{equation*}
\left \vert \partial _{x_{i}}\left[ w^{-1}\left( \partial _{t}w+\xi \cdot
\nabla _{x}w\right) \right] \right \vert \leq \frac{2\left \vert \xi \right
\vert }{\ell \left \langle x\right \rangle }\text{, \ }\left \vert \partial
_{x_{i}x_{j}}^{2}\left[ w^{-1}\left( \partial _{t}w+\xi \cdot \nabla
_{x}w\right) \right] \right \vert \leq \frac{6\left \vert \xi \right \vert }{%
\ell \left \langle x\right \rangle ^{2}}
\end{equation*}%
for $1\leq i$, $j\leq 3$. Let $0<\delta \ll 1$. Consider the quantity%
\begin{equation*}
B\left( t\right) :=\left \Vert \mathcal{R}_{w}^{\left( 6\right) }\right
\Vert _{L_{\xi }^{2}L_{x}^{2}}^{2}+\frac{\delta ^{2}}{4\mathfrak{D}^{4}}%
\sum_{i=1}^{3}\left \Vert \partial _{x_{i}}\mathcal{R}_{w}^{\left( 6\right)
}\right \Vert _{L_{\xi }^{2}L_{x}^{2}}^{2}+\left( \frac{\delta ^{2}}{4%
\mathfrak{D}^{4}}\right) ^{2}\sum_{i,j=1}^{3}\left \Vert \partial
_{x_{i}x_{j}}^{2}\mathcal{R}_{w}^{\left( 6\right) }\right \Vert _{L_{\xi
}^{2}L_{x}^{2}}^{2}\text{,}
\end{equation*}%
where the constant $\mathfrak{D}=\left( \int_{\mathbb{R}^{3}}\sum_{j=0}^{4}%
\left \vert \xi \right \vert \chi _{j}^{2}d\xi \right) ^{1/2}>1$. The direct
computation gives
\begin{eqnarray*}
\frac{1}{2}\frac{d}{dt}\left \Vert \mathcal{R}_{w}^{\left( 6\right) }\right
\Vert _{L_{\xi }^{2}L_{x}^{2}}^{2} &=&\int_{\mathbb{R}^{3}}\left \langle
\left( -\frac{M}{\ell }+\frac{\xi \cdot x}{\ell \left \langle x\right
\rangle }\right) \left( \mathrm{P}_{0}\mathcal{R}_{w}^{\left( 6\right) }+%
\mathrm{P}_{1}\mathcal{R}_{w}^{\left( 6\right) }\right) ,\left( \mathrm{P}%
_{0}\mathcal{R}_{w}^{\left( 6\right) }+\mathrm{P}_{1}\mathcal{R}_{w}^{\left(
6\right) }\right) \right \rangle _{\xi }dx \\
&&+\int_{\mathbb{R}^{n}}\left \langle L\mathcal{R}_{w}^{\left( 6\right) },%
\mathcal{R}_{w}^{\left( 6\right) }\right \rangle _{\xi }+\int_{\mathbb{R}%
^{3}}\left \langle Ku_{2}^{\left( 6\right) },\mathcal{R}_{w}^{\left(
6\right) }\right \rangle _{\xi }dx\text{,}
\end{eqnarray*}%
\begin{eqnarray*}
&&\frac{1}{2}\frac{d}{dt}\left \Vert \partial _{x_{i}}\mathcal{R}%
_{w}^{\left( 6\right) }\right \Vert _{L_{\xi }^{2}L_{x}^{2}}^{2} \\
&=&\int_{\mathbb{R}^{3}}\left \langle \left( -\frac{M}{\ell }+\frac{\xi
\cdot x}{\ell \left \langle x\right \rangle }\right) \left( \mathrm{P}%
_{0}\partial _{x_{i}}\mathcal{R}_{w}^{\left( 6\right) }+\mathrm{P}%
_{1}\partial _{x_{i}}\mathcal{R}_{w}^{\left( 6\right) }\right) ,\left(
\mathrm{P}_{0}\partial _{x_{i}}\mathcal{R}_{w}^{\left( 6\right) }+\mathrm{P}%
_{1}\partial _{x_{i}}\mathcal{R}_{w}^{\left( 6\right) }\right) \right
\rangle _{\xi }dx \\
&&+\int_{\mathbb{R}^{3}}\left \langle L\left( \partial _{x_{i}}\mathcal{R}%
_{w}^{\left( 6\right) }\right) ,\partial _{x_{i}}\mathcal{R}_{w}^{\left(
6\right) }\right \rangle _{\xi }dx \\
&&+\int_{\mathbb{R}^{3}}\left \langle \mathcal{R}_{w}^{\left( 6\right)
}\partial _{x_{i}}\left[ w^{-1}\left( \partial _{t}w+\xi \cdot \nabla
_{x}w\right) \right] ,\partial _{x_{i}}\mathcal{R}_{w}^{\left( 6\right)
}\right \rangle _{\xi }dx+\int_{\mathbb{R}^{3}}\left \langle K\partial
_{x_{i}}u_{2}^{\left( 6\right) },\partial _{x_{i}}\mathcal{R}_{w}^{\left(
6\right) }\right \rangle _{\xi }dx\text{,}
\end{eqnarray*}%
\begin{eqnarray*}
&&\frac{1}{2}\frac{d}{dt}\left \Vert \partial _{x_{i}x_{j}}^{2}\mathcal{R}%
_{w}^{\left( 6\right) }\right \Vert _{L_{\xi }^{2}L_{x}^{2}}^{2} \\
&=&\int_{\mathbb{R}^{3}}\left \langle \left( -\frac{M}{\ell }+\frac{\xi
\cdot x}{\ell \left \langle x\right \rangle }\right) \left( \mathrm{P}%
_{0}\partial _{x_{i}x_{j}}^{2}\mathcal{R}_{w}^{\left( 6\right) }+\mathrm{P}%
_{1}\partial _{x_{i}x_{j}}^{2}\mathcal{R}_{w}^{\left( 6\right) }\right)
,\left( \mathrm{P}_{0}\partial _{x_{i}x_{j}}^{2}\mathcal{R}_{w}^{\left(
6\right) }+\mathrm{P}_{1}\partial _{x_{i}x_{j}}^{2}\mathcal{R}_{w}^{\left(
6\right) }\right) \right \rangle _{\xi }dx \\
&&+\int_{\mathbb{R}^{3}}\left \langle L\left( \partial _{x_{i}x_{j}}^{2}%
\mathcal{R}_{w}^{\left( 6\right) }\right) ,\partial _{x_{i}x_{j}}^{2}%
\mathcal{R}_{w}^{\left( 6\right) }\right \rangle _{\xi }+\int_{\mathbb{R}%
^{3}}\left \langle \partial _{x_{i}}\mathcal{R}_{w}^{\left( 6\right) }\cdot
\partial _{x_{j}}\left[ w^{-1}\left( \partial _{t}w+\xi \cdot \nabla
_{x}w\right) \right] ,\partial _{x_{i}x_{j}}^{2}\mathcal{R}_{w}^{\left(
6\right) }\right \rangle _{\xi }dx \\
&&+\int_{\mathbb{R}^{3}}\left \langle \partial _{x_{j}}\mathcal{R}%
_{w}^{\left( 6\right) }\cdot \partial _{x_{i}}\left[ w^{-1}\left( \partial
_{t}w+\xi \cdot \nabla _{x}w\right) \right] ,\partial _{x_{i}x_{j}}^{2}%
\mathcal{R}_{w}^{\left( 6\right) }\right \rangle _{\xi }dx \\
&&+\int_{\mathbb{R}^{3}}\left \langle \mathcal{R}_{w}^{\left( 6\right)
}\partial _{x_{i}x_{j}}^{2}\left[ w^{-1}\left( \partial _{t}w+\xi \cdot
\nabla _{x}w\right) \right] ,\partial _{x_{i}x_{j}}^{2}\mathcal{R}%
_{w}^{\left( 6\right) }\right \rangle _{\xi }dx \\
&&+\int_{\mathbb{R}^{3}}\left \langle K\left( \partial
_{x_{i}x_{j}}^{2}u_{2}^{\left( 6\right) }\right) ,\partial _{x_{i}x_{j}}^{2}%
\mathcal{R}_{w}^{\left( 6\right) }\right \rangle _{\xi }dx\text{.}
\end{eqnarray*}%
Using the fact that in (\ref{speed})
\begin{equation*}
\mathrm{P}_{0}\left( \frac{x\cdot \xi }{\left \vert x\right \vert }\right)
\mathrm{P}_{0}g=\lambda _{1}\left \langle g,E_{1}\right \rangle _{\xi
}E_{1}+\lambda _{2}\left \langle g,E_{2}\right \rangle _{\xi }E_{2}\text{, }%
x\neq 0\text{,}
\end{equation*}%
where $\lambda _{1}=-\lambda _{2}=\mathbf{c}$, we have
\begin{eqnarray*}
&&\frac{1}{2}\frac{d}{dt}B\left( t\right) \\
&=&-\frac{1}{\ell }\left( M-\mathbf{c}-\delta -3\delta -\frac{81\delta ^{3}}{%
4\mathfrak{D}^{4}}\right) \left \Vert \mathrm{P}_{0}\mathcal{R}_{w}^{\left(
6\right) }\right \Vert _{L_{\xi }^{2}L_{x}^{2}}^{2}-\left( \nu _{0}-\frac{1}{%
\ell }-\frac{\mathfrak{D}^{2}}{4\ell \delta }-\frac{3\delta }{\ell \mathfrak{%
D}^{2}}-\frac{81\delta ^{3}}{4\ell \mathfrak{D}^{6}}\right) \left \Vert
\mathrm{P}_{1}\mathcal{R}_{w}^{\left( 6\right) }\right \Vert _{L_{\sigma
}^{2}L_{x}^{2}}^{2} \\
&&+\frac{\delta ^{2}}{4\mathfrak{D}^{4}}\sum_{i=1}^{3}\left \{ -\frac{1}{%
\ell }\left( M-\mathbf{c}-\delta -6\delta \right) \left \Vert \mathrm{P}%
_{0}\partial _{x_{i}}\mathcal{R}_{w}^{\left( 6\right) }\right \Vert _{L_{\xi
}^{2}L_{x}^{2}}^{2}-\left( \nu _{0}-\frac{1}{\ell }-\frac{\mathfrak{D}^{2}}{%
4\ell \delta }-\frac{6\delta ^{3}}{4\ell \mathfrak{D}^{6}}\right) \left
\Vert \mathrm{P}_{1}\partial _{x_{i}}\mathcal{R}_{w}^{\left( 6\right)
}\right \Vert _{L_{\sigma }^{2}L_{x}^{2}}^{2}\right \} \\
&&+\left( \frac{\delta ^{2}}{4\mathfrak{D}^{4}}\right)
^{2}\sum_{i,j=1}^{3}\left \{ -\frac{1}{\ell }\left( M-\mathbf{c}-4\delta
\right) \left \Vert \mathrm{P}_{0}\partial _{x_{i}x_{j}}^{2}\mathcal{R}%
^{\left( 6\right) }\right \Vert _{L_{\xi }^{2}L_{x}^{2}}^{2}-\left( \nu _{0}-%
\frac{1}{\ell }-\frac{\mathfrak{D}^{2}}{4\ell \delta }-\frac{3\delta }{\ell
\mathfrak{D}^{2}}\right) \left \Vert \mathrm{P}_{1}\partial _{x_{i}x_{j}}^{2}%
\mathcal{R}_{w}^{\left( 6\right) }\right \Vert _{L_{\sigma }^{2}L_{x}^{2}\xi
}^{2}\right \} \\
&&+C\left \Vert u_{2}^{\left( 6\right) }\right \Vert _{L_{\xi
}^{2}L_{x}^{2}}\left \Vert \mathcal{R}_{w}^{\left( 6\right) }\right \Vert
_{L_{\xi }^{2}L_{x}^{2}}+\frac{\delta ^{2}}{4\mathfrak{D}^{4}}\sum_{i=1}^{3}%
\left[ C\left \Vert \partial _{x_{i}}u_{2}^{\left( 6\right) }\right \Vert
_{L_{\xi }^{2}L_{x}^{2}}\left \Vert \partial _{x_{i}}\mathcal{R}_{w}^{\left(
6\right) }\right \Vert _{L_{\xi }^{2}L_{x}^{2}}\right] \\
&&+\left( \frac{\delta ^{2}}{4\mathfrak{D}^{4}}\right) ^{2}\sum_{i,j=1}^{3}%
\left[ C\left \Vert \partial _{x_{i}x_{j}}^{2}u_{2}^{\left( 6\right) }\right
\Vert _{L_{\xi }^{2}L_{x}^{2}}\left \Vert \partial _{x_{i}x_{j}}^{2}\mathcal{%
R}_{w}^{\left( 6\right) }\right \Vert _{L_{\xi }^{2}L_{x}^{2}}\right] \\
&\leq &3C\left[ \left \Vert u_{2}^{\left( 6\right) }\right \Vert _{L_{\xi
}^{2}L_{x}^{2}}^{2}+\frac{\delta ^{2}}{4\mathfrak{D}^{4}}\sum_{i=1}^{3}\left%
\Vert \partial _{x_{i}}u_{2}^{\left( 6\right) }\right \Vert _{L_{\xi
}^{2}L_{x}^{2}}^{2}+\left( \frac{\delta ^{2}}{4\mathfrak{D}^{4}}\right)
^{2}\sum_{i,j=1}^{3}\left \Vert \partial _{x_{i}x_{j}}^{2}u_{2}^{\left(
6\right) }\right \Vert _{L_{\xi }^{2}L_{x}^{2}}^{2}\right] ^{1/2}\sqrt{%
B\left( t\right) } \\
&\leq &3C\left \Vert u_{2}^{\left( 6\right) }\right \Vert _{L_{\xi
}^{2}H_{x}^{2}}\sqrt{B\left( t\right) }
\end{eqnarray*}%
if we choose $M=$ $\mathbf{c}+25\delta $, and then choose $\ell >0$
sufficiently large such that $\nu _{0}-\frac{166}{\ell }-\frac{\mathfrak{D}%
^{2}}{4\ell \delta }>0$. Under this choice, we have
\begin{equation*}
\sqrt{B\left( t\right) }\leq \int_{0}^{t}3C\left \Vert u_{2}^{\left(
6\right) }\right \Vert _{L_{\xi }^{2}H_{x}^{2}}d\tau \lesssim \eps+\eps%
^{2}C_{u_{2},T}^{2}
\end{equation*}%
by Lemma \ref{Regularization -u^(6)}, so that
\begin{equation*}
\left \Vert \mathcal{R}_{w}^{\left( 6\right) }\right \Vert _{L_{\xi
}^{2}H_{x}^{2}}\leq C_{\delta }\left( \eps+\eps^{2}C_{u_{2},T}^{2}\right)
\text{,}
\end{equation*}%
for some positive constant $C_{\delta }$, $C_{\delta }\rightarrow \infty $
as $\delta \rightarrow 0$. The proof of this proposition is complete.
\end{proof}

Therefore, the Sobolev inequality implies that
\begin{equation*}
\left \Vert \mathcal{R}_{w}^{\left( 6\right) }\right \Vert _{L_{\xi
}^{2}L_{x}^{\infty }}\lesssim C_{\delta }\left( \eps+\eps^{2}C_{u_{2},T}^{2}%
\right) .
\end{equation*}%
Combining this with Lemma \ref{Estimate-u^(j)}, we have%
\begin{eqnarray*}
\left \Vert u_{2}\right \Vert _{L_{\xi }^{2}L_{x}^{\infty }} &\leq &\left
\Vert W^{\left( 6\right) }\right \Vert _{L_{\xi }^{2}L_{x}^{\infty }}+\left
\Vert \mathcal{R}_{w}^{\left( 6\right) }\right \Vert _{L_{\xi
}^{2}L_{x}^{\infty }} \\
&\lesssim &\left \Vert W^{\left( 6\right) }\right \Vert _{L_{\xi ,\beta
}^{\infty }L_{x}^{\infty }}+\left \Vert \mathcal{R}_{w}^{\left( 6\right)
}\right \Vert _{L_{\xi }^{2}L_{x}^{\infty }} \\
&\lesssim &\eps e^{-c_{0}t}+\eps^{2}C_{u_{2},T}^{\infty }\left( 1+t\right)
^{-\frac{3}{2}}+C_{\delta }\left( \eps+\eps^{2}C_{u_{2},T}^{2}\right) \,.
\end{eqnarray*}%
According to the wave-remainder decomposition, $u_{2}=W^{\left( 7\right) }+%
\mathcal{R}_{w}^{\left( 7\right) }$ and
\begin{equation*}
\mathcal{R}_{w}^{\left( 7\right) }=\int_{0}^{t}\exp \left( -\int_{\tau }^{t}%
\widetilde{\nu }\left( r,x-\left( t-r\right) \xi ,\xi \right) dr\right)
\left( K\mathcal{R}_{w}^{\left( 6\right) }\right) \left( \tau \right) d\tau
\text{.}
\end{equation*}%
Since
\begin{equation*}
\left \Vert \mathcal{R}_{w}^{\left( 7\right) }\right \Vert _{L_{\xi
}^{\infty }L_{x}^{\infty }}\lesssim \int_{0}^{t}e^{-\frac{3}{4}\nu
_{0}\left( t-\tau \right) }\left \Vert \mathcal{R}_{w}^{\left( 6\right)
}\right \Vert _{L_{\xi }^{2}L_{x}^{\infty }}\left( \tau \right) d\tau \text{,%
}
\end{equation*}%
we get%
\begin{equation*}
\left \Vert u_{2}\right \Vert _{L_{\xi }^{\infty }L_{x}^{\infty }}\lesssim %
\eps e^{-c_{0}t}+\eps^{2}C_{u_{2},T}^{\infty }\left( 1+t\right) ^{-\frac{3}{2%
}}+C_{\delta }\left( \eps+\eps^{2}C_{u_{2},T}^{2}\right) \,.
\end{equation*}%
By finite steps of bootstrap argument, we obtain the desired $L_{\xi ,\beta
}^{\infty }L_{x}^{\infty }$ estimate for $u_{2}$. Similarly, by Lemma \ref%
{Estimate-u^(j)}, Proposition \ref{Regularization-R^(6)}, and the bootstrap
argument, we obtain $L_{\xi ,\beta }^{\infty }L_{x}^{2}$ estimate for $u_{2}$
as well. We summarize the estimates for $u_{2}$ as below.

\begin{proposition}
\label{Estimate-u2}Let $\beta >4$ be large enough and let $0<\delta \ll 1$.
Then for $M=\mathbf{c}+25\delta $ and large $\ell >0$, the corresponding $%
u_{2}$ satisfies
\begin{equation*}
C_{u_{2},T}^{2}\leq C_{1}\left[ 1+\eps C_{u_{2},T}^{2}+C_{\delta }\left( 1+%
\eps C_{u_{2},T}^{2}\right) \right]  \label{u2-L2}
\end{equation*}%
\begin{equation*}
C_{u_{2},T}^{\infty }\leq C_{2}\left[ 1+\eps C_{u_{2},T}^{\infty }+C_{\delta
}\left( 1+\eps C_{u_{2},T}^{2}\right) \right]  \label{u2-L^infinity}
\end{equation*}%
for some positive constants $C_{1}$ and $C_{2}$ dependent on $\Vert
f_{0}\Vert _{L_{\xi ,\beta }^{\infty }\left( L_{x}^{\infty }\cap
L_{x}^{2}\right) }$ but independent of $T$ and $\delta $.
\end{proposition}

Now, we are ready to prove Theorem \ref{outside the cone}. For any fixed $%
0<\delta \ll 1$, we take $M=\mathbf{c}+25\delta $, and consider the weight
function%
\begin{equation*}
w\left( x,t\right) =\exp \left( \frac{\left \langle x\right \rangle -Mt}{%
\ell }\right) \, \text{,}
\end{equation*}%
with $\ell >0$ being chosen large. In view of Proposition \ref{Estimate-u2},
choosing $\varepsilon >0$ sufficiently small gives
\begin{equation*}
\left \vert u_{2}\right \vert _{L_{\xi ,\beta }^{\infty }}\leq CC_{\delta }%
\eps \text{,}
\end{equation*}%
for some positive constant $C$ dependent on $\Vert f_{0}\Vert _{L_{\xi
,\beta }^{\infty }\left( L_{x}^{\infty }\cap L_{x}^{2}\right) }$.

Note that $\left \vert x\right \vert >\left( M+\delta \right) t=\left(
\mathbf{c}+26\delta \right) t$, we have%
\begin{eqnarray*}
\left \langle x\right \rangle -Mt &=&\frac{\frac{\delta }{2M}\left \langle
x\right \rangle }{2+\frac{\delta }{2M}}+\frac{2\left \langle x\right \rangle
}{2+\frac{\delta }{2M}}-Mt \\
&>&\frac{\frac{\delta }{2M}\left \langle x\right \rangle }{2+\frac{\delta }{%
2M}}+\frac{\frac{\delta }{2}t}{2+\frac{\delta }{2M}}\geq \frac{\delta \left(
\left \langle x\right \rangle +t\right) }{4M+\delta }\text{.}
\end{eqnarray*}%
Then for $\left \vert x\right \vert >\left( \mathbf{c}+26\delta \right) t$,
\begin{equation*}
e^{\frac{\delta \left( \left \langle x\right \rangle +t\right) }{\left(
4M+\delta \right) \ell }}\left \vert f_{2}\right \vert _{L_{\xi ,\beta
}^{\infty }}\leq \left \vert u_{2}\right \vert _{L_{\xi ,\beta }^{\infty
}}\leq C C_{\delta }\eps
\end{equation*}%
which implies that
\begin{equation*}
\left \vert f_{2}\right \vert _{L_{\xi ,\beta }^{\infty }}\leq C C_{\delta }%
\eps e^{-\frac{\delta \left( \left \langle x\right \rangle +t\right) }{%
\left( 4M+\delta \right) \ell }} \text{.}
\end{equation*}%
The proof of Theorem \ref{outside the cone} is complete.

\section{The transition from polynomial tail to Gaussian tail}

\label{f1-velocity}

In this section, we study the behavior of $f_{1}$ in the microscopic
variable as $t$ increases. We provide a quantitative description of how the
velocity variable transitions from a polynomial tail to a Gaussian tail. The
result is as follows:

\begin{theorem}
\label{Improvement-f1}Let $\beta >4$, $R>0$ be sufficiently large and $%
0<\kappa <\min \{1/4$, $\nu _{0}/2\}$. Assume that $f_{0}$ satisfies the
same condition as in Theorem \ref{thm-main}. If $\varepsilon >0$ is
sufficiently small, then there exists a constant $\overline{C}_{\beta }>0$
only depending on $\beta $ such that%
\begin{equation*}
\left \vert f_{1}\left( t,x,\xi \right) \right \vert \leq \overline{C}%
_{\beta }\varepsilon \left \langle \xi \right \rangle ^{-\beta }e^{-\kappa
\rho \left( \xi ,t\right) }\left \Vert f_{0}\left( x,\xi \right) \right
\Vert _{L_{\xi ,\beta }^{\infty }L_{x}^{\infty }}
\end{equation*}%
for all $t\geq 0$, $x\in \mathbb{R}^{3}$, $\xi \in \mathbb{R}^{3}$, where $%
\rho \left( \xi ,t\right) =\left \langle \xi \right \rangle \left(
\left
\langle \xi \right \rangle \wedge t\right) $. Consequently, for any
fixed $t>0$, then
\begin{equation*}
|\mathbf{1}_{\{ \left<\xi \right> \leq t\}}f_{1}|\lesssim \varepsilon \left
\langle \xi \right \rangle ^{-\beta } e^{-\ka \left<\xi \right>^{2}}
\end{equation*}
and
\begin{equation*}
|\mathbf{1}_{\{ \left<\xi \right> >t\}}f_{1}|\lesssim \varepsilon \left
\langle \xi \right \rangle ^{-\beta } e^{-\ka t^{2}}\,.
\end{equation*}
\end{theorem}

Firstly, we need some estimate for the weight function $e^{\kappa \rho
\left( \xi ,t\right) }$. For simplicity, we define
\begin{equation}
\rho \left( \xi ,t\right) =\overline{\rho }\left( \left \vert \xi \right
\vert ,t\right) \hbox{, }\xi \in \mathbb{R}^{3}\hbox{, }t\geq 0\hbox{,}
\label{weight-exponent}
\end{equation}%
with $\overline{\rho }\left( z,t\right) =\left( 1+z^{2}\right) ^{1/2}\left(
\left( 1+z^{2}\right) ^{1/2}\wedge t\right) $, $t\geq 0$. Now, we give an
inequality regarding the function $\overline{\rho }$.

\begin{lemma}
\label{Weight-estimate}Let $\overline{\rho }\left( z,t\right) $ be a
function defined by
\begin{equation*}
\overline{\rho }\left( z,t\right) =\left( 1+z^{2}\right) ^{1/2}\left( \left(
1+z^{2}\right) ^{1/2}\wedge t\right)
\end{equation*}%
for $z\in \mathbb{R}$ and $t\geq 0$. Then
\begin{equation*}
\overline{\rho }\left( a,t\right) +\overline{\rho }\left( b,t\right) \geq
\overline{\rho }\left( \sqrt{a^{2}+b^{2}},t\right)
\end{equation*}%
for all $a$, $b\geq 0$.
\end{lemma}

\proof%
If either $a=0$ or $b=0$, it is trivial. By symmetry, we may assume that $%
b>a>0$ and so $\sqrt{a^{2}+b^{2}}>b>a>0$. In the following we discuss the
inequality in four cases.$%
\vspace {3mm}%
$

Case 1: $t\geq \sqrt{a^{2}+b^{2}}>b>a>0$. Then
\begin{equation*}
\overline{\rho }\left( a,t\right) +\overline{\rho }\left( b,t\right)
=2+a^{2}+b^{2}=1+\overline{\rho }\left( \sqrt{a^{2}+b^{2}},t\right) \geq
\overline{\rho }\left( \sqrt{a^{2}+b^{2}},t\right) \hbox{.}
\end{equation*}

Case 2: $\sqrt{a^{2}+b^{2}}>b>a>t$. Then
\begin{eqnarray*}
\overline{\rho }\left( \sqrt{a^{2}+b^{2}},t\right) &=&\left(
1+a^{2}+b^{2}\right) ^{1/2}t\leq \left( 1+a^{2}+1+b^{2}\right) ^{1/2}t \\
&\leq &\left( 1+a^{2}\right) ^{1/2}t+\left( 1+b^{2}\right) ^{1/2}t=\overline{%
\rho }\left( a,t\right) +\overline{\rho }\left( b,t\right) \hbox{.}
\end{eqnarray*}

Case 3: $\sqrt{a^{2}+b^{2}}>t>b>a$. Then%
\begin{equation*}
\overline{\rho }\left( a,t\right) +\overline{\rho }\left( b,t\right)
=2+a^{2}+b^{2}\geq \left( 1+a^{2}+b^{2}\right) ^{1/2}t=\overline{\rho }%
\left( \sqrt{a^{2}+b^{2}},t\right) \hbox{.}
\end{equation*}

Case 4: $\sqrt{a^{2}+b^{2}}>b>t>a$. Then%
\begin{eqnarray*}
&&\left[ \overline{\rho }\left( a,t\right) +\overline{\rho }\left(
b,t\right) \right] ^{2}-\left[ \overline{\rho }\left( \sqrt{a^{2}+b^{2}}%
,t\right) \right] ^{2} \\
&=&\left[ 1+a^{2}+\left( 1+b^{2}\right) ^{1/2}t\right] ^{2}-\left(
1+a^{2}+b^{2}\right) t^{2} \\
&=&\left( 1+a^{2}\right) ^{2}+2\left( 1+a^{2}\right) \left( 1+b^{2}\right)
^{1/2}t-a^{2}t^{2}\geq 2\left( 1+a^{2}\right) t^{2}-a^{2}t^{2}\geq 0\hbox{.}
\end{eqnarray*}

As $a=b>0$, it is a consequence of Case 1-Case 3. Gathering all the cases,
the proof is complete. $%
\hfill%
\square $

According to this lemma, we can prove the following weighted estimate
regarding $Q$ and $\mathcal{K}$.

\begin{lemma}
\label{Weighted-Q}Let $\beta >4$, $\kappa >0$ and let $\rho \left( \xi
,t\right) $ be defined by $(\ref{weight-exponent})$. Then there exists a
constant $C_{\beta }>0$ depending only on $\beta $ such that
\begin{equation}  \label{rho1}
\int_{\mathbb{R}^{3}}\int_{\mathbb{S}^{2}}\left \vert \left( \xi -\xi _{\ast
}\right) \cdot n\right \vert \frac{\left \langle \xi \right \rangle ^{\beta }%
}{\left \langle \xi _{\ast }^{\prime }\right \rangle ^{\beta }\left \langle
\xi ^{\prime }\right \rangle ^{\beta }}\frac{e^{\kappa \rho \left( \xi
,t\right) }}{e^{\kappa \rho \left( \xi _{\ast }^{\prime },t\right)
}e^{\kappa \rho \left( \xi ^{\prime },t\right) }}dnd\xi_{*}\leq \left( \frac{%
C}{\beta }+\frac{C_{\beta }}{\left \langle \xi \right \rangle ^{2}}\right)
\nu \left( \xi \right) \hbox{.}
\end{equation}
Moreover, if $0<\ka<\frac{1}{4}$, we have
\begin{equation}  \label{rho2}
\int_{\mathbb{R}^{3}}\int_{\mathbb{S}^{2}}\left \vert \left( \xi -\xi _{\ast
}\right) \cdot n\right \vert \frac{\left \langle \xi \right \rangle ^{\beta
} }{\left \langle \xi _{\ast }^{\prime }\right \rangle ^{\beta }}e^{-\frac{1%
}{4}|\xi^{\prime 2}}\frac{e^{\kappa \rho \left( \xi ,t\right) }}{e^{\kappa
\rho \left( \xi _{\ast }^{\prime },t\right) }e^{\kappa \rho \left( \xi
^{\prime },t\right) }}dnd\xi_{*}\leq \left( \frac{C}{\beta }+\frac{C_{\beta }%
}{\left \langle \xi \right \rangle ^{2}}\right) \nu \left( \xi \right) %
\hbox{.}
\end{equation}
\end{lemma}

\begin{proof}
Let $\eta =\xi ^{\prime }-\xi $ and $\omega =\xi _{\ast }-\xi ^{\prime }$.
Then $\eta \bot \omega $, $\xi ^{\prime }=\xi +\eta $, $\xi _{\ast }=\xi
+\eta +\omega $, $\xi _{\ast }^{\prime }=\xi +\omega $. By change of
variables,
\begin{eqnarray*}
&&\int_{\mathbb{R}^{3}}\int_{\mathbb{S}^{2}}\left \vert \left( \xi -\xi
_{\ast }\right) \cdot n\right \vert \frac{\left \langle \xi \right \rangle
^{\beta }}{\left \langle \xi _{\ast }^{\prime }\right \rangle ^{\beta }\left
\langle \xi ^{\prime }\right \rangle ^{\beta }}\frac{e^{\kappa \rho \left(
\xi ,t\right) }}{e^{\kappa \rho \left( \xi _{\ast }^{\prime },t\right)
}e^{\kappa \rho \left( \xi ^{\prime },t\right) }}dnd\xi_{*} \\
&\leq &\int_{\mathbb{R}^{3}}\int_{\{ \omega \hbox{ }:\hbox{ }\eta \bot
\omega \}}\frac{1}{\left \vert \eta \right \vert }\frac{\left \langle \xi
\right \rangle ^{\beta }}{\left \langle \xi +\omega \right \rangle ^{\beta
}\left \langle \xi +\eta \right \rangle ^{\beta }}\frac{e^{\kappa \rho
\left( \xi ,t\right) }}{e^{\kappa \rho \left( \xi +\omega ,t\right)
}e^{\kappa \rho \left( \xi +\eta ,t\right) }}d\omega d\eta =:\mathrm{I}%
\hbox{.}
\end{eqnarray*}%
We split $\xi $ into two parts%
\begin{equation*}
\xi =\xi _{\Vert }+\xi _{\bot }\hbox{, }\xi _{\bot }=\frac{\left( \xi \cdot
\eta \right) \eta }{\left \vert \eta \right \vert ^{2}}\hbox{, }\xi _{\Vert
}\Vert \omega \hbox{, }\xi _{\bot }\bot \omega \hbox{, }\left \vert \xi
\right \vert ^{2}=\left \vert \xi _{\Vert }\right \vert ^{2}+\left \vert \xi
_{\bot }\right \vert ^{2}\hbox{.}
\end{equation*}%
Then%
\begin{eqnarray*}
\mathrm{I} &=&\int_{\mathbb{R}^{3}}\frac{\left \langle \xi \right \rangle
^{\beta }e^{\kappa \left[ \rho \left( \xi ,t\right) -\rho \left( \xi +\eta
,t\right) \right] }}{\left \langle \xi +\eta \right \rangle ^{\beta }\left
\vert \eta \right \vert }\int_{\{ \omega \hbox{ }:\hbox{ }\eta \bot \omega
\}}\frac{e^{-\kappa \left( 1+\left \vert \xi _{\Vert }+\omega \right \vert
^{2}+\left \vert \xi _{\bot }\right \vert ^{2}\right) ^{1/2}\left( \left(
1+\left \vert \xi _{\Vert }+\omega \right \vert ^{2}+\left \vert \xi _{\bot
}\right \vert ^{2}\right) ^{1/2}\wedge t\right) }}{\left( 1+\left \vert \xi
_{\Vert }+\omega \right \vert ^{2}+\left \vert \xi _{\bot }\right \vert
^{2}\right) ^{\beta /2}}d\omega d\eta \\
&=&\int_{\mathbb{R}^{3}}\frac{\left \langle \xi \right \rangle ^{\beta
}e^{\kappa \left[ \rho \left( \xi ,t\right) -\rho \left( \xi +\eta ,t\right) %
\right] }}{\left \langle \xi +\eta \right \rangle ^{\beta }\left \vert \eta
\right \vert }\int_{\mathbb{R}^{2}}\frac{e^{-\kappa \left( 1+\left \vert
\varpi \right \vert ^{2}+\left \vert \xi _{\bot }\right \vert ^{2}\right)
^{1/2}\left[ \left( 1+\left \vert \varpi \right \vert ^{2}+\left \vert \xi
_{\bot }\right \vert ^{2}\right) ^{1/2}\wedge t\right] }}{\left( 1+\left
\vert \varpi \right \vert ^{2}+\left \vert \xi _{\bot }\right \vert
^{2}\right) ^{\beta /2}}d\varpi d\eta \hbox{.}
\end{eqnarray*}%
Making a change of variable by $\varpi =\sqrt{1+\left \vert \xi _{\bot
}\right \vert ^{2}}z$ gives%
\begin{eqnarray*}
&&\int_{\mathbb{R}^{2}}\frac{e^{-\kappa \left( 1+\left \vert \varpi \right
\vert ^{2}+\left \vert \xi _{\bot }\right \vert ^{2}\right) ^{1/2}\left[
\left( 1+\left \vert \varpi \right \vert ^{2}+\left \vert \xi _{\bot }\right
\vert ^{2}\right) ^{1/2}\wedge t\right] }}{\left( 1+\left \vert \varpi
\right \vert ^{2}+\left \vert \xi _{\bot }\right \vert ^{2}\right) ^{\beta
/2}}d\varpi \\
&=&\left( 1+\left \vert \xi _{\bot }\right \vert ^{2}\right) ^{-\frac{\beta
-2}{2}}\int_{\mathbb{R}^{2}}\frac{e^{-\kappa \left( 1+\left \vert \xi _{\bot
}\right \vert ^{2}\right) ^{1/2}\left( 1+\left \vert z\right \vert
^{2}\right) ^{1/2}\left[ \left( 1+\left \vert \xi _{\bot }\right \vert
^{2}\right) ^{1/2}\left( 1+\left \vert z\right \vert ^{2}\right)
^{1/2}\wedge t\right] }}{\left( 1+\left \vert z\right \vert ^{2}\right)
^{\beta /2}}dz \\
&=&\frac{2\pi }{\left \langle \xi _{\bot }\right \rangle ^{\beta -2}}%
\int_{1}^{\infty }\frac{e^{-\kappa \left \langle \xi _{\bot }\right \rangle
\zeta \left[ \left \langle \xi _{\bot }\right \rangle \zeta \wedge t\right] }%
}{\zeta ^{\beta -1}}d\zeta \hbox{.}
\end{eqnarray*}%
If $\left \langle \xi _{\bot }\right \rangle \geq t$, then
\begin{eqnarray*}
\frac{2\pi }{\left \langle \xi _{\bot }\right \rangle ^{\beta -2}}%
\int_{1}^{\infty }\frac{e^{-\kappa \left \langle \xi _{\bot }\right \rangle
\zeta \left[ \left \langle \xi _{\bot }\right \rangle \zeta \wedge t\right] }%
}{\zeta ^{\beta -1}}d\zeta &=&\frac{2\pi }{\left \langle \xi _{\bot }\right
\rangle ^{\beta -2}}\int_{1}^{\infty }\frac{e^{-\kappa \left \langle \xi
_{\bot }\right \rangle t\zeta }}{\zeta ^{\beta -1}}d\zeta \\
&\leq &\frac{2\pi }{\beta -2}\frac{e^{-\kappa \left \langle \xi _{\bot
}\right \rangle t}}{\left \langle \xi _{\bot }\right \rangle ^{\beta -2}}=%
\frac{2\pi }{\beta -2}\frac{e^{-\kappa \rho (\xi _{\bot },t)}}{\left \langle
\xi _{\bot }\right \rangle ^{\beta -2}}\hbox{.}
\end{eqnarray*}%
If $\left \langle \xi _{\bot }\right \rangle <t$, then $\left \langle \xi
_{\bot }\right \rangle \zeta <t$ for $\zeta <t/\left \langle \xi _{\bot
}\right \rangle $ and thus%
\begin{eqnarray*}
\frac{2\pi }{\left \langle \xi _{\bot }\right \rangle ^{\beta -2}}%
\int_{1}^{\infty }\frac{e^{-\kappa \left \langle \xi _{\bot }\right \rangle
\zeta \left[ \left \langle \xi _{\bot }\right \rangle \zeta \wedge \tau %
\right] }}{\zeta ^{\beta -1}}d\zeta &=&\frac{2\pi }{\left \langle \xi _{\bot
}\right \rangle ^{\beta -2}}\left[ \int_{1}^{\frac{t}{\left \langle \xi
_{\bot }\right \rangle }}\frac{e^{-\kappa \left \langle \xi _{\bot }\right
\rangle ^{2}\zeta ^{2}}}{\zeta ^{\beta -1}}d\zeta +\int_{\frac{t}{\left
\langle \xi _{\bot }\right \rangle }}^{\infty }\frac{e^{-\kappa \left
\langle \xi _{\bot }\right \rangle \zeta t}}{\zeta ^{\beta -1}}d\zeta \right]
\\
&\leq &\frac{2\pi }{\beta -2}\frac{1}{\left \langle \xi _{\bot }\right
\rangle ^{\beta -2}}\left( e^{-\kappa \left \langle \xi _{\bot }\right
\rangle ^{2}}+e^{-\kappa t^{2}}\right) \leq \frac{4\pi }{\beta -2}\frac{%
e^{-\kappa \rho (\xi _{\bot },t)}}{\left \langle \xi _{\bot }\right \rangle
^{\beta -2}}\hbox{.}
\end{eqnarray*}%
Therefore,
\begin{equation*}
\mathrm{I}\leq \frac{4\pi }{\beta -2}\int_{\mathbb{R}^{3}}\frac{\left
\langle \xi \right \rangle ^{\beta }}{\left \langle \xi +\eta \right \rangle
^{\beta }\left \langle \xi _{\bot }\right \rangle ^{\beta -2}}\frac{%
e^{\kappa \left[ \rho \left( \xi ,t\right) -\rho \left( \xi +\eta ,t\right)
-\rho (\xi _{\bot },t)\right] }}{\left \vert \eta \right \vert }d\eta %
\hbox{.}
\end{equation*}%
Observe that $\left \vert \xi _{\bot }\right \vert ^{2}+\left \vert \xi
+\eta \right \vert ^{2}=\left \vert \xi _{\bot }\right \vert
^{2}+\left
\vert \xi \right \vert ^{2}+2\xi \cdot \eta +\left \vert \eta
\right \vert ^{2}=\left \vert \xi _{\bot }\right \vert ^{2}+\left \vert \xi
\right \vert ^{2}+2\xi _{\bot }\cdot \eta +\left \vert \eta \right \vert
^{2}=\left \vert \xi \right \vert ^{2}+\left \vert \xi _{\bot }+\eta
\right
\vert ^{2}$. Hence,
\begin{eqnarray*}
\rho \left( \xi ,t\right) -\rho \left( \xi +\eta ,t\right) -\rho (\xi _{\bot
},t) &=&\overline{\rho }\left( \left \vert \xi \right \vert ,t\right) -%
\overline{\rho }\left( \left \vert \xi +\eta \right \vert ,t\right) -%
\overline{\rho }(\left \vert \xi _{\bot }\right \vert ,t) \\
&\leq &\overline{\rho }\left( \sqrt{\left \vert \xi _{\bot }\right \vert
^{2}+\left \vert \xi +\eta \right \vert ^{2}},t\right) -\overline{\rho }%
\left( \left \vert \xi +\eta \right \vert ,t\right) -\overline{\rho }(\left
\vert \xi _{\bot }\right \vert ,t)\leq 0\hbox{,}
\end{eqnarray*}%
the last inequality being valid due to Lemma \ref{Weight-estimate}. It
follows
\begin{equation*}
\mathrm{I}\leq \frac{4\pi }{\beta -2}\int_{\mathbb{R}^{3}}\frac{1}{\left
\vert \eta \right \vert }\frac{\left \langle \xi \right \rangle ^{\beta }}{%
\left \langle \xi +\eta \right \rangle ^{\beta }\left \langle \xi _{\bot
}\right \rangle ^{\beta -2}}d\eta =\frac{4\pi }{\beta -2}\int_{\mathbb{R}%
^{3}}\frac{1}{\left \vert \xi ^{\prime }-\xi \right \vert }\frac{\left
\langle \xi \right \rangle ^{\beta }}{\left \langle \xi ^{\prime }\right
\rangle ^{\beta }\left \langle \xi _{\bot }\right \rangle ^{\beta -2}}d\xi
^{\prime }\hbox{.}
\end{equation*}%
According to the estimate in the proof of \cite[Lemma 2.11]{[Cao]}, there
exists a constant $C_{\beta }>0$ depending only on $\beta $ such that%
\begin{equation*}
\mathrm{I}\leq \frac{4\pi }{\beta -2}\int_{\mathbb{R}^{3}}\frac{1}{\left
\vert \xi ^{\prime }-\xi \right \vert }\frac{\left \langle \xi \right
\rangle ^{\beta }}{\left \langle \xi ^{\prime }\right \rangle ^{\beta }\left
\langle \xi _{\bot }\right \rangle ^{\beta -2}}d\xi ^{\prime }\leq \left(
\frac{C}{\beta }+\frac{C_{\beta }}{\left \langle \xi \right \rangle ^{2}}%
\right) \nu \left( \xi \right) \hbox{,}
\end{equation*}%
where $C>0$ is a universal constant. This completes the proof of (\ref{rho1}%
).

For the estimate of (\ref{rho2}), applying the same argument as (\ref{rho1}%
), one gets
\begin{align*}
&\quad \int_{\mathbb{R}^{3}}\int_{\mathbb{S}^{2}}\left \vert \left( \xi -\xi
_{\ast }\right) \cdot n\right \vert \frac{\left \langle \xi \right \rangle
^{\beta } }{\left \langle \xi _{\ast }^{\prime }\right \rangle ^{\beta }}e^{-%
\frac{1}{4}|\xi^{\prime 2}}\frac{e^{\kappa \rho \left( \xi ,t\right) }}{%
e^{\kappa \rho \left( \xi _{\ast }^{\prime },t\right) }e^{\kappa \rho \left(
\xi ^{\prime },t\right) }}dnd\xi_{*} \\
&\leq \frac{4\pi }{\beta -2}\int_{\mathbb{R}^{3}}\frac{1}{\left \vert \xi
^{\prime }-\xi \right \vert }\frac{\left \langle \xi \right \rangle ^{\beta }%
}{e^{\frac{1}{4}|\xi^{\prime 2}}\left \langle \xi _{\bot }\right \rangle
^{\beta -2}}d\xi ^{\prime }\,.
\end{align*}
Then one can modify the argument in \cite[Lemma 2.11]{[Cao]} (in fact, it is
easier) to conclude our result.
\end{proof}

\begin{corollary}
\label{Cor-Q}Let $\beta >4$, $\kappa >0$ and let $\rho \left( \xi ,t\right) $
be defined by $(\ref{weight-exponent})$. Then there exists a constant $%
C_{\beta }^{^{\prime \prime }}>0$ depending only on $\beta $ such that%
\begin{equation*}
e^{\kappa \rho \left( \xi ,t\right) }\left \langle \xi \right \rangle
^{\beta }\left \vert Q\left( g,h\right) \right \vert \leq C_{\beta
}^{^{\prime \prime }}\nu \left( \xi \right) \left \vert e^{\kappa \rho
\left( \xi ,t\right) }g\right \vert _{L_{\xi ,\beta }^{\infty }}\left \vert
e^{\kappa \rho \left( \xi ,t\right) }h\right \vert _{L_{\xi ,\beta }^{\infty
}}\hbox{.}
\end{equation*}
Moreover, for any $R>0$, and $0<\ka<\frac{1}{4}$, we have
\begin{equation*}
\chi _{\{ \left \vert \xi \right \vert \geq R\}}e^{\kappa \rho \left( \xi
,t\right) }\left \langle \xi \right \rangle ^{\beta }\left \vert \mathcal{K}%
f\right \vert \leq \left( \frac{C}{\beta }+\frac{C_{\beta }}{R^{2}}\right)
\nu \left( \xi \right) \left \vert e^{\ka \rho(\xi,t)}f\right \vert _{L_{\xi
,\beta }^{\infty }}\,.
\end{equation*}
\end{corollary}

\begin{proof}
We only prove the estimate of $Q$ since the estimate of $\mathcal{K}$ is
similar. By definition of $Q$,
\begin{eqnarray*}
e^{\kappa \rho \left( \xi ,t\right) }\left \langle \xi \right \rangle
^{\beta }\left \vert Q\left( g,h\right) \right \vert &\leq &\left \vert
e^{\kappa \rho \left( \xi ,t\right) }g\right \vert _{L_{\xi ,\beta }^{\infty
}}\left \vert e^{\kappa \rho \left( \xi ,t\right) }h\right \vert _{L_{\xi
,\beta }^{\infty }} \\
&&\cdot \left[ \int_{\mathbb{R}^{3}}\int_{\mathbb{S}^{2}}\left \vert \left(
\xi -\xi _{\ast }\right) \cdot n\right \vert \frac{\left \langle \xi \right
\rangle ^{\beta }}{\left \langle \xi _{\ast }^{\prime }\right \rangle
^{\beta }\left \langle \xi ^{\prime }\right \rangle ^{\beta }}\frac{%
e^{\kappa \rho \left( \xi ,t\right) }}{e^{\kappa \rho \left( \xi _{\ast
}^{\prime },t\right) }e^{\kappa \rho \left( \xi ^{\prime },t\right) }}dnd\xi
_{\ast }\right. \\
&&\left. +\int_{\mathbb{R}^{3}}\int_{\mathbb{S}^{2}}\left \vert \left( \xi
-\xi _{\ast }\right) \cdot n\right \vert \frac{1}{\left \langle \xi _{\ast
}\right \rangle ^{\beta }e^{\kappa \rho \left( \xi _{\ast },\tau \right) }}%
dnd\xi _{\ast }\right] \hbox{.}
\end{eqnarray*}%
In view of Lemma \ref{Weighted-Q},%
\begin{equation*}
\int_{\mathbb{R}^{3}}\int_{\mathbb{S}^{2}}\left \vert \left( \xi -\xi _{\ast
}\right) \cdot n\right \vert \frac{\left \langle \xi \right \rangle ^{\beta }%
}{\left \langle \xi _{\ast }^{\prime }\right \rangle ^{\beta }\left \langle
\xi ^{\prime }\right \rangle ^{\beta }}\frac{e^{\kappa \rho \left( \xi
,t\right) }}{e^{\kappa \rho \left( \xi _{\ast }^{\prime },t\right)
}e^{\kappa \rho \left( \xi ^{\prime },t\right) }}dnd\xi _{\ast }\leq
C_{\beta }^{\prime }\nu \left( \xi \right)
\end{equation*}%
for some $C_{\beta }^{\prime }>0$. By straightforward computation,
\begin{eqnarray*}
&&\int_{\mathbb{R}^{3}}\int_{\mathbb{S}^{2}}\left \vert \left( \xi -\xi
_{\ast }\right) \cdot n\right \vert \frac{1}{\left \langle \xi _{\ast
}\right \rangle ^{\beta }e^{\kappa \rho \left( \xi _{\ast },\tau \right) }}%
dnd\xi _{\ast } \\
&\leq &4\pi \left[ \int_{\left \vert \xi _{\ast }\right \vert \leq \left
\vert \xi \right \vert }\left \vert \xi -\xi _{\ast }\right \vert \frac{1}{%
\left \langle \xi _{\ast }\right \rangle ^{\beta }}d\xi _{\ast }+\int_{\left
\vert \xi _{\ast }\right \vert >\left \vert \xi \right \vert }\left \vert
\xi -\xi _{\ast }\right \vert \frac{1}{\left \langle \xi _{\ast }\right
\rangle ^{\beta }}d\xi _{\ast }\right] \\
&\leq &\left( 4\pi \right) ^{2}\left( \frac{2\left \vert \xi \right \vert }{%
\beta -3}+\frac{2}{\beta -4}\frac{1}{\left \langle \xi \right \rangle
^{\beta -4}}\right) \hbox{.}
\end{eqnarray*}%
Hence,
\begin{equation*}
e^{\kappa \rho \left( \xi ,t\right) }\left \langle \xi \right \rangle
^{\beta }\left \vert Q\left( g,h\right) \right \vert \leq C_{\beta
}^{^{\prime \prime }}\nu \left( \xi \right) \left \vert e^{\kappa \rho
\left( \xi ,t\right) }g\right \vert _{L_{\xi ,\beta }^{\infty }}\left \vert
e^{\kappa \rho \left( \xi ,t\right) }h\right \vert _{L_{\xi ,\beta }^{\infty
}}
\end{equation*}%
for some constant $C_{\beta }^{^{\prime \prime }}>0$ only depending upon $%
\beta $, as desired.
\end{proof}

\begin{proof}[\textbf{Proof of Theorem \protect \ref{Improvement-f1}}]
In virtue of $\left( \ref{decom-System}\right) $, $f_{1}$ can be expressed
as
\begin{equation*}
f_{1}=\mathbb{S}^{t}f_{0}+\int_{0}^{t}\mathbb{S}^{t-\tau }\left[
K_{s}f_{1}+Q\left( f_{1},f_{1}\right) +Q\left( f_{1},\sqrt{\mathcal{M}}%
f_{2}\right) +Q\left( \sqrt{\mathcal{M}}f_{2},f_{1}\right) \right] \left(
\tau \right) d\tau \hbox{.}
\end{equation*}%
Let $T>0$ be any number and denote
\begin{equation*}
u\left( t,x,\xi \right) =\left \langle \xi \right \rangle ^{\beta }e^{\kappa
\rho \left( \xi ,t\right) }f_{1}\left( t,x,\xi \right) \hbox{.}
\end{equation*}%
Multiplying $\left \langle \xi \right \rangle ^{\beta }e^{\kappa \rho \left(
\xi ,t\right) }$ on both sides of the integral equation gives
\begin{eqnarray*}
\left \langle \xi \right \rangle ^{\beta }e^{\kappa \rho \left( \xi
,t\right) }f_{1} &=&\varepsilon \mathbb{S}^{t}\left( \left \langle \xi
\right \rangle ^{\beta }e^{\kappa \rho \left( \xi ,t\right) }f_{0}\right)
+\int_{0}^{t}\left \langle \xi \right \rangle ^{\beta }e^{\kappa \rho \left(
\xi ,t\right) }\mathbb{S}^{t-\tau }Q\left( f_{1},f_{1}\right) \left( \tau
\right) d\tau \\
&&+\int_{0}^{t}\left \langle \xi \right \rangle ^{\beta }e^{\kappa \rho
\left( \xi ,t\right) }\mathbb{S}^{t-\tau }\left[ K_{s}f_{1}+Q\left( f_{1},%
\sqrt{\mathcal{M}}f_{2}\right) +Q\left( \sqrt{\mathcal{M}}f_{2},f_{1}\right) %
\right] \left( \tau \right) d\tau \\
&=&: \mathrm{I}+\mathrm{II}+\mathrm{III}\hbox{.}
\end{eqnarray*}

At first, it is easy to see
\begin{equation}
\left \vert \mathrm{I}\right \vert \leq \varepsilon \left \Vert \left
\langle \xi \right \rangle ^{\beta }e^{-\nu _{0}\left \langle \xi \right
\rangle t+\kappa \left \langle \xi \right \rangle t}\left \vert f_{0}\right
\vert \right \Vert _{L_{\xi }^{\infty }L_{x}^{\infty }}\leq \varepsilon
\left \Vert f_{0}\right \Vert _{L_{\xi ,\beta }^{\infty }L_{x}^{\infty }}%
\hbox{,}  \label{weighted-I}
\end{equation}%
since $\kappa \rho \left( \xi ,t\right) \leq \kappa \left \langle \xi
\right
\rangle t$.

As for $\mathrm{II}$ and $\mathrm{III}$, we can find
\begin{equation*}
e^{-\nu \left( \xi \right) \left( t-\tau \right) }e^{\kappa \rho \left( \xi
,t\right) }\leq e^{-\left( \nu _{0}-\kappa \right) \left \langle \xi \right
\rangle \left( t-\tau \right) }e^{\kappa \rho \left( \xi ,\tau \right) }
\end{equation*}%
for all $\xi $ and $0\leq \tau \leq t$. To see this, for $\left \langle \xi
\right \rangle >t$, we have $\rho \left( \xi ,t\right) =\left \langle \xi
\right \rangle t$ and thus%
\begin{equation*}
e^{-\nu \left( \xi \right) \left( t-\tau \right) }e^{\kappa \rho \left( \xi
,t\right) }=e^{-\nu \left( \xi \right) \left( t-\tau \right) }e^{\kappa
\left \langle \xi \right \rangle t}\leq e^{-\left( \nu _{0}-\kappa \right)
\left \langle \xi \right \rangle \left( t-\tau \right) }e^{\kappa \left
\langle \xi \right \rangle \tau }=e^{-\left( \nu _{0}-\kappa \right) \left
\langle \xi \right \rangle \left( t-\tau \right) }e^{\kappa \rho \left( \xi
,\tau \right) }\hbox{;}
\end{equation*}%
for $\left \langle \xi \right \rangle \leq t$, we have $\rho \left( \xi
,t\right) =\left \langle \xi \right \rangle ^{2}$, so that
\begin{eqnarray*}
e^{-\nu \left( \xi \right) \left( t-\tau \right) }e^{\kappa \rho \left( \xi
,t\right) } &\leq &e^{-\nu _{0}\left \langle \xi \right \rangle \left(
t-\tau \right) }e^{\kappa \left \langle \xi \right \rangle ^{2}}=e^{-\left(
\nu _{0}-\kappa \right) \left \langle \xi \right \rangle \left( t-\tau
\right) }e^{\kappa \left \langle \xi \right \rangle \left( \left \langle \xi
\right \rangle -t\right) +\kappa \left \langle \xi \right \rangle \tau } \\
&\leq &e^{-\left( \nu _{0}-\kappa \right) \left \langle \xi \right \rangle
\left( t-\tau \right) }e^{\kappa \rho \left( \xi ,\tau \right) }\hbox{.}
\end{eqnarray*}%
Therefore,%
\begin{equation*}
\left \vert \mathrm{II}\right \vert \leq \int_{0}^{t}e^{-\left( \nu
_{0}-\kappa \right) \left \langle \xi \right \rangle \left( t-\tau \right)
}e^{\kappa \rho \left( \xi ,\tau \right) }\left \langle \xi \right \rangle
^{\beta }\left \vert Q\left( f_{1},f_{1}\right) \left( \tau \right) \right
\vert d\tau \hbox{,}
\end{equation*}%
\begin{equation*}
\left \vert \mathrm{III}\right \vert \leq \int_{0}^{\infty }e^{-\left( \nu
_{0}-\kappa \right) \left \langle \xi \right \rangle \left( t-\tau \right)
}\left \langle \xi \right \rangle ^{\beta }e^{\kappa \rho \left( \xi ,\tau
\right) }\left \vert \left[ K_{s}f_{1}+Q\left( f_{1},\sqrt{\mathcal{M}}%
f_{2}\right) +Q\left( \sqrt{\mathcal{M}}f_{2},f_{1}\right) \right] \left(
\tau \right) \right \vert d\tau \hbox{.}
\end{equation*}%
By Corollary \ref{Cor-Q}, we have
\begin{equation}
\left \vert \mathrm{II}\right \vert \leq C_{\beta }^{\prime \prime
}\sup_{0\leq t\leq T}\left \Vert u\right \Vert _{L_{\xi }^{\infty
}L_{x}^{\infty }}^{2}\int_{0}^{t}e^{-\left( \nu _{0}-\kappa \right) \left
\langle \xi \right \rangle \left( t-\tau \right) }\nu \left( \xi \right)
d\tau \leq \frac{C_{\beta }^{\prime \prime }\nu_{1}}{\nu _{0}-\kappa }\left(
\sup_{0\leq t\leq T}\left \Vert u\right \Vert _{L_{\xi }^{\infty
}L_{x}^{\infty }}\right) ^{2}\hbox{.}  \label{weighted-II}
\end{equation}%
Regarding $\mathrm{III}$, in view of Theorem \ref{thm-main} and Corollary %
\ref{Cor-Q}, together the fact that%
\begin{equation*}
e^{\kappa \rho \left( \xi ,\tau \right) }\sqrt{\mathcal{M}}=e^{\kappa \left
\langle \xi \right \rangle \left( \left \langle \xi \right \rangle \wedge
\tau \right) }e^{-\frac{\left \vert \xi \right \vert ^{2}}{4}}\leq e^{\kappa
\left \langle \xi \right \rangle ^{2}-\frac{\left \vert \xi \right \vert ^{2}%
}{4}}\leq e^{1/4}\hbox{,}
\end{equation*}%
we have
\begin{eqnarray*}
&&\int_{0}^{\infty }e^{-\left( \nu _{0}-\kappa \right) \left \langle \xi
\right \rangle \left( t-\tau \right) }\left \langle \xi \right \rangle
^{\beta }e^{\kappa \rho \left( \xi ,\tau \right) }\left \vert Q\left( f_{1},%
\sqrt{\mathcal{M}}f_{2}\right) +Q\left( \sqrt{\mathcal{M}}f_{2},f_{1}\right)
\right \vert d\tau \\
&\leq &2C_{\beta }^{\prime \prime }\sup_{0\leq t\leq T}\left \Vert u\right
\Vert _{L_{\xi }^{\infty }L_{x}^{\infty }}\cdot \sup_{0\leq t\leq T}\left
\Vert e^{\kappa \rho \left( \xi ,\tau \right) }\sqrt{\mathcal{M}}f_{2}\right
\Vert _{L_{\xi ,\beta }^{\infty }L_{x}^{\infty }}\int_{0}^{t}e^{-\left( \nu
_{0}-\kappa \right) \left \langle \xi \right \rangle \left( t-\tau \right)
}\nu \left( \xi \right) d\tau \\
&\leq &\left( \frac{\widetilde{C}_{\beta }\varepsilon }{\nu _{0}-\kappa }%
\left \Vert f_{0}\right \Vert _{L_{\xi ,\beta }^{\infty }L_{x}^{\infty
}}\right) \sup_{0\leq t\leq T}\left \Vert u\right \Vert _{L_{\xi }^{\infty
}L_{x}^{\infty }}\hbox{.}
\end{eqnarray*}%
By Corollary \ref{Cor-Q},
\begin{eqnarray*}
&&\int_{0}^{\infty }e^{-\left( \nu _{0}-\kappa \right) \left \langle \xi
\right \rangle \left( t-\tau \right) }\left \langle \xi \right \rangle
^{\beta }e^{\kappa \rho \left( \xi ,\tau \right) }\left \vert
K_{s}f_{1}\right \vert d\tau \\
&\leq &\left( \frac{C}{\beta }+\frac{C_{\beta }}{R^{2}}\right) \sup_{0\leq
t\leq T}\left \Vert u\right \Vert _{L_{\xi }^{\infty }L_{x}^{\infty
}}\int_{0}^{\infty }e^{-\left( \nu _{0}-\kappa \right) \left \langle \xi
\right \rangle \left( t-\tau \right) }\nu \left( \xi \right) d\tau \\
&\leq &\frac{\nu_{1}}{\nu _{0}-\kappa }\left( \frac{C}{\beta }+\frac{%
C_{\beta }}{R^{2}}\right) \sup_{0\leq t\leq T}\left \Vert u\right \Vert
_{L_{\xi }^{\infty }L_{x}^{\infty }}\leq \frac{1}{4}\sup_{0\leq t\leq
T}\left \Vert u\right \Vert _{L_{\xi }^{\infty }L_{x}^{\infty }}\hbox{,}
\end{eqnarray*}%
after choosing $\beta >4$ and $R>0$ sufficiently large. Hence,%
\begin{equation}
\left \vert \mathrm{III}\right \vert \leq \left( \frac{1}{4}+\frac{%
\widetilde{C}_{\beta }\varepsilon }{\nu _{0}-\kappa }\left \Vert f_{0}\right
\Vert _{L_{\xi ,\beta }^{\infty }L_{x}^{\infty }}\right) \sup_{0\leq t\leq
T}\left \Vert u\right \Vert _{L_{\xi }^{\infty }L_{x}^{\infty }}\hbox{.}
\label{weighted-III}
\end{equation}

Combining (\ref{weighted-I}), (\ref{weighted-II}) and (\ref{weighted-III}),
we have%
\begin{eqnarray*}
\sup_{0\leq t\leq T}\left \Vert u\right \Vert _{L_{\xi }^{\infty
}L_{x}^{\infty }} &\leq &\varepsilon \left \Vert f_{0}\right \Vert _{L_{\xi
,\beta }^{\infty }L_{x}^{\infty }}+\left( \frac{1}{4}+\frac{\widetilde{C}%
_{\beta }\varepsilon }{\nu _{0}-\kappa }\left \Vert f_{0}\right \Vert
_{L_{\xi ,\beta }^{\infty }L_{x}^{\infty }}\right) \sup_{0\leq t\leq T}\left
\Vert u\right \Vert _{L_{\xi }^{\infty }L_{x}^{\infty }} \\
&&+\frac{C_{\beta }^{\prime \prime }\nu_{1}}{\nu _{0}-\kappa }\left(
\sup_{0\leq t\leq T}\left \Vert u\right \Vert _{L_{\xi }^{\infty
}L_{x}^{\infty }}\right) ^{2}\hbox{.}
\end{eqnarray*}%
We may assume that $\frac{C_{\beta }^{\prime \prime }\nu_{1}}{\nu
_{0}-\kappa }>1$. Choosing $\varepsilon >0$ sufficiently small such that%
\begin{equation*}
\frac{\widetilde{C}_{\beta }\varepsilon }{\nu _{0}-\kappa }\left \Vert
f_{0}\right \Vert _{L_{\xi ,\beta }^{\infty }L_{x}^{\infty }}<\frac{1}{4}%
\hbox{ and }\left( \frac{4C_{\beta }^{\prime \prime }\nu_{1}}{\nu
_{0}-\kappa }\right) ^{2}\varepsilon \left \Vert f_{0}\left( x,\xi \right)
\right \Vert _{L_{\xi ,\beta }^{\infty }L_{x}^{\infty }}<1\hbox{,}
\end{equation*}%
we obtain%
\begin{equation*}
\sup_{0\leq t\leq T}\left \Vert u\right \Vert _{L_{\xi }^{\infty
}L_{x}^{\infty }}\leq 2\varepsilon \left \Vert f_{0}\right \Vert _{L_{\xi
,\beta }^{\infty }L_{x}^{\infty }}+\frac{2C_{\beta }^{\prime \prime }\nu_{1}%
}{\nu _{0}-\kappa }\left( \sup_{0\leq t\leq T}\left \Vert u\right \Vert
_{L_{\xi }^{\infty }L_{x}^{\infty }}\right) ^{2}\hbox{.}
\end{equation*}%
Since $\left \Vert u\left( 0,x,\xi \right) \right \Vert _{L_{\xi }^{\infty
}L_{x}^{\infty }}=\left \Vert f_{1}\left( 0,x,\xi \right) \right \Vert
_{L_{\xi ,\beta }^{\infty }L_{x}^{\infty }}=\varepsilon \left \Vert
f_{0}\left( x,\xi \right) \right \Vert _{L_{\xi ,\beta }^{\infty
}L_{x}^{\infty }}$,
\begin{equation*}
\sup_{0\leq t\leq T}\left \Vert u\right \Vert _{L_{\xi }^{\infty
}L_{x}^{\infty }}\leq \left( \frac{4C_{\beta }^{\prime \prime }\nu_{1}}{\nu
_{0}-\kappa }\right) \varepsilon \left \Vert f_{0}\left( x,\xi \right)
\right \Vert _{L_{\xi ,\beta }^{\infty }L_{x}^{\infty }}\hbox{,}
\end{equation*}%
for any finite $T>0$. Consequently,
\begin{equation*}
\left \langle \xi \right \rangle ^{\beta }e^{\kappa \rho \left( \xi
,t\right) }\left \vert f_{1}\left( t,x,\xi \right) \right \vert =\left \vert
u\left( t,x,\xi \right) \right \vert \leq \overline{C}_{\beta }\varepsilon
\left \Vert f_{0}\left( x,\xi \right) \right \Vert _{L_{\xi ,\beta }^{\infty
}L_{x}^{\infty }}
\end{equation*}%
for some constant $\overline{C}_{\beta }>0$ depending only on $\beta $,
i.e.,
\begin{equation*}
\left \vert f_{1}\left( t,x,\xi \right) \right \vert \leq \overline{C}%
_{\beta }\varepsilon \left \langle \xi \right \rangle ^{-\beta }e^{-\kappa
\rho \left( \xi ,t\right) }\left \Vert f_{0}\left( x,\xi \right) \right
\Vert _{L_{\xi ,\beta }^{\infty }L_{x}^{\infty }}
\end{equation*}%
for all $t\geq 0$, $x\in \mathbb{R}^{3}$, $\xi \in \mathbb{R}^{3}$.
\end{proof}

\section{Some convolution estimates}

\label{wave-interaction} In this section, we will compute the interactions
between different wave patterns, which are essential for determining the
precise space-time structure of the solution. Although these estimates
appear complicated, there is a clear physical picture behind them (see
Section \ref{nonlinear wave interaction} for some illustrations). The proofs
in fact aim to translate this heuristic picture into refined convolution
estimates.

To facilitate the estimates, we decompose space-time domain into the
following $5$ regions:
\begin{align*}
D_{1}& =\left \{ |x|\leq \sqrt{1+t}\right \} \,, \\
\ D_{2}& =\left \{ \mathbf{c}t-\sqrt{1+t}\leq |x|\leq \mathbf{c}t+\sqrt{1+t}%
\right \} \,, \\
D_{3}& =\left \{ |x|\geq \mathbf{c}t+\sqrt{1+t}\right \} \,, \\
D_{4}& =\left \{ \sqrt{1+t}\leq |x|\leq \frac{1}{2}\mathbf{c}t\right \} \,,
\\
D_{5}& =\left \{ \frac{1}{2}\mathbf{c}t\leq |x|\leq \mathbf{c}t-\sqrt{1+t}%
\right \} \,.
\end{align*}

\subsection{Linear interaction}

\begin{lemma}[Diffusion wave convolved with exponential decay]
\label{exponential-easy-1}
\begin{equation*}
\left( 1+t\right) ^{-3/2}e^{-\frac{\left \vert x\right \vert ^{2}}{%
D_{0}\left( 1+t\right) }}\ast _{x,t}e^{-\frac{t+\left \vert x\right \vert }{%
c_{0}}}\lesssim \left( 1+t\right) ^{-3/2}e^{-\frac{\left \vert x\right \vert
^{2}}{\widehat{D}\left( 1+t\right) }}+e^{-\frac{t+\left \vert x\right \vert
}{\widehat{c}}}\hbox{,}
\end{equation*}%
for some constants $\widehat{c}$ and $\widehat{D}>0$.
\end{lemma}

\begin{lemma}[Huygens wave convolved with exponential decay]
\label{moving-exponential}
\begin{eqnarray*}
A &=&\left( 1+t\right) ^{-2}e^{-\frac{\left( \left \vert x\right \vert -%
\mathbf{c}t\right) ^{2}}{D_{0}\left( 1+t\right) }}\ast _{x,t}e^{-\frac{%
t+\left \vert x\right \vert }{c_{0}}} \\
&\lesssim &\left( 1+t\right) ^{-2}e^{-\frac{\left( \left \vert x\right \vert
-\mathbf{c}t\right) ^{2}}{\widehat{D}\left( 1+t\right) }}+e^{-\frac{t+\left
\vert x\right \vert }{\widehat{c}}}\hbox{,}
\end{eqnarray*}%
for some constants $\widehat{c}$ and $\widehat{D}>0$.
\end{lemma}

\begin{lemma}[Riesz wave convolved with exponential decay]
\label{poly-exponential}
\begin{eqnarray*}
B &=&\mathbf{1}_{\{ \left \vert x\right \vert \leq \mathbf{c}t\}}\left(
1+t\right) ^{-3/2}\left( 1+\frac{\left \vert x\right \vert ^{2}}{1+t}\right)
^{-3/2}\ast _{x,t}e^{-\frac{t+\left \vert x\right \vert }{c_{0}}} \\
&\lesssim &\mathbf{1}_{\{ \left \vert x\right \vert <\mathbf{c}t\}}\left(
1+t\right) ^{-3/2}\left( 1+\frac{\left \vert x\right \vert ^{2}}{1+t}\right)
^{-3/2}+e^{-\frac{t+\left \vert x\right \vert }{\widehat{c}_{0}}}+\left(
1+t\right) ^{-2}e^{-\frac{\left( \left \vert x\right \vert -\mathbf{c}%
t\right) ^{2}}{\widehat{D}_{0}\left( 1+t\right) }}\hbox{,}
\end{eqnarray*}%
for some constants $\widehat{c}_{0},\widehat{D}_{0}>0$.
\end{lemma}

\bigskip The proof of Lemma \ref{exponential-easy-1} is easy and hence we
omit it.

\begin{proof}[\textbf{Proof of Lemma \protect \ref{moving-exponential}}
(Huygens wave convolved with exponential decay)]
We rewrite%
\begin{equation*}
A=\int_{0}^{t}\int_{\mathbb{R}^{3}}\left( 1+\tau \right) ^{-2}e^{-\frac{%
\left( \left \vert y\right \vert -\mathbf{c}\tau \right) ^{2}}{D_{0}\left(
1+\tau \right) }}e^{-\frac{\left( t-\tau \right) +\left \vert x-y\right
\vert }{c_{0}}}dyd\tau \hbox{.}
\end{equation*}%
We discuss the integral $A$ in each domain $D_{i}$ $\left( 1\leq i\leq
5\right) $ for which $\left( x,t\right) $ belongs to. \newline
\newline
\textbf{Case 1:} $\left( x,t\right) \in D_{1}$. Direct computation gives%
\begin{eqnarray*}
A &\lesssim &\int_{0}^{\frac{t}{2}}\int_{\mathbb{R}^{3}}\left( 1+\tau
\right) ^{-2}e^{-\frac{t}{2c_{0}}-\frac{\left \vert x-y\right \vert }{c_{0}}%
}dyd\tau +\int_{\frac{t}{2}}^{t}\int_{\left \vert y\right \vert \leq \frac{%
\mathbf{c}\tau }{2}}\left( 1+\tau \right) ^{-2}e^{-\frac{\left( \left \vert
y\right \vert -\mathbf{c}\tau \right) ^{2}}{D_{0}\left( 1+\tau \right) }}e^{-%
\frac{\left( t-\tau \right) +\left \vert x-y\right \vert }{c_{0}}}dyd\tau \\
&&+\int_{\frac{t}{2}}^{t}\int_{\left \vert y\right \vert >\frac{\mathbf{c}%
\tau }{2}}\left( 1+\tau \right) ^{-2}e^{-\frac{\left( \left \vert y\right
\vert -\mathbf{c}\tau \right) ^{2}}{D_{0}\left( 1+\tau \right) }}e^{-\frac{%
\left( t-\tau \right) +\left \vert x-y\right \vert }{c_{0}}}dyd\tau \\
&\lesssim &e^{-\frac{t}{2c_{0}}}+e^{-\frac{\mathbf{c}^{2}t}{32D_{0}}}\int_{%
\frac{t}{2}}^{t}\int_{\mathbb{R}^{3}}\left( 1+\tau \right) ^{-2}e^{-\frac{%
\left \vert x-y\right \vert }{c_{0}}}dyd\tau \\
&&+\int_{\frac{t}{2}}^{t}\int_{\left \vert y\right \vert >\frac{\mathbf{c}%
\tau }{2}}\left( 1+\tau \right) ^{-2}e^{-\frac{\left( \left \vert y\right
\vert -\mathbf{c}\tau \right) ^{2}}{D_{0}\left( 1+\tau \right) }}e^{-\frac{%
\left( t-\tau \right) +\left \vert x-y\right \vert }{c_{0}}}dyd\tau \\
&\lesssim &e^{-\frac{t}{2c_{0}}}+e^{-\frac{\mathbf{c}t}{32D_{0}}}+\int_{%
\frac{t}{2}}^{t}\int_{\left \vert y\right \vert >\frac{\mathbf{c}\tau }{2}%
}\left( 1+\tau \right) ^{-2}e^{-\frac{\left( \left \vert y\right \vert -%
\mathbf{c}\tau \right) ^{2}}{D_{0}\left( 1+\tau \right) }}e^{-\frac{\left(
t-\tau \right) +\left \vert x-y\right \vert }{c_{0}}}dyd\tau \hbox{.}
\end{eqnarray*}

Note that if $\frac{t}{2}\leq \tau \leq t$, $\left \vert y\right \vert >%
\frac{\mathbf{c}\tau }{2}$, then%
\begin{equation*}
\left \vert x-y\right \vert \geq \left \vert y\right \vert -\left \vert
x\right \vert \geq \frac{\mathbf{c}t}{4}-\sqrt{1+t}\geq \frac{\mathbf{c}t}{8}%
+\frac{\mathbf{c}t}{8}-\sqrt{1+t}\geq \frac{\mathbf{c}t}{8}
\end{equation*}%
for $t\geq 40$, so that%
\begin{eqnarray*}
&&\int_{\frac{t}{2}}^{t}\int_{\left \vert y\right \vert >\frac{\mathbf{c}%
\tau }{2}}\left( 1+\tau \right) ^{-2}e^{-\frac{\left( \left \vert y\right
\vert -\mathbf{c}\tau \right) ^{2}}{D_{0}\left( 1+\tau \right) }}e^{-\frac{%
\left( t-\tau \right) +\left \vert x-y\right \vert }{c_{0}}}dyd\tau \\
&\lesssim &e^{-\frac{\mathbf{c}t}{8c_{0}}}\int_{\frac{t}{2}}^{t}\int_{\left
\vert y\right \vert >\frac{\mathbf{c}\tau }{2}}\left( 1+\tau \right)
^{-2}e^{-\frac{\left( \left \vert y\right \vert -\mathbf{c}\tau \right) ^{2}%
}{D_{0}\left( 1+\tau \right) }}dyd\tau \\
&\lesssim &e^{-\frac{\mathbf{c}t}{8c_{0}}}\int_{\frac{t}{2}}^{t}\left(
1+\tau \right) ^{-2+\frac{5}{2}}d\tau \lesssim e^{-\frac{\mathbf{c}t}{16c_{0}%
}}\hbox{,}
\end{eqnarray*}%
for $t\geq 40$. Consequently,
\begin{equation*}
A\lesssim e^{-\frac{t}{2c_{0}}}+e^{-\frac{\mathbf{c}t}{32D_{0}}}+e^{-\frac{%
\mathbf{c}t}{16c_{0}}}\lesssim e^{-\frac{t+\left \vert x\right \vert }{%
\widehat{c}}}
\end{equation*}%
for all $t\geq 0$, where $\frac{1}{\widehat{c}}=\frac{1}{2}\min \{ \frac{1}{%
2c_{0}},\frac{\mathbf{c}}{32D_{0}},\frac{\mathbf{c}}{16c_{0}}\}$. \bigskip
\newline
\textbf{Case 2:} $\left( x,t\right) \in D_{2}$.
\begin{eqnarray*}
A &\lesssim &\int_{0}^{\frac{t}{2}}\int_{\mathbb{R}^{3}}\left( 1+\tau
\right) ^{-2}e^{-\frac{t}{2c_{0}}}e^{-\frac{\left \vert x-y\right \vert }{%
c_{0}}}dyd\tau +\int_{\frac{t}{2}}^{t}\int_{\mathbb{R}^{3}}\left( 1+t\right)
^{-2}e^{-\frac{\left( t-\tau \right) +\left \vert x-y\right \vert }{c_{0}}%
}dyd\tau \\
&\lesssim &e^{-\frac{t}{2c_{0}}}+\left( 1+t\right) ^{-2}\lesssim \left(
1+t\right) ^{-2}\lesssim \left( 1+t\right) ^{-2}e^{-\frac{\left( \left \vert
x\right \vert -\mathbf{c}t\right) ^{2}}{D\left( 1+t\right) }}\hbox{.}
\end{eqnarray*}
\newline
\textbf{Case 3:} $\left( x,t\right) \in D_{3}$. We split the integral $A$
into four parts%
\begin{eqnarray*}
A &=&\int_{0}^{\frac{t}{2}}\left( \int_{\left \vert y\right \vert \leq \frac{%
\left \vert x\right \vert +\mathbf{c}\tau }{2}}+\int_{\left \vert y\right
\vert >\frac{\left \vert x\right \vert +\mathbf{c}\tau }{2}}\right) \left(
\cdots \right) dyd\tau +\int_{\frac{t}{2}}^{t}\left( \int_{\left \vert
y\right \vert \leq \frac{\left \vert x\right \vert +\mathbf{c}\tau }{2}%
}+\int_{\left \vert y\right \vert >\frac{\left \vert x\right \vert +\mathbf{c%
}\tau }{2}}\right) \left( \cdots \right) dyd\tau \\
&=&:A_{11}+A_{12}+A_{21}+A_{22}\hbox{.}
\end{eqnarray*}

Note that if $\left \vert y\right \vert \leq \frac{\left \vert x\right \vert
+\mathbf{c}\tau }{2}$, then
\begin{equation*}
\left \vert x-y\right \vert \geq \left \vert x\right \vert -\left \vert
y\right \vert \geq \frac{\left \vert x\right \vert -\mathbf{c}\tau }{2}=%
\frac{\left \vert x\right \vert -\mathbf{c}t}{2}+\frac{\mathbf{c}\left(
t-\tau \right) }{2}\hbox{.}
\end{equation*}%
if $\left \vert y\right \vert >\frac{\left \vert x\right \vert +\mathbf{c}%
\tau }{2}$, then%
\begin{equation*}
\left \vert y\right \vert -\mathbf{c}\tau >\frac{\left \vert x\right \vert -%
\mathbf{c}\tau }{2}=\frac{\left \vert x\right \vert -\mathbf{c}t}{2}+\frac{%
\mathbf{c}\left( t-\tau \right) }{2}\hbox{.}
\end{equation*}%
It immediately follows that%
\begin{eqnarray*}
A_{11} &\lesssim &\int_{0}^{\frac{t}{2}}\int_{\left \vert y\right \vert \leq
\frac{\left \vert x\right \vert +\mathbf{c}\tau }{2}}\left( 1+\tau \right)
^{-2}e^{-\frac{\left \vert x-y\right \vert }{2c_{0}}}e^{-\frac{t}{2c_{0}}%
}e^{-\frac{\left \vert x-y\right \vert }{2c_{0}}}dyd\tau \\
&\lesssim &\int_{0}^{\frac{t}{2}}\int_{\mathbb{R}^{3}}\left( 1+\tau \right)
^{-2}e^{-\frac{\left \vert x-y\right \vert }{2c_{0}}}e^{-\frac{t}{2c_{0}}-%
\frac{1}{2c_{0}}\left( \frac{\left \vert x\right \vert }{2}-\frac{\mathbf{c}t%
}{4}\right) }dyd\tau \lesssim e^{-\frac{t}{2c_{0}}-\frac{1}{2c_{0}}\left(
\frac{\left \vert x\right \vert }{2}-\frac{\mathbf{c}t}{4}\right) }\lesssim
e^{-\frac{t}{2c_{0}}-\frac{\left \vert x\right \vert }{8c_{0}}}\hbox{,}
\end{eqnarray*}%
\begin{equation*}
A_{12}\lesssim \int_{0}^{\frac{t}{2}}\int_{\mathbb{R}^{3}}\left( 1+\tau
\right) ^{-2}e^{-\frac{\left( \left \vert x\right \vert -\mathbf{c}t\right)
^{2}}{4D_{0}\left( 1+t\right) }}e^{-\frac{t}{2c_{0}}}e^{-\frac{\left \vert
x-y\right \vert }{c_{0}}}\lesssim e^{-\frac{\left( \left \vert x\right \vert
-\mathbf{c}t\right) ^{2}}{4D_{0}\left( 1+t\right) }}e^{-\frac{t}{2c_{0}}}%
\hbox{,}
\end{equation*}%
\begin{equation*}
A_{22}\lesssim \int_{\frac{t}{2}}^{t}\int_{\left \vert y\right \vert \geq
\frac{\left \vert x\right \vert +\mathbf{c}\tau }{2}}\left( 1+t\right)
^{-2}e^{-\frac{\left( \left \vert x\right \vert -\mathbf{c}t\right) ^{2}}{%
4D_{0}\left( 1+t\right) }}e^{-\frac{\left( t-\tau \right) +\left \vert
x-y\right \vert }{c_{0}}}dyd\tau \lesssim \left( 1+t\right) ^{-2}e^{-\frac{%
\left( \left \vert x\right \vert -\mathbf{c}t\right) ^{2}}{4D_{0}\left(
1+t\right) }}\hbox{.}
\end{equation*}%
As for $A_{21}$,%
\begin{eqnarray*}
A_{21} &\lesssim &\int_{\frac{t}{2}}^{t}\int_{\left \vert y\right \vert \leq
\frac{\left \vert x\right \vert +\mathbf{c}\tau }{2}}\left( 1+\tau \right)
^{-2}e^{-\frac{\left( \left \vert y\right \vert -\mathbf{c}\tau \right) ^{2}%
}{D_{0}\left( 1+\tau \right) }}e^{-\frac{\left \vert x\right \vert -\mathbf{c%
}t}{2c_{0}}}dyd\tau \\
&\lesssim &e^{-\frac{\left \vert x\right \vert -\mathbf{c}t}{2c_{0}}}\int_{%
\frac{t}{2}}^{t}\left( 1+\tau \right) ^{-2+\frac{5}{2}}d\tau \lesssim \left(
1+t\right) ^{\frac{3}{2}}e^{-\frac{\left \vert x\right \vert -\mathbf{c}t}{%
2c_{0}}}\lesssim \left( 1+t\right) ^{\frac{3}{2}}e^{-\frac{\sqrt{1+t}}{4c_{0}%
}}e^{-\frac{\left \vert x\right \vert -\mathbf{c}t}{4c_{0}}} \\
&\lesssim &\left( 1+t\right) ^{-2}e^{-\frac{\left \vert x\right \vert -%
\mathbf{c}t}{4c_{0}}}\hbox{,}
\end{eqnarray*}%
since $\left \vert x\right \vert -\mathbf{c}t\geq \sqrt{1+t}$. Now we
discuss two cases: (i) $\left \vert x\right \vert -\mathbf{c}t\leq \frac{1}{2%
}\left( 1+t\right) $, and (ii) $\left \vert x\right \vert -\mathbf{c}t\geq
\frac{1}{2}\left( 1+t\right) $. For case (i),

\begin{equation*}
\left \vert x\right \vert -\mathbf{c}t\geq \frac{2\left( \left \vert x\right
\vert -\mathbf{c}t\right) ^{2}}{1+t}\hbox{,}
\end{equation*}%
which follows that
\begin{equation*}
e^{-\frac{\left \vert x\right \vert -\mathbf{c}t}{4c_{0}}}\lesssim e^{-\frac{%
\left( \left \vert x\right \vert -\mathbf{c}t\right) ^{2}}{2c_{0}\left(
1+t\right) }}\hbox{.}
\end{equation*}%
For case (ii), it is easy to see
\begin{equation*}
\left( 1+t\right) \leq 2\left( \left \vert x\right \vert -\mathbf{c}t\right)
\end{equation*}%
and thus
\begin{equation*}
\left \vert x\right \vert =\left \vert \left \vert x\right \vert -\mathbf{c}%
t+\mathbf{c}t\right \vert \leq \left( \left \vert x\right \vert -\mathbf{c}%
t\right) +\mathbf{c}t\leq \left( 1+2\mathbf{c}\right) \left( \left \vert
x\right \vert -\mathbf{c}t\right) <4\left( \left \vert x\right \vert -%
\mathbf{c}t\right) \hbox{,}
\end{equation*}
so that%
\begin{equation*}
e^{-\frac{\left \vert x\right \vert -\mathbf{c}t}{4c_{0}}}=e^{-\frac{\left
\vert x\right \vert -\mathbf{c}t}{8c_{0}}}e^{-\frac{\left \vert x\right
\vert -\mathbf{c}t}{8c_{0}}}\lesssim e^{-\frac{\left( 1+t\right) }{16c_{0}}%
}e^{-\frac{\left \vert x\right \vert }{32c_{0}}}\lesssim e^{-\frac{t}{16c_{0}%
}}e^{-\frac{\left \vert x\right \vert }{32c_{0}}}\hbox{.}
\end{equation*}%
Hence,
\begin{equation*}
A_{21}\lesssim \left( 1+t\right) ^{-2}e^{-\frac{\left( \left \vert x\right
\vert -\mathbf{c}t\right) ^{2}}{2c_{0}\left( 1+t\right) }}+e^{-\frac{t}{%
16c_{0}}}e^{-\frac{\left \vert x\right \vert }{32c_{0}}}\hbox{.}
\end{equation*}%
Combining all above estimates, we have%
\begin{equation*}
A\lesssim \left( 1+t\right) ^{-2}e^{-\frac{\left( \left \vert x\right \vert -%
\mathbf{c}t\right) ^{2}}{\widehat{D}\left( 1+t\right) }}+e^{-\frac{t+\left
\vert x\right \vert }{\widehat{c}}}
\end{equation*}%
for some constants $\widehat{c}$ and $\widehat{D}>0$. \bigskip \newline
\textbf{Case 4}: $\left( x,t\right) \in D_{4}$. We split the integral into
two parts%
\begin{equation*}
A=\left( \int_{0}^{\frac{2}{3}t}+\int_{\frac{2}{3}t}^{t}\right) \int_{%
\mathbb{R}^{3}}\left( \cdots \right) dyd\tau =:A_{1}+A_{2}\hbox{.}
\end{equation*}

For $A_{1}$,%
\begin{equation*}
A_{1}\lesssim \int_{0}^{\frac{2}{3}t}\int_{\mathbb{R}^{3}}\left( 1+\tau
\right) ^{-2}e^{-\frac{\left \vert x-y\right \vert }{c_{0}}}e^{-\frac{t}{%
3c_{0}}}dyd\tau \lesssim e^{-\frac{t}{3c_{0}}}\lesssim e^{-\frac{t}{6c_{0}}%
}e^{-\frac{\left \vert x\right \vert }{3c_{0}\mathbf{c}}}
\end{equation*}%
since $\sqrt{1+t}\leq \left \vert x\right \vert \leq \mathbf{c}t/2$.

For $A_{2}$, we decompose $\mathbb{R}^{3}$ into two parts
\begin{equation*}
A_{2}=\int_{\frac{2}{3}t}^{t}\left( \int_{\left \vert y\right \vert \leq
\frac{\left \vert x\right \vert +\mathbf{c}\tau }{2}}+\int_{\left \vert
y\right \vert >\frac{\left \vert x\right \vert +\mathbf{c}\tau }{2}}\right)
\left( \cdots \right) dyd\tau =:A_{21}+A_{22}\hbox{.}
\end{equation*}%
If $\frac{2}{3}t\leq \tau \leq t$, $\left \vert y\right \vert \leq \frac{%
\left \vert x\right \vert +\mathbf{c}\tau }{2}$, then%
\begin{equation*}
\mathbf{c}\tau -\left \vert y\right \vert \geq \frac{\mathbf{c}\tau -\left
\vert x\right \vert }{2}\geq \frac{\mathbf{c}t}{3}-\frac{\mathbf{c}t}{4}=%
\frac{\mathbf{c}t}{12}\geq \frac{\mathbf{c}t-\left \vert x\right \vert }{12}%
\hbox{.}
\end{equation*}%
If $\frac{2}{3}t\leq \tau \leq t$, $\left \vert y\right \vert >\frac{%
\left
\vert x\right \vert +\mathbf{c}\tau }{2}$, then%
\begin{equation*}
\left \vert x-y\right \vert \geq \left \vert y\right \vert -\left \vert
x\right \vert \geq \frac{\mathbf{c}\tau -\left \vert x\right \vert }{2}\geq
\frac{\mathbf{c}t}{12}\geq \frac{\mathbf{c}t-\left \vert x\right \vert }{12}%
\hbox{.}
\end{equation*}%
Hence,
\begin{eqnarray*}
A_{21} &\lesssim &\int_{\frac{2}{3}t}^{t}\int_{\left \vert y\right \vert
\leq \frac{\left \vert x\right \vert +\mathbf{c}\tau }{2}}\left( 1+\tau
\right) ^{-2}e^{-\frac{\left( \left \vert y\right \vert -\mathbf{c}\tau
\right) ^{2}}{D_{0}\left( 1+\tau \right) }}e^{-\frac{\left( t-\tau \right)
+\left \vert x-y\right \vert }{c_{0}}}dyd\tau \\
&\lesssim &\left( 1+t\right) ^{-2}e^{-\frac{\left( \mathbf{c}t-\left \vert
x\right \vert \right) ^{2}}{144D_{0}\left( 1+t\right) }}\hbox{,}
\end{eqnarray*}%
and%
\begin{eqnarray*}
A_{22} &\lesssim &\int_{\frac{2}{3}t}^{t}\int_{\left \vert y\right \vert >%
\frac{\left \vert x\right \vert +\mathbf{c}\tau }{2}}\left( 1+t\right)
^{-2}e^{-\frac{t-\tau }{c_{0}}}e^{-\frac{\left \vert x-y\right \vert }{2c_{0}%
}}e^{-\frac{\mathbf{c}t-\left \vert x\right \vert }{24c_{0}}}dyd\tau \\
&\lesssim &\left( 1+t\right) ^{-2}e^{-\frac{\mathbf{c}t-\left \vert x\right
\vert }{24c_{0}}}\hbox{.}
\end{eqnarray*}%
Since $t\geq 1$ and $\mathbf{c}t-\left \vert x\right \vert >\frac{\mathbf{c}t%
}{2}\geq \frac{\mathbf{c}}{4}\left( 1+t\right) $ for $\left( x,t\right) \in
D_{4}$,
\begin{equation*}
\mathbf{c}\left( 1+t\right) \leq 4\left( \mathbf{c}t-\left \vert x\right
\vert \right) \hbox{,}
\end{equation*}%
and thus%
\begin{equation*}
\left \vert x\right \vert =\left \vert \left \vert x\right \vert -\mathbf{c}%
t+\mathbf{c}t\right \vert \leq \mathbf{c}t-\left \vert x\right \vert +%
\mathbf{c}t\leq 5\left( \mathbf{c}t-\left \vert x\right \vert \right) %
\hbox{,}
\end{equation*}%
which implies that
\begin{equation*}
e^{-\frac{\mathbf{c}t-\left \vert x\right \vert }{24c_{0}}}=e^{-\frac{%
\mathbf{c}t-\left \vert x\right \vert }{48c_{0}}}e^{-\frac{\mathbf{c}t-\left
\vert x\right \vert }{48c_{0}}}\lesssim e^{-\frac{\mathbf{c}\left(
1+t\right) }{192c_{0}}}e^{-\frac{\left \vert x\right \vert }{240c_{0}}%
}\lesssim e^{-\frac{\mathbf{c}t}{192c_{0}}}e^{-\frac{\left \vert x\right
\vert }{240c_{0}}}\hbox{.}
\end{equation*}%
Therefore,
\begin{equation*}
A_{22}\lesssim e^{-\frac{\mathbf{t}}{96c_{0}}}e^{-\frac{\left \vert x\right
\vert }{192c_{0}}}\hbox{.}
\end{equation*}%
Combining this with $A_{1}$ and $A_{11}$, we get the desired estimate%
\begin{equation*}
A\lesssim e^{-\frac{t}{6c_{0}}}e^{-\frac{\left \vert x\right \vert }{3c_{0}%
\mathbf{c}}}+\left( 1+t\right) ^{-2}e^{-\frac{\left( \mathbf{c}t-\left \vert
x\right \vert \right) ^{2}}{144D_{0}\left( 1+t\right) }}+e^{-\frac{\mathbf{t}%
}{96c_{0}}}e^{-\frac{\left \vert x\right \vert }{192c_{0}}}\hbox{.}
\end{equation*}
\newline
\textbf{Case 5:} $\left( x,t\right) \in D_{5}$. We split the integral $A$
into three parts%
\begin{equation*}
A=\left( \int_{0}^{\frac{t}{2}}+\int_{\frac{t}{2}}^{\frac{t}{2}+\frac{\left
\vert x\right \vert }{2\mathbf{c}}}+\int_{\frac{t}{2}+\frac{\left \vert
x\right \vert }{2\mathbf{c}}}^{t}\right) \int_{\mathbb{R}^{3}}\left( \cdots
\right) dyd\tau =:A_{1}+A_{2}+A_{3}\hbox{.}
\end{equation*}

It immediately follows that%
\begin{equation*}
A_{1}\lesssim \int_{0}^{\frac{t}{2}}\int_{\mathbb{R}^{3}}\left( 1+\tau
\right) ^{-2}e^{-\frac{\left \vert x-y\right \vert }{c_{0}}}e^{-\frac{t}{%
2c_{0}}}dyd\tau \lesssim e^{-\frac{t}{2c_{0}}}\lesssim e^{-\frac{t}{4c_{0}}-%
\frac{\left \vert x\right \vert }{4\mathbf{c}c_{0}}}\hbox{,}
\end{equation*}%
and%
\begin{eqnarray*}
A_{2} &\lesssim &\int_{\frac{t}{2}}^{\frac{t}{2}+\frac{\left \vert x\right
\vert }{2\mathbf{c}}}\int_{\mathbb{R}^{3}}\left( 1+\tau \right) ^{-2}e^{-%
\frac{\left( \left \vert y\right \vert -\mathbf{c}\tau \right) ^{2}}{%
D_{0}\left( 1+\tau \right) }}e^{-\frac{\left( t-\tau \right) +\left \vert
x-y\right \vert }{c_{0}}}dyd\tau \\
&\lesssim &\left( 1+t\right) ^{-2}\int_{\frac{t}{2}}^{\frac{t}{2}+\frac{%
\left \vert x\right \vert }{2\mathbf{c}}}\int_{\mathbb{R}^{3}}e^{-\frac{%
\left \vert x-y\right \vert }{c_{0}}}e^{-\frac{t-\tau }{2c_{0}}}e^{-\frac{1}{%
2c_{0}}\left( \frac{t}{2}-\frac{\left \vert x\right \vert }{2\mathbf{c}}%
\right) }dyd\tau \\
&\lesssim &\left( 1+t\right) ^{-2}e^{-\frac{\mathbf{c}t-\left \vert x\right
\vert }{4c_{0}\mathbf{c}}}\hbox{.}
\end{eqnarray*}

For $A_{3}$, we decompose $\mathbb{R}^{3}$ into two parts%
\begin{equation*}
A_{3}=\int_{\frac{t}{2}+\frac{\left \vert x\right \vert }{2\mathbf{c}}%
}^{t}\left( \int_{\left \vert y\right \vert \leq \frac{\left \vert x\right
\vert +\mathbf{c}\tau }{2}}+\int_{\left \vert y\right \vert >\frac{\left
\vert x\right \vert +\mathbf{c}\tau }{2}}\right) \left( \cdots \right)
dyd\tau =A_{31}+A_{32}\hbox{.}
\end{equation*}%
If $\frac{t}{2}+\frac{\left \vert x\right \vert }{2\mathbf{c}}\leq \tau \leq
t$, $\left \vert y\right \vert \leq \frac{\left \vert x\right \vert +\mathbf{%
c}\tau }{2}$, then
\begin{equation*}
\mathbf{c}\tau -\left \vert y\right \vert \geq \frac{\mathbf{c}\tau -\left
\vert x\right \vert }{2}\geq \frac{\mathbf{c}t-\left \vert x\right \vert }{4}%
\hbox{.}
\end{equation*}%
If $\frac{t}{2}+\frac{\left \vert x\right \vert }{2\mathbf{c}}\leq \tau \leq
t$, $\left \vert y\right \vert >\frac{\left \vert x\right \vert +\mathbf{c}%
\tau }{2}$, then%
\begin{equation*}
\left \vert x-y\right \vert \geq \left \vert y\right \vert -\left \vert
x\right \vert >\frac{\mathbf{c}\tau -\left \vert x\right \vert }{2}\geq
\frac{\mathbf{c}t-\left \vert x\right \vert }{4}\hbox{.}
\end{equation*}%
Hence,
\begin{eqnarray*}
A_{3} &=&A_{31}+A_{32} \\
&\lesssim &\int_{\frac{t}{2}+\frac{\left \vert x\right \vert }{2\mathbf{c}}%
}^{t}\int_{\left \vert y\right \vert \leq \frac{\left \vert x\right \vert +%
\mathbf{c}\tau }{2}}\left( 1+t\right) ^{-2}e^{-\frac{\left( \mathbf{c}%
t-\left \vert x\right \vert \right) ^{2}}{16D_{0}\left( 1+t\right) }}e^{-%
\frac{\left( t-\tau \right) +\left \vert x-y\right \vert }{c_{0}}}dyd\tau \\
&&+\int_{\frac{t}{2}+\frac{\left \vert x\right \vert }{2\mathbf{c}}%
}^{t}\int_{\left \vert y\right \vert >\frac{\left \vert x\right \vert +%
\mathbf{c}\tau }{2}}\left( 1+t\right) ^{-2}e^{-\frac{\left \vert x-y\right
\vert }{2c_{0}}}e^{-\frac{t-\tau }{c_{0}}}e^{-\frac{\mathbf{c}t-\left \vert
x\right \vert }{8c_{0}}}dyd\tau \\
&\lesssim &\left( 1+t\right) ^{-2}e^{-\frac{\left( \mathbf{c}t-\left \vert
x\right \vert \right) ^{2}}{16D_{0}\left( 1+t\right) }}+\left( 1+t\right)
^{-2}e^{-\frac{\mathbf{c}t-\left \vert x\right \vert }{8c_{0}}}\hbox{.}
\end{eqnarray*}%
Since $\sqrt{1+t}\leq \mathbf{c}t-\left \vert x\right \vert \leq \frac{%
\mathbf{c}t}{2}\leq \frac{\mathbf{c}}{2}\left( 1+t\right) $, we have%
\begin{equation*}
\mathbf{c}t-\left \vert x\right \vert \geq \frac{2\left( \mathbf{c}t-\left
\vert x\right \vert \right) ^{2}}{\mathbf{c}\left( 1+t\right) }\hbox{,}
\end{equation*}%
and thus
\begin{equation*}
e^{-\frac{\mathbf{c}t-\left \vert x\right \vert }{4c_{0}\mathbf{c}}}\lesssim
e^{-\frac{\left( \mathbf{c}t-\left \vert x\right \vert \right) ^{2}}{2c_{0}%
\mathbf{c}^{2}\left( 1+t\right) }}\hbox{, \  \ }e^{-\frac{\mathbf{c}t-\left
\vert x\right \vert }{8c_{0}}}\lesssim e^{-\frac{\left( \mathbf{c}t-\left
\vert x\right \vert \right) ^{2}}{4c_{0}\mathbf{c}\left( 1+t\right) }}%
\hbox{.}
\end{equation*}%
Consequently,
\begin{equation*}
A_{2}\lesssim \left( 1+t\right) ^{-2}e^{-\frac{\left( \mathbf{c}t-\left
\vert x\right \vert \right) ^{2}}{2c_{0}\mathbf{c}^{2}\left( 1+t\right) }}%
\hbox{, \
\ }A_{3}\lesssim \left( 1+t\right) ^{-2}e^{-\frac{\left( \mathbf{c}t-\left
\vert x\right \vert \right) ^{2}}{16D_{0}\left( 1+t\right) }}+\left(
1+t\right) ^{-2}e^{-\frac{\left( \mathbf{c}t-\left \vert x\right \vert
\right) ^{2}}{4c_{0}\mathbf{c}\left( 1+t\right) }}\hbox{.}
\end{equation*}%
Combining this with $A_{1}$, we get the desired estimate.
\end{proof}

\begin{proof}[\textbf{Proof of Lemma \protect \ref{poly-exponential}} (Riesz
wave convolved with exponential decay). ]
We compute the integral for two cases: $\left \vert x\right \vert >\mathbf{c}%
t$ and $\left \vert x\right \vert \leq \mathbf{c}t$. \newline
\newline
\textbf{Case 1:} $\left \vert x\right \vert >\mathbf{c}t$. Let $0\leq \tau
\leq t$. If $\left \vert x\right \vert >\mathbf{c}t$ and $\left \vert
y\right
\vert <\mathbf{c}\tau $, then
\begin{equation*}
\left \vert x-y\right \vert \geq \left \vert x\right \vert -\left \vert
y\right \vert \geq \left \vert x\right \vert -\mathbf{c}t+\mathbf{c}t-\left
\vert y\right \vert \geq \left \vert x\right \vert -\mathbf{c}t\geq 0\hbox{.}
\end{equation*}%
Hence,
\begin{eqnarray*}
&&\int_{0}^{t}\int_{\mathbb{R}^{3}}e^{-\frac{\left( t-\tau \right) +\left
\vert x-y\right \vert }{c_{0}}}1_{\{ \left \vert y\right \vert <\mathbf{c}%
\tau \}}\left( 1+\tau \right) ^{-3/2}\left( 1+\frac{\left \vert y\right
\vert ^{2}}{1+\tau }\right) ^{-3/2}dyd\tau \\
&\lesssim &e^{-\frac{\left \vert x\right \vert -\mathbf{c}t}{2c_{0}}%
}\int_{0}^{t}\int_{\mathbb{R}^{3}}e^{-\frac{t-\tau }{c_{0}}}e^{-\frac{\left
\vert x-y\right \vert }{2c_{0}}}\left( 1+\tau \right) ^{-3/2}\left( 1+\frac{%
\left \vert y\right \vert ^{2}}{1+\tau }\right) ^{-3/2}dyd\tau \\
&=&e^{-\frac{\left \vert x\right \vert -\mathbf{c}t}{2c_{0}}}\left(
\int_{0}^{\frac{t}{2}}+\int_{\frac{t}{2}}^{t}\right) \int_{\mathbb{R}%
^{3}}\left( \cdots \right) dyd\tau =:B_{1}+B_{2}\hbox{.}
\end{eqnarray*}%
It immediately follows that
\begin{equation*}
B_{1}\leq e^{-\frac{t}{2c_{0}}}e^{-\frac{\left \vert x\right \vert -\mathbf{c%
}t}{2c_{0}}}\int_{0}^{\frac{t}{2}}\int_{\mathbb{R}^{3}}e^{-\frac{\left \vert
x-y\right \vert }{2c_{0}}}\left( 1+\tau \right) ^{-3/2}dyd\tau \lesssim e^{-%
\frac{t}{2c_{0}}}e^{-\frac{\left \vert x\right \vert -\mathbf{c}t}{2c_{0}}}%
\hbox{.}
\end{equation*}%
As for $B_{2}$,%
\begin{eqnarray*}
B_{2} &=&e^{-\frac{\left \vert x\right \vert -\mathbf{c}t}{2c_{0}}}\int_{%
\frac{t}{2}}^{t}\left( \int_{\left \vert y\right \vert \leq \frac{\left
\vert x\right \vert }{2}}+\int_{\left \vert y\right \vert >\frac{\left \vert
x\right \vert }{2}}\right) e^{-\frac{t-\tau }{c_{0}}}e^{-\frac{\left \vert
x-y\right \vert }{2c_{0}}}\left( 1+\tau \right) ^{-3/2}\left( 1+\frac{\left
\vert y\right \vert ^{2}}{1+\tau }\right) ^{-3/2}dyd\tau \\
&\leq &e^{-\frac{\left \vert x\right \vert -\mathbf{c}t}{2c_{0}}}e^{-\frac{%
\left \vert x\right \vert }{8c_{0}}}\int_{\frac{t}{2}}^{t}\int_{\left \vert
y\right \vert \leq \frac{\left \vert x\right \vert }{2}}e^{-\frac{t-\tau }{%
c_{0}}}e^{-\frac{\left \vert x-y\right \vert }{4c_{0}}}dyd\tau \\
&&+\left( 1+t\right) ^{-3/2}\left( 1+\frac{\left \vert x\right \vert ^{2}}{%
1+t}\right) ^{-3/2}e^{-\frac{\left \vert x\right \vert -\mathbf{c}t}{2c_{0}}%
}\int_{\frac{t}{2}}^{t}\int_{\left \vert y\right \vert >\frac{\left \vert
x\right \vert }{2}}e^{-\frac{t-\tau }{c_{0}}}e^{-\frac{\left \vert x-y\right
\vert }{2c_{0}}}dyd\tau \\
&\lesssim &e^{-\frac{\left \vert x\right \vert -\mathbf{c}t}{2c_{0}}}e^{-%
\frac{\left \vert x\right \vert }{8c_{0}}}+\left( 1+t\right) ^{-3}e^{-\frac{%
\left \vert x\right \vert -\mathbf{c}t}{2c_{0}}}\lesssim \left( 1+t\right)
^{-3}e^{-\frac{\left \vert x\right \vert -\mathbf{c}t}{2c_{0}}}\hbox{.}
\end{eqnarray*}%
since $\left \vert x\right \vert >\mathbf{c}t$.

Now, we consider the cases: $0\leq \left \vert x\right \vert -\mathbf{c}%
t\leq 2\mathbf{c}\left( 1+t\right) $ and $\left \vert x\right \vert -\mathbf{%
c}t\geq 2\mathbf{c}\left( 1+t\right) $. For $0\leq \left \vert x\right \vert
-\mathbf{c}t\leq 2\mathbf{c}\left( 1+t\right) $, we have%
\begin{equation*}
\left \vert x\right \vert -\mathbf{c}t\geq \frac{\left( \left \vert x\right
\vert -\mathbf{c}t\right) ^{2}}{2\mathbf{c}\left( 1+t\right) }\hbox{,}
\end{equation*}%
so that%
\begin{equation*}
e^{-\frac{\left \vert x\right \vert -\mathbf{c}t}{2c_{0}}}\leq e^{-\frac{%
\left( \left \vert x\right \vert -\mathbf{c}t\right) ^{2}}{4\mathbf{c}%
c_{0}\left( 1+t\right) }}\hbox{.}
\end{equation*}%
For $\left \vert x\right \vert -\mathbf{c}t\geq 2\mathbf{c}\left( 1+t\right)
$, it follows that%
\begin{equation*}
\left \vert x\right \vert -\mathbf{c}t=\frac{\left \vert x\right \vert }{2}+%
\frac{\left \vert x\right \vert }{2}-\mathbf{c}t\geq \frac{\left \vert
x\right \vert }{2}+\frac{\mathbf{c}t}{2}\geq \frac{\left \vert x\right \vert
}{2}+\frac{t}{2}\hbox{,}
\end{equation*}%
and thus
\begin{equation*}
e^{-\frac{\left \vert x\right \vert -\mathbf{c}t}{2c_{0}}}\leq e^{-\frac{%
\left \vert x\right \vert +t}{4c_{0}}}\hbox{.}
\end{equation*}%
Consequently,
\begin{eqnarray*}
&&\int_{0}^{t}\int_{\mathbb{R}^{3}}e^{-\frac{\left( t-\tau \right) +\left
\vert x-y\right \vert }{c_{0}}}1_{\{ \left \vert y\right \vert <\mathbf{c}%
\tau \}}\left( 1+\tau \right) ^{-3/2}\left( 1+\frac{\left \vert y\right
\vert ^{2}}{1+\tau }\right) ^{-3/2}dyd\tau \\
&\lesssim &\left( 1+t\right) ^{-3}e^{-\frac{\left( \left \vert x\right \vert
-\mathbf{c}t\right) ^{2}}{4\mathbf{c}c_{0}\left( 1+t\right) }}+e^{-\frac{%
\left \vert x\right \vert +t}{4c_{0}}}\hbox{.}
\end{eqnarray*}
\newline
\textbf{Case 2:} $\left \vert x\right \vert <\mathbf{c}t$. Direct
computation gives
\begin{eqnarray*}
B &\lesssim &e^{-\frac{t}{2c_{0}}}\int_{0}^{t/2}\int_{\mathbb{R}^{3}}e^{-%
\frac{\left( t-\tau \right) +\left \vert x-y\right \vert }{c_{0}}}dyd\tau \\
&&+\left( 1+t\right) ^{-3/2}\left( 1+\frac{\left \vert x\right \vert ^{2}}{%
1+t}\right) ^{-3/2}\int_{t/2}^{t}\int_{\left \vert y\right \vert >\frac{%
\left \vert x\right \vert }{2}}e^{-\frac{\left( t-\tau \right) +\left \vert
x-y\right \vert }{c_{0}}}dyd\tau \\
&&+\left( 1+t\right) ^{-3/2}e^{-\frac{\left \vert x\right \vert }{4c_{0}}%
}\int_{t/2}^{t}\int_{\left \vert y\right \vert \leq \frac{\left \vert
x\right \vert }{2}}e^{-\frac{t-\tau }{c_{0}}}e^{-\frac{\left \vert x-y\right
\vert }{2c_{0}}}dyd\tau \\
&\lesssim &e^{-\frac{t}{2c_{0}}}+\left( 1+t\right) ^{-3/2}\left( 1+\frac{%
\left \vert x\right \vert ^{2}}{1+t}\right) ^{-3/2} \\
&\lesssim &e^{-\frac{t+\left \vert x\right \vert }{4cc_{0}}}+\mathbf{1}_{\{
\left \vert x\right \vert <\mathbf{c}t\}}\left( 1+t\right) ^{-3/2}\left( 1+%
\frac{\left \vert x\right \vert ^{2}}{1+t}\right) ^{-3/2}\hbox{.}
\end{eqnarray*}

To sum up,
\begin{equation*}
B\lesssim \mathbf{1}_{\{ \left \vert x\right \vert <\mathbf{c}t \}}\left(
1+t\right) ^{-3/2}\left( 1+\frac{\left \vert x\right \vert ^{2}}{1+t}\right)
^{-3/2}+e^{-\frac{t+\left \vert x\right \vert }{\widehat{c}_{0}}}+\left(
1+t\right) ^{-2}e^{-\frac{\left( \left \vert x\right \vert -\mathbf{c}%
t\right) ^{2}}{\widehat{D}_{0}\left( 1+t\right) }}
\end{equation*}%
for some constants $\widehat{D}_{0}$, $\widehat{c}_{0}>0$.
\end{proof}

\subsection{Nonlinear wave interaction}

\label{nonlinear wave interaction}

\begin{lemma}[Diffusion wave convolved with Diffusion wave]
\label{Easy1}
\begin{eqnarray*}
&&\left( 1+t \right) ^{-2}\left( 1+\frac{\left \vert x\right \vert ^{2}}{1+t
}\right) ^{-\frac{3}{2}}*_{x,t}\left( 1+t \right) ^{-3}\left( 1+\frac{\left
\vert x\right \vert ^{2}}{ 1+t }\right) ^{-3}  \notag \\
&\lesssim& \left( 1+t \right) ^{-2}\left( 1+\frac{\left \vert x\right \vert
^{2}}{ 1+t }\right) ^{-\frac{3}{2}}\hbox{.}  \label{E1}
\end{eqnarray*}
\end{lemma}

\begin{lemma}[Space-time exponential decay convolved with Huygens wave]
\label{Easy2}
\begin{eqnarray*}
&&e^{-\frac{t+|x|}{c_{0}}}*_{x,t}\left(1+t\right) ^{-4}\left( 1+\frac{\left(
\left \vert x\right \vert -\mathbf{c}t \right) ^{2}}{1+t }\right) ^{-2}
\notag \\
&\lesssim &\left( 1+t\right) ^{-4}\left( 1+\frac{\left( \left \vert x\right
\vert -\mathbf{c}t\right) ^{2}}{1+t}\right) ^{-2}\hbox{.}  \label{E2}
\end{eqnarray*}
\end{lemma}

\begin{lemma}[Riesz wave convolved with Huygens wave]
\label{poly-moving}
\begin{eqnarray*}
I &=&\mathbf{1}_{\{ \left \vert x\right \vert \leq \mathbf{c}t\}}\left(
1+t\right) ^{-2}\left( 1+\frac{\left \vert x\right \vert ^{2}}{1+t}\right)
^{-\frac{3}{2}}\ast _{x,t}\left( 1+t\right) ^{-4}\left( 1+\frac{\left( \left
\vert x\right \vert -\mathbf{c}t\right) ^{2}}{1+t}\right) ^{-2}  \notag \\
&\lesssim &\left( 1+t\right) ^{-2}\left( 1+\frac{\left \vert x\right \vert
^{2}}{1+t}\right) ^{-\frac{3}{2}}+\left( 1+t\right) ^{-5/2}\left( 1+\frac{%
\left( \left \vert x\right \vert -\mathbf{c}t\right) ^{2}}{1+t}\right) ^{-1}%
\hbox{.}  \label{C11}
\end{eqnarray*}
\end{lemma}

\begin{lemma}[Huygens wave convolved with Huygens wave]
\label{moving-moving}
\begin{eqnarray*}
J&=&\left( 1+t\right) ^{-5/2}e^{-\frac{\left( \left \vert x\right \vert -%
\mathbf{c}t \right) ^{2}}{D_{0}\left( 1+t\right) }}*_{x,t}\left(1+t\right)
^{-4}\left( 1+\frac{\left( \left \vert x\right \vert -\mathbf{c}t \right)
^{2}}{1+t }\right) ^{-2}  \label{C10} \\
&\lesssim &\left( 1+t\right) ^{-2}\left( 1+\frac{\left \vert x\right \vert
^{2}}{1+t}\right) ^{-\frac{3}{2}}+\left( 1+t\right) ^{-5/2}\left( 1+\frac{%
\left( \left \vert x\right \vert -\mathbf{c}t\right) ^{2}}{1+t}\right) ^{-1}%
\hbox{.}  \notag
\end{eqnarray*}
\end{lemma}

\begin{lemma}[Huygens wave convolved with diffusion wave]
\label{moving-nonmoving}
\begin{eqnarray*}
K&=&\left( 1+t\right) ^{-5/2}e^{-\frac{\left( \left \vert x\right \vert -%
\mathbf{c}t \right) ^{2}}{D_{0}\left( 1+t\right) }}*_{x,t}\left(1+t\right)
^{-3}\left( 1+\frac{\left \vert x\right \vert ^{2}}{1+t }\right) ^{-3}
\label{C7} \\
&\lesssim &\left( 1+t\right) ^{-2}\left( 1+\frac{\left \vert x\right \vert
^{2}}{1+t}\right) ^{-\frac{3}{2}}+\left( 1+t\right) ^{-5/2}\left( 1+\frac{%
\left( \left \vert x\right \vert -\mathbf{c}t\right) ^{2}}{1+t}\right) ^{-1}%
\hbox{.}  \notag
\end{eqnarray*}
\end{lemma}

\begin{lemma}[Diffusion wave convolutied with Huygens wave]
\label{nonmoving-moving}
\begin{eqnarray*}
L &=&\left( 1+t\right) ^{-2}e^{-\frac{\left \vert x\right \vert ^{2}}{%
D_{0}\left( 1+t\right) }}\ast _{x,t}\left( 1+t\right) ^{-4}\left( 1+\frac{%
\left( \left \vert x\right \vert -\mathbf{c}t\right) ^{2}}{1+t}\right) ^{-2}
\\
&\lesssim &\left( 1+t\right) ^{-2}\left( 1+\frac{\left \vert x\right \vert
^{2}}{1+t}\right) ^{-\frac{3}{2}}+\left( 1+t\right) ^{-5/2}\left( 1+\frac{%
\left( \left \vert x\right \vert -\mathbf{c}t\right) ^{2}}{1+t}\right) ^{-1}%
\hbox{.}
\end{eqnarray*}
\end{lemma}

\bigskip

We omit the proof of Lemma \ref{Easy1} and Lemma \ref{Easy2}. Before we
proceed to the detailed proof of Lemmas \ref{poly-moving}-\ref%
{nonmoving-moving}, let us present some heuristic calculations to help
understand the mechanism of nonlinear wave interactions. We use two examples
for illustration: the convolution of two Huygens waves (Lemma \ref%
{moving-moving}) and convolution of diffusion wave with a Huygens wave
(Lemma \ref{nonmoving-moving}).

{\color{black} }

\begin{proof}[\protect \underline{Heuristic estimate for Lemma \protect \ref%
{moving-moving}}]
Set $\mathbf{c}=1$ for simplicity. Then
\begin{equation}  \label{eq:J}
J=\int_{0}^{t} \int_{{\mathbb{R}}^3} (1+t-s)^{-5/2} e^{-\frac{(\lvert
x-y\rvert-(t-s))^2}{D_0(t-s)}} (1+s)^{-4} \Bigl( 1+\frac{\left( \left \vert
y\right \vert -s\right) ^{2}}{1+s}\Bigr) ^{-2} dy ds.
\end{equation}
We can interpret $(1+t-s)^{-5/2} e^{-\frac{(\lvert x-y\rvert-(t-s))^2}{%
D_0(t-s)}} $ as a receiver located at $(x,t)$ that can only receive signals
along the wave cone concentrated on $\lvert x-y\rvert=(t-s)$ with thickness $%
\sqrt{t-s}$. Similarly, $(1+s)^{-4} \Bigl( 1+\frac{\left( \left
\vert
y\right \vert -s\right) ^{2}}{1+s}\Bigr) ^{-2} $ can be viewed as a sender
located at $(0,0)$ that sends signals along the wave cone concentrated on $%
\lvert y\rvert=s$ with thickness $\sqrt{s}$.

\begin{figure}[htbp]
\centering
\subfigure[]{
		\includegraphics[width=0.55\textwidth]{f-mov-mov.jpg}
		\label{fig:image1}
	} 
\subfigure[]{
		\includegraphics[width=0.4\textwidth]{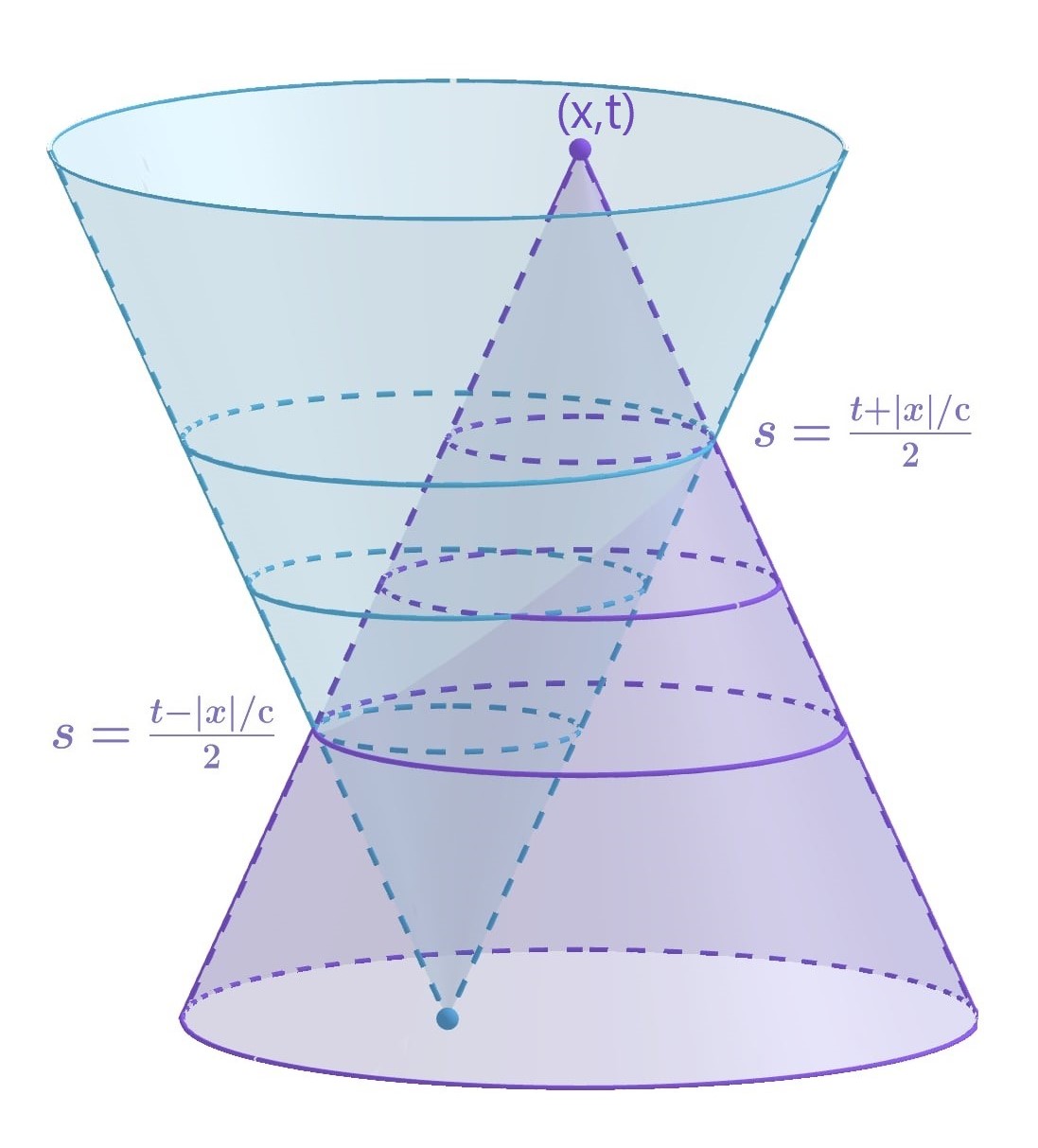}
		\label{fig:image2}
	} 
\caption{Interaction between two Huygens waves}
\end{figure}

During the interaction process, the interaction becomes strong in the
following space-time region (See Figure \ref{fig:image1})
\begin{equation}
E=\left \{ (y,s)\big|\, \lvert \lvert x-y\rvert -(t-s)\rvert \leq O(1)\sqrt{%
t-s}\mbox{ and }\lvert \lvert y\rvert -s\rvert \leq O(1)\sqrt{s}\right \} .
\label{eq:int-dom}
\end{equation}%
Inside this region, the space decay terms in the convolution are of order $%
O(1)$, and the decay is mainly from time factor. The key point is the sharp
estimate of the volume for the strong interaction region.

If $\lvert x \rvert>t$, as $s$ increases from $0$ to $t$, the receiving and
sending wave cones are almost disjoint, so the interaction is very weak.

Now we focus the case where $O(1)\sqrt{t}<\lvert x \rvert<t-O(1)\sqrt{t}$.
The interaction process starts at $s=(t-\lvert x\rvert )/2$ and ends at $%
s=(t+\lvert x\rvert )/2$ (See Figure \ref{fig:image2}). To satisfy the
condition in \eqref{eq:int-dom}, one has
\begin{equation*}
\left \{ \begin{aligned} &\lvert \lvert x-y \rvert -(t-s)\rvert \leq
O(1)\sqrt{t-s},\\ &\lvert \lvert y \rvert -s\rvert \leq O(1)\sqrt{s}.
\end{aligned} \right.
\end{equation*}
Let $r=\lvert y\rvert $, $\theta$ be the angle between $x$ and $y$, and set $%
O(1)=1$. The above constraints are equivalent to
\begin{equation}  \label{eq:mov-mov-heu}
\left \{ \begin{aligned} &\lvert \sqrt{\lvert x\rvert^2+r^2-2r|x|\cos
\theta} -(t-s)\rvert \leq \sqrt{t-s},\\ &\lvert r -s\rvert \leq \sqrt{s},
\end{aligned} \right.
\end{equation}
The first equation of \eqref{eq:mov-mov-heu} implies {\small
\begin{equation*}
\frac{\lvert x\rvert^2+r^2-(t-s)^2-(t-s)}{2r\lvert x\rvert } - \frac{%
(t-s)^{3/2}}{r\lvert x\rvert }\leq \cos \theta \leq \frac{\lvert
x\rvert^2+r^2-(t-s)^2-(t-s)}{2r\lvert x\rvert } + \frac{(t-s)^{3/2}}{r\lvert
x\rvert }
\end{equation*}
} This leads to
\begin{equation*}
\Delta \cos \theta \approx \frac{(t-s)^{3/2}}{r\lvert x\rvert }.
\end{equation*}
For $s$ large, $r\approx s$, $\Delta r\approx \sqrt{s}$. Then inside the
strong interaction region, \eqref{eq:J} is approximated by

\begin{align*}
J& \approx \int_{\frac{t-\lvert x\rvert }{2}}^{\frac{t+\lvert x\rvert }{2}%
}(1+t-s)^{-5/2}(1+s)^{-4}\underset{r^{2}dr}{\underbrace{s^{2}\sqrt{s}}}\,%
\underset{\sin \theta d\theta }{\underbrace{\frac{(t-s)^{3/2}}{s\lvert
x\rvert }}}ds \\
& \lesssim \int_{\frac{t-\lvert x\rvert }{2}}^{\frac{t+\lvert x\rvert }{2}%
}(1+t-s)^{-1}(1+s)^{-5/2}ds\frac{1}{\lvert x\rvert } \\
& \lesssim \Bigl(\int_{\frac{t-\lvert x\rvert }{2}}^{\frac{t}{2}%
}(1+t)^{-1}(1+s)^{-5/2}ds+\int_{\frac{t}{2}}^{\frac{t+\lvert x\rvert }{2}%
}(1+t-s)^{-1}(1+t)^{-5/2}ds\Bigr)\frac{1}{\lvert x\rvert } \\
& \lesssim \Bigl[(1+t)^{-1}(1+t-\lvert x\rvert )^{-3/2}+(1+t)^{-5/2}\ln
\frac{t}{t-\lvert x\rvert }\Bigr]\frac{1}{\lvert x\rvert }.
\end{align*}%
When $\sqrt{t}\leq \lvert x\rvert \leq t/2$, it is bounded by
\begin{equation*}
(1+t)^{-5/2}\lvert x\rvert ^{-1}\lesssim (1+t)^{-2}\Bigl(1+\frac{\lvert
x\rvert ^{2}}{1+t}\Bigr)^{-3/2}.
\end{equation*}%
When $t/2\leq \lvert x\rvert \leq t-\sqrt{t}$, it is bounded by
\begin{equation*}
t^{-2}(1+t-\lvert x\rvert )^{-3/2}+(1+t)^{-5/2}\frac{t^{2}}{(t-\lvert
x\rvert )^{2}}(1+t)^{-1}\lesssim t^{-5/2}\Bigl(1+\frac{(t-\lvert x\rvert
)^{2}}{1+t}\Bigr)^{-1}.
\end{equation*}%
This is the desired estimate in Lemma \ref{moving-moving}.
\end{proof}

\begin{proof}[\protect \underline{Heuristic estimate for Lemma \protect \ref%
{moving-nonmoving}}]
Set $\mathbf{c}=1$, and one has
\begin{equation*}
K=\int_{0}^{t}\int_{\mathbb{R}^3}\left( 1+t-s\right) ^{-5/2}e^{-\frac{\left(
\left \vert x-y\right \vert -(t-s) \right) ^{2}}{D_{0}\left( 1+t-s\right) }%
}\left(1+s\right) ^{-3}\left( 1+\frac{\left \vert y\right \vert ^{2}}{1+s }%
\right) ^{-3}dyds.
\end{equation*}
Similar as heuristics argument of Lemma \ref{moving-moving}, we view $%
(1+t-s)^{-5/2}e^{-\frac{\left( \left \vert x-y\right \vert -(t-s) \right)
^{2}}{D_{0}\left( 1+t-s\right) }}$ as a receiver and $\left(1+s\right)
^{-3}\left( 1+\frac{\left \vert y\right \vert ^{2}}{1+s }\right) ^{-3} $ as
a sender, to have the following figure representing the interaction:
\begin{figure}[h]
\centering
\includegraphics[width=0.6\textwidth]{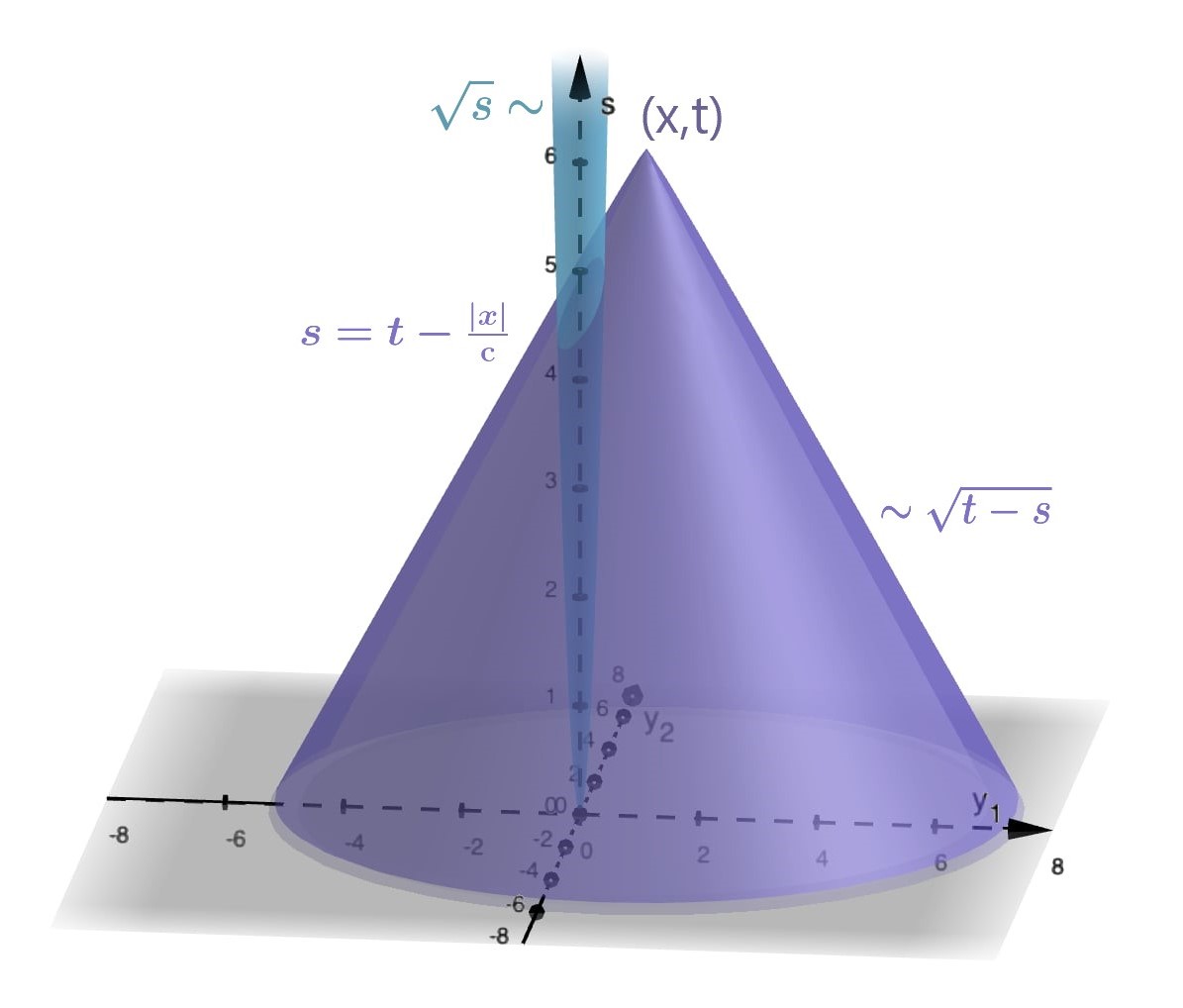}
\caption{Interaction between Huygens and diffusion wave}
\end{figure}

For $O(1)\sqrt{t}\leq \lvert x \rvert \leq t-O(1)\sqrt{t}$, the interaction
occurs at $s\approx t-\lvert x\rvert$ and its duration is of the order $%
\sqrt{s}$.
So the volume of the interaction region in space-time is approximated by
\begin{equation*}
\underset{\mbox{vol of space}}{\underbrace{s^{3/2}}} \cdot \underset{%
\mbox{vol of time}}{\underbrace{\sqrt{s}}}
\end{equation*}
The contribution of the integrand during the interaction is approximated by
\begin{equation*}
(1+t-s)^{-5/2}(1+s)^{-3}\Big|_{s\approx t- \lvert x\rvert }.
\end{equation*}
So we conclude
\begin{align*}
K & \approx (1+t-s)^{-5/2}(1+s)^{-3} s^2 \Big|_{s\approx t-\lvert x\rvert
}\approx (1+\lvert x \rvert)^{-5/2} (1+t- \lvert x\rvert) ^{-1} \\
&\lesssim
\begin{cases}
(1+t)^{-1} (1+\lvert x\rvert )^{-5/2}\lesssim (1+t)^{-2}\Bigl(1+\frac{\lvert
x \rvert^2 }{1+t}\Bigr)^{-3/2}, & \sqrt{t}\leq \lvert x\rvert \leq t/2, \\
(1+t)^{-5/2}(1+t-\lvert x\rvert)^{-1}\lesssim (1+t)^{-5/2} \Bigl(1+\frac{%
(t-\lvert x\rvert)^2}{1+t}\Bigr)^{-1}, & t/2 \leq \lvert x\rvert \leq t-%
\sqrt{t}.%
\end{cases}%
\end{align*}
This yields the desired estimate in Lemma \ref{moving-nonmoving}.
\end{proof}

Other convolutions can be estimated heuristically in a similar manner. It
turns out that the convolutions in Lemmas \ref{moving-moving} and \ref%
{moving-nonmoving} are the dominated ones among all the nonlinear wave
couplings.

We now begin the rigorous proofs, transforming the previous heuristic
calculations into refined (and complex) convolution estimates!

\begin{proof}[\textbf{Proof of Lemma \protect \ref{poly-moving}}]
(Riesz wave convolved with Huygens wave). \newline
\newline
\textbf{Case 1:} $\left( x,t\right) \in D_{1}$. Direct computation gives%
\begin{eqnarray*}
I &=&\left( \int_{0}^{\frac{t}{2}}+\int_{\frac{t}{2}}^{t}\right) \int_{%
\mathbb{R}^{3}}\left( \cdots \right) dyd\tau \\
&\lesssim &\left( 1+t\right) ^{-2}\int_{0}^{\frac{t}{2}}\int_{\mathbb{R}%
^{3}}\left( 1+\tau \right) ^{-4}\left( 1+\frac{\left( \left \vert y\right
\vert -\mathbf{c}\tau \right) ^{2}}{1+\tau }\right) ^{-2}dyd\tau \\
&&+\left( 1+t\right) ^{-4}\int_{\frac{t}{2}}^{t}\int_{\left \vert x-y\right
\vert \leq \mathbf{c}\left( t-\tau \right) }\left( 1+t-\tau \right)
^{-2}\left( 1+\frac{\left \vert x-y\right \vert ^{2}}{D_{0}\left( 1+t-\tau
\right) }\right) ^{-\frac{3}{2}}dyd\tau \\
&\lesssim &\left( 1+t\right) ^{-2}+\left( 1+t\right) ^{-4}\left(
1+t\right)^{3/4} \lesssim \left( 1+t\right) ^{-2}\left( 1+\frac{\left \vert
x\right \vert ^{2}}{1+t}\right) ^{-\frac{3}{2}}\hbox{.}
\end{eqnarray*}
\newline
\textbf{Case 2:} $\left( x,t\right) \in D_{2}$. We split the integral $I$
into two parts%
\begin{equation*}
I=\left( \int_{0}^{\frac{t}{4}}+\int_{\frac{t}{4}}^{t}\right) \int_{\mathbb{R%
}^{3}}\left( \cdots \right) dyd\tau =:I_{1}+I_{2}\hbox{.}
\end{equation*}

For $I_{2}$, one can see that
\begin{eqnarray*}
I_{2} &\lesssim &\left( 1+t\right) ^{-4}\int_{\frac{t}{4}}^{t}\int_{\mathbb{R%
}^{3}}\mathbf{1}_{\{ \left \vert x-y\right \vert \leq \mathbf{c}\left(
t-\tau \right) \}}\left( 1+t-\tau \right) ^{-2}\left( 1+\frac{\left \vert
x-y\right \vert ^{2}}{D_{0}\left( 1+t-\tau \right) }\right) ^{-\frac{3}{2}%
}dyd\tau \\
&\lesssim &\left( 1+t\right) ^{-4}\int_{\frac{t}{4}}^{t} \left( 1+t-\tau
\right)^{-1/4} d\tau \lesssim \left( 1+t\right) ^{-4}\left( 1+t\right)^{3/4}
\\
&\lesssim &\left( 1+t\right) ^{-5/2}\left( 1+\frac{\left( \left \vert
x\right \vert -\mathbf{c}t\right) ^{2}}{1+t}\right) ^{-1}\hbox{.}
\end{eqnarray*}

For $I_{1}$, we further decompose $\mathbb{R}^{3}$ into two parts%
\begin{equation*}
I_{1}=\left( \int_{0}^{\frac{t}{4}}\int_{\left \vert y\right \vert \leq
\frac{2}{3}\left \vert x\right \vert }+\int_{0}^{\frac{t}{4}}\int_{\left
\vert y\right \vert >\frac{2}{3}\left \vert x\right \vert }\right) \left(
\cdots \right) dyd\tau =I_{11}+I_{12}\hbox{.}
\end{equation*}%
If $\left \vert y\right \vert \leq \frac{2}{3}\left \vert x\right \vert $,
then we have
\begin{equation*}
\left \vert x-y\right \vert \geq \left \vert x\right \vert -\left \vert
y\right \vert \geq \frac{\left \vert x\right \vert }{3}\hbox{,}
\end{equation*}%
and thus
\begin{eqnarray*}
I_{11} &\lesssim &\left( 1+t\right) ^{-2}\left( 1+\frac{\left \vert x\right
\vert ^{2}}{1+t}\right) ^{-\frac{3}{2}}\int_{0}^{\frac{t}{4}}\int_{\mathbb{R}%
^{3}}\left( 1+\tau \right) ^{-4}\left( 1+\frac{\left( \left \vert y\right
\vert -\mathbf{c}\tau \right) ^{2}}{1+\tau }\right) ^{-2}dyd\tau \\
&\lesssim &\left( 1+t\right) ^{-2}\left( 1+\frac{\left \vert x\right \vert
^{2}}{1+t}\right) ^{-\frac{3}{2}}\int_{0}^{\frac{t}{4}}\left( 1+\tau \right)
^{-4}\left( 1+\tau \right) ^{\frac{5}{2}}d\tau \\
&\lesssim &\left( 1+t\right) ^{-2}\left( 1+\frac{\left \vert x\right \vert
^{2}}{1+t}\right) ^{-\frac{3}{2}}\hbox{.}
\end{eqnarray*}%
If $\left \vert y\right \vert >\frac{2}{3}\left \vert x\right \vert $ and $%
0\leq \tau \leq \frac{t}{4}$,
\begin{equation*}
\left \vert y\right \vert -\mathbf{c}\tau \geq \frac{2}{3}\left \vert
x\right \vert -\mathbf{c}\tau \geq \frac{\left \vert x\right \vert }{3}+%
\frac{1}{3}\left( \mathbf{c}t-\sqrt{1+t}\right) -\frac{\mathbf{c}t}{4}\geq
\frac{\left \vert x\right \vert }{3}+\frac{1}{24}\mathbf{c}t
\end{equation*}%
$\allowbreak \allowbreak \allowbreak $for $t\geq 40$. We mention here that $%
\left( x,t\right) \in D_{2}$ with $t\geq \frac{2+2\sqrt{1+\mathbf{c}^{2}}}{%
\mathbf{c}^{2}}(\doteqdot 3.15)$ implies that $\left \vert x\right \vert
\geq \sqrt{1+t}$. Hence for $t\geq 40$,
\begin{eqnarray*}
I_{12} &\lesssim &\left( 1+\left \vert x\right \vert ^{2}\right) ^{-\frac{3}{%
2}}\int_{0}^{\frac{t}{4}}\int_{\left \vert y\right \vert >\frac{2}{3}\left
\vert x\right \vert }\mathbf{1}_{\{ \left \vert x-y\right \vert \leq \mathbf{%
c}\left( t-\tau \right) \}}\left( 1+t-\tau \right) ^{-2}\left( 1+\frac{\left
\vert x-y\right \vert ^{2}}{D_{0}\left( 1+t-\tau \right) }\right) ^{-\frac{3%
}{2}} \\
&&\cdot \left( 1+\tau \right) ^{-\frac{5}{2}}\left( 1+\frac{\left( \left
\vert y\right \vert -\mathbf{c}\tau \right) ^{2}}{1+\tau }\right) ^{-\frac{1%
}{2}}dyd\tau \\
&\lesssim &\left( 1+\left \vert x\right \vert ^{2}\right) ^{-\frac{3}{2}%
}\int_{0}^{\frac{t}{4}}\int_{\left \{ \left \vert y\right \vert >\frac{2}{3}%
\left \vert x\right \vert \right \} \cap \{ \left \vert x-y\right \vert \leq
\mathbf{c}\left( t-\tau \right) \}}\left( 1+t-\tau \right) ^{-2}\left( 1+%
\frac{\left \vert x-y\right \vert ^{2}}{D_{0}\left( 1+t-\tau \right) }%
\right) ^{-\frac{3}{2}} \\
&&\cdot \left( 1+\tau \right) ^{-\frac{5}{2}}\left( 1+\frac{\left \vert
x-y\right \vert ^{2}}{1+t-\tau }\right) ^{-\frac{1}{2}}dyd\tau \\
&\lesssim &\left( 1+\left \vert x\right \vert ^{2}\right) ^{-\frac{3}{2}%
}\int_{0}^{\frac{t}{4}}\left( 1+\tau \right) ^{-\frac{5}{2}}\int_{\mathbb{R}%
^{3}}\left( 1+t-\tau \right) ^{-2}\left( 1+\frac{\left \vert x-y\right \vert
^{2}}{D_{0}\left( 1+t-\tau \right) }\right) ^{-2}dyd\tau \\
&\lesssim &(1+t)^{-1/2}\left( 1+\left \vert x\right \vert ^{2}\right) ^{-%
\frac{3}{2}}\lesssim \left( 1+t\right) ^{-2}\left( 1+\frac{\left \vert
x\right \vert ^{2}}{1+t}\right) ^{-\frac{3}{2}}\hbox{,}
\end{eqnarray*}%
and for $0\leq t\leq 40$,
\begin{eqnarray*}
I_{12} &\lesssim &\left( 1+t\right) ^{-\frac{3}{2}}\int_{0}^{\frac{t}{4}%
}\int_{\mathbb{R}^{3}}\left( 1+\tau \right) ^{-4}\left( 1+\frac{\left( \left
\vert y\right \vert -\mathbf{c}\tau \right) ^{2}}{1+\tau }\right)
^{-2}dyd\tau \\
&\lesssim &C\lesssim \left( 1+t\right) ^{-5/2}\left( 1+\frac{\left( \left
\vert x\right \vert -\mathbf{c}t\right) ^{2}}{1+t}\right) ^{-1}\hbox{.}
\end{eqnarray*}%
Thereupon,
\begin{equation*}
I_{1}\lesssim \left( 1+t\right) ^{-2}\left( 1+\frac{\left \vert x\right
\vert ^{2}}{1+t}\right) ^{-\frac{3}{2}}+\left( 1+t\right) ^{-5/2}\left( 1+%
\frac{\left( \left \vert x\right \vert -\mathbf{c}t\right) ^{2}}{1+t}\right)
^{-1}\hbox{,}
\end{equation*}%
and so
\begin{equation*}
I\lesssim \left( 1+t\right) ^{-2}\left( 1+\frac{\left \vert x\right \vert
^{2}}{1+t}\right) ^{-\frac{3}{2}}+\left( 1+t\right) ^{-5/2}\left( 1+\frac{%
\left( \left \vert x\right \vert -\mathbf{c}t\right) ^{2}}{1+t}\right) ^{-1}%
\hbox{.}
\end{equation*}
\newline
\textbf{Case 3:} $\left( x,t\right) \in D_{3}$. We split the integral $I$
into two parts%
\begin{equation*}
I=\left( \int_{0}^{\frac{t}{2}}+\int_{\frac{t}{2}}^{t}\right) \int_{\mathbb{R%
}^{3}}\left( \cdots \right) dyd\tau =:I_{1}+I_{2}\hbox{.}
\end{equation*}

For $I_{1}$, we decompose $\mathbb{R}^{3}$ into two parts as follows:
\begin{equation*}
I_{1}=\int_{0}^{\frac{t}{2}}\left( \int_{\left \vert y\right \vert \leq
\frac{2}{3}\left \vert x\right \vert }+\int_{\left \vert y\right \vert >%
\frac{2}{3}\left \vert x\right \vert }\right) \left( \ldots \right) dyd\tau
=:I_{11}+I_{12}\hbox{.}
\end{equation*}%
If $\left \vert y\right \vert \leq \frac{2}{3}\left \vert x\right \vert $,
then
\begin{equation*}
\left \vert x-y\right \vert \geq \frac{\left \vert x\right \vert }{3}\hbox{,}
\end{equation*}%
and thus%
\begin{eqnarray*}
I_{11} &\lesssim &\left( 1+t\right) ^{-2}\left( 1+\frac{\left \vert x\right
\vert ^{2}}{1+t}\right) ^{-\frac{3}{2}}\int_{0}^{\frac{t}{2}}\int_{\mathbb{R}%
^{3}}\left( 1+\tau \right) ^{-4}\left( 1+\frac{\left( \left \vert y\right
\vert -\mathbf{c}\tau \right) ^{2}}{1+\tau }\right) ^{-2}dyd\tau \\
&\lesssim &\left( 1+t\right) ^{-2}\left( 1+\frac{\left \vert x\right \vert
^{2}}{1+t}\right) ^{-\frac{3}{2}}\hbox{.}
\end{eqnarray*}%
If $\left \vert y\right \vert >\frac{2}{3}\left \vert x\right \vert $ and $%
0\leq \tau \leq \frac{t}{2}$, we have%
\begin{equation*}
\left \vert y\right \vert -\mathbf{c}\tau \geq \frac{2}{3}\left \vert
x\right \vert -\frac{\mathbf{c}t}{2}=\frac{1}{6}\left \vert x\right \vert +%
\frac{1}{2}\left( \left \vert x\right \vert -\mathbf{c}t\right) \geq \frac{1%
}{6}\mathbf{c}t\hbox{,}
\end{equation*}%
so that%
\begin{eqnarray*}
I_{12} &\lesssim &\left( 1+\left \vert x\right \vert ^{2}\right) ^{-\frac{3}{%
2}}\int_{0}^{\frac{t}{2}}\int_{\left \vert y\right \vert >\frac{2}{3}\left
\vert x\right \vert }\mathbf{1}_{\{ \left \vert x-y\right \vert \leq \mathbf{%
c}\left( t-\tau \right) \}}\left( 1+t-\tau \right) ^{-2}\left( 1+\frac{\left
\vert x-y\right \vert ^{2}}{D_{0}\left( 1+t-\tau \right) }\right) ^{-\frac{3%
}{2}} \\
&&\cdot \left( 1+\tau \right) ^{-\frac{5}{2}}\left( 1+\frac{\left( \left
\vert y\right \vert -\mathbf{c}\tau \right) ^{2}}{1+\tau }\right) ^{-\frac{1%
}{2}}dyd\tau \\
&\lesssim &\left( 1+\left \vert x\right \vert ^{2}\right) ^{-\frac{3}{2}%
}\int_{0}^{\frac{t}{2}}\int_{\left \{ \left \vert y\right \vert \geq \frac{2%
}{3}\left \vert x\right \vert \right \} \cap \{ \left \vert x-y\right \vert
\leq \mathbf{c}\left( t-\tau \right) \}}\left( 1+t-\tau \right) ^{-2}\left(
1+\tau \right) ^{-\frac{5}{2}}\cdot \left( 1+\frac{\left \vert x-y\right
\vert ^{2}}{\left( 1+t-\tau \right) }\right) ^{-2}dyd\tau \\
&\lesssim &\left( 1+\left \vert x\right \vert ^{2}\right) ^{-\frac{3}{2}%
}\int_{0}^{\frac{t}{2}}\left( 1+\tau \right) ^{-\frac{5}{2}}\cdot \int_{%
\mathbb{R}^{3}}\left( 1+t-\tau \right) ^{-2}\left( 1+\frac{\left \vert
x-y\right \vert ^{2}}{\left( 1+t-\tau \right) }\right) ^{-2}dyd\tau \\
&\lesssim &(1+t)^{-1/2}\left( 1+\left \vert x\right \vert ^{2}\right) ^{-%
\frac{3}{2}}\lesssim (1+t)^{-1/2}\left( \frac{1+t}{2}+\frac{\left \vert
x\right \vert ^{2}}{2}\right) ^{-\frac{3}{2}}\lesssim \left( 1+t\right)
^{-2}\left( 1+\frac{\left \vert x\right \vert ^{2}}{1+t}\right) ^{-\frac{3}{2%
}}
\end{eqnarray*}%
since $\left \vert x\right \vert \geq \sqrt{1+t}$. Therefore,%
\begin{equation*}
I_{1}\lesssim \left( 1+t\right) ^{-2}\left( 1+\frac{\left \vert x\right
\vert ^{2}}{1+t}\right) ^{-\frac{3}{2}}\hbox{.}
\end{equation*}

For $I_{2}$, we decompose $\mathbb{R}^{3}$ into two parts
\begin{equation*}
I_{2}=\int_{\frac{t}{2}}^{t}\left( \int_{\left \vert y\right \vert -\mathbf{c%
}\tau \leq \frac{\left \vert x\right \vert -ct}{2}}+\int_{\left \vert
y\right \vert -\mathbf{c}\tau >\frac{\left \vert x\right \vert -ct}{2}%
}\right) \left( \cdots \right) dyd\tau =I_{21}+I_{22}\hbox{. }
\end{equation*}%
If $\left \vert y\right \vert -\mathbf{c}\tau \leq \frac{\left \vert
x\right
\vert -\mathbf{c}t}{2}$, we have
\begin{equation*}
\left \vert x-y\right \vert \geq \left \vert x\right \vert -\left( \frac{%
\left \vert x\right \vert -\mathbf{c}t}{2}\right) -\mathbf{c}\tau \geq \frac{%
\left \vert x\right \vert -\mathbf{c}t}{2}\hbox{,}
\end{equation*}%
and thus%
\begin{eqnarray*}
I_{21} &\lesssim &\left( 1+\left( \left \vert x\right \vert -\mathbf{c}%
t\right) ^{2}\right) ^{-\frac{3}{2}}\int_{\frac{t}{2}}^{t}\int_{\mathbb{R}%
^{3}}(1+t-\tau)^{-1/2}\left( 1+\tau \right) ^{-4}\left( 1+\frac{\left( \left
\vert y\right \vert -\mathbf{c}\tau \right) ^{2}}{1+\tau }\right)
^{-2}dyd\tau \\
&\lesssim &\left( 1+t\right) ^{-\frac{3}{2}}\left( 1+\frac{\left( \left
\vert x\right \vert -\mathbf{c}t\right) ^{2}}{1+t}\right) ^{-\frac{3}{2}%
}\int_{\frac{t}{2}}^{t}(1+t-\tau)^{-1/2}\left( 1+\tau \right) ^{-4}\left(
1+\tau \right) ^{\frac{5}{2}}d\tau \\
&\lesssim &\left( 1+t\right) ^{-5/2}\left( 1+\frac{\left( \left \vert
x\right \vert -\mathbf{c}t\right) ^{2}}{1+t}\right) ^{-\frac{3}{2}}\hbox{.}
\end{eqnarray*}%
If $\left \vert y\right \vert -\mathbf{c}\tau >\frac{\left \vert
x\right
\vert -\mathbf{c}t}{2}$, then
\begin{eqnarray*}
I_{22} &\lesssim &\left( 1+t\right) ^{-4}\left( 1+\frac{\left( \left \vert
x\right \vert -\mathbf{c}t\right) ^{2}}{1+t}\right) ^{-2}\int_{\frac{t}{2}%
}^{t}\int_{\{ \left \vert x-y\right \vert \leq \mathbf{c}\left( t-\tau
\right) \}}\left( 1+t-\tau \right) ^{-2}\left( 1+\frac{\left \vert x-y\right
\vert ^{2}}{D_{0}\left( 1+t-\tau \right) }\right) ^{-\frac{3}{2}}dyd\tau \\
&\lesssim &\left( 1+t\right) ^{-4}\left( 1+\frac{\left( \left \vert x\right
\vert -\mathbf{c}t\right) ^{2}}{1+t}\right) ^{-2}\int_{\frac{t}{2}%
}^{t}(1+t-\tau)^{-1/2}\ln \left( 1+t-\tau \right) d\tau \\
& \lesssim& \left( 1+t\right) ^{-3}\left( 1+\frac{\left( \left \vert x\right
\vert -\mathbf{c}t\right) ^{2}}{1+t}\right) ^{-2}\hbox{.}
\end{eqnarray*}%
Therefore,
\begin{equation*}
I_{2}\lesssim \left( 1+t\right) ^{-5/2}\left( 1+\frac{\left( \left \vert
x\right \vert -\mathbf{c}t\right) ^{2}}{1+t}\right) ^{-\frac{3}{2}}\hbox{.}
\end{equation*}%
Combining all the estimates, we get desired estimate%
\begin{equation*}
I\lesssim \left( 1+t\right) ^{-2}\left( 1+\frac{\left \vert x\right \vert
^{2}}{1+t}\right) ^{-\frac{3}{2}}+\left( 1+t\right) ^{-5/2}\left( 1+\frac{%
\left( \left \vert x\right \vert -\mathbf{c}t\right) ^{2}}{1+t}\right) ^{-%
\frac{3}{2}}\hbox{.}
\end{equation*}
\newline
\textbf{Case 4:} $\left( x,t\right) \in D_{4}$. In this region $\sqrt{1+t}%
\leq \left \vert x\right \vert \leq \frac{\mathbf{c}t}{2}$. We split the
integral $I $ into three parts:%
\begin{eqnarray*}
I &=&\int_{0}^{t}\int_{\mathbb{R}^{3}}\left( 1+t-\tau \right) ^{-4}\left( 1+%
\frac{\left( \left \vert x-y\right \vert -\mathbf{c}\left( t-\tau \right)
\right) ^{2}}{1+t-\tau }\right) ^{-2}\left( 1+\tau \right) ^{-2}\left( 1+%
\frac{\left \vert y\right \vert ^{2}}{D_{0}\left( 1+\tau \right) }\right) ^{-%
\frac{3}{2}}\mathbf{1}_{\{ \left \vert y\right \vert \leq \mathbf{c}\tau
\}}dyd\tau \\
&=&\left( \int_{0}^{\frac{t}{2}-\frac{\left \vert x\right \vert }{2\mathbf{c}%
}}+\int_{\frac{t}{2}-\frac{\left \vert x\right \vert }{2\mathbf{c}}}^{t-%
\frac{\left \vert x\right \vert }{2\mathbf{c}}}+\int_{t-\frac{\left \vert
x\right \vert }{2\mathbf{c}}}^{t}\right) \int_{\mathbb{R}^{3}}\left( \cdots
\right) dyd\tau =:I_{1}+I_{2}+I_{3}\hbox{.}
\end{eqnarray*}

For $I_{1}$, we decompose $\mathbb{R}^{3}$ into two parts%
\begin{equation*}
I_{1}=\int_{0}^{\frac{t}{2}-\frac{\left \vert x\right \vert }{2\mathbf{c}}%
}\left( \int_{\left \vert y\right \vert \leq \frac{\mathbf{c}\left( t-\tau
\right) -\left \vert x\right \vert }{2}}+\int_{\left \vert y\right \vert >%
\frac{\mathbf{c}\left( t-\tau \right) -\left \vert x\right \vert }{2}%
}\right) =:I_{11}+I_{12}\hbox{.}
\end{equation*}%
If $0\leq \tau \leq \frac{t}{2}-\frac{\left \vert x\right \vert }{2\mathbf{c}%
}$ and $\left \vert y\right \vert \leq \frac{\mathbf{c}\left( t-\tau \right)
-\left \vert x\right \vert }{2}$, then%
\begin{equation*}
\mathbf{c}\left( t-\tau \right) -\left \vert x-y\right \vert \geq \mathbf{c}%
\left( t-\tau \right) -\left \vert x\right \vert -\left \vert y\right \vert
\geq \frac{\mathbf{c}t-\left \vert x\right \vert }{4}\hbox{,}
\end{equation*}%
so that
\begin{eqnarray*}
I_{11} &\lesssim &\left( 1+t+\left \vert x\right \vert \right) ^{-4}\left( 1+%
\frac{\left( \mathbf{c}t-\left \vert x\right \vert \right) ^{2}}{1+t}\right)
^{-2}\int_{0}^{\frac{t}{2}}\int_{\left \vert y\right \vert \leq \mathbf{c}%
\tau }\left( 1+\tau \right) ^{-2}\left( 1+\frac{\left \vert y\right \vert
^{2}}{D_{0}\left( 1+\tau \right) }\right) ^{-\frac{3}{2}}dyd\tau \\
&\lesssim &\left( 1+t\right) ^{-4}\left( 1+\frac{\left( \mathbf{c}t-\left
\vert x\right \vert \right) ^{2}}{1+t}\right) ^{-2}\left( 1+t\right)^{3/4}
\lesssim \left( 1+t\right) ^{-13/4}\left( 1+\frac{\left( \mathbf{c}t-\left
\vert x\right \vert \right) ^{2}}{1+t}\right) ^{-2}\hbox{.}
\end{eqnarray*}%
\ If $0\leq \tau \leq \frac{t}{2}-\frac{\left \vert x\right \vert }{2\mathbf{%
c}}$ and $\left \vert y\right \vert >\frac{\mathbf{c}\left( t-\tau \right)
-\left \vert x\right \vert }{2}$, then%
\begin{equation*}
\left \vert y\right \vert >\frac{\mathbf{c}\left( t-\tau \right) -\left
\vert x\right \vert }{2}\geq \frac{\mathbf{c}t-\left \vert x\right \vert }{4}%
\hbox{,}
\end{equation*}%
and thus%
\begin{eqnarray*}
I_{12} &\lesssim &\left( 1+\frac{t}{2}+\frac{\left \vert x\right \vert }{2%
\mathbf{c}}\right) ^{-4}\left( 1+\left( \mathbf{c}t-\left \vert x\right
\vert \right) ^{2}\right) ^{-\frac{3}{2}}\int_{0}^{\frac{t}{2}-\frac{\left
\vert x\right \vert }{\mathbf{c}}}\int_{\mathbb{R}^{3}}(1+\tau)^{-1/2}\left(
1+\frac{\left( \left \vert x-y\right \vert -\mathbf{c}\left( t-\tau \right)
\right) ^{2}}{1+t-\tau }\right) ^{-2}dyd\tau \\
&\lesssim &\left( 1+t\right) ^{-4}\left( 1+\left( \mathbf{c}t-\left \vert
x\right \vert \right) ^{2}\right) ^{-\frac{3}{2}}\int_{0}^{\frac{t}{2}%
}(1+\tau)^{-1/2}\left( 1+t-\tau \right) ^{\frac{5}{2}}d\tau \\
&\lesssim &\left( 1+t\right) ^{-4}\left( \frac{\left( \mathbf{c}t-\left
\vert x\right \vert \right) ^{2}}{2}+\frac{1+t}{2}\right) ^{-\frac{3}{2}%
}\left( 1+t\right) ^{\frac{7}{2}}\lesssim \left( 1+t\right) ^{-5/2}\left( 1+%
\frac{\left( \mathbf{c}t-\left \vert x\right \vert \right) ^{2}}{1+t}\right)
^{-1}
\end{eqnarray*}%
since $\mathbf{c}t-\left \vert x\right \vert \geq \frac{\mathbf{c}t}{2}$.
Therefore,%
\begin{equation*}
I_{1}\lesssim \left( 1+t\right) ^{-5/2}\left( 1+\frac{\left( \mathbf{c}%
t-\left \vert x\right \vert \right) ^{2}}{1+t}\right) ^{-1}\hbox{.}
\end{equation*}

For $I_{2}$, we use the spherical coordinates to obtain
\begin{eqnarray*}
I_{2} &\lesssim &\int_{\frac{t}{2}-\frac{\left \vert x\right \vert }{2%
\mathbf{c}}}^{t-\frac{\left \vert x\right \vert }{2\mathbf{c}}%
}\int_{0}^{\infty }\int_{0}^{\pi }\left( 1+t-\tau \right) ^{-4}\left( 1+%
\frac{\left( \sqrt{\left \vert x\right \vert ^{2}+r^{2}-2r\left \vert
x\right \vert \cos \theta }-\mathbf{c}\left( t-\tau \right) \right) ^{2}}{%
1+t-\tau }\right) ^{-2} \\
&&\cdot \left( 1+\tau \right) ^{-2}\left( 1+\frac{\left \vert r\right \vert
^{2}}{\left( 1+\tau \right) }\right) ^{-\frac{3}{2}}\mathbf{1}_{\{r\leq
\mathbf{c}\tau \}}r^{2}\sin \theta d\theta drd\tau \\
&\lesssim &\int_{\frac{t}{2}-\frac{\left \vert x\right \vert }{2\mathbf{c}}%
}^{t-\frac{\left \vert x\right \vert }{2\mathbf{c}}}\int_{0}^{\infty
}\int_{\left \vert \left \vert x\right \vert -r\right \vert }^{\left \vert
x\right \vert +r}\left( 1+t-\tau \right) ^{-4}\left( 1+\frac{\left( z-%
\mathbf{c}\left( t-\tau \right) \right) ^{2}}{1+t-\tau }\right) ^{-2}z\left(
1+\tau \right) ^{-2} \\
&&\cdot \left( 1+\frac{r^{2}}{\left( 1+\tau \right) }\right) ^{-\frac{3}{2}}%
\frac{r}{\left \vert x\right \vert }\mathbf{1}_{\{r\leq \mathbf{c}\tau
\}}dzdrd\tau \\
&\lesssim &\left \vert x\right \vert ^{-1}\int_{\frac{t}{2}-\frac{\left
\vert x\right \vert }{2\mathbf{c}}}^{t-\frac{\left \vert x\right \vert }{2%
\mathbf{c}}}\int_{0}^{\infty }\left( 1+t-\tau \right) ^{-4+\frac{3}{2}%
}\left( 1+\tau \right) ^{-2}\cdot r\left( 1+\frac{r^{2}}{\left( 1+\tau
\right) }\right) ^{-\frac{3}{2}}\mathbf{1}_{\{r\leq \mathbf{c}\tau \}}drd\tau
\\
&\lesssim &\left \vert x\right \vert ^{-1}\int_{\frac{t}{2}-\frac{\left
\vert x\right \vert }{2\mathbf{c}}}^{t-\frac{\left \vert x\right \vert }{2%
\mathbf{c}}}\left( 1+t-\tau \right) ^{-\frac{5}{2}}\left( 1+\tau \right)
^{-1 }d\tau \\
&\lesssim &\left \vert x\right \vert ^{-1}\left( 1+\frac{\left \vert x\right
\vert }{2\mathbf{c}}\right) ^{-\frac{3}{2}}\left( 1+\frac{t}{2}-\frac{\left
\vert x\right \vert }{2\mathbf{c}}\right) ^{-1} \\
&\lesssim &(1+t)^{-1/2}\left( 1+\left \vert x\right \vert \right)
^{-3}\lesssim \left( 1+t\right) ^{-2}\left( 1+\frac{\left \vert x\right
\vert ^{2}}{1+t}\right) ^{-\frac{3}{2}}\hbox{.}
\end{eqnarray*}%
by setting $z=\sqrt{\left \vert x\right \vert ^{2}+r^{2}-2r\left \vert
x\right \vert \cos \theta }$ and $\sin \theta d\theta =\frac{z}{r\left \vert
x\right \vert }dz$. $%
\vspace {3mm}%
$

For $I_{3}$, we decompose $\mathbb{R}^{3}$ into two parts%
\begin{equation*}
I_{3}=\int_{t-\frac{\left \vert x\right \vert }{2\mathbf{c}}}^{t}\left(
\int_{\left \vert y\right \vert \leq \frac{\left \vert x\right \vert -%
\mathbf{c}\left( t-\tau \right) }{2}}+\int_{\left \vert y\right \vert >\frac{%
\left \vert x\right \vert -\mathbf{c}\left( t-\tau \right) }{2}}\right)
=:I_{31}+I_{32}\hbox{.}
\end{equation*}%
If $t-\frac{\left \vert x\right \vert }{2\mathbf{c}}\leq \tau \leq t$, $%
\left \vert y\right \vert \leq \frac{\left \vert x\right \vert -\mathbf{c}%
\left( t-\tau \right) }{2}$, then
\begin{equation*}
\left \vert x-y\right \vert -\mathbf{c}\left( t-\tau \right) \geq \left
\vert x\right \vert -\left \vert y\right \vert -\mathbf{c}\left( t-\tau
\right) \geq \frac{\left \vert x\right \vert -\mathbf{c}\left( t-\tau
\right) }{2}\geq \frac{\left \vert x\right \vert }{4}\hbox{.}
\end{equation*}%
If $t-\frac{\left \vert x\right \vert }{2\mathbf{c}}\leq \tau \leq t$, $%
\left \vert y\right \vert >\frac{\left \vert x\right \vert -\mathbf{c}\left(
t-\tau \right) }{2}$, then%
\begin{equation*}
\left \vert y\right \vert >\frac{\left \vert x\right \vert -\mathbf{c}\left(
t-\tau \right) }{2}\geq \frac{\left \vert x\right \vert }{4}\hbox{.}
\end{equation*}%
Hence,%
\begin{eqnarray*}
I_{31} &\lesssim &\left( 1+\left \vert x\right \vert ^{2}\right) ^{-\frac{3}{%
2}}\int_{t-\frac{\left \vert x\right \vert }{2\mathbf{c}}}^{t}\int_{\left
\vert y\right \vert \leq \frac{\left \vert x\right \vert -\mathbf{c}\left(
t-\tau \right) }{2}}\left( 1+t-\tau \right) ^{-\frac{5}{2}}\left( 1+\tau
\right) ^{-2}\left( 1+\frac{\left \vert y\right \vert ^{2}}{D_{0}\left(
1+\tau \right) }\right) ^{-2}dyd\tau \\
&\lesssim &(1+t)^{-1/2}\left( 1+\left \vert x\right \vert ^{2}\right) ^{-%
\frac{3}{2}}\lesssim \left( 1+t\right) ^{-2}\left( 1+\frac{\left \vert
x\right \vert ^{2}}{\left( 1+t\right) }\right) ^{-\frac{3}{2}}\hbox{,}
\end{eqnarray*}%
since $\left \vert x\right \vert \geq \sqrt{1+t}$, and%
\begin{eqnarray*}
I_{32} &\lesssim &\left( 1+t\right) ^{-2}\left( 1+\frac{\left \vert x\right
\vert ^{2}}{\left( 1+t\right) }\right) ^{-\frac{3}{2}}\int_{t-\frac{\left
\vert x\right \vert }{2\mathbf{c}}}^{t}\int_{\mathbb{R}^{3}}\left( 1+t-\tau
\right) ^{-4}\left( 1+\frac{\left( \left \vert x-y\right \vert -\mathbf{c}%
\left( t-\tau \right) \right) ^{2}}{1+t-\tau }\right) ^{-2}dyd\tau \\
&\lesssim &\left( 1+t\right) ^{-2}\left( 1+\frac{\left \vert x\right \vert
^{2}}{\left( 1+t\right) }\right) ^{-\frac{3}{2}}\hbox{.}
\end{eqnarray*}%
Therefore,
\begin{equation*}
I_{3}\lesssim \left( 1+t\right) ^{-2}\left( 1+\frac{\left \vert x\right
\vert ^{2}}{\left( 1+t\right) }\right) ^{-\frac{3}{2}}\hbox{.}
\end{equation*}%
Combining all the estimates yields
\begin{equation*}
I\lesssim \left( 1+t\right) ^{-2}\left( 1+\frac{\left \vert x\right \vert
^{2}}{\left( 1+t\right) }\right) ^{-\frac{3}{2}}+\left( 1+t\right)
^{-5/2}\left( 1+\frac{\left( \mathbf{c}t-\left \vert x\right \vert \right)
^{2}}{1+t}\right) ^{-1}\hbox{.}
\end{equation*}
\newline
\textbf{Case 5:} $\left( x,t\right) \in D_{5}$. In this region $\mathbf{c}%
t/2\leq \left \vert x\right \vert \leq \mathbf{c}t-\sqrt{1+t}$. Now we split
the integral $I$ into three parts%
\begin{eqnarray*}
I &=&\int_{0}^{t}\int_{\mathbb{R}^{3}}\left( 1+t-\tau \right) ^{-4}\left( 1+%
\frac{\left( \left \vert x-y\right \vert -\mathbf{c}\left( t-\tau \right)
\right) ^{2}}{1+t-\tau }\right) ^{-2}\left( 1+\tau \right) ^{-2 }\left( 1+%
\frac{\left \vert y\right \vert ^{2}}{D_{0}\left( 1+\tau \right) }\right) ^{-%
\frac{3}{2}}\mathbf{1}_{\{ \left \vert y\right \vert \leq \mathbf{c}\tau
\}}dyd\tau \\
&=&\left( \int_{0}^{\frac{t}{2}-\frac{\left \vert x\right \vert }{2\mathbf{c}%
}}+\int_{\frac{t}{2}-\frac{\left \vert x\right \vert }{2\mathbf{c}}}^{\frac{1%
}{2}\left( t+\frac{3}{2}\left( t-\frac{\left \vert x\right \vert }{\mathbf{c}%
}\right) \right) }+\int_{\frac{1}{2}\left( t+\frac{3}{2}\left( t-\frac{\left
\vert x\right \vert }{\mathbf{c}}\right) \right) }^{t}\right) \int_{\mathbb{R%
}^{3}}\left( \cdots \right) dyd\tau =:I_{1}+I_{2}+I_{3}\hbox{.}
\end{eqnarray*}

For $I_{1}$, the estimate is the same as the $I_{1}$ of Case 4, so
\begin{equation*}
I_{1}\lesssim \left( 1+t\right) ^{-5/2}\left( 1+\frac{\left( \mathbf{c}%
t-\left \vert x\right \vert \right) ^{2}}{1+t}\right) ^{-1}\hbox{.}
\end{equation*}

For $I_{2}$, we use the spherical coordinates to obtain%
\begin{eqnarray*}
I_{2} &=&\int_{\frac{t}{2}-\frac{\left \vert x\right \vert }{2\mathbf{c}}}^{%
\frac{1}{2}\left( t+\frac{3}{2}\left( t-\frac{\left \vert x\right \vert }{%
\mathbf{c}}\right) \right) }\int_{0}^{\infty }\int_{0}^{\pi }\left( 1+t-\tau
\right) ^{-4}\left( 1+\frac{\left( \sqrt{\left \vert x\right \vert
^{2}+r^{2}-2r\left \vert x\right \vert \cos \theta }-\mathbf{c}\left( t-\tau
\right) \right) ^{2}}{1+t-\tau }\right) ^{-2} \\
&&\cdot \left( 1+\tau \right) ^{-2}\left( 1+\frac{r^{2}}{D_{0}\left( 1+\tau
\right) }\right) ^{-\frac{3}{2}}\mathbf{1}_{\{r\leq \mathbf{c}\tau
\}}r^{2}\sin \theta d\theta drd\tau \\
&=&\frac{1}{\left \vert x\right \vert }\int_{\frac{t}{2}-\frac{\left \vert
x\right \vert }{2\mathbf{c}}}^{\frac{1}{2}\left( t+\frac{3}{2}\left( t-\frac{%
\left \vert x\right \vert }{\mathbf{c}}\right) \right) }\int_{0}^{\infty
}\int_{\left \vert \left \vert x\right \vert -r\right \vert }^{\left \vert
x\right \vert +r}\left( 1+t-\tau \right) ^{-4}\left( 1+\frac{\left( z-%
\mathbf{c}\left( t-\tau \right) \right) ^{2}}{1+t-\tau }\right) ^{-2}z\left(
1+\tau \right) ^{-2} \\
&&\cdot r\left( 1+\frac{r^{2}}{D_{0}\left( 1+\tau \right) }\right) ^{-\frac{3%
}{2}}\mathbf{1}_{\{r\leq \mathbf{c}\tau \}}dzdrd\tau \\
&\lesssim &\left \vert x\right \vert ^{-1}\int_{\frac{t}{2}-\frac{\left
\vert x\right \vert }{2\mathbf{c}}}^{\frac{1}{2}\left( t+\frac{3}{2}\left( t-%
\frac{\left \vert x\right \vert }{\mathbf{c}}\right) \right) }\left(
1+t-\tau \right) ^{-\frac{5}{2}}\left( 1+\tau \right) ^{-1}d\tau \\
&\lesssim &\left \vert x\right \vert ^{-1}\left( 1+\frac{3\left \vert
x\right \vert }{4\mathbf{c}}-\frac{t}{4}\right) ^{-\frac{3}{2}}\left( 1+%
\frac{t}{2}-\frac{\left \vert x\right \vert }{2\mathbf{c}}\right) ^{-1} \\
&\lesssim &\left \vert x\right \vert ^{-1}(1+t)^{-1/2}\left( 1+\frac{t}{4}-%
\frac{\left \vert x\right \vert }{4\mathbf{c}}+\frac{\left \vert x\right
\vert }{\mathbf{c}}-\frac{t}{2}\right) ^{-1}\left( 1+\frac{t}{2}-\frac{\left
\vert x\right \vert }{2\mathbf{c}}\right) ^{-1} \\
&\lesssim & \left( 1+t\right) ^{-5/2}\left( 1+\frac{\left( \mathbf{c}t-\left
\vert x\right \vert \right) ^{2}}{1+t}\right) ^{-1}
\end{eqnarray*}%
by setting $z=\sqrt{\left \vert x\right \vert ^{2}+r^{2}-2r\left \vert
x\right \vert \cos \theta }$ and $\sin \theta d\theta =\frac{z}{r\left \vert
x\right \vert }dz$.$%
\vspace {3mm}%
$

For $I_{3}$, we decompose $\mathbb{R}^{3}$ into two parts%
\begin{equation*}
I_{3}=\int_{\frac{1}{2}\left( t+\frac{3}{2}\left( t-\frac{\left \vert
x\right \vert }{\mathbf{c}}\right) \right) }^{t}\left( \int_{\left \vert
y\right \vert \leq \frac{\left \vert x\right \vert -\mathbf{c}\left( t-\tau
\right) }{2}}+\int_{\left \vert y\right \vert >\frac{\left \vert x\right
\vert -\mathbf{c}\left( t-\tau \right) }{2}}\right) \left( \cdots \right)
dyd\tau =:I_{31}+I_{32}\hbox{.}
\end{equation*}%
If $\frac{1}{2}\left( t+\frac{3}{2}\left( t-\frac{\left \vert x\right \vert
}{\mathbf{c}}\right) \right) \leq \tau \leq t$, $\left \vert y\right \vert
\leq \frac{\left \vert x\right \vert -\mathbf{c}\left( t-\tau \right) }{2}$,
then
\begin{equation*}
\left \vert x-y\right \vert -\mathbf{c}\left( t-\tau \right) \geq \frac{%
\left \vert x\right \vert -\mathbf{c}\left( t-\tau \right) }{2}\geq \frac{%
\mathbf{c}t+\left \vert x\right \vert }{8}= \frac{\mathbf{c}t-\left \vert
x\right \vert }{8}+\frac{|x|}{4}\hbox{.}
\end{equation*}%
If $\frac{1}{2}\left( t+\frac{3}{2}\left( t-\frac{\left \vert x\right \vert
}{\mathbf{c}}\right) \right) \leq \tau \leq t$, $\left \vert y\right \vert >%
\frac{\left \vert x\right \vert -\mathbf{c}\left( t-\tau \right) }{2}$, then
\begin{equation*}
\left \vert y\right \vert >\frac{\left \vert x\right \vert -\mathbf{c}\left(
t-\tau \right) }{2}\geq \frac{\mathbf{c}t-\left \vert x\right \vert }{8}+%
\frac{|x|}{4} \hbox{.}
\end{equation*}%
Hence,
\begin{eqnarray*}
I_{31} &\lesssim &\left( 1+\left( \mathbf{c}t-\left \vert x\right \vert
\right) ^{2}\right) ^{-2}\int_{\frac{1}{2}\left( t+\frac{3}{2}\left( t-\frac{%
\left \vert x\right \vert }{\mathbf{c}}\right) \right) }^{t}\int_{|y|\leq
\mathbf{c}\tau}\left( 1+t-\tau \right) ^{-2}\left( 1+\tau \right) ^{-2
}\left( 1+\frac{\left \vert y\right \vert ^{2}}{D_{0}\left( 1+\tau \right) }%
\right) ^{-\frac{3}{2}}dyd\tau \\
&\lesssim &\left( 1+\left( \mathbf{c}t-\left \vert x\right \vert \right)
^{2}\right) ^{-2}\int_{\frac{1}{2}\left( t+\frac{3}{2}\left( t-\frac{\left
\vert x\right \vert }{\mathbf{c}}\right) \right) }^{t}\left( 1+t-\tau
\right) ^{-2}(1+\tau)^{-1/4}d\tau \\
&\lesssim &(1+t)^{-1/4}\left( 1+\left( \mathbf{c}t-\left \vert x\right \vert
\right) ^{2}\right) ^{-2}\lesssim \left( 1+t\right) ^{-9/4}\left( 1+\frac{%
\left( \mathbf{c}t-\left \vert x\right \vert \right) ^{2}}{1+t}\right) ^{-2}%
\hbox{.}
\end{eqnarray*}
On the other hand,
\begin{eqnarray*}
I_{31} &\lesssim &\left( 1+|x|^{2}\right) ^{-2}\int_{\frac{1}{2}\left( t+%
\frac{3}{2}\left( t-\frac{\left \vert x\right \vert }{\mathbf{c}}\right)
\right) }^{t}\int_{|y|\leq \mathbf{c}\tau}\left( 1+t-\tau \right)
^{-2}\left( 1+\tau \right) ^{-2 }\left( 1+\frac{\left \vert y\right \vert
^{2}}{D_{0}\left( 1+\tau \right) }\right) ^{-\frac{3}{2}}dyd\tau \\
&\lesssim &(1+t)^{-17/4}\hbox{.}
\end{eqnarray*}
By interpolation, one has
\begin{equation*}
I_{31} \lesssim (1+t)^{-13/4}\left( 1+\frac{\left( \mathbf{c}t-\left \vert
x\right \vert \right) ^{2}}{1+t}\right) ^{-1}\,.
\end{equation*}
For $I_{32}$,
\begin{eqnarray*}
I_{32} &\lesssim &\int_{\frac{1}{2}\left( t+\frac{3}{2}\left( t-\frac{\left
\vert x\right \vert }{\mathbf{c}}\right) \right) }^{t}\left( 1+\left \vert
x\right \vert -\mathbf{c}\left( t-\tau \right) \right) ^{-3}(1+\tau)^{-1/2}
\\
&&\cdot \int_{\left \vert y\right \vert >\frac{\left \vert x\right \vert -%
\mathbf{c}\left( t-\tau \right) }{2}}\left( 1+t-\tau \right) ^{-4}\left( 1+%
\frac{\left( \left \vert x-y\right \vert -\mathbf{c}\left( t-\tau \right)
\right) ^{2}}{1+t-\tau }\right) ^{-2}dyd\tau \\
&\lesssim &\int_{\frac{1}{2}\left( t+\frac{3}{2}\left( t-\frac{\left \vert
x\right \vert }{\mathbf{c}}\right) \right) }^{t}\left( 1+\left \vert x\right
\vert -\mathbf{c}\left( t-\tau \right) \right) ^{-3}\left( 1+t-\tau \right)
^{-4+\frac{5}{2}}\left( 1+\tau \right) ^{-\frac{1}{2}}d\tau \\
&\lesssim &\left( 1+t\right) ^{-7/2}\lesssim (1+t)^{-1/2}\left( 1+\left
\vert x\right \vert \right) ^{-3}\lesssim \left( 1+t\right) ^{-2}\left( 1+%
\frac{\left \vert x\right \vert ^{2}}{\left( 1+t\right) }\right) ^{-\frac{3}{%
2}}\hbox{,}
\end{eqnarray*}%
so that%
\begin{equation*}
I_{3}\lesssim \left( 1+t\right) ^{-2}\left( 1+\frac{\left \vert x\right
\vert ^{2}}{\left( 1+t\right) }\right) ^{-\frac{3}{2}}+\left( 1+t\right)
^{-5/2}\left( 1+\frac{\left( \mathbf{c}t-\left \vert x\right \vert \right)
^{2}}{1+t}\right) ^{-1}\hbox{.}
\end{equation*}%
To sum up,
\begin{equation*}
I\leq \left( 1+t\right) ^{-2}\left( 1+\frac{\left \vert x\right \vert ^{2}}{%
\left( 1+t\right) }\right) ^{-\frac{3}{2}}+\left( 1+t\right) ^{-5/2}\left( 1+%
\frac{\left( \mathbf{c}t-\left \vert x\right \vert \right) ^{2}}{1+t}\right)
^{-1}\hbox{.}
\end{equation*}
\end{proof}

\begin{proof}[\textbf{Proof of Lemma \protect \ref{moving-moving}}]
(Huygens wave convolved with Huygens wave). \newline
\newline
\textbf{Case 1:} $\left( x,t\right) \in D_{1}\cup D_{2}$. Direct computation
gives
\begin{eqnarray*}
J &\lesssim &\int_{0}^{\frac{t}{2}}\int_{\mathbb{R}^{3}}\left( 1+t\right)
^{-5/2}\left( 1+\tau \right) ^{-4}\left( 1+\frac{\left( \left \vert y\right
\vert -\mathbf{c}\tau \right) ^{2}}{1+\tau }\right) ^{-2}dyd\tau \\
&&+\int_{\frac{t}{2}}^{t}\int_{\mathbb{R}^{3}}\left( 1+t-\tau \right)
^{-5/2}e^{-\frac{\left( \left \vert x-y\right \vert -\mathbf{c}\left( t-\tau
\right) \right) ^{2}}{D_{0}\left( 1+t-\tau \right) }}\left( 1+t\right)
^{-4}dyd\tau \\
&\lesssim &\left( 1+t\right) ^{-5/2}\int_{0}^{\frac{t}{2}}\left( 1+\tau
\right) ^{-4}\left( 1+\tau \right) ^{\frac{5}{2}}d\tau +\left( 1+t\right)
^{-4}\int_{\frac{t}{2}}^{t}\left( 1+t-\tau \right) ^{-5/2}\left( 1+t-\tau
\right) ^{\frac{5}{2}}d\tau \\
&\lesssim &\left( 1+t\right) ^{-5/2}\hbox{.}
\end{eqnarray*}
\newline
\textbf{Case 2:} $\left( x,t\right) \in D_{3}$. We split the integral $J$
into four parts%
\begin{eqnarray*}
J &=&\int_{0}^{\frac{t}{2}}\left( \int_{\left \vert y\right \vert -\mathbf{c}%
\tau \leq \frac{\left \vert x\right \vert -\mathbf{c}t}{2}}+\int_{\left
\vert y\right \vert -\mathbf{c}\tau >\frac{\left \vert x\right \vert -%
\mathbf{c}t}{2}}\right) \left( \cdots \right) dyd\tau +\int_{\frac{t}{2}%
}^{t}\left( \int_{\left \vert y\right \vert -\mathbf{c}\tau \leq \frac{\left
\vert x\right \vert -\mathbf{c}t}{2}}+\int_{\left \vert y\right \vert -%
\mathbf{c}\tau >\frac{\left \vert x\right \vert -\mathbf{c}t}{2}}\right)
\left( \cdots \right) dyd\tau \\
&=&J_{11}+J_{12}+J_{21}+J_{22}\hbox{.}
\end{eqnarray*}

Note that if $\left \vert y\right \vert -\mathbf{c}\tau \leq \frac{%
\left
\vert x\right \vert -\mathbf{c}t}{2}$, then
\begin{equation*}
\left \vert x-y\right \vert -\mathbf{c}\left( t-\tau \right) \geq \left
\vert x\right \vert -\left \vert y\right \vert -\mathbf{c}\left( t-\tau
\right) \geq \frac{\left \vert x\right \vert -\mathbf{c}t}{2}\hbox{.}
\end{equation*}%
Hence,
\begin{eqnarray*}
J_{11} &\lesssim &\int_{0}^{\frac{t}{2}}\int_{\left \vert y\right \vert -%
\mathbf{c}\tau \leq \frac{\left \vert x\right \vert -\mathbf{c}t}{2}}\left(
1+t\right) ^{-5/2}e^{-\frac{\left( \left \vert x\right \vert -\mathbf{c}%
t\right) ^{2}}{D_{0}\left( 1+t\right) }}\left( 1+\tau \right) ^{-4}\left( 1+%
\frac{\left( \left \vert y\right \vert -\mathbf{c}\tau \right) ^{2}}{1+\tau }%
\right) ^{-2}dyd\tau \\
&\lesssim &\int_{0}^{\frac{t}{2}}\left( 1+t\right) ^{-5/2}e^{-\frac{\left(
\left \vert x\right \vert -\mathbf{c}t\right) ^{2}}{D_{0}\left( 1+t\right) }%
}\left( 1+\tau \right) ^{-4}\left( 1+\tau \right) ^{\frac{5}{2}}d\tau
\lesssim \left( 1+t\right) ^{-5/2}e^{-\frac{\left( \left \vert x\right \vert
-\mathbf{c}t\right) ^{2}}{D_{0}\left( 1+t\right) }}\hbox{.}
\end{eqnarray*}

For $J_{12}$, we have
\begin{eqnarray*}
J_{12} &\lesssim &\int_{0}^{\frac{t}{2}}\int_{\left \vert y\right \vert -%
\mathbf{c}\tau >\frac{\left \vert x\right \vert -\mathbf{c}t}{2}}\left(
1+t-\tau \right) ^{-5/2}e^{-\frac{\left( \left \vert x-y\right \vert -%
\mathbf{c}\left( t-\tau \right) \right) ^{2}}{D_{0}\left( 1+t-\tau \right) }%
}\left( 1+\tau \right) ^{-4}\left( 1+\frac{\left( \left \vert y\right \vert -%
\mathbf{c}\tau \right) ^{2}}{1+\tau }\right) ^{-2}dyd\tau \\
&\lesssim &\int_{0}^{\frac{t}{2}}\left( 1+t-\tau \right) ^{-5/2}\left(
1+t-\tau \right) ^{\frac{5}{2}}\left( 1+\tau \right) ^{-2}\left( 1+\left(
\left \vert x\right \vert -\mathbf{c}t\right) ^{2}\right) ^{-2}d\tau \\
&\lesssim &\left( 1+\left \vert x\right \vert -\mathbf{c}t\right) ^{-4}%
\hbox{,}
\end{eqnarray*}%
and
\begin{eqnarray*}
J_{12} &\lesssim &\int_{0}^{\frac{t}{2}}\int_{\left \vert y\right \vert -%
\mathbf{c}\tau >\frac{\left \vert x\right \vert -\mathbf{c}t}{2}}\left(
1+t\right) ^{-5/2}\left( 1+\tau \right) ^{-2}\left( 1+\left \vert y\right
\vert -\mathbf{c}\tau \right) ^{-4}dyd\tau \\
&\lesssim &\int_{0}^{\frac{t}{2}}\left( 1+t\right) ^{-5/2}\left( 1+\tau
\right) ^{-2}\int_{\mathbf{c}\tau +\frac{\left \vert x\right \vert -\mathbf{c%
}t}{2}}^{\infty }\left( 1+r-\mathbf{c}\tau \right) ^{-4}r^{2}drd\tau \\
&\lesssim &\int_{0}^{\frac{t}{2}}\left( 1+t\right) ^{-5/2}\left( 1+\tau
\right) ^{-2}\int_{\frac{\left \vert x\right \vert -\mathbf{c}t}{2}}^{\infty
}\left( 1+r\right) ^{-4}\left( r+\mathbf{c}\tau \right) ^{2}drd\tau \\
&\lesssim &\int_{0}^{\frac{t}{2}}\left( 1+t\right) ^{-5/2}\left( 1+\tau
\right) ^{-2}\left[ r^{-1}+\mathbf{c}\tau \cdot r^{-2}+\left( \mathbf{c}\tau
\right) ^{2}r^{-3}\right] _{r=\frac{\left \vert x\right \vert -\mathbf{c}t}{2%
}}d\tau \\
&\lesssim &\left( 1+t\right) ^{-5/2}\left[ \left( 1+\left \vert x\right
\vert -\mathbf{c}t\right) ^{-1}+\ln \left( 2+t\right) \left( 1+\left \vert
x\right \vert -\mathbf{c}t\right) ^{-2}+t\left( 1+\left \vert x\right \vert -%
\mathbf{c}t\right) ^{-3}\right] \\
&\lesssim &\left( 1+t\right) ^{-5/2}\left( 1+\left \vert x\right \vert -%
\mathbf{c}t\right) ^{-1}\left[ 1+\ln \left( 2+t\right) \left( 1+\sqrt{1+t}%
\right) ^{-1}+t\left( 1+\sqrt{1+t}\right) ^{-2}\right] \\
&\lesssim &\left( 1+t\right) ^{-5/2}\left( 1+\left \vert x\right \vert -%
\mathbf{c}t\right) ^{-1}\hbox{,}
\end{eqnarray*}%
which implies that
\begin{eqnarray*}
J_{12} &=&\left( J_{12}\right) ^{\frac{1}{3}}\left( J_{12}\right) ^{\frac{2}{%
3}}\lesssim \left[ \left( 1+\left \vert x\right \vert -\mathbf{c}t\right)
^{-4}\right] ^{\frac{1}{3}}\left[ \left( 1+t\right) ^{-5/2}\left( 1+\left
\vert x\right \vert -\mathbf{c}t\right) ^{-1}\right] ^{\frac{2}{3}} \\
&\lesssim &\left( 1+t\right) ^{-\frac{5}{3}}\left( 1+\left \vert x\right
\vert -\mathbf{c}t\right) ^{-2} \\
&\lesssim &\left( 1+t\right) ^{-\frac{8}{3}}\left( 1+\frac{\left( \left
\vert x\right \vert -\mathbf{c}t\right) ^{2}}{1+t}\right) ^{-1}\hbox{.}
\end{eqnarray*}

For $J_{21}$ and $J_{22}$, it follows that
\begin{eqnarray*}
J_{21} &\lesssim &\int_{\frac{t}{2}}^{t}\int_{\left \vert y\right \vert -%
\mathbf{c}\tau \leq \frac{\left \vert x\right \vert -\mathbf{c}t}{2}}\left(
1+t-\tau \right) ^{-5/2}e^{-\frac{\left( \left \vert x\right \vert -\mathbf{c%
}t\right) ^{2}}{2D_{0}\left( 1+t\right) }}e^{-\frac{\left( \left \vert
x-y\right \vert -\mathbf{c}\left( t-\tau \right) \right) ^{2}}{2D_{0}\left(
1+t-\tau \right) }}\left( 1+t\right) ^{-4}\left( 1+\frac{\left( \left \vert
y\right \vert -\mathbf{c}\tau \right) ^{2}}{1+\tau }\right) ^{-2}dyd\tau \\
&\lesssim &\left( 1+t\right) ^{-4}e^{-\frac{\left( \left \vert x\right \vert
-\mathbf{c}t\right) ^{2}}{2D_{0}\left( 1+t\right) }}\int_{\frac{t}{2}%
}^{t}\left( 1+t-\tau \right) ^{-5/2}\left( 1+t-\tau \right) ^{\frac{5}{2}%
}d\tau \\
&\lesssim &\left( 1+t\right) ^{-3}e^{-\frac{\left( \left \vert x\right \vert
-\mathbf{c}t\right) ^{2}}{2D_{0}\left( 1+t\right) }}\hbox{,}
\end{eqnarray*}%
\begin{eqnarray*}
J_{22} &\lesssim &\int_{\frac{t}{2}}^{t}\int_{\left \vert y\right \vert -%
\mathbf{c}\tau >\frac{\left \vert x\right \vert -\mathbf{c}t}{2}}\left(
1+t-\tau \right) ^{-5/2}e^{-\frac{\left( \left \vert x-y\right \vert -%
\mathbf{c}\left( t-\tau \right) \right) ^{2}}{D_{0}\left( 1+t-\tau \right) }%
}\left( 1+t\right) ^{-4}\left( 1+\frac{\left( \left \vert y\right \vert -%
\mathbf{c}\tau \right) ^{2}}{1+\tau }\right) ^{-2}dyd\tau \\
&\lesssim &\int_{\frac{t}{2}}^{t}\left( 1+t-\tau \right) ^{-5/2}\left(
1+t-\tau \right) ^{\frac{5}{2}}\left( 1+t\right) ^{-4}\left( 1+\frac{\left(
\left \vert x\right \vert -\mathbf{c}t\right) ^{2}}{1+t}\right) ^{-2}d\tau \\
&\lesssim &\left( 1+t\right) ^{-3}\left( 1+\frac{\left( \left \vert x\right
\vert -\mathbf{c}t\right) ^{2}}{1+t}\right) ^{-2}\hbox{.}
\end{eqnarray*}

Gathering all the estimates, we can conclude that
\begin{equation*}
J\lesssim \left( 1+t\right) ^{-5/2}\left( 1+\frac{\left( \left \vert x\right
\vert -\mathbf{c}t\right) ^{2}}{1+t}\right) ^{-1}\hbox{.}
\end{equation*}
\newline
\textbf{Case 3:} $\left( x,t\right) \in D_{4}$. We split the integral into
four parts%
\begin{eqnarray*}
J &=&\left( \int_{0}^{\frac{t}{4}-\frac{\left \vert x\right \vert }{4\mathbf{%
c}}}+\int_{\frac{t}{4}-\frac{\left \vert x\right \vert }{4\mathbf{c}}}^{%
\frac{t}{2}}+\int_{\frac{t}{2}}^{\frac{t}{2}+\frac{1}{4}\left( t+\frac{\left
\vert x\right \vert }{\mathbf{c}}\right) }+\int_{\frac{t}{2}+\frac{1}{4}%
\left( t+\frac{\left \vert x\right \vert }{\mathbf{c}}\right) }^{t}\right)
\int_{\mathbb{R}^{3}}\left( \cdots \right) dyd\tau \\
&=&:J_{1}+J_{2}+J_{3}+J_{4}\hbox{.}
\end{eqnarray*}

For $J_{1}$, we decompose $\mathbb{R}^{3}$ into two parts%
\begin{equation*}
J_{1}=\int_{0}^{\frac{t}{4}-\frac{\left \vert x\right \vert }{4\mathbf{c}}%
}\left( \int_{\left \vert y\right \vert \leq \frac{\mathbf{c}t-\left \vert
x\right \vert }{2}}+\int_{\left \vert y\right \vert >\frac{\mathbf{c}t-\left
\vert x\right \vert }{2}}\right) \int_{\mathbb{R}^{3}}\left( \cdots \right)
dyd\tau =:J_{11}+J_{12}\hbox{.}
\end{equation*}%
If $0\leq \tau \leq \frac{t}{4}-\frac{\left \vert x\right \vert }{4\mathbf{c}%
}$, $\left \vert y\right \vert \leq \frac{\mathbf{c}t-\left \vert
x\right
\vert }{2}$, then
\begin{equation*}
\mathbf{c}\left( t-\tau \right) -\left \vert x-y\right \vert \geq \mathbf{c}%
\left( t-\tau \right) -\left \vert x\right \vert -\left \vert y\right \vert
\geq \frac{\mathbf{c}t-\left \vert x\right \vert }{4}\geq \frac{\mathbf{c}t}{%
8}\hbox{.}
\end{equation*}%
Hence,
\begin{eqnarray*}
J_{11} &\lesssim &\int_{0}^{\frac{t}{4}-\frac{\left \vert x\right \vert }{4%
\mathbf{c}}}\int_{\left \vert y\right \vert \leq \frac{\mathbf{c}t-\left
\vert x\right \vert }{2}}\left( 1+t-s\right) ^{-5/2}e^{-\frac{\left( \left
\vert x\right \vert -\mathbf{c}t\right) ^{2}}{2D_{0}\left( 1+t\right) }%
}\left( 1+\tau \right) ^{-4}\left( 1+\frac{\left( \left \vert y\right \vert -%
\mathbf{c}\tau \right) ^{2}}{1+\tau }\right) ^{-2}dyd\tau \\
&\lesssim &\left( 1+t\right) ^{-5/2}e^{-\frac{\left( \left \vert x\right
\vert -\mathbf{c}t\right) ^{2}}{2D_{0}\left( 1+t\right) }}\int_{0}^{\frac{t}{%
4}-\frac{\left \vert x\right \vert }{4\mathbf{c}}}\left( 1+\tau \right)
^{-4}\left( 1+\tau \right) ^{\frac{5}{2}}d\tau \lesssim \left( 1+t\right)
^{-5/2}e^{-\frac{\left( \left \vert x\right \vert -\mathbf{c}t\right) ^{2}}{%
2D_{0}\left( 1+t\right) }}\hbox{,}
\end{eqnarray*}%
and%
\begin{eqnarray*}
J_{12} &\lesssim &\int_{0}^{\frac{t}{4}-\frac{\left \vert x\right \vert }{4%
\mathbf{c}}}\int_{\left \vert y\right \vert >\frac{\mathbf{c}t-\left \vert
x\right \vert }{2}}\left( 1+t-\tau \right) ^{-5/2}e^{-\frac{\left( \left
\vert x-y\right \vert -\mathbf{c}\left( t-\tau \right) \right) ^{2}}{%
D_{0}\left( 1+t-\tau \right) }}\left( 1+\tau \right) ^{-2}\left( 1+\left(
\mathbf{c}t-\left \vert x\right \vert \right) ^{2}\right) ^{-2}dyd\tau \\
&\lesssim &\left( 1+t\right) ^{-4}\int_{0}^{\frac{t}{4}-\frac{\left \vert
x\right \vert }{4\mathbf{c}}}\left( 1+t-\tau \right) ^{-5/2}\left( 1+t-\tau
\right) ^{\frac{5}{2}}\left( 1+\tau \right) ^{-2}d\tau \\
&\lesssim &\left( 1+t\right) ^{-4}\lesssim \left( 1+t\right) ^{-1}\left(
1+\left \vert x\right \vert \right) ^{-3}\lesssim \left( 1+t\right)
^{-5/2}\left( 1+\frac{\left \vert x\right \vert ^{2}}{1+t}\right) ^{-\frac{3%
}{2}}\hbox{,}
\end{eqnarray*}%
since $\sqrt{1+t}\leq \left \vert x\right \vert \leq \frac{\mathbf{c}t}{2}$,
so that
\begin{equation*}
J_{1}\lesssim \left( 1+t\right) ^{-5/2}e^{-\frac{\left( \left \vert x\right
\vert -\mathbf{c}t\right) ^{2}}{2D_{0}\left( 1+t\right) }}+\left( 1+t\right)
^{-5/2}\left( 1+\frac{\left \vert x\right \vert ^{2}}{1+t}\right) ^{-\frac{3%
}{2}}\hbox{.}
\end{equation*}

For $J_{2}$ and $J_{3}$,
\begin{eqnarray*}
J_{2} &\lesssim &\int_{\frac{t}{4}-\frac{\left \vert x\right \vert }{4%
\mathbf{c}}}^{\frac{t}{2}}\int_{\mathbb{R}^{3}}\left( 1+t\right) ^{-5/2}e^{-%
\frac{\left( \left \vert x-y\right \vert -\mathbf{c}\left( t-\tau \right)
\right) ^{2}}{D_{0}\left( 1+t-\tau \right) }}\left( 1+\tau \right)
^{-4}\left( 1+\frac{\left( \left \vert y\right \vert -\mathbf{c}\tau \right)
^{2}}{1+\tau }\right) ^{-2}dyd\tau \\
&\lesssim &\left( 1+t\right) ^{-5/2}\left( 1+\frac{t}{4}-\frac{\left \vert
x\right \vert }{4\mathbf{c}}\right) ^{-4}\int_{0}^{t}\int_{\mathbb{R}%
^{3}}e^{-\frac{\left( \left \vert x-y\right \vert -\mathbf{c}\left( t-\tau
\right) \right) ^{2}}{D_{0}\left( 1+t-\tau \right) }}\left( 1+\frac{\left(
\left \vert y\right \vert -\mathbf{c}\tau \right) ^{2}}{1+\tau }\right)
^{-2}dyd\tau \\
&\lesssim &\left( 1+t\right) ^{-5/2}\left( 1+t\right) ^{-4}\left( 1+t\right)
^{3}\lesssim \left( 1+t\right) ^{-7/2}\lesssim \left( 1+t\right) ^{-2}\left(
1+\frac{\left \vert x\right \vert ^{2}}{1+t}\right) ^{-\frac{3}{2}}\hbox{,}
\end{eqnarray*}%
\begin{eqnarray*}
J_{3} &\lesssim &\int_{\frac{t}{2}}^{\frac{t}{2}+\frac{1}{4}\left( t+\frac{%
\left \vert x\right \vert }{\mathbf{c}}\right) }\int_{\mathbb{R}^{3}}\left(
1+t-\tau \right) ^{-5/2}e^{-\frac{\left( \left \vert x-y\right \vert -%
\mathbf{c}\left( t-\tau \right) \right) ^{2}}{D_{0}\left( 1+t-\tau \right) }%
}\left( 1+t\right) ^{-4}\left( 1+\frac{\left( \left \vert y\right \vert -%
\mathbf{c}\tau \right) ^{2}}{1+\tau }\right) ^{-2}dyd\tau \\
&\lesssim &\left( 1+\frac{1}{4}\left( t-\frac{\left \vert x\right \vert }{%
\mathbf{c}}\right) \right) ^{-5/2}\left( 1+t\right) ^{-4}\int_{0}^{t}\int_{%
\mathbb{R}^{3}}e^{-\frac{\left( \left \vert x-y\right \vert -\mathbf{c}%
\left( t-\tau \right) \right) ^{2}}{D_{0}\left( 1+t-\tau \right) }}\left( 1+%
\frac{\left( \left \vert y\right \vert -\mathbf{c}\tau \right) ^{2}}{1+\tau }%
\right) ^{-2}dyd\tau \\
&\lesssim &\left( 1+t\right) ^{-13/2}\left( 1+t\right) ^{3}\lesssim \left(
1+t\right) ^{-2}\left( 1+\frac{\left \vert x\right \vert ^{2}}{1+t}\right)
^{-\frac{3}{2}}\hbox{,}
\end{eqnarray*}
due to Lemma 5.4 of \cite{[Liu-Wang]}.

For $J_{4}$, we decompose $\mathbb{R}^{3}$ into two parts%
\begin{equation*}
J_{4}=\int_{\frac{t}{2}+\frac{1}{4}\left( t+\frac{\left \vert x\right \vert
}{\mathbf{c}}\right) }^{t}\left( \int_{\left \vert y\right \vert \leq \frac{%
\left \vert x\right \vert +\mathbf{c}t}{2}}+\int_{\left \vert y\right \vert >%
\frac{\left \vert x\right \vert +\mathbf{c}t}{2}}\right) \left( \cdots
\right) dyd\tau =:J_{41}+J_{42}\hbox{.}
\end{equation*}%
If $\frac{t}{2}+\frac{1}{4}\left( t+\frac{\left \vert x\right \vert }{%
\mathbf{c}}\right) \leq \tau \leq t$, $\left \vert y\right \vert \leq \frac{%
\left \vert x\right \vert +\mathbf{c}t}{2}$, then%
\begin{equation*}
\mathbf{c}\tau -\left \vert y\right \vert \geq \frac{\mathbf{c}t}{2}+\frac{1%
}{4}\left( \mathbf{c}t+\left \vert x\right \vert \right) -\frac{\left \vert
x\right \vert +\mathbf{c}t}{2}\geq \frac{\mathbf{c}t-\left \vert x\right
\vert }{4}\geq \frac{\mathbf{c}t}{8}\hbox{.}
\end{equation*}%
If $\frac{t}{2}+\frac{1}{4}\left( t+\frac{\left \vert x\right \vert }{%
\mathbf{c}}\right) \leq \tau \leq t$, $\left \vert y\right \vert >\frac{%
\left \vert x\right \vert +\mathbf{c}t}{2}$, then%
\begin{eqnarray*}
\left \vert x-y\right \vert -\mathbf{c}\left( t-\tau \right) &\geq &\left
\vert y\right \vert -\left \vert x\right \vert -\mathbf{c}\left( t-\tau
\right) \geq \frac{\left \vert x\right \vert +\mathbf{c}t}{2}-\left \vert
x\right \vert -\mathbf{c}t+\frac{\mathbf{c}t}{2}+\frac{1}{4}\left( \mathbf{c}%
t+\left \vert x\right \vert \right) \\
&\geq &\frac{\mathbf{c}t-\left \vert x\right \vert }{4}\geq \frac{\mathbf{c}t%
}{8}\hbox{.}
\end{eqnarray*}%
Hence,
\begin{eqnarray*}
J_{41} &\lesssim &\int_{\frac{t}{2}+\frac{1}{4}\left( t+\frac{\left \vert
x\right \vert }{\mathbf{c}}\right) }^{t}\int_{\left \vert y\right \vert \leq
\frac{\left \vert x\right \vert +\mathbf{c}t}{2}}\left( 1+t-\tau \right)
^{-5/2}e^{-\frac{\left( \left \vert x-y\right \vert -\mathbf{c}\left( t-\tau
\right) \right) ^{2}}{D_{0}\left( 1+t-\tau \right) }}\left( 1+t\right)
^{-4}\left( 1+\frac{\left( \mathbf{c}t-\left \vert x\right \vert \right) ^{2}%
}{1+t}\right) ^{-2}dyd\tau \\
&\lesssim &\left( 1+t\right) ^{-4}\left( 1+\frac{\left( \mathbf{c}t-\left
\vert x\right \vert \right) ^{2}}{1+t}\right) ^{-2}(1+t) \\
&\lesssim &\left( 1+t\right) ^{-5}\lesssim \left( 1+t\right) ^{-2}\left(
1+\left \vert x\right \vert \right) ^{-3}\lesssim \left( 1+t\right)
^{-7/2}\left( 1+\frac{\left \vert x\right \vert ^{2}}{1+t}\right) ^{-\frac{3%
}{2}}\hbox{,}
\end{eqnarray*}%
\begin{eqnarray*}
J_{42} &\lesssim &\int_{\frac{t}{2}+\frac{1}{4}\left( t+\frac{\left \vert
x\right \vert }{\mathbf{c}}\right) }^{t}\int_{\left \vert y\right \vert >%
\frac{\left \vert x\right \vert +\mathbf{c}t}{2}}\left( 1+t-\tau \right)
^{-5/2}e^{-\frac{\left( \mathbf{c}t-\left \vert x\right \vert \right) ^{2}}{%
2D_{0}\left( 1+t\right) }}e^{-\frac{\left( \left \vert x-y\right \vert -%
\mathbf{c}\left( t-\tau \right) \right) ^{2}}{2D_{0}\left( 1+t-\tau \right) }%
}\left( 1+t\right) ^{-4}dyd\tau \\
&\lesssim &\left( 1+t\right) ^{-4}e^{-\frac{\left( \mathbf{c}t-\left \vert
x\right \vert \right) ^{2}}{2D_{0}\left( 1+t\right) }}\int_{\frac{t}{2}+%
\frac{1}{4}\left( t+\frac{\left \vert x\right \vert }{\mathbf{c}}\right)
}^{t}\left( 1+t-\tau \right) ^{-\frac{5}{2}+\frac{5}{2}}d\tau \\
&\lesssim &\left( 1+t\right) ^{-3}e^{-\frac{t}{C}}\lesssim \left( 1+t\right)
^{-7/2}\left( 1+\frac{\left \vert x\right \vert ^{2}}{1+t}\right) ^{-\frac{3%
}{2}}\hbox{.}
\end{eqnarray*}%
for some $C>0$, so that
\begin{equation*}
J_{4}\lesssim \left( 1+t\right) ^{-7/2}\left( 1+\frac{\left \vert x\right
\vert ^{2}}{1+t}\right) ^{-\frac{3}{2}}\hbox{.}
\end{equation*}

As a result,
\begin{equation*}
J\lesssim \left( 1+t\right) ^{-2}\left( 1+\frac{\left \vert x\right \vert
^{2}}{1+t}\right) ^{-\frac{3}{2}}+\left( 1+t\right) ^{-5/2}e^{-\frac{\left(
\left \vert x\right \vert -\mathbf{c}t\right) ^{2}}{2D_{0}\left( 1+t\right) }%
}
\end{equation*}
for $\left( x,t\right) \in D_{4}$. \bigskip \newline
\textbf{Case 4:} $\left( x,t\right) \in D_{5}$. We split the integral $J$
into five parts%
\begin{equation*}
J=\left( \int_{0}^{\frac{t}{4}-\frac{\left \vert x\right \vert }{4\mathbf{c}}%
}+\int_{\frac{t}{4}-\frac{\left \vert x\right \vert }{4\mathbf{c}}}^{\frac{t%
}{2}}+\int_{\frac{t}{2}}^{\frac{|x|}{2\mathbf{c}}+\frac{1}{4}(t+\frac{|x|}{%
\mathbf{c}})} +\int_{\frac{|x|}{2\mathbf{c}}+\frac{1}{4}(t+\frac{|x|}{%
\mathbf{c}})}^{\frac{t}{2}+\frac{1}{4}(t+\frac{|x|}{\mathbf{c}})} +\int^{t}_{%
\frac{t}{2}+\frac{1}{4}(t+\frac{|x|}{\mathbf{c}})}\right) \int_{\mathbb{R}%
^{3}}\left( \cdots \right) dyd\tau =:\sum_{i=1}^{5}J_{i}\hbox{.}
\end{equation*}

For $J_{1}$, we decompose $\mathbb{R}^{3}$ into two parts%
\begin{equation*}
J_{1}=\int_{0}^{\frac{t}{4}-\frac{\left \vert x\right \vert }{4\mathbf{c}}%
}\left( \int_{\left \vert y\right \vert \leq \frac{\mathbf{c}t-\left \vert
x\right \vert }{2}}+\int_{\left \vert y\right \vert >\frac{\mathbf{c}t-\left
\vert x\right \vert }{2}}\right) \left( \cdots \right) dyd\tau
=:J_{11}+J_{12}\hbox{.}
\end{equation*}%
If $0\leq \tau \leq \frac{t}{4}-\frac{\left \vert x\right \vert }{4\mathbf{c}%
}$, $\left \vert y\right \vert \leq \frac{\mathbf{c}t-\left \vert
x\right
\vert }{2}$, then
\begin{equation*}
\mathbf{c}\left( t-\tau \right) -\left \vert x-y\right \vert \geq \mathbf{c}%
\left( t-\tau \right) -\left( \left \vert x\right \vert +\left \vert y\right
\vert \right) \geq \frac{\mathbf{c}t-\left \vert x\right \vert }{4}\hbox{.}
\end{equation*}%
If $0\leq \tau \leq \frac{t}{4}-\frac{\left \vert x\right \vert }{4\mathbf{c}%
}$, $\left \vert y\right \vert >\frac{\mathbf{c}t-\left \vert x\right \vert
}{2}$, then%
\begin{equation*}
\left \vert y\right \vert -\mathbf{c}\tau \geq \frac{\mathbf{c}t-\left \vert
x\right \vert }{4}\hbox{.}
\end{equation*}%
Hence,%
\begin{eqnarray*}
J_{11} &\lesssim &\int_{0}^{\frac{t}{4}-\frac{\left \vert x\right \vert }{4%
\mathbf{c}}}\int_{\left \vert y\right \vert \leq \frac{\mathbf{c}t-\left
\vert x\right \vert }{2}}\left( 1+t\right) ^{-5/2}e^{-\frac{\left( \mathbf{c}%
t-\left \vert x\right \vert \right) ^{2}}{D_{0}\left( 1+t\right) }}\left(
1+\tau \right) ^{-4}\left( 1+\frac{\left( \left \vert y\right \vert -\mathbf{%
c}\tau \right) ^{2}}{1+\tau }\right) ^{-2}dyd\tau \\
&\lesssim &\left( 1+t\right) ^{-5/2}e^{-\frac{\left( \mathbf{c}t-\left \vert
x\right \vert \right) ^{2}}{D_{0}\left( 1+t\right) }}\int_{0}^{\frac{\mathbf{%
c}t-\left \vert x\right \vert }{4}}\left( 1+\tau \right) ^{-4+\frac{5}{2}%
}d\tau \lesssim \left( 1+t\right) ^{-5/2}e^{-\frac{\left( \mathbf{c}t-\left
\vert x\right \vert \right) ^{2}}{D_{0}\left( 1+t\right) }}\hbox{.}
\end{eqnarray*}%
For $J_{12}$, we have%
\begin{eqnarray*}
J_{12} &\lesssim &\int_{0}^{\frac{t}{4}-\frac{\left \vert x\right \vert }{4%
\mathbf{c}}}\int_{\left \vert y\right \vert >\frac{\mathbf{c}t-\left \vert
x\right \vert }{2}}\left( 1+t-\tau \right) ^{-5/2}e^{-\frac{\left( \left
\vert x-y\right \vert -\mathbf{c}\left( t-\tau \right) \right) ^{2}}{%
D_{0}\left( 1+t-\tau \right) }}\left( 1+\tau \right) ^{-2}\left( 1+\left(
\mathbf{c}t-\left \vert x\right \vert \right) ^{2}\right) ^{-2}dyd\tau \\
&\lesssim &\left( 1+\left( \mathbf{c}t-\left \vert x\right \vert \right)
^{2}\right) ^{-2}\int_{0}^{\frac{\mathbf{c}t-\left \vert x\right \vert }{4}%
}\left( 1+\tau \right) ^{-2}d\tau \\
&\lesssim &\left( 1+\mathbf{c}t-\left \vert x\right \vert \right) ^{-4}%
\hbox{,}
\end{eqnarray*}%
and%
\begin{eqnarray*}
J_{12} &\lesssim &\int_{0}^{\frac{t}{4}-\frac{\left \vert x\right \vert }{4%
\mathbf{c}}}\left( 1+t\right) ^{-5/2}\left( 1+\tau \right) ^{-2}\int_{\frac{%
\mathbf{c}t-\left \vert x\right \vert }{2}}^{\infty }\left( 1+r-\mathbf{c}%
\tau \right) ^{-4}r^{2}drd\tau \\
&\lesssim &\int_{0}^{\frac{\mathbf{c}t-\left \vert x\right \vert }{4}}\left(
1+t\right) ^{-5/2}\left( 1+\tau \right) ^{-2}\left( 1+\mathbf{c}t-\left
\vert x\right \vert \right) ^{-1}d\tau \\
&\lesssim &\left( 1+t\right) ^{-5/2}\left( 1+\mathbf{c}t-\left \vert x\right
\vert \right) ^{-1}\hbox{,}
\end{eqnarray*}%
which implies that
\begin{eqnarray*}
J_{12} &\lesssim &\left( J_{12}\right) ^{\frac{1}{3}}\left( J_{12}\right) ^{%
\frac{2}{3}}\lesssim \left[ \left( 1+\mathbf{c}t-\left \vert x\right \vert
\right) ^{-4}\right] ^{\frac{1}{3}}\left[ \left( 1+t\right) ^{-5/2}\left( 1+%
\mathbf{c}t-\left \vert x\right \vert \right) ^{-1}\right] ^{\frac{2}{3}} \\
&\lesssim &\left( 1+t\right) ^{-\frac{5}{3}}\left( 1+\mathbf{c}t-\left \vert
x\right \vert \right) ^{-2}\lesssim \left( 1+t\right) ^{-\frac{8}{3}}\left(
1+\frac{\left( \mathbf{c}t-\left \vert x\right \vert \right) ^{2}}{1+t}%
\right) ^{-1}\hbox{.}
\end{eqnarray*}%
Therefore, we get%
\begin{equation*}
J_{1}\lesssim \left( 1+t\right) ^{-5/2}e^{-\frac{\left( \mathbf{c}t-\left
\vert x\right \vert \right) ^{2}}{D_{0}\left( 1+t\right) }}+\left(
1+t\right) ^{-\frac{8}{3}}\left( 1+\frac{\left( \mathbf{c}t-\left \vert
x\right \vert \right) ^{2}}{1+t}\right) ^{-1}\hbox{.}
\end{equation*}

For $J_{2}$, we use the spherical coordinates to obtain

\begin{eqnarray*}
J_{2} &\lesssim &\int_{\frac{t}{4}-\frac{\left \vert x\right \vert }{4%
\mathbf{c}}}^{\frac{t}{2}}\int_{0}^{\infty }\int_{0}^{\pi }\left( 1+t-\tau
\right) ^{-5/2}e^{-\frac{\left( \sqrt{\left \vert x\right \vert
^{2}+r^{2}-2r\left \vert x\right \vert \cos \theta }-\mathbf{c}\left( t-\tau
\right) \right) ^{2}}{D_{0}\left( 1+t-\tau \right) }}\left( 1+\tau \right)
^{-4}\left( 1+\frac{\left( r-\mathbf{c}\tau \right) ^{2}}{1+\tau }\right)
^{-2}r^{2}\sin \theta d\theta drd\tau \\
&\lesssim &\int_{\frac{t}{4}-\frac{\left \vert x\right \vert }{4\mathbf{c}}%
}^{\frac{t}{2}}\int_{0}^{\infty }\int_{\left \vert \left \vert x\right \vert
-r\right \vert }^{\left \vert x\right \vert +r}\left( 1+t-\tau \right)
^{-5/2}e^{-\frac{\left( z-\mathbf{c}\left( t-\tau \right) \right) ^{2}}{%
D_{0}\left( 1+t-\tau \right) }}\left( 1+\tau \right) ^{-4}\left( 1+\frac{%
\left( r-\mathbf{c}\tau \right) ^{2}}{1+\tau }\right) ^{-2}rz\frac{1}{\left
\vert x\right \vert }dzdrd\tau \\
&\lesssim &\int_{\frac{t}{4}-\frac{\left \vert x\right \vert }{4\mathbf{c}}%
}^{\frac{t}{2}}\int_{0}^{\infty }\int_{0}^{\infty }\left( 1+t-\tau \right)
^{-5/2}e^{-\frac{\left( z-\mathbf{c}\left( t-\tau \right) \right) ^{2}}{%
D_{0}\left( 1+t-\tau \right) }}\left( 1+\tau \right) ^{-4}\left( 1+\frac{%
\left( r-\mathbf{c}\tau \right) ^{2}}{1+\tau }\right) ^{-2}rz\frac{1}{\left
\vert x\right \vert }dzdrd\tau \\
&\lesssim &\int_{\frac{t}{4}-\frac{\left \vert x\right \vert }{4\mathbf{c}}%
}^{\frac{t}{2}}\int_{0}^{\infty }\left( 1+t-\tau \right) ^{-\frac{5}{2}+%
\frac{3}{2}}\left( 1+\tau \right) ^{-4}\left( 1+\frac{\left( r-\mathbf{c}%
\tau \right) ^{2}}{1+\tau }\right) ^{-2}\frac{r}{\left \vert x\right \vert }%
drd\tau \\
&\lesssim &\left( 1+t\right) ^{-\frac{5}{2}+\frac{3}{2}}\left \vert x\right
\vert ^{-1}\int_{\frac{t}{4}-\frac{\left \vert x\right \vert }{4\mathbf{c}}%
}^{\frac{t}{2}}\left( 1+\tau \right) ^{-4+\frac{3}{2}}d\tau \\
&\lesssim &\left( 1+t\right) ^{-2}\left( 1+\mathbf{c}t-\left \vert x\right
\vert \right) ^{-\frac{3}{2}}\lesssim \left( 1+t\right) ^{-2}\left( 1+%
\mathbf{c}t-\left \vert x\right \vert \right) ^{\frac{1}{2}}\left( 1+\mathbf{%
c}t-\left \vert x\right \vert \right) ^{-2} \\
&\lesssim &\left( 1+t\right) ^{-5/2}\left( 1+\frac{\left( \mathbf{c}t-\left
\vert x\right \vert \right) ^{2}}{1+t}\right) ^{-1}\hbox{,}
\end{eqnarray*}%
by setting $z=\sqrt{\left \vert x\right \vert ^{2}+r^{2}-2r\left \vert
x\right \vert \cos \theta }$ and $\sin \theta d\theta =\frac{z}{r\left \vert
x\right \vert }dz$.

For $J_{3}$, we use the spherical coordinates again to obtain%
\begin{eqnarray*}
J_{3} &\lesssim &\int_{\frac{t}{2}}^{\frac{3|x|}{4\mathbf{c}}+\frac{t}{4}%
}\int_{0}^{\infty }\int_{0}^{\pi }\left( 1+t-\tau \right) ^{-5/2}e^{-\frac{%
\left( \sqrt{\left \vert x\right \vert ^{2}+r^{2}-2r\left \vert x\right
\vert \cos \theta }-\mathbf{c}\left( t-\tau \right) \right) ^{2}}{%
D_{0}\left( 1+t-\tau \right) }}\left( 1+t\right) ^{-4}\left( 1+\frac{\left(
r-\mathbf{c}\tau \right) ^{2}}{1+\tau }\right) ^{-2}r^{2}\sin \theta d\theta
drd\tau \\
&\lesssim &\int_{\frac{t}{2}}^{\frac{3|x|}{4\mathbf{c}}+\frac{t}{4}%
}\int_{0}^{\infty }\int_{\left \vert \left \vert x\right \vert -r\right
\vert }^{\left \vert x\right \vert +r}\left( 1+t-\tau \right) ^{-5/2}e^{-%
\frac{\left( z-\mathbf{c}\left( t-\tau \right) \right) ^{2}}{D_{0}\left(
1+t-\tau \right) }}\left( 1+t\right) ^{-4}\left( 1+\frac{\left( r-\mathbf{c}%
\tau \right) ^{2}}{1+\tau }\right) ^{-2}rz\frac{1}{\left \vert x\right \vert
}dzdrd\tau \\
&\lesssim &\left( 1+t\right) ^{-4}\left \vert x\right \vert ^{-1}\int_{\frac{%
t}{2}}^{\frac{3|x|}{4\mathbf{c}}+\frac{t}{4}} \int_{0}^{\infty }\left(
1+t-\tau \right) ^{-\frac{5}{2}+\frac{3}{2}}\left( 1+\frac{\left( r-\mathbf{c%
}\tau \right) ^{2}}{1+\tau }\right) ^{-2}rdrd\tau \\
&\lesssim &\left( 1+t\right) ^{-7/2}\int_{\frac{t}{2}}^{\frac{3|x|}{4\mathbf{%
c}}+\frac{t}{4}}\left( 1+t-\tau \right) ^{-1}d\tau \\
&\lesssim &\left( 1+t\right) ^{-7/2}\left[\ln \left(1+\frac{t}{2}\right)-\ln
\left(1+\frac{\mathbf{c}t-|x|}{4\mathbf{c}}\right)\right] \\
&\lesssim& \left( 1+t\right) ^{-5/2}\left( 1+\frac{\left( \mathbf{c}t-\left
\vert x\right \vert \right) ^{2}}{1+t}\right) ^{-1}\hbox{.}
\end{eqnarray*}
For $J_{4}$, similar to $J_{3}$, we have
\begin{eqnarray*}
J_{4} &\lesssim &\left( 1+t\right) ^{-7/2}\int_{\frac{3|x|}{4\mathbf{c}}+%
\frac{t}{4}}^{\frac{|x|}{4\mathbf{c}}+\frac{3t}{4}}\left( 1+t-\tau \right)
^{-1}d\tau \\
&\lesssim &\left( 1+t\right) ^{-7/2}\left[\left(1+\frac{3(\mathbf{c}t-|x|)}{4%
\mathbf{c}}\right)^{-1} \left(1+\frac{(\mathbf{c}t-|x|)}{2\mathbf{c}}\right)%
\right] \\
&\lesssim& \left( 1+t\right) ^{-5/2}\left( 1+\frac{\left( \mathbf{c}t-\left
\vert x\right \vert \right) ^{2}}{1+t}\right) ^{-1}\hbox{.}
\end{eqnarray*}
For $J_{5}$, we decompose $\mathbb{R}^{3}$ into two parts%
\begin{equation*}
J_{5}=\int_{\frac{t}{2}+\frac{1}{4}\left( t+\frac{\left \vert x\right \vert
}{\mathbf{c}}\right) }^{t}\left( \int_{\left \vert y\right \vert \leq \frac{%
\left \vert x\right \vert +\mathbf{c}t}{2}}+\int_{\left \vert y\right \vert >%
\frac{\left \vert x\right \vert +\mathbf{c}t}{2}}\right) \left( \cdots
\right) dyd\tau =:J_{51}+J_{52}\hbox{.}
\end{equation*}
If $\frac{t}{2}+\frac{1}{4}\left( t+\frac{\left \vert x\right \vert }{%
\mathbf{c}}\right) \leq \tau \leq t$, $\left \vert y\right \vert \leq \frac{%
\left \vert x\right \vert +\mathbf{c}t}{2}$, then%
\begin{equation*}
\mathbf{c}\tau -\left \vert y\right \vert \geq \frac{\mathbf{c}t}{2}+\frac{1%
}{4}\left( \mathbf{c}t+\left \vert x\right \vert \right) -\frac{\left \vert
x\right \vert +\mathbf{c}t}{2}\geq \frac{\mathbf{c}t-\left \vert x\right
\vert }{4}\hbox{.}
\end{equation*}%
If $\frac{t}{2}+\frac{1}{4}\left( t+\frac{\left \vert x\right \vert }{%
\mathbf{c}}\right) \leq \tau \leq t$, $\left \vert y\right \vert >\frac{%
\left \vert x\right \vert +\mathbf{c}t}{2}$, then%
\begin{eqnarray*}
\left \vert x-y\right \vert -\mathbf{c}\left( t-\tau \right) &\geq &\left
\vert y\right \vert -\left \vert x\right \vert -\mathbf{c}\left( t-\tau
\right) \geq \frac{\left \vert x\right \vert +\mathbf{c}t}{2}-\left \vert
x\right \vert -\mathbf{c}t+\frac{\mathbf{c}t}{2}+\frac{1}{4}\left( \mathbf{c}%
t+\left \vert x\right \vert \right) \geq \frac{\mathbf{c}t-\left \vert
x\right \vert }{4}\hbox{.}
\end{eqnarray*}%
Hence,
\begin{eqnarray*}
J_{51} &\lesssim &\int_{\frac{t}{2}+\frac{1}{4}\left( t+\frac{\left \vert
x\right \vert }{\mathbf{c}}\right) }^{t}\int_{\left \vert y\right \vert \leq
\frac{\left \vert x\right \vert +\mathbf{c}t}{2}}\left( 1+t-\tau \right)
^{-5/2}e^{-\frac{\left( \left \vert x-y\right \vert -\mathbf{c}\left( t-\tau
\right) \right) ^{2}}{D_{0}\left( 1+t-\tau \right) }}\left( 1+t\right)
^{-4}\left( 1+\frac{\left( \mathbf{c}t-\left \vert x\right \vert \right) ^{2}%
}{1+t}\right) ^{-2}dyd\tau \\
&\lesssim &\left( 1+t\right) ^{-4}\left( 1+\frac{\left( \mathbf{c}t-\left
\vert x\right \vert \right) ^{2}}{1+t}\right) ^{-2}\left( \mathbf{c}t-\left
\vert x\right \vert \right) \\
&\lesssim &\left( 1+t\right) ^{-3}\left( 1+\frac{\left( \mathbf{c}t-\left
\vert x\right \vert \right) ^{2}}{1+t}\right) ^{-2}\hbox{,}
\end{eqnarray*}%
\begin{eqnarray*}
J_{52} &\lesssim &\int_{\frac{t}{2}+\frac{1}{4}\left( t+\frac{\left \vert
x\right \vert }{\mathbf{c}}\right) }^{t}\int_{\left \vert y\right \vert >%
\frac{\left \vert x\right \vert +\mathbf{c}t}{2}}\left( 1+t-\tau \right)
^{-5/2}e^{-\frac{\left( \mathbf{c}t-\left \vert x\right \vert \right) ^{2}}{%
2D_{0}\left( 1+\mathbf{c}t-\left \vert x\right \vert \right) }}e^{-\frac{%
\left( \left \vert x-y\right \vert -\mathbf{c}\left( t-\tau \right) \right)
^{2}}{2D_{0}\left( 1+t-\tau \right) }}\left( 1+t\right) ^{-4}dyd\tau \\
&\lesssim &\left( 1+t\right) ^{-4}e^{-\frac{\left( \mathbf{c}t-\left \vert
x\right \vert \right)}{2D_{0}}}\int_{\frac{t}{2}+\frac{1}{4}\left( t+\frac{%
\left \vert x\right \vert }{\mathbf{c}}\right) }^{t}\left( 1+t-\tau \right)
^{-\frac{5}{2}+\frac{5}{2}}d\tau \\
&\lesssim &\left( 1+t\right) ^{-3}e^{-\frac{\left( \mathbf{c}t-\left \vert
x\right \vert \right)}{2D_{0}}}\hbox{.}
\end{eqnarray*}%
Gathering all the estimates, we have
\begin{equation*}
J\lesssim \left( 1+t\right) ^{-2}\left( 1+\frac{\left \vert x\right \vert
^{2}}{1+t}\right) ^{-\frac{3}{2}}+\left( 1+t\right) ^{-5/2}\left( 1+\frac{%
\left( \mathbf{c}t-\left \vert x\right \vert \right) ^{2}}{1+t}\right) ^{-1}%
\hbox{.}
\end{equation*}
\end{proof}

\begin{proof}[\textbf{Proof of Lemma \protect \ref{moving-nonmoving}}]
(Huygens wave convolved with diffusion wave). \newline
\newline
\textbf{Case 1:} $\left( x,t\right) \in D_{1}$. Direct computation gives
\begin{eqnarray*}
K &\lesssim &\left( 1+t\right) ^{-5/2}\int_{0}^{\frac{t}{2}}\int_{\mathbb{R}%
^{3}}\left( 1+\tau \right) ^{-3}\left( 1+\frac{\left \vert y\right \vert ^{2}%
}{1+\tau }\right) ^{-3}dyd\tau \\
&&+\left( 1+t\right) ^{-3}\int_{\frac{t}{2}}^{t}\int_{\mathbb{R}^{3}}\left(
1+t-\tau \right) ^{-5/2}\left( 1+\frac{\left( \left \vert x-y\right \vert -%
\mathbf{c}\left( t-\tau \right) \right) ^{2}}{D_{0}\left( 1+t-\tau \right) }%
\right) ^{-2}dyd\tau \\
&\lesssim &\left( 1+t\right) ^{-5/2}\int_{0}^{\frac{t}{2}}\left( 1+\tau
\right) ^{-3}\left( 1+\tau \right) ^{\frac{3}{2}}dyd\tau +\left( 1+t\right)
^{-3}\int_{\frac{t}{2}}^{t}\left( 1+t-\tau \right) ^{-5/2}\left( 1+t-\tau
\right) ^{\frac{5}{2}}d\tau \\
&\lesssim &\left( 1+t\right) ^{-5/2}+\left( 1+t\right) ^{-3}\left(
1+t\right)\lesssim \left( 1+t\right) ^{-2}\left( 1+\frac{\left \vert x\right
\vert ^{2}}{1+t}\right) ^{-\frac{3}{2}}\hbox{.}
\end{eqnarray*}
\newline
\textbf{Case 2:} $\left( x,t\right) \in D_{2}$. We split the integral $K$
into two parts%
\begin{equation*}
K=\left( \int_{0}^{\frac{3}{4}t}+\int_{\frac{3}{4}t}^{t}\right) \int_{%
\mathbb{R}^{3}}\left( \cdots \right) dyd\tau =:K_{1}+K_{2}\hbox{.}
\end{equation*}

For $K_{1}$, it is easy to see%
\begin{eqnarray*}
K_{1} &\lesssim &\int_{0}^{\frac{3}{4}t}\int_{\mathbb{R}^{3}}\left(
1+t\right) ^{-5/2}\left( 1+\tau \right) ^{-3}\left( 1+\frac{\left \vert
y\right \vert ^{2}}{1+\tau }\right) ^{-3}dyd\tau \\
&\lesssim &\int_{0}^{\frac{3}{4}t}\left( 1+t\right) ^{-5/2}\left( 1+\tau
\right) ^{-3}\left( 1+\tau \right) ^{\frac{3}{2}}d\tau \lesssim \left(
1+t\right) ^{-5/2}\lesssim \left( 1+t\right) ^{-5/2}\left( 1+\frac{\left(
\left \vert x\right \vert -\mathbf{c}t\right) ^{2}}{1+t}\right) ^{-1}\hbox{.}
\end{eqnarray*}

For $K_{2}$, we decompose $\mathbb{R}^{3}$ into two parts
\begin{equation*}
K_{2}=\int_{\frac{3}{4}t}^{t}\left( \int_{\left \vert y\right \vert \leq
\frac{\left \vert x\right \vert }{2}}+\int_{\left \vert y\right \vert >\frac{%
\left \vert x\right \vert }{2}}\right) \left( \cdots \right) dyd\tau
=:K_{21}+K_{22}.
\end{equation*}%
It readily follows that%
\begin{eqnarray*}
K_{22} &\lesssim &\left( 1+t\right) ^{-3}\left( 1+\frac{\left \vert x\right
\vert ^{2}}{1+t}\right) ^{-3}\int_{\frac{3}{4}t}^{t}\int_{\left \vert
y\right \vert >\frac{\left \vert x\right \vert }{2}}\left( 1+t-\tau \right)
^{-5/2}e^{-\frac{\left( \left \vert x-y\right \vert -\mathbf{c}\left( t-\tau
\right) \right) ^{2}}{D_{0}\left( 1+t-\tau \right) }}dyd\tau \\
&\lesssim &\left( 1+t\right) ^{-3}\left( 1+\frac{\left \vert x\right \vert
^{2}}{1+t}\right) ^{-3}\int_{\frac{3}{4}t}^{t}\left( 1+t-\tau \right)
^{-5/2}\left( 1+t-\tau \right) ^{\frac{5}{2}}d\tau \\
&\lesssim &\left( 1+t\right) ^{-2}\left( 1+\frac{\left \vert x\right \vert
^{2}}{1+t}\right) ^{-3}\hbox{.}
\end{eqnarray*}%
If $\frac{3}{4}t\leq \tau \leq t$, $\left \vert y\right \vert \leq \frac{%
\left \vert x\right \vert }{2}$, then
\begin{eqnarray*}
\left \vert x-y\right \vert -\mathbf{c}\left( t-\tau \right) &\geq &\left
\vert x\right \vert -\left \vert y\right \vert -\mathbf{c}\left( t-\tau
\right) \geq \frac{\left \vert x\right \vert }{2}-\frac{\mathbf{c}t}{4}\geq
\frac{1}{2}\left( \mathbf{c}t-\sqrt{1+t}\right) -\frac{\mathbf{c}t}{4} \\
&\geq &\frac{\mathbf{c}t}{8}+\left( \frac{\mathbf{c}t}{8}-\frac{\sqrt{1+t}}{2%
}\right) \geq \frac{\mathbf{c}t}{8}
\end{eqnarray*}%
for $t\geq 11$. Hence, there exists a universal constant $C>0$ such that%
\begin{eqnarray*}
K_{21} &\lesssim &\left( 1+t\right) ^{-3}e^{-\frac{t}{C}}\int_{\frac{3}{4}%
t}^{t}\int_{\mathbb{R}^{3}}\left( 1+t-\tau \right) ^{-5/2}\left( 1+\frac{%
\left \vert y\right \vert ^{2}}{1+\tau }\right) ^{-3}dyd\tau \\
&\lesssim &\left( 1+t\right) ^{-3}e^{-\frac{t}{C}}\int_{\frac{3}{4}%
t}^{t}\left( 1+t-\tau \right) ^{-5/2}\left( 1+\tau \right) ^{\frac{3}{2}%
}d\tau \\
&\lesssim &e^{-\frac{t}{C}}\lesssim \left( 1+t\right) ^{-5/2}\left( 1+\frac{%
\left( \left \vert x\right \vert -\mathbf{c}t\right) ^{2}}{1+t}\right) ^{-1}
\end{eqnarray*}%
for $t\geq 11$, and
\begin{eqnarray*}
K_{21} &\lesssim &\left( 1+t\right) ^{-3}\int_{\frac{3}{4}t}^{t}\int_{%
\mathbb{R}^{3}}\left( 1+t-\tau \right) ^{-5/2}\left( 1+\frac{\left \vert
y\right \vert ^{2}}{1+\tau }\right) ^{-3}dyd\tau \\
&\lesssim &C\lesssim \left( 1+t\right) ^{-5/2}\left( 1+\frac{\left( \left
\vert x\right \vert -\mathbf{c}t\right) ^{2}}{1+t}\right) ^{-1}
\end{eqnarray*}%
for $0\leq t\leq 11$. Therefore we can conclude that
\begin{equation*}
K\lesssim \left( 1+t\right) ^{-2}\left( 1+\frac{\left \vert x\right \vert
^{2}}{1+t}\right) ^{-3}+\left( 1+t\right) ^{-5/2}\left( 1+\frac{\left( \left
\vert x\right \vert -\mathbf{c}t\right) ^{2}}{1+t}\right) ^{-1}\hbox{.}
\end{equation*}
\newline
\textbf{Case 3:} $\left( x,t\right) \in D_{3}$. We split the integral $K$
into four parts%
\begin{eqnarray*}
K &=&\int_{0}^{\frac{t}{2}}\left( \int_{\left \vert y\right \vert \leq \frac{%
\left \vert x\right \vert -\mathbf{c}\left( t-\tau \right) }{2}}+\int_{\left
\vert y\right \vert >\frac{\left \vert x\right \vert -\mathbf{c}\left(
t-\tau \right) }{2}}\right) \left( \cdots \right) dyd\tau +\int_{\frac{t}{2}%
}^{t}\left( \int_{\left \vert y\right \vert \leq \frac{\left \vert x\right
\vert -\mathbf{c}\left( t-\tau \right) }{2}}+\int_{\left \vert y\right \vert
>\frac{\left \vert x\right \vert -\mathbf{c}\left( t-\tau \right) }{2}%
}\right) \left( \cdots \right) dyd\tau \\
&=&K_{11}+K_{12}+K_{21}+K_{22}\hbox{.}
\end{eqnarray*}

If $\left \vert y\right \vert \leq \frac{\left \vert x\right \vert -\mathbf{c%
}\left( t-\tau \right) }{2}$, then
\begin{equation*}
\left \vert x-y\right \vert -\mathbf{c}\left( t-\tau \right) \geq \left
\vert x\right \vert -\left \vert y\right \vert -\mathbf{c}\left( t-\tau
\right) \geq \frac{\left \vert x\right \vert -\mathbf{c}\left( t-\tau
\right) }{2}=\frac{\left \vert x\right \vert -\mathbf{c}t+\mathbf{c}\tau }{2}%
\hbox{.}
\end{equation*}%
Hence,

\begin{eqnarray*}
K_{11} &\lesssim &\int_{0}^{\frac{t}{2}}\int_{\left \vert y\right \vert \leq
\frac{\left \vert x\right \vert -\mathbf{c}\left( t-\tau \right) }{2}}\left(
1+t-\tau \right) ^{-5/2}e^{-\frac{\left( \left \vert x\right \vert -\mathbf{c%
}t\right) ^{2}+\left( \mathbf{c}\tau \right) ^{2}}{4D_{0}\left( 1+t-\tau
\right) }}\left( 1+\tau \right) ^{-3}\left( 1+\frac{\left \vert y\right
\vert ^{2}}{1+\tau }\right) ^{-3}dyd\tau \\
&\lesssim &\left( 1+t\right) ^{-5/2}e^{-\frac{\left( \left \vert x\right
\vert -\mathbf{c}t\right) ^{2}}{4D_{0}\left( 1+t\right) }}\int_{0}^{\frac{t}{%
2}}\left( 1+\tau \right) ^{-3}\left( 1+\tau \right) ^{\frac{3}{2}}dyd\tau
\lesssim \left( 1+t\right) ^{-5/2}e^{-\frac{\left( \left \vert x\right \vert
-\mathbf{c}t\right) ^{2}}{4D_{0}\left( 1+t\right) }}\hbox{,}
\end{eqnarray*}%
\begin{eqnarray*}
K_{12} &\lesssim &\left( 1+t\right) ^{-5/2}\int_{0}^{\frac{t}{2}}\int_{\left
\vert y\right \vert >\frac{\left \vert x\right \vert -\mathbf{c}\left(
t-\tau \right) }{2}}\left( 1+\tau +\left \vert y\right \vert ^{2}\right)
^{-3}dyd\tau \\
&\lesssim &\left( 1+t\right) ^{-5/2}\int_{0}^{\frac{t}{2}}\int_{\frac{\left
\vert x\right \vert -\mathbf{c}\left( t-\tau \right) }{2}}^{\infty }\left(
1+r\right) ^{-6}r^{2}drd\tau \lesssim \left( 1+t\right) ^{-5/2}\int_{0}^{%
\frac{t}{2}}\left( 1+\left \vert x\right \vert -\mathbf{c}t+\mathbf{c}\tau
\right) ^{-3}d\tau \\
&\lesssim &\left( 1+t\right) ^{-5/2}\left( 1+\left \vert x\right \vert -%
\mathbf{c}t\right) ^{-2}\lesssim \left( 1+t\right) ^{-7/2}\left( 1+\frac{%
\left( \left \vert x\right \vert -\mathbf{c}t\right) ^{2}}{1+t}\right) ^{-1}%
\hbox{,}
\end{eqnarray*}%
\begin{eqnarray*}
K_{21} &\lesssim &\int_{\frac{t}{2}}^{t}\int_{\left \vert y\right \vert \leq
\frac{\left \vert x\right \vert -\mathbf{c}\left( t-\tau \right) }{2}}\left(
1+t-\tau \right) ^{-5/2}e^{-\frac{\left( \left \vert x\right \vert -\mathbf{c%
}t\right) ^{2}+\left( \mathbf{c}\tau \right) ^{2}}{4D_{0}\left( 1+t-\tau
\right) }}\left( 1+\tau \right) ^{-3}\left( 1+\frac{\left \vert y\right
\vert ^{2}}{1+\tau }\right) ^{-3}dyd\tau \\
&\lesssim &e^{-\frac{\left( \left \vert x\right \vert -\mathbf{c}t\right)
^{2}}{4D_{0}\left( 1+t\right) }}e^{-\frac{t}{C}}\int_{\frac{t}{2}}^{t}\left(
1+t-\tau \right) ^{-5/2}\left( 1+\tau \right) ^{-\frac{3}{2}}d\tau \lesssim
e^{-\frac{t}{C}}e^{-\frac{\left( \left \vert x\right \vert -\mathbf{c}%
t\right) ^{2}}{4D_{0}\left( 1+t\right) }}
\end{eqnarray*}%
for some $C>0$, and%
\begin{eqnarray*}
K_{22} &\lesssim &\left( 1+t\right) ^{-3}\left( 1+\frac{\left( \left \vert
x\right \vert -\mathbf{c}t\right) ^{2}+\mathbf{c}^{2}t^{2}}{1+t}\right)
^{-3}\int_{\frac{t}{2}}^{t}\int_{\mathbb{R}^{3}}\left( 1+t-\tau \right)
^{-5/2}e^{-\frac{\left( \left \vert x-y\right \vert -\mathbf{c}\left( t-\tau
\right) \right) ^{2}}{D_{0}\left( 1+t-\tau \right) }}dyd\tau \\
&\lesssim &\left( 1+t\right) ^{-3}\left( 1+\frac{\left( \left \vert x\right
\vert -\mathbf{c}t\right) ^{2}}{1+t}+t\right) ^{-3}\int_{\frac{t}{2}%
}^{t}\left( 1+t-\tau \right) ^{-5/2}\left( 1+t-\tau \right) ^{\frac{5}{2}%
}d\tau \\
&\lesssim &\left( 1+t\right) ^{-4}\left( 1+\frac{\left( \left \vert x\right
\vert -\mathbf{c}t\right) ^{2}}{1+t}\right) ^{-1}\hbox{.}
\end{eqnarray*}%
It implies that
\begin{equation*}
K\lesssim \left( 1+t\right) ^{-5/2}\left( 1+\frac{\left( \left \vert x\right
\vert -\mathbf{c}t\right) ^{2}}{1+t}\right) ^{-1}\hbox{.}
\end{equation*}
\newline
\textbf{Case 4:} $\left( x,t\right) \in D_{4}$. First note that $\sqrt{1+t}%
\leq \left \vert x\right \vert \leq \mathbf{c}t/2$ in this region. Now we
split the integral $K$ into three parts
\begin{equation*}
K=\left( \int_{0}^{\frac{t}{2}-\frac{\left \vert x\right \vert }{2\mathbf{c}}%
}+\int_{\frac{t}{2}-\frac{\left \vert x\right \vert }{2\mathbf{c}}}^{t-\frac{%
\left \vert x\right \vert }{2\mathbf{c}}}+\int_{t-\frac{\left \vert x\right
\vert }{2\mathbf{c}}}^{t}\right) \int_{\mathbb{R}^{3}}\left( \cdots \right)
dyd\tau =:K_{1}+K_{2}+K_{3}\hbox{.}
\end{equation*}

For $K_{1}$, we decompose $\mathbb{R}^{3}$ into two parts
\begin{equation*}
K_{1}=\int_{0}^{\frac{t}{2}-\frac{\left \vert x\right \vert }{2\mathbf{c}}%
}\left( \int_{\left \vert y\right \vert \leq \frac{\mathbf{c}\left( t-\tau
\right) -\left \vert x\right \vert }{2}}+\int_{\left \vert y\right \vert >%
\frac{\mathbf{c}\left( t-\tau \right) -\left \vert x\right \vert }{2}%
}\right) \left( \cdots \right) dyd\tau =:K_{11}+K_{12}\hbox{.}
\end{equation*}%
If $0\leq \tau \leq \frac{t}{2}-\frac{\left \vert x\right \vert }{2\mathbf{c}%
}$, $\left \vert y\right \vert \leq \frac{\mathbf{c}\left( t-\tau \right)
-\left \vert x\right \vert }{2}$, then
\begin{equation*}
\mathbf{c}\left( t-\tau \right) -\left \vert x\right \vert -\left \vert
y\right \vert \geq \frac{\mathbf{c}\left( t-\tau \right) -\left \vert
x\right \vert }{2}\geq \frac{\mathbf{c}t-\left \vert x\right \vert }{4}%
\hbox{.}
\end{equation*}%
If $0\leq \tau \leq \frac{t}{2}-\frac{\left \vert x\right \vert }{2\mathbf{c}%
}$, $\left \vert y\right \vert >\frac{\mathbf{c}\left( t-\tau \right)
-\left
\vert x\right \vert }{2}$, then%
\begin{equation*}
\left \vert y\right \vert >\frac{\mathbf{c}t-\left \vert x\right \vert }{4}%
\hbox{.}
\end{equation*}%
Hence, we have
\begin{eqnarray*}
K_{11} &\lesssim &\left( 1+\frac{t}{2}+\frac{\left \vert x\right \vert }{2%
\mathbf{c}}\right) ^{-5/2}e^{-\frac{\left( \left \vert x\right \vert -%
\mathbf{c}t\right) ^{2}}{16D_{0}\left( 1+t\right) }}\int_{0}^{\frac{t}{2}-%
\frac{\left \vert x\right \vert }{2\mathbf{c}}}\int_{\left \vert y\right
\vert \leq \frac{\mathbf{c}\left( t-\tau \right) -\left \vert x\right \vert
}{2}}\left( 1+\tau \right) ^{-3}\left( 1+\frac{\left( \left \vert y\right
\vert -\mathbf{c}\tau \right) ^{2}}{1+\tau }\right) ^{-3}dyd\tau \\
&\lesssim &\left( 1+t\right) ^{-5/2}e^{-\frac{\left( \left \vert x\right
\vert -\mathbf{c}t\right) ^{2}}{16D_{0}\left( 1+t\right) }}\int_{0}^{\frac{t%
}{2}-\frac{\left \vert x\right \vert }{2\mathbf{c}}}\left( 1+\tau \right)
^{-3+\frac{3}{2}}d\tau \lesssim \left( 1+t\right) ^{-5/2}e^{-\frac{\left(
\left \vert x\right \vert -\mathbf{c}t\right) ^{2}}{4D_{0}\left( 1+t\right) }%
}\hbox{,}
\end{eqnarray*}%
\begin{eqnarray*}
K_{12} &\lesssim &\left( 1+t\right) ^{-5/2}\int_{0}^{\frac{t}{2}-\frac{\left
\vert x\right \vert }{2\mathbf{c}}}\int_{\left \vert y\right \vert \geq
\frac{\mathbf{c}\left( t-\tau \right) -\left \vert x\right \vert }{2}}\left(
1+\tau \right) ^{-3}\left( 1+\frac{\left \vert y\right \vert ^{2}}{1+\tau }%
\right) ^{-3}dyd\tau \\
&\lesssim &\left( 1+t\right) ^{-5/2}\int_{0}^{\frac{t}{2}-\frac{\left \vert
x\right \vert }{2\mathbf{c}}}\int_{\frac{\mathbf{c}\left( t-\tau \right)
-\left \vert x\right \vert }{2}}^{\infty }r^{-6}r^{2}drd\tau \lesssim \left(
1+t\right) ^{-5/2}\left( 1+\mathbf{c}t-\left \vert x\right \vert \right)
^{-2} \\
&\lesssim &\left( 1+t\right) ^{-7/2}\left( 1+\frac{\left( \left \vert
x\right \vert -\mathbf{c}t\right) ^{2}}{1+t}\right) ^{-1}\hbox{,}
\end{eqnarray*}%
so that%
\begin{equation*}
K_{1}\lesssim \left( 1+t\right) ^{-5/2}\left( 1+\frac{\left( \left \vert
x\right \vert -\mathbf{c}t\right) ^{2}}{1+t}\right) ^{-1}\hbox{.}
\end{equation*}

For $K_{2}$, we use the spherical coordinates to obtain
\begin{eqnarray*}
K_{2} &=&\int_{\frac{t}{2}-\frac{\left \vert x\right \vert }{2\mathbf{c}}%
}^{t-\frac{\left \vert x\right \vert }{2\mathbf{c}}}\int_{0}^{\infty
}\int_{0}^{\pi }\left( 1+t-\tau \right) ^{-5/2}e^{-\frac{\left( \sqrt{\left
\vert x\right \vert ^{2}+r^{2}-2r\left \vert x\right \vert \cos \theta }-%
\mathbf{c}\left( t-\tau \right) \right) ^{2}}{D_{0}\left( 1+t-\tau \right) }%
}\left( 1+\tau \right) ^{-3}\left( 1+\frac{r^{2}}{1+\tau }\right)
^{-3}r^{2}\sin \theta d\theta drd\tau \\
&=&\int_{\frac{t}{2}-\frac{\left \vert x\right \vert }{2\mathbf{c}}}^{t-%
\frac{\left \vert x\right \vert }{2\mathbf{c}}}\int_{0}^{\infty }\int_{\left
\vert \left \vert x\right \vert -r\right \vert }^{\left \vert x\right \vert
+r}\left( 1+t-\tau \right) ^{-5/2}e^{-\frac{\left( z-\mathbf{c}\left( t-\tau
\right) \right) ^{2}}{D_{0}\left( 1+t-\tau \right) }}\left( 1+\tau \right)
^{-3}\left( 1+\frac{r^{2}}{1+\tau }\right) ^{-3}r^{2}\frac{zdz}{r\left \vert
x\right \vert }drd\tau \\
&\lesssim &\int_{\frac{t}{2}-\frac{\left \vert x\right \vert }{2\mathbf{c}}%
}^{t-\frac{\left \vert x\right \vert }{2\mathbf{c}}}\int_{0}^{\infty
}\int_{\left \vert \left \vert x\right \vert -r\right \vert }^{\left \vert
x\right \vert +r}\left( 1+\left \vert x\right \vert \right) ^{-5/2}e^{-\frac{%
\left( c\tau -\left( \mathbf{c}t-z\right) \right) ^{2}}{D_{0}\left(
1+t\right) }}\left( 1+t\right) ^{-3}\left( 1+\frac{r^{2}}{1+t}\right) ^{-3}r%
\frac{zdz}{\left \vert x\right \vert }drd\tau \\
&\lesssim &\frac{1}{\left \vert x\right \vert }\int_{0}^{\infty }\int_{\left
\vert \left \vert x\right \vert -r\right \vert }^{\left \vert x\right \vert
+r}\left( 1+\left \vert x\right \vert \right) ^{-5/2}\sqrt{1+t}\left(
1+t\right) ^{-3}\left( 1+\frac{r^{2}}{1+t}\right) ^{-3}rzdzdr \\
&\lesssim &\left( 1+\left \vert x\right \vert \right) ^{-7/2}\left(
1+t\right) ^{-\frac{5}{2}}\int_{0}^{\infty }\int_{\left \vert \left \vert
x\right \vert -r\right \vert }^{\left \vert x\right \vert +r}\left( 1+\frac{%
r^{2}}{1+t}\right) ^{-3}rzdzdr \\
&\lesssim &\left( 1+\left \vert x\right \vert \right) ^{-7/2}\left(
1+t\right) ^{-\frac{5}{2}}\left[ \int_{0}^{\left \vert x\right \vert
}\int_{\left \vert x\right \vert -r}^{\left \vert x\right \vert +r}zr\left(
1+\frac{r^{2}}{1+t}\right) ^{-3}dzdr+\int_{\left \vert x\right \vert
}^{\infty }\int_{r-\left \vert x\right \vert }^{r+\left \vert x\right \vert
}zr\left( 1+\frac{r^{2}}{1+t}\right) ^{-3}dzdr\right] \\
&\lesssim &\left( 1+\left \vert x\right \vert \right) ^{-7/2}\left(
1+t\right) ^{-\frac{5}{2}}\left[ \int_{0}^{\left \vert x\right \vert
}r^{2}\left \vert x\right \vert \left( 1+\frac{r^{2}}{1+t}\right)
^{-3}dr+\int_{\left \vert x\right \vert }^{\infty }r^{2}\left \vert x\right
\vert \left( 1+\frac{r^{2}}{1+t}\right) ^{-3}dr\right] \\
&\lesssim &\left( 1+\left \vert x\right \vert \right) ^{-7/2}\left(
1+t\right) ^{-\frac{5}{2}}\left[ \left \vert x\right \vert \left( 1+t\right)
^{\frac{3}{2}}+\left \vert x\right \vert \left( 1+t\right) ^{\frac{3}{2}%
}\left( 1+\frac{\left \vert x\right \vert }{\sqrt{1+t}}\right) ^{-3}\right]
\\
&\lesssim &\left( 1+\left \vert x\right \vert \right) ^{-5/2}\left(
1+t\right) ^{-1}\lesssim \left( 1+t\right) ^{-2}\left( 1+\frac{\left \vert
x\right \vert ^{2}}{ 1+t }\right) ^{-\frac{3}{2}}\hbox{.}
\end{eqnarray*}%
by setting $z=\sqrt{\left \vert x\right \vert ^{2}+r^{2}-2r\left \vert
x\right \vert \cos \theta }$ and $\sin \theta d\theta =\frac{z}{r\left \vert
x\right \vert }dz$.

For $K_{3}$, we split $\mathbb{R}^{3}$ into two parts
\begin{equation*}
K_{3}=\int_{t-\frac{\left \vert x\right \vert }{2\mathbf{c}}}^{t}\left(
\int_{\left \vert y\right \vert \leq \frac{\left \vert x\right \vert -%
\mathbf{c}\left( t-\tau \right) }{2}}+\int_{\left \vert y\right \vert >\frac{%
\left \vert x\right \vert -\mathbf{c}\left( t-\tau \right) }{2}}\right)
\left( \cdots \right) dyd\tau =:K_{31}+K_{32}\hbox{.}
\end{equation*}%
If $t-\frac{\left \vert x\right \vert }{2\mathbf{c}}<\tau \leq t$, $%
\left
\vert y\right \vert \leq \frac{\left \vert x\right \vert -\mathbf{c}%
\left( t-\tau \right) }{2}$, then
\begin{equation*}
\left \vert x-y\right \vert -\mathbf{c}\left( t-\tau \right) \geq \left
\vert x\right \vert -\left \vert y\right \vert -\mathbf{c}\left( t-\tau
\right) \geq \frac{\left \vert x\right \vert -\mathbf{c}\left( t-\tau
\right) }{2}\geq \frac{\left \vert x\right \vert }{4}\hbox{.}
\end{equation*}%
If $t-\frac{\left \vert x\right \vert }{2\mathbf{c}}<\tau \leq t$, $%
\left
\vert y\right \vert >\frac{\left \vert x\right \vert -\mathbf{c}%
\left( t-\tau \right) }{2}$, then%
\begin{equation*}
\left \vert y\right \vert >\frac{\left \vert x\right \vert -\mathbf{c}\left(
t-\tau \right) }{2}\geq \frac{\left \vert x\right \vert }{4}\hbox{.}
\end{equation*}%
Hence,
\begin{eqnarray*}
K_{31} &\lesssim &\left( 1+|x|^{2}\right) ^{-\frac{5}{2}}\int_{t-\frac{\left
\vert x\right \vert }{2\mathbf{c}}}^{t}\int_{\mathbb{R}^{3}}\left( 1+\tau
\right) ^{-3}\left( 1+\frac{\left \vert y\right \vert ^{2}}{1+\tau }\right)
^{-3}dyd\tau \\
&\lesssim &\left( 1+|x|^{2}\right) ^{-\frac{5}{2}}\int_{t-\frac{\left \vert
x\right \vert }{2\mathbf{c}}}^{t}\left( 1+\tau \right) ^{-\frac{3}{2}}d\tau
\\
&\lesssim & \left( 1+t\right) ^{-3}\left( 1+\frac{\left \vert x\right \vert
^{2}}{1+t }\right) ^{-\frac{3}{2}}\hbox{,}
\end{eqnarray*}%
\begin{eqnarray*}
K_{32} &\lesssim &\left( 1+t-\frac{\left \vert x\right \vert }{2\mathbf{c}}%
\right) ^{-3}\left( 1+\frac{\left \vert x\right \vert ^{2}}{ 1+t }\right)
^{-3}\int_{t-\frac{\left \vert x\right \vert }{2\mathbf{c}}}^{t}\int_{\left
\vert y\right \vert >\frac{\left \vert x\right \vert -\mathbf{c}\left(
t-\tau \right) }{2}}\left( 1+t-\tau \right) ^{-5/2}e^{-\frac{\left( \left
\vert x-y\right \vert -\mathbf{c}\left( t-\tau \right) \right) ^{2}}{%
D_{0}\left( 1+t-\tau \right) }}dyd\tau \\
&\lesssim &\left( 1+t-\frac{\left \vert x\right \vert }{2\mathbf{c}}\right)
^{-3}\left( 1+\frac{\left \vert x\right \vert ^{2}} { 1+t }\right)
^{-3}\int_{t-\frac{\left \vert x\right \vert }{2\mathbf{c}}}^{t}\left(
1+t-\tau \right) ^{-5/2}\left( 1+t-\tau \right) ^{\frac{5}{2}}d\tau \\
&\lesssim & \left( 1+t\right) ^{-2}\left( 1+\frac{\left \vert x\right \vert
^{2}}{ 1+t }\right) ^{-3}\hbox{,}
\end{eqnarray*}%
so that
\begin{equation*}
K_{3}\lesssim \left( 1+t\right) ^{-2}\left( 1+\frac{\left \vert x\right
\vert ^{2}}{ 1+t }\right) ^{-\frac{3}{2}}\hbox{.}
\end{equation*}%
To sum up,
\begin{equation*}
K\lesssim \left( 1+t\right) ^{-2}\left( 1+\frac{\left \vert x\right \vert
^{2}}{ 1+t }\right) ^{-\frac{3}{2}}+\left( 1+t\right) ^{-5/2}\left( 1+\frac{%
\left( \left \vert x\right \vert -\mathbf{c}t\right) ^{2}}{1+t}\right) ^{-1}%
\hbox{.}
\end{equation*}
\newline
\textbf{Case 5:} $\left( x,t\right) \in D_{5}$. We split the integral $K$
into three parts%
\begin{equation*}
K=\left( \int_{0}^{\frac{t}{2}-\frac{\left \vert x\right \vert }{2\mathbf{c}}%
}+\int_{\frac{t}{2}-\frac{\left \vert x\right \vert }{2\mathbf{c}}}^{\frac{3%
}{2}\left( t-\frac{\left \vert x\right \vert }{\mathbf{c}}\right) }+\int_{%
\frac{3}{2}\left( t-\frac{\left \vert x\right \vert }{\mathbf{c}}\right)
}^{t}\right) \int_{\mathbb{R}^{3}}\left( \cdots \right) dyd\tau
=:K_{1}+K_{2}+K_{3}\hbox{.}
\end{equation*}

For $K_{1}$, the estimate is the same as the term $K_{1}$ of Case 4 and so
\begin{equation*}
K_{1}\lesssim \left( 1+t\right) ^{-5/2}\left( 1+\frac{\left( \left \vert
x\right \vert -\mathbf{c}t\right) ^{2}}{1+t}\right) ^{-1}\hbox{.}
\end{equation*}

For $K_{2}$, we use the spherical coordinates to obtain
\begin{eqnarray*}
K_{2} &\lesssim &\int_{\frac{t}{2}-\frac{\left \vert x\right \vert }{2%
\mathbf{c}}}^{\frac{3}{2}\left( t-\frac{\left \vert x\right \vert }{\mathbf{c%
}}\right) }\int_{\mathbb{R}^{3}}\left( 1+t\right) ^{-5/2}e^{-\frac{\left(
\left \vert x-y\right \vert -\mathbf{c}\left( t-\tau \right) \right) ^{2}}{%
D_{0}\left( 1+t-\tau \right) }}\left( 1+\tau \right) ^{-3}\left( 1+\frac{%
\left \vert y\right \vert ^{2}}{1+\tau }\right) ^{-3}dyd\tau \\
&\lesssim &\int_{\frac{t}{2}-\frac{\left \vert x\right \vert }{2\mathbf{c}}%
}^{\frac{3}{2}\left( t-\frac{\left \vert x\right \vert }{\mathbf{c}}\right)
}\int_{0}^{\infty }\int_{0}^{\pi }\left( 1+t\right) ^{-5/2}e^{-\frac{\left(
\sqrt{\left \vert x\right \vert ^{2}+r^{2}-2r\left \vert x\right \vert \cos
\theta }-\mathbf{c}\left( t-\tau \right) \right) ^{2}}{D_{0}\left(
1+t\right) }}\left( 1+\tau \right) ^{-3}\left( 1+\frac{r^{2}}{1+\tau }%
\right) ^{-3}r^{2}\sin \theta d\theta drd\tau \\
&\lesssim &\int_{\frac{t}{2}-\frac{\left \vert x\right \vert }{2\mathbf{c}}%
}^{\frac{3}{2}\left( t-\frac{\left \vert x\right \vert }{\mathbf{c}}\right)
}\int_{0}^{\infty }\int_{\left \vert \left \vert x\right \vert -r\right
\vert }^{\left \vert x\right \vert +r}\left( 1+t\right) ^{-5/2}e^{-\frac{%
\left( z-\mathbf{c}\left( t-\tau \right) \right) ^{2}}{D_{0}\left(
1+t\right) }}z\left( 1+\mathbf{c}t-\left \vert x\right \vert \right)
^{-3}\left( 1+\frac{r^{2}}{1+\tau }\right) ^{-3}\frac{r}{\left \vert x\right
\vert }dzdrd\tau \\
&\lesssim &\left( 1+t\right) ^{-7/2}\left( 1+\mathbf{c}t-\left \vert x\right
\vert \right) ^{-3}\int_{\frac{t}{2}-\frac{\left \vert x\right \vert }{2%
\mathbf{c}}}^{\frac{3}{2}\left( t-\frac{\left \vert x\right \vert }{\mathbf{c%
}}\right) }\int_{0}^{\infty }\int_{\left \vert \left \vert x\right \vert
-r\right \vert }^{\left \vert x\right \vert +r}e^{-\frac{\left( c\tau
-\left( \mathbf{c}t-z\right) \right) ^{2}}{D_{0}\left( 1+t\right) }}z\left(
1+\frac{r^{2}}{1+\tau }\right) ^{-3}rdzdrd\tau \\
&\lesssim &\left( 1+t\right) ^{-7/2}\left( 1+\mathbf{c}t-\left \vert x\right
\vert \right) ^{-3}\left( 1+t\right) ^{\frac{1}{2}} \\
&&\cdot \left[ \int_{0}^{\left \vert x\right \vert }\int_{\left \vert
x\right \vert -r}^{\left \vert x\right \vert +r}z\left( 1+\frac{r^{2}}{1+%
\mathbf{c}t-\left \vert x\right \vert }\right) ^{-3}rdzdr+\int_{\left \vert
x\right \vert }^{\infty }\int_{r-\left \vert x\right \vert }^{r+\left \vert
x\right \vert }z\left( 1+\frac{r^{2}}{1+\mathbf{c}t-\left \vert x\right
\vert }\right) ^{-3}rdzdr\right] \\
&\lesssim &\left( 1+t\right) ^{-3}\left( 1+\mathbf{c}t-\left \vert x\right
\vert \right) ^{-3}\left[ \int_{0}^{\left \vert x\right \vert }r^{2}\left
\vert x\right \vert \left( 1+\frac{r^{2}}{1+\mathbf{c}t-\left \vert x\right
\vert }\right) ^{-3}dr+\int_{\left \vert x\right \vert }^{\infty }r^{2}\left
\vert x\right \vert \left( 1+\frac{r^{2}}{1+\mathbf{c}t-\left \vert x\right
\vert }\right) ^{-3}dr\right] \\
&\lesssim &\left( 1+t\right) ^{-3}\left( 1+\mathbf{c}t-\left \vert x\right
\vert \right) ^{-3}\left[ \left \vert x\right \vert \left( 1+\mathbf{c}%
t-\left \vert x\right \vert \right) ^{\frac{3}{2}}+\left \vert x\right \vert
\left( 1+\mathbf{c}t-\left \vert x\right \vert \right) ^{3}\left( 1+\left
\vert x\right \vert \right) ^{-3}\right] \\
&\lesssim &\left( 1+t\right) ^{-2}\left( 1+\mathbf{c}t-\left \vert x\right
\vert \right) ^{-\frac{3}{2}}\lesssim \left( 1+t\right) ^{-2 }\left( 1+%
\mathbf{c}t-\left \vert x\right \vert \right) ^{\frac{1}{2}}\left( 1+\mathbf{%
c}t-\left \vert x\right \vert \right) ^{-2} \\
&\lesssim & \left( 1+t\right) ^{-5/2}\left( 1+\frac{\left( \mathbf{c}t-\left
\vert x\right \vert \right) ^{2}}{1+t}\right) ^{-1}\hbox{.}
\end{eqnarray*}

For $K_{3}$, we decompose $\mathbb{R}^{3}$ into two parts%
\begin{equation*}
K_{3}=\int_{\frac{3}{2}\left( t-\frac{\left \vert x\right \vert }{\mathbf{c}}%
\right) }^{t}\left( \int_{\left \vert y\right \vert \leq \frac{\left \vert
x\right \vert -\mathbf{c}\left( t-\tau \right) }{2}}+\int_{\left \vert
y\right \vert >\frac{\left \vert x\right \vert -\mathbf{c}\left( t-\tau
\right) }{2}}\right) \left( \cdots \right) dyd\tau =K_{31}+K_{32}\hbox{.}
\end{equation*}%
If $\frac{3}{2}\left( t-\frac{\left \vert x\right \vert }{\mathbf{c}}\right)
\leq \tau \leq t$, $\left \vert y\right \vert \leq \frac{\left \vert
x\right
\vert -\mathbf{c}\left( t-\tau \right) }{2}$, then
\begin{equation*}
\left \vert x-y\right \vert -\mathbf{c}\left( t-\tau \right) \geq \left
\vert x\right \vert -\left \vert y\right \vert -\mathbf{c}\left( t-\tau
\right) \geq \frac{\left \vert x\right \vert -\mathbf{c}\left( t-\tau
\right) }{2}\geq \frac{\mathbf{c}t-\left \vert x\right \vert }{4}\hbox{.}
\end{equation*}%
If $\frac{3}{2}\left( t-\frac{\left \vert x\right \vert }{\mathbf{c}}\right)
\leq \tau \leq t$, $\left \vert y\right \vert >\frac{\left \vert
x\right
\vert -\mathbf{c}\left( t-\tau \right) }{2}$, then
\begin{equation*}
\left \vert y\right \vert >\frac{\left \vert x\right \vert -\mathbf{c}\left(
t-\tau \right) }{2}\geq \frac{\mathbf{c}t-\left \vert x\right \vert }{4}%
\hbox{.}
\end{equation*}%
Hence, we use the spherical coordinates to obtain
\begin{eqnarray*}
K_{31} &\lesssim &\int_{\frac{3}{2}\left( t-\frac{\left \vert x\right \vert
}{\mathbf{c}}\right) }^{t}\int_{0}^{\frac{\left \vert x\right \vert -\mathbf{%
c}\left( t-\tau \right) }{2}}\int_{0}^{\pi }\left( 1+t-\tau \right)
^{-5/2}e^{-\frac{\left( \mathbf{c}t-\left \vert x\right \vert \right) ^{2}}{%
32D_{0}\left( 1+t-\tau \right) }}e^{-\frac{\left( \sqrt{\left \vert x\right
\vert ^{2}+r^{2}-2r\left \vert x\right \vert \cos \theta }-\mathbf{c}\left(
t-\tau \right) \right) ^{2}}{2D_{0}\left( 1+t-\tau \right) }} \\
&&\cdot \left( 1+\tau \right) ^{-3}\left( 1+\frac{r^{2}}{1+\tau }\right)
^{-3}r^{2}\sin \theta d\theta drd\tau \\
&\lesssim &e^{-\frac{\left( \mathbf{c}t-\left \vert x\right \vert \right)
^{2}}{32D_{0}\left( 1+t\right) }}\int_{\frac{3}{2}\left( t-\frac{\left \vert
x\right \vert }{\mathbf{c}}\right) }^{t}\int_{0}^{\frac{\left \vert x\right
\vert -\mathbf{c}\left( t-\tau \right) }{2}}\int_{\left \vert x\right \vert
-r}^{\left \vert x\right \vert +r}\left( 1+t-\tau \right) ^{-5/2}e^{-\frac{%
\left( z-\mathbf{c}\left( t-\tau \right) \right) ^{2}}{2D_{0}\left( 1+t-\tau
\right) }}z\left( 1+\tau \right) ^{-3} \\
&&\cdot \left( 1+\frac{r^{2}}{1+\tau }\right) ^{-3}\frac{r}{\left \vert
x\right \vert }dzdrd\tau \\
&\lesssim &\frac{1}{\left \vert x\right \vert }e^{-\frac{\left( \mathbf{c}%
t-\left \vert x\right \vert \right) ^{2}}{32D_{0}\left( 1+t\right) }}\int_{%
\frac{3}{2}\left( t-\frac{\left \vert x\right \vert }{\mathbf{c}}\right)
}^{t}\left( 1+t-\tau \right) ^{-5/2}\left( 1+\tau \right) ^{-3}\left(
1+t-\tau \right) ^{\frac{3}{2}}\left( 1+\tau \right) d\tau \\
&\lesssim &\frac{1}{\left \vert x\right \vert }e^{-\frac{\left( \mathbf{c}%
t-\left \vert x\right \vert \right) ^{2}}{32D_{0}\left( 1+t\right) }}\left(
\int_{\frac{3}{2}\left( t-\frac{\left \vert x\right \vert }{\mathbf{c}}%
\right) }^{\frac{1}{2}\left( t+\frac{3}{2}\left( t-\frac{\left \vert x\right
\vert }{\mathbf{c}}\right) \right) }+\int_{\frac{1}{2}\left( t+\frac{3}{2}%
\left( t-\frac{\left \vert x\right \vert }{\mathbf{c}}\right) \right)
}^{t}\right) \left( 1+t-\tau \right) ^{-1}\left( 1+\tau \right) ^{-2}d\tau \\
&\lesssim &\frac{e^{-\frac{\left( \mathbf{c}t-\left \vert x\right \vert
\right) ^{2}}{32D_{0}\left( 1+t\right) }}}{\left \vert x\right \vert }\left[
\left( 1+\frac{3\left \vert x\right \vert }{4\mathbf{c}}-\frac{t}{4}\right)
^{-1}\left( 1+\frac{3}{2}\left( t-\frac{\left \vert x\right \vert }{\mathbf{c%
}}\right) \right) ^{-1}+\left( 1+t\right) ^{-2}\ln \left(1+\frac{3|x|}{4%
\mathbf{c}}-\frac{t}{4}\right) \right] \\
&\lesssim &\frac{e^{-\frac{\left( \mathbf{c}t-\left \vert x\right \vert
\right) ^{2}}{32D_{0}\left( 1+t\right) }}}{1+t}\left(
1+t\right)^{-3/2}\lesssim \left( 1+t\right) ^{-5/2}\left( 1+\frac{\left(
\mathbf{c}t-x\right) ^{2}}{1+t}\right) ^{-1}\hbox{,}
\end{eqnarray*}%
by setting $z=\sqrt{\left \vert x\right \vert ^{2}+r^{2}-2r\left \vert
x\right \vert \cos \theta }$ and $\sin \theta d\theta =\frac{z}{r\left \vert
x\right \vert }dz$. As for $K_{32}$, direct computation gives%
\begin{eqnarray*}
K_{32} &\lesssim &\int_{\frac{3}{2}\left( t-\frac{\left \vert x\right \vert
}{\mathbf{c}}\right) }^{t}\int_{\left \vert y\right \vert >\frac{\left \vert
x\right \vert -\mathbf{c}\left( t-\tau \right) }{2}}\left( 1+t-\tau \right)
^{-5/2}e^{-\frac{\left( \left \vert x-y\right \vert -\mathbf{c}\left( t-\tau
\right) \right) ^{2}}{D_{0}\left( 1+t-\tau \right) }}\left( 1+\left \vert
y\right \vert ^{2}\right) ^{-3}dyd\tau \\
&\lesssim &\left( \int_{\frac{3}{2}\left( t-\frac{\left \vert x\right \vert
}{\mathbf{c}}\right) }^{\frac{1}{2}\left( t+\frac{3}{2}\left( t-\frac{\left
\vert x\right \vert }{\mathbf{c}}\right) \right) }+\int_{\frac{1}{2}\left( t+%
\frac{3}{2}\left( t-\frac{\left \vert x\right \vert }{\mathbf{c}}\right)
\right) }^{t}\right) \int_{\left \vert y\right \vert >\frac{\left \vert
x\right \vert -\mathbf{c}\left( t-\tau \right) }{2}}(\cdot \cdot \cdot) dy
d\tau \\
&= &K_{321}+K_{322}\,,
\end{eqnarray*}%
we then have
\begin{eqnarray*}
K_{321} &\lesssim &\int_{\frac{3}{2}\left( t-\frac{\left \vert x\right \vert
}{\mathbf{c}}\right) }^{\frac{1}{2}\left( t+\frac{3}{2}\left( t-\frac{\left
\vert x\right \vert }{\mathbf{c}}\right) \right) }\int_{\left \vert y\right
\vert >\frac{\left \vert x\right \vert -\mathbf{c}\left( t-\tau \right) }{2}%
}\left( 1+t-\tau \right) ^{-5/2}e^{-\frac{\left( \left \vert x-y\right \vert
-\mathbf{c}\left( t-\tau \right) \right) ^{2}}{D_{0}\left( 1+t-\tau \right) }%
}\left( 1+\left \vert y\right \vert ^{2}\right) ^{-3}dyd\tau \\
&\lesssim &\int_{\frac{3}{2}\left( t-\frac{\left \vert x\right \vert }{%
\mathbf{c}}\right) }^{\frac{1}{2}\left( t+\frac{3}{2}\left( t-\frac{\left
\vert x\right \vert }{\mathbf{c}}\right) \right) } \left( 1+t-\tau \right)
^{-5/2}\left( 1+\left \vert x\right \vert -\mathbf{c}\left( t-\tau \right)
\right) ^{-3}d\tau \\
&\lesssim &\left( 1+t\right) ^{-5/2}\left( 1+\mathbf{c}t-\left \vert x\right
\vert \right) ^{-2} \\
&\lesssim &\left( 1+t\right) ^{-7/2}\left( 1+\frac{\left( \mathbf{c}t-\left
\vert x\right \vert \right) ^{2}}{1+t}\right) ^{-1}\hbox{.}
\end{eqnarray*}%
and
\begin{eqnarray*}
K_{322} &\lesssim &\int_{\frac{1}{2}\left( t+\frac{3}{2}\left( t-\frac{\left
\vert x\right \vert }{\mathbf{c}}\right) \right) }^{t}\int_{\left \vert
y\right \vert >\frac{\left \vert x\right \vert -\mathbf{c}\left( t-\tau
\right) }{2}}\left( 1+t-\tau \right) ^{-5/2}e^{-\frac{\left( \left \vert
x-y\right \vert -\mathbf{c}\left( t-\tau \right) \right) ^{2}}{D_{0}\left(
1+t-\tau \right) }}\left( 1+\left \vert y\right \vert ^{2}\right)
^{-3}dyd\tau \\
&\lesssim &\int_{\frac{1}{2}\left( t+\frac{3}{2}\left( t-\frac{\left \vert
x\right \vert }{\mathbf{c}}\right) \right) }^{t}\left( 1+\left \vert x\right
\vert -\mathbf{c}\left( t-\tau \right) \right) ^{-6}d\tau \\
&\lesssim &\left( 1+\left \vert x\right \vert +\mathbf{c}t\right) ^{-5} \\
&\lesssim &\left( 1+t\right) ^{-\frac{7}{2}}\left( 1+\frac{\left \vert
x\right \vert ^{2}}{1+t}\right) ^{-\frac{3}{2}}\hbox{.}
\end{eqnarray*}%
It implies that
\begin{equation*}
K_{3}\lesssim \left( 1+t\right) ^{-5/2}\left( 1+\frac{\left( \mathbf{c}%
t-\left \vert x\right \vert \right) ^{2}}{1+t}\right) ^{-1}+\left(
1+t\right) ^{-2}\left( 1+\frac{\left \vert x\right \vert ^{2}}{1+t}\right)
^{-\frac{3}{2}}\hbox{.}
\end{equation*}%
As a consequence,
\begin{equation*}
K\lesssim \left( 1+t\right) ^{-5/2}\left( 1+\frac{\left( \mathbf{c}t-\left
\vert x\right \vert \right) ^{2}}{1+t}\right) ^{-1}+\left( 1+t\right)
^{-2}\left( 1+\frac{\left \vert x\right \vert ^{2}}{1+t}\right) ^{-\frac{3}{2%
}}\hbox{.}
\end{equation*}
\end{proof}

\begin{proof}[\textbf{Proof of Lemma \protect \ref{nonmoving-moving}}]
(Diffusion wave convolution with Huygens wave). \newline
\newline
\textbf{Case 1:} $\left( x,t\right) \in D_{1}$. We split the integral into
two parts to obtain%
\begin{eqnarray*}
L &=&\left( \int_{0}^{\frac{t}{2}}+\int_{\frac{t}{2}}^{t}\right) \int_{%
\mathbb{R}^{3}}\left( \cdots \right) dyd\tau \\
&\lesssim &\left( 1+t\right) ^{-2}\int_{0}^{\frac{t}{2}}\int_{\mathbb{R}%
^{3}}\left( 1+\tau \right) ^{-4}\left( 1+\frac{\left( \left \vert y\right
\vert -\mathbf{c}\tau \right) ^{2}}{1+\tau }\right) ^{-2}dyd\tau \\
&&+\left( 1+t\right) ^{-4}\int_{\frac{t}{2}}^{t}\int_{\mathbb{R}^{3}}\left(
1+t-\tau \right) ^{-2}\left( 1+\frac{\left \vert x-y\right \vert ^{2}}{%
D_{0}\left( 1+t-\tau \right) }\right) ^{-2}dyd\tau \\
&\lesssim &\left( 1+t\right) ^{-2}\int_{0}^{\frac{t}{2}}\left( 1+\tau
\right) ^{-4}\left( 1+\tau \right) ^{\frac{5}{2}}d\tau +\left( 1+t\right)
^{-4}\int_{t/2}^{t}\left( 1+t-\tau \right) ^{-\frac{1}{2}}d\tau \\
&\lesssim &\left( 1+t\right) ^{-2}\lesssim \left( 1+t\right) ^{-2}\left( 1+%
\frac{\left \vert x\right \vert ^{2}}{1+t}\right) ^{-\frac{3}{2}}\hbox{.}
\end{eqnarray*}
\newline
\textbf{Case 2:} $\left( x,t\right) \in D_{2}$. We split the integral into
two parts:
\begin{equation*}
L=\left( \int_{0}^{\frac{t}{2}}+\int_{\frac{t}{2}}^{t}\right) \int_{\mathbb{R%
}^{3}}\left( \cdots \right) dyd\tau =:L_{1}+L_{2}\hbox{.}
\end{equation*}

First one can see%
\begin{eqnarray*}
L_{2} &\lesssim &\left( 1+t\right) ^{-4}\int_{\frac{t}{2}}^{t}\int_{\mathbb{R%
}^{3}}\left( 1+t-\tau \right) ^{-2}\left( 1+\frac{\left \vert x-y\right
\vert ^{2}}{D_{0}\left( 1+t-\tau \right) }\right) ^{-2}dyd\tau \\
&\lesssim &\left( 1+t\right) ^{-7/2}\left( 1+\frac{\left( \left \vert
x\right \vert -\mathbf{c}t\right) ^{2}}{1+t}\right) ^{-1}\hbox{.}
\end{eqnarray*}

As for $L_{1}$, we further decompose $\mathbb{R}^{3}$ into two parts:
\begin{equation*}
L_{1}=\int_{0}^{\frac{t}{2}}\left( \int_{\left \vert y\right \vert \leq
\frac{2}{3}\left \vert x\right \vert }+\int_{\left \vert y\right \vert >%
\frac{2}{3}\left \vert x\right \vert }\right) \left( \ldots \right) dyd\tau
=:L_{11}+L_{12}.
\end{equation*}%
If $\left \vert y\right \vert \leq \frac{2}{3}\left \vert x\right \vert $,
then we have
\begin{equation*}
\left \vert x-y\right \vert \geq \left \vert x\right \vert -\left \vert
y\right \vert \geq \frac{\left \vert x\right \vert }{3}\hbox{,}
\end{equation*}%
and thus
\begin{eqnarray*}
&&L_{11}\lesssim \left( 1+t\right) ^{-2}\left( 1+\frac{\left \vert x\right
\vert ^{2}}{1+t}\right) ^{-2}\int_{0}^{\frac{t}{2}}\int_{\left \vert y\right
\vert \leq \frac{2}{3}\left \vert x\right \vert }\left( 1+\tau \right)
^{-4}\left( 1+\frac{\left( \left \vert y\right \vert -\mathbf{c}\tau \right)
^{2}}{1+\tau }\right) ^{-2}dyd\tau \\
&\lesssim &\left( 1+t\right) ^{-2}\left( 1+\frac{\left \vert x\right \vert
^{2}}{1+t}\right) ^{-2}\int_{0}^{\frac{t}{2}}\left( 1+\tau \right)
^{-4}\left( 1+\tau \right) ^{\frac{5}{2}}d\tau \lesssim \left( 1+t\right)
^{-2}\left( 1+\frac{\left \vert x\right \vert ^{2}}{1+t}\right) ^{-2}\hbox{.}
\end{eqnarray*}%
If $\left \vert y\right \vert \geq \frac{2}{3}\left \vert x\right \vert $
and $0\leq \tau \leq \frac{t}{2}$, then
\begin{equation*}
\left \vert y\right \vert -\mathbf{c}\tau \geq \frac{2}{3}\left \vert
x\right \vert -\mathbf{c}\tau \geq \frac{2}{3}\left( \mathbf{c}t-\sqrt{1+t}%
\right) -\frac{\mathbf{c}t}{2}\geq \frac{\mathbf{c}t}{12}+\frac{\mathbf{c}t}{%
12}-\frac{2}{3}\sqrt{1+t}\geq \frac{\mathbf{c}t}{12}
\end{equation*}%
$\allowbreak $for $t\geq 40$. Hence,
\begin{eqnarray*}
&&\int_{0}^{\frac{t}{2}}\int_{\left \vert y\right \vert \geq \frac{2}{3}%
\left \vert x\right \vert }\left( 1+t-\tau \right) ^{-2}\left( 1+\frac{\left
\vert x-y\right \vert ^{2}}{D_{0}\left( 1+t-\tau \right) }\right)
^{-2}\left( 1+\tau \right) ^{-4}\left( 1+\frac{\left( \left \vert y\right
\vert -\mathbf{c}\tau \right) ^{2}}{1+\tau }\right) ^{-2}dyd\tau \\
&\lesssim &\left( 1+t\right) ^{-4}\int_{0}^{\frac{t}{2}}\int_{\mathbb{R}%
^{3}}\left( 1+t-\tau \right) ^{-2}\left( 1+\frac{\left \vert x-y\right \vert
^{2}}{D_{0}\left( 1+t-\tau \right) }\right) ^{-2}\left( 1+\tau \right)
^{-2}dyd\tau \\
&\lesssim &\left( 1+t\right) ^{-9/2}\left( 1+\frac{\left( \left \vert
x\right \vert -\mathbf{c}t\right) ^{2}}{1+t}\right) ^{-1}
\end{eqnarray*}%
whenever $t\geq 40$, and
\begin{eqnarray*}
&&\int_{0}^{\frac{t}{2}}\int_{\left \vert y\right \vert \geq \frac{2}{3}%
\left \vert x\right \vert }\left( 1+t-\tau \right) ^{-2}\left( 1+\frac{\left
\vert x-y\right \vert ^{2}}{D_{0}\left( 1+t-\tau \right) }\right)
^{-2}\left( 1+\tau \right) ^{-4}\left( 1+\frac{\left( \left \vert y\right
\vert -\mathbf{c}\tau \right) ^{2}}{1+\tau }\right) ^{-2}dyd\tau \\
&\lesssim &C\lesssim \left( 1+t\right) ^{-5/2}\left( 1+\frac{\left( \left
\vert x\right \vert -\mathbf{c}t\right) ^{2}}{1+t}\right) ^{-1}
\end{eqnarray*}%
whenever $0\leq t\leq 40$. Therefore, we get%
\begin{equation*}
L_{12}\lesssim \left( 1+t\right) ^{-5/2}\left( 1+\frac{\left( \left \vert
x\right \vert -\mathbf{c}t\right) ^{2}}{1+t}\right) ^{-1}\hbox{.}
\end{equation*}%
Combining all the estimates, we have%
\begin{equation*}
L\lesssim \left( 1+t\right) ^{-2}\left( 1+\frac{\left \vert x\right \vert
^{2}}{1+t}\right) ^{-2}+\left( 1+t\right) ^{-5/2}\left( 1+\frac{\left( \left
\vert x\right \vert -\mathbf{c}t\right) ^{2}}{1+t}\right) ^{-1}\hbox{.}
\end{equation*}
\newline
\textbf{Case 3:} $\left( x,t\right) \in D_{3}$. We split the integral into
two parts%
\begin{equation*}
L=\left( \int_{0}^{\frac{t}{2}}+\int_{\frac{t}{2}}^{t}\right) \int_{\mathbb{R%
}^{3}}\left( \cdots \right) dyd\tau =:L_{1}+L_{2}\hbox{.}
\end{equation*}

For $L_{1}$, we further decompose the space domain into two parts:
\begin{equation*}
L_{1}=\int_{0}^{\frac{t}{2}}\left( \int_{\left \vert y\right \vert \leq
\frac{2}{3}\left \vert x\right \vert }+\int_{\left \vert y\right \vert >%
\frac{2}{3}\left \vert x\right \vert }\right) \left( \cdots \right) dyd\tau
=:L_{11}+L_{12}\hbox{.}
\end{equation*}%
If $\left \vert y\right \vert \leq \frac{2}{3}\left \vert x\right \vert $,
then
\begin{equation*}
\left \vert x-y\right \vert \geq \frac{\left \vert x\right \vert }{3}\hbox{,}
\end{equation*}%
and thus%
\begin{eqnarray*}
L_{11} &\lesssim &\left( 1+t\right) ^{-2}\left( 1+\frac{\left \vert x\right
\vert ^{2}}{1+t}\right) ^{-\frac{3}{2}}\int_{0}^{\frac{t}{2}}\int_{\mathbb{R}%
^{3}}\left( 1+\tau \right) ^{-4}\left( 1+\frac{\left( \left \vert y\right
\vert -\mathbf{c}\tau \right) ^{2}}{1+\tau }\right) ^{-2}dyd\tau \\
&\lesssim &\left( 1+t\right) ^{-2}\left( 1+\frac{\left \vert x\right \vert
^{2}}{1+t}\right) ^{-\frac{3}{2}}\hbox{.}
\end{eqnarray*}%
If $\left \vert y\right \vert \geq \frac{2}{3}\left \vert x\right \vert $
and $0\leq \tau \leq \frac{t}{2}$, we have%
\begin{equation*}
\left \vert y\right \vert -\mathbf{c}\tau \geq \frac{2}{3}\left \vert
x\right \vert -\frac{\mathbf{c}t}{2}=\frac{1}{6}\left \vert x\right \vert +%
\frac{1}{2}\left( \left \vert x\right \vert -\mathbf{c}t\right) \hbox{,}
\end{equation*}%
so that%
\begin{eqnarray*}
L_{12} &\lesssim &\int_{0}^{\frac{t}{2}}\int_{\left \vert y\right \vert \geq
\frac{2}{3}\left \vert x\right \vert }\left( 1+t-\tau \right) ^{-2 }e^{-%
\frac{\left \vert x-y\right \vert ^{2}}{D_{0}\left( 1+t-\tau \right) }%
}\left( 1+\tau \right) ^{-2}\left( 1+\tau +\left( \left \vert x\right \vert -%
\mathbf{c}t\right) ^{2}\right) ^{-2}dyd\tau \\
&\lesssim &\left( 1+\left( \left \vert x\right \vert -\mathbf{c}t\right)
^{2}\right) ^{-2}\int_{0}^{\frac{t}{2}}\int_{\mathbb{R}^{3}}\left( 1+t-\tau
\right) ^{-2}e^{-\frac{\left \vert x-y\right \vert ^{2}}{D_{0}\left(
1+t-\tau \right) }}\left( 1+\tau \right) ^{-2}dyd\tau \\
&\lesssim &\left( 1+\left( \left \vert x\right \vert -\mathbf{c}t\right)
^{2}\right) ^{-2}(1+t)^{-1/2}\lesssim \left( \frac{1+t}{2}+\frac{\left(
\left \vert x\right \vert -\mathbf{c}t\right) ^{2}}{2}\right)
^{-2}(1+t)^{-1/2} \\
&\lesssim& \left( 1+t\right) ^{-5/2}\left( 1+\frac{\left( \left \vert
x\right \vert -\mathbf{c}t\right) ^{2}}{1+t}\right) ^{-2}
\end{eqnarray*}%
since $\left \vert x\right \vert -\mathbf{c}t\geq \sqrt{1+t}$ for $\left(
x,t\right) \in D_{3}$. Therefore, we obtain%
\begin{equation*}
L_{1}\lesssim \left( 1+t\right) ^{-2}\left( 1+\frac{\left \vert x\right
\vert ^{2}}{1+t}\right) ^{-\frac{3}{2}}+\left( 1+t\right) ^{-5/2}\left( 1+%
\frac{\left( \left \vert x\right \vert -\mathbf{c}t\right) ^{2}}{1+t}\right)
^{-2}\hbox{.}
\end{equation*}

Next for $L_{2}$, we decompose $\mathbb{R}^{3}$ into two parts%
\begin{equation*}
L_{2}=\int_{\frac{t}{2}}^{t}\left( \int_{\left \vert y\right \vert -\mathbf{c%
}\tau \leq \frac{\left \vert x\right \vert -ct}{2}}+\int_{\left \vert
y\right \vert -\mathbf{c}\tau >\frac{\left \vert x\right \vert -ct}{2}%
}\right) \left( \cdots \right) dyd\tau =:L_{21}+L_{22}\hbox{.}
\end{equation*}%
If $\left \vert y\right \vert -\mathbf{c}\tau \leq \frac{\left \vert
x\right
\vert -\mathbf{c}t}{2}$, we have
\begin{equation*}
\left \vert x-y\right \vert \geq \left \vert x\right \vert -\left( \frac{%
\left \vert x\right \vert -\mathbf{c}t}{2}\right) -\mathbf{c}\tau \geq \frac{%
\left \vert x\right \vert -\mathbf{c}t}{2}\hbox{,}
\end{equation*}%
so
\begin{eqnarray*}
L_{21} &\lesssim &\left( 1+\frac{\left( \left \vert x\right \vert -\mathbf{c}%
t\right) ^{2}}{1+t}\right) ^{-2}\int_{\frac{t}{2}}^{t}\int_{\mathbb{R}%
^{3}}\left( 1+t-\tau \right) ^{-2}e^{-\frac{\left \vert x-y\right \vert ^{2}%
}{2D_{0}\left( 1+t-\tau \right) }}\left( 1+\tau \right) ^{-4}dyd\tau \\
&\lesssim &\left( 1+t\right) ^{-7/2}\left( 1+\frac{\left( \left \vert
x\right \vert -\mathbf{c}t\right) ^{2}}{1+t}\right) ^{-2}\hbox{,}
\end{eqnarray*}
and
\begin{eqnarray*}
L_{22} &\lesssim &\left( 1+\frac{\left( \left \vert x\right \vert -\mathbf{c}%
t\right) ^{2}}{1+t}\right) ^{-2}\int_{\frac{t}{2}}^{t}\int_{\mathbb{R}%
^{3}}\left( 1+t-\tau \right) ^{-2}e^{-\frac{\left \vert x-y\right \vert ^{2}%
}{D_{0}\left( 1+t-\tau \right) }}\left( 1+\tau \right) ^{-4}dyd\tau \\
&\lesssim &\left( 1+t\right) ^{-7/2}\left( 1+\frac{\left( \left \vert
x\right \vert -\mathbf{c}t\right) ^{2}}{1+t}\right) ^{-2}\hbox{.}
\end{eqnarray*}

Therefore,
\begin{equation*}
L\lesssim \left( 1+t\right) ^{-2}\left( 1+\frac{\left \vert x\right \vert
^{2}}{1+t}\right) ^{-\frac{3}{2}}+\left( 1+t\right) ^{-5/2}\left( 1+\frac{%
\left( \left \vert x\right \vert -\mathbf{c}t\right) ^{2}}{1+t}\right) ^{-2}%
\hbox{.}
\end{equation*}
\newline
\textbf{Case 4:} $\left( x,t\right) \in D_{4}\cup D_{5}$. Observe that%
\begin{equation*}
\frac{\sqrt{1+t}}{2\mathbf{c}}\leq \frac{\left \vert x\right \vert }{2%
\mathbf{c}}\leq \frac{t}{2}\hbox{.}
\end{equation*}

We split the integral into three parts:%
\begin{equation*}
L=\left( \int_{0}^{\frac{\left \vert x\right \vert }{2\mathbf{c}}}+\int_{%
\frac{\left \vert x\right \vert }{2\mathbf{c}}}^{\frac{t}{2}}+\int_{\frac{t}{%
2}}^{t}\right) \int_{\mathbb{R}^{3}}\left( \ldots \right) dyd\tau
=:L_{1}+L_{2}+L_{3}\hbox{.}
\end{equation*}

For $L_{1}$, we decompose $\mathbb{R}^{3}$ into two parts $\{ \left \vert
y\right \vert \leq \frac{3}{4}\left \vert x\right \vert \}$ and $\{
\left
\vert y\right \vert >\frac{3}{4}\left \vert x\right \vert \}$. If $%
\left
\vert y\right \vert \geq \frac{3}{4}\left \vert x\right \vert $ and $%
0\leq \tau \leq \frac{\left \vert x\right \vert }{2\mathbf{c}}$, then
\begin{equation*}
\left \vert y\right \vert -\mathbf{c}\tau \geq \frac{1}{4}\left \vert
x\right \vert \hbox{.}
\end{equation*}
Thus,
\begin{eqnarray*}
L_{1} &\lesssim &\left( 1+t\right) ^{-2}\left( 1+\frac{\left \vert x\right
\vert ^{2}}{1+t}\right) ^{-\frac{3}{2}}\int_{0}^{\frac{t}{2}}\int_{\left
\vert y\right \vert \leq \frac{3\left \vert x\right \vert }{4}}\left( 1+\tau
\right) ^{-4}\left( 1+\frac{\left( \left \vert y\right \vert -\mathbf{c}\tau
\right) ^{2}}{1+\tau }\right) ^{-2}dyd\tau \\
&&+\left( 1+\left \vert x\right \vert ^{2}\right) ^{-2}\int_{0}^{\frac{t}{2}%
}\int_{\left \vert y\right \vert >\frac{3\left \vert x\right \vert }{4}%
}\left( 1+t-\tau \right) ^{-2}e^{-\frac{\left \vert x-y\right \vert ^{2}}{%
D_{0}\left( 1+t-\tau \right) }}\left( 1+\tau \right) ^{-2}dyd\tau \\
&\lesssim &\left( 1+t\right) ^{-2}\left( 1+\frac{\left \vert x\right \vert
^{2}}{1+t}\right) ^{-\frac{3}{2}}+\left( \frac{\left \vert x\right \vert ^{2}%
}{2}+\frac{1+t}{2}\right) ^{-2} (1+t)^{-1/2} \\
&\lesssim &\left( 1+t\right) ^{-2}\left( 1+\frac{\left \vert x\right \vert }{%
1+t}\right) ^{-\frac{3}{2}}\hbox{.}
\end{eqnarray*}

As for $L_{2}$ and $L_{3}$, it immediately follows that
\begin{eqnarray*}
L_{2} &\lesssim &\int_{\frac{\left \vert x\right \vert }{2\mathbf{c}}}^{%
\frac{t}{2}}\int_{\mathbb{R}^{3}}\left( 1+t-\tau \right) ^{-2}e^{-\frac{%
\left \vert x-y\right \vert ^{2}}{D_{0}\left( 1+t-\tau \right) }}\left(
1+\tau \right) ^{-4}\left( 1+\frac{\left( \left \vert y\right \vert -\mathbf{%
c}\tau \right) ^{2}}{1+\tau }\right) ^{-2}dyd\tau \\
&\lesssim &\left( 1+t\right) ^{-2}\left( 1+\frac{\left \vert x\right \vert }{%
2\mathbf{c}}\right) ^{-2}\int_{\frac{\left \vert x\right \vert }{2\mathbf{c}}%
}^{\frac{t}{2}}\left( 1+\tau \right) ^{-2}\int_{\mathbb{R}^{3}}e^{-\frac{%
\left \vert x-y\right \vert ^{2}}{D_{0}\left( 1+t-\tau \right) }}dyd\tau \\
&\lesssim &\left( 1+t\right) ^{-2}\left( 1+\frac{\left \vert x\right \vert }{%
2\mathbf{c}}\right) ^{-3}\left( 1+t\right) ^{\frac{3}{2}} \\
&\lesssim &\left( 1+t\right) ^{-5/2}\left( 1+\frac{\left \vert x\right \vert
^{2}}{1+t}\right) ^{-2}\hbox{,}
\end{eqnarray*}

and%
\begin{eqnarray*}
L_{3} &\lesssim &\left( 1+t\right) ^{-4}\int_{\frac{t}{2}}^{t}\int_{\mathbb{R%
}^{3}}\left( 1+t-\tau \right) ^{-2}e^{-\frac{\left \vert x-y\right \vert ^{2}%
}{D_{0}\left( 1+t-\tau \right) }}dyd\tau \\
&\lesssim &\left( 1+t\right) ^{-7/2}\lesssim \left( 1+\left \vert x\right
\vert ^{2}\right) ^{-\frac{3}{2}}(1+t)^{-1/2}\lesssim \left( 1+t\right)
^{-2}\left( 1+\frac{\left \vert x\right \vert }{1+t}\right) ^{-\frac{3}{2}}%
\hbox{,}
\end{eqnarray*}%
since $\sqrt{1+t}\leq \left \vert x\right \vert \leq \mathbf{c}t$.
Therefore, we have%
\begin{equation*}
L\lesssim \left( 1+t\right) ^{-2}\left( 1+\frac{\left \vert x\right \vert }{%
1+t}\right) ^{-\frac{3}{2}}\hbox{.}
\end{equation*}
\end{proof}

\textbf{Acknowledgments:}

This work is partially supported by the National Key R\&D Program of China
under grant 2022YFA1007300. Y.-C. Lin is supported by the National Science
and Technology Council under the grant NSTC 112-2115-M-006-006-MY2. H.-T.
Wang is supported by NSFC under Grant No. 12371220 and 12031013, the
Strategic Priority Research Program of Chinese Academy of Sciences under
Grant No. XDA25010403. K.-C. Wu is supported by the National Science and
Technology Council under the grant NSTC 112-2628-M-006-006 -MY4 and National
Center for Theoretical Sciences.

\end{document}